\documentclass[a4paper,11pt]{amsart} 
\usepackage[T1]{fontenc}
\usepackage[utf8]{inputenc}
\usepackage{amsmath,amsthm, amssymb}
\usepackage{mathtools}
\usepackage{enumitem}
\usepackage{gensymb}
\usepackage[scr=boondoxo,scrscaled=1.05]{mathalfa}
\usepackage{lmodern}
\usepackage{float}
\usepackage{esint}
\usepackage[usenames,dvipsnames]{pstricks}
\usepackage{enumitem} 
\usepackage{chngcntr}
\usepackage{dsfont}
\usepackage[margin=2.5cm]{geometry}
\usepackage{color}
\usepackage{graphicx}
\usepackage[titletoc,toc,title]{appendix}
\usepackage[pdfusetitle=true,bookmarksdepth=3]{hyperref}
\usepackage{amsrefs}
\usepackage{tikz}
\usepackage{todonotes}
\usepackage{soul}

\colorlet{jvs}{cyan!20}
\colorlet{bvv}{red!20}
\usepackage{IEEEtrantools}

\newtheorem{theorem}{Theorem}[section]
\newtheorem{lemma}[theorem]{Lemma}
\newtheorem{prop}[theorem]{Proposition}

\theoremstyle{definition}
\newtheorem{defn}[theorem]{Definition}
\newtheorem{cor}[theorem]{Corollary}
\newtheorem{rem}[theorem]{Remark}

\DeclareMathOperator*\tr{tr}

\DeclareMathOperator*\dist{dist}

\DeclareMathOperator*\diam{diam}

\DeclareMathOperator*\ess{ess\,sup}
\DeclareMathOperator*\loc{loc}

\DeclarePairedDelimiter{\brk}{(}{)}
\DeclarePairedDelimiter{\abs}{\lvert}{\rvert}
\DeclarePairedDelimiter{\norm}{\lVert}{\rVert}

\DeclarePairedDelimiter{\floor}{\lfloor}{\rfloor}
\DeclarePairedDelimiterX{\intvc}[2]{[}{]}{#1,#2}
\DeclarePairedDelimiterX{\intvl}[2]{(}{]}{#1,#2}
\DeclarePairedDelimiterX{\intvr}[2]{[}{)}{#1,#2}
\DeclarePairedDelimiterX{\intvo}[2]{(}{)}{#1,#2}

\newcommand{\csubset}{\Subset}

\DeclarePairedDelimiterX\set[1]\{\}{%

#1
}

\makeatletter
\newcommand{\widerof}[3][c]{\mathpalette\widerof@{{#1}{#2}{#3}}}
\newcommand{\widerof@}[2]{\widerof@@{#1}#2}
\newcommand{\widerof@@}[4]{%
    \begingroup
    \sbox\z@{$\m@th#1#3$}%
    \sbox\tw@{$\m@th#1#4$}%
    \makebox[\ifdim\wd\z@>\wd\tw@ \wd\z@\else \wd\tw@\fi][#2]{$\m@th#1#3$}%
    \endgroup
}
\newcommand*\diff{\mathop{}\, d}

\newcommand{\mres}{\mathbin{\vrule height 1.6ex depth 0pt width
        0.13ex\vrule height 0.13ex depth 0pt width 1.3ex}}
\makeatother

\numberwithin{equation}{section}
\numberwithin{figure}{section}

\author{Bohdan Bulanyi, Jean Van Schaftingen, Benoît Van Vaerenbergh}

\keywords{Singular harmonic map, $p$-minimizer, $p$-stationary map, Plateau problem, varifold, systole,  topological obstruction, finite homotopy group, stress-energy tensor, extension of Sobolev mappings}

\subjclass[2020]{58E20, 49J45, 49Q10, 35B25, 76A15}

\date{\today}

\begin{document}

\address{
Universit\`{a} di Bologna, Dipartimento di Matematica, Piazza di Porta S. Donato 5, 40126 Bologna, Italy}

\email{bohdan.bulanyi@unibo.it}

\address{
Universit\'e catholique de Louvain, Institut de Recherche en Math\'ematique et Physique, Chemin du Cyclotron 2 bte L7.01.01, 1348 Louvain-la-Neuve, Belgium}

\email{jean.vanschaftingen@uclouvain.be}

\address{
Universit\'e catholique de Louvain, Institut de Recherche en Math\'ematique et Physique, Chemin du Cyclotron 2 bte L7.01.01, 1348 Louvain-la-Neuve, Belgium}

\email{benoit.vanvaerenbergh@uclouvain.be}

    \title[Limiting behavior of minimizing $p$-harmonic maps] 
    {Limiting behavior of minimizing $p$-harmonic maps in 3d as  $p$ goes to $2$ with finite fundamental group}
   
      \begin{abstract}
      We study the limiting behavior of minimizing $p$-harmonic maps from a bounded Lipschitz  domain $\Omega \subset \mathbb{R}^{3}$ to a compact connected Riemannian manifold without boundary and with finite fundamental group as $p \nearrow 2$. We prove that there exists a closed set $S_{*}$ of finite length such that minimizing $p$-harmonic maps converge to a locally minimizing harmonic map in $\Omega \setminus S_{*}$. We prove that locally inside $\Omega$ the singular set $S_{*}$ is a finite union of straight line segments, and it minimizes the mass in the appropriate class of admissible chains. Furthermore, we establish local and global estimates for the limiting singular harmonic map. Under additional assumptions, we prove that globally in $\overline{\Omega}$ the set $S_{*}$ is a finite union of straight line segments, and it minimizes the mass in the appropriate class of admissible chains, which is defined by a given boundary datum and $\Omega$.
  \end{abstract}
  
  \thanks{B. Bulanyi and J. Van Schaftingen were supported by the Projet de Recherche T.0229.21 ``Singular Harmonic Maps and Asymptotics of Ginzburg-Landau Relaxations'' of the Fonds de la Recherche Scientifique-FNRS. B. Van Vaerenbergh was supported by a FRIA fellowship. The authors would like to thank the anonymous referee for carefully reading  and commenting on the paper.}
  
   \maketitle

    \tableofcontents
    \section{Introduction}   
Given a bounded domain \(\Omega \subset \mathbb{R}^3\) with a smooth boundary \(\partial \Omega\) and a compact connected smooth Riemannian manifold \(\mathcal{N}\) without boundary and of finite dimension, which will be assumed, thanks to Nash's embedding theorem, to be isometrically embedded into the Euclidean space $\mathbb{R}^{\nu}$ for some $\nu\in \mathbb{N}\setminus\{0\}$,
we are interested in the problem of finding a minimizing harmonic map with boundary datum \(g : \partial \Omega \to \mathcal{N}\), that is, find a minimizer of the Dirichlet integral 
\begin{equation}
\label{eq_WaiCho7AVur1Iicho7ohthoo}
 \int_{\Omega} \frac{\abs{D u}^2}{2}\diff x
\end{equation}
among all Sobolev mappings \(u \in W^{1, 2} (\Omega, \mathcal{N})\) that have \(g\) as a trace on \(\partial \Omega\).

If there exists some  \(v \in W^{1, 2} (\Omega, \mathcal{N})\) having \(g\) as a trace on \(\partial \Omega\), standard methods in the calculus of variations show the existence of a minimizer \(u \in W^{1, 2} (\Omega, \mathcal{N})\) of the energy \eqref{eq_WaiCho7AVur1Iicho7ohthoo}.
Following the classical works of Schoen and Uhlenbeck, the map \(u\) is continuous outside a discrete set of points in \(\Omega\) \cite{SchoUhl}; moreover, the solution \(u\) inherits further regularity of the boundary datum: for example, if \(g \in C^{2, \alpha} (\partial \Omega, \mathcal{N})\), then \(u\) has the same regularity in a neighborhood of the boundary  \cite{SU}.

When the target manifold \(\mathcal{N}\) is simply connected, or equivalently has a trivial fundamental group (\emph{i.e.}, \(\pi_1 (\mathcal{N}) \simeq \set{0}\)), the space of admissible traces is the fractional Sobolev space \(W^{1/2, 2} (\partial \Omega, \mathcal{N})\)  \cite{H-L}, similarly to the trace theory of linear spaces \cite{Gagliardo_1957}, giving then rise to a nice theory of harmonic extension of mappings into \(\mathcal{N}\).

When the manifold \(\mathcal{N}\) is not simply connected, the admissible traces form a strict subset of \(W^{1/2, 2} (\partial \Omega, \mathcal{N})\) due to several obstruction phenomena.
First, one has global topological obstructions created by the combination of the topologies of \(\partial \Omega\) and \(\mathcal{N}\) \citelist{\cite{Bethuel-Demengel}*{Theorem~5}}: if \(\Omega \simeq B^{2}_{1} \times \mathbb{S}^1\) is a solid torus, taking \(\gamma \in C^\infty(\mathbb{S}^1, \mathcal{N})\), which is not homotopic to a constant, the function \(g \in C^\infty (\partial \Omega, \mathcal{N})\) defined by \(g (x', x'') = \gamma(x')\) for \((x', x'') \in \mathbb{S}^1 \times \mathbb{S}^1 \simeq \partial \Omega\) has no extension in \(W^{1, 2} (\Omega, \mathcal{N})\).
Next, one has local topological obstructions with boundary data that behave like \(g (x) = \gamma(x'/\abs{x'})\) \citelist{\cite{H-L}*{Section~6.3}\cite{Bethuel-Demengel}*{Theorem~4}}. 
Finally, there are some analytical obstructions coming from the infiniteness of \(\pi_1 (\mathcal{N})\) \cite{Bethuel_2014}, and since the integrability exponent \(2\) is an integer, from the nontriviality of \(\pi_1 (\mathcal{N})\) \cite{Mironescu_VanSchaftingen_2021}*{Theorem~1.10}; these obstructions arise as limits of juxtaposition of boundary data, which have arbitrarily large lower bounds on the extension energy. 

These obstructions to trace extension disappear if one relaxes the problem and considers the problem of minimizing, for \(1 < p < 2\), the \(p\)-energy 
\begin{equation*}
 \int_{\Omega} \frac{\abs{D u}^p}{p} \diff x.
\end{equation*}
Indeed, given \(g \in W^{1/2, 2} (\partial \Omega, \mathcal{N}) \subset W^{1 - 1/p, p} (\partial \Omega, \mathcal{N})\), there exists \(v \in W^{1, p} (\Omega, \mathcal{N})\) such that \(\operatorname{tr}_{\partial \Omega} (v) = g\).
One way then to find a good replacement for solutions to the problem of minimizing harmonic maps is to define \(u_p\) to be a minimizing \(p\)-harmonic map under the boundary condition \(u = g\) on \(\partial \Omega\) and  study the asymptotics of the solutions \(u_p\) of this \(p\)-harmonic relaxation as \(p \nearrow 2\).

Such an approach has been developed successfully for planar domains \(\Omega \subset \mathbb{R}^2\) when \(\mathcal{N}\) is a circle \cite{Hardt_Lin_1995} or a general compact Riemannian manifold  \cite{VanSchaftingen_VanVaerenbergh}, where, up to a subsequence, \(u_p\) converges  as \(p \nearrow 2\) then to a harmonic map outside a finite set of singular points whose positions and topological charges minimize a suitable renormalized energy \citelist{\cite{Bethuel_Brezis_Helein_1994}\cite{Renormalmap}}.
Similar results hold for \(p\)-energy minimizing maps from \(\Omega \subset \mathbb{R}^m\) to \(\mathbb{S}^{m - 1}\) as \(p \nearrow m\) \cite{Hardt_Lin_Wang_1997}.

In higher dimensions, this problem has been studied 
for maps taking their values in $\mathbb{S}^{n}$:
Hardt and Lin \cite[Section~3]{Lin_2011} have proved that the $p$-energy measures associated to the energy densities of \(p\)-energy minimizing maps from an $m$-dimensional domain to $\mathbb{S}^{n}$ as $p \nearrow n < m$ converge, up to subsequences, to a limit measure whose support coincides with some area minimizing current of dimension $m - n$; Stern \cite{Stern_2020} has obtained the convergence of the $p$-energy measures associated to the energy densities of stationary \(p\)-harmonic maps from a closed oriented Riemannian manifold to the circle \(\mathbb{S}^1\) as \(p \nearrow 2\) to the weight measure of a stationary varifold (we recall that the most-studied class of stationary $p$-harmonic maps (for arbitrary target manifolds) are the $p$-energy minimizers).

In three dimensions, these asymptotics have been analyzed for maps from a closed oriented Riemannian manifold into the circle \(\mathbb{S}^1\) \cite{Lin_2011,Stern_2020}.

Our first result is the following.

\begin{theorem}
\label{Interiorbehavior1th}
If \(\Omega \subset \mathbb{R}^{3}\) is a bounded Lipschitz domain, \(\pi_1 (\mathcal{N})\) is finite, \(g \in W^{1/2, 2} (\partial \Omega, \mathcal{N})\), $(p_{n})_{n\in \mathbb{N}}\subset [1,2)$, $p_{n}\nearrow 2$ as $ n\to +\infty$ and if for each $n\in \mathbb{N}$, $u_{n}\in W^{1,p_{n}}(\Omega, \mathcal{N})$ is a minimizing \(p_n\)-harmonic map with boundary datum $g$, then 
\begin{equation*}
   \limsup_{n \to +\infty} (2-p_{n}) \int_{\Omega} \frac{\abs{D u_n}^{p_n}}{p_n}\diff x \leq C |g|_{W^{1/2, 2}(\partial \Omega, \mathbb{R}^{\nu})}^2,
\end{equation*}
where $C=C(\partial \Omega, \mathcal{N})>0$. Furthermore, there exists a closed set $S_{*} \subset \overline{\Omega}$, at most a countable union of segments \((L_j)_{j \in J}\) lying in $\overline{\Omega}$, mappings \((\gamma_j)_{j \in J}\) in \(C (\mathbb{S}^1, \mathcal{N})\), a map \(u_* : \Omega \to \mathcal{N}\) such that  \(S_{*} \cap \Omega= \bigcup_{j \in J} L_j \cap \Omega\) and, up to a subsequence (not relabeled), the following assertions hold.
\begin{enumerate}[label=(\roman*)]
 \item \(u_n \to u_*\) almost everywhere in \(\Omega\).
 \item  For each $q \in [1,2)$,  $u_{*} \in W^{1,q}_{\loc}(\Omega, \mathcal{N})$, $u_{n} \to u_{*}$  in $W^{1,q}_{\loc}(\Omega, \mathbb{R}^{\nu})$ and 		
 \begin{equation*}
\limsup_{n \to +\infty}\int_{\Omega}(\dist(x,\partial \Omega)|Du_{n}(x)|)^{q}\diff x \leq C |\Omega|+\frac{C}{2 - q}\limsup_{n\to+\infty}(2-p_{n})\int_{\Omega} \frac{|Du_{n}|^{p_{n}}}{p_{n}}\diff x,
\end{equation*}
where $C=C(\diam(\Omega), \mathcal{N})>0$. 
\item Assume that there exists $r_{0}=r_{0}(\Omega)>0$ such that if $x \in \mathbb{R}^{3} \setminus \Omega$ and $\dist(x, \partial \Omega)\leq r_{0}$, there exists a unique point $P(x) \in \partial \Omega$ such that $\dist(x, \partial \Omega)=|x-P(x)|$. If $g$ is a $C^{1}$ map outside a finite  set of points around each of which $g$ is homotopically nontrivial and whose Euclidean norm of the gradient behaves like the distance to the power of \(-1\) close to these points, then for each $q \in [1,2)$, $u_{*} \in W^{1,q}(\Omega, \mathcal{N})$, $u_{n} \to u_{*}$ in $W^{1,q}(\Omega, \mathbb{R}^{\nu})$ and
\begin{equation*}
\limsup_{n \to +\infty}\int_{\Omega}|Du_{n}|^{q}\diff x \leq C+\frac{C}{2 - q}\limsup_{n\to+\infty}(2-p_{n})\int_{\Omega} \frac{|Du_{n}|^{p_{n}}}{p_{n}}\diff x,
\end{equation*}
where $C=C(g, \Omega, \mathcal{N})>0$. 

 \item Any point of \(\Omega\) has a neighborhood that only intersects finitely many \(L_j\).
 \item \(u_* \in W^{1, 2}_{\mathrm{loc}} (\Omega \setminus S_{*}, \mathcal{N})\) is a locally minimizing harmonic map in \(\Omega  \setminus S_{*}\): if \(K \csubset \Omega \setminus S_{*}\) is an open set and \(v \in W^{1, 2} (K, \mathbb{R}^{\nu})\) satisfies $v-u_{*} \in W^{1,2}_{0}(K, \mathbb{R}^{\nu})$, then 
 \begin{equation*}
  \int_{K} \frac{\abs{D u_{*}}^2}{2}\diff x \leq  \int_{K} \frac{\abs{D v}^2}{2} \diff x. 
 \end{equation*}
 Furthermore, there exists at most a countable and locally finite subset $S_{0}$ of $\Omega \backslash S_{*}$ such that $u_{*}\in~C^{\infty}(\Omega\setminus (S_{*} \cup S_{0}), \mathcal{N})$. 
  \item The varifold $V_{*}=(\textup{Id}, A_{*})_{\#}\left(\sum_{j \in J}\mathcal{E}^{\textup{sg}}_{2}(\gamma_{j})\mathcal{H}^{1}\mres (L_{j} \cap \Omega)\right)$ supported by \(\bigcup_{j \in J} L_j \cap \Omega\) and with a multiplicity \(\mathcal{E}^{\mathrm{sg}}_{2} (\gamma_j)\) is a stationary  varifold which is the limit of the rescaled stress-energy tensors and 
 \begin{equation*}
 V_{*}(\Omega \times \mathrm{G(3,1)})= \sum_{j \in J} \mathcal{E}^{\mathrm{sg}}_{2} (\gamma_j) \mathcal{H}^1 (L_j \cap \Omega) 
  \le  \liminf_{n \to + \infty} (2 - p_n) \int_{\Omega} \frac{|D u_n|^{p_n}}{p_n} \diff x,
 \end{equation*}
where $\textup{Id}:\mathbb{R}^{3}\to \mathbb{R}^{3}$ is the identity map, $A_{*}(x) \in \mathbb{R}^{3}\otimes \mathbb{R}^{3}$ represents the orthogonal projection onto the approximate tangent plane $T_{x}S_{*}$ to $S_{*}$ at $x$ and $\mathrm{G}(3,1)$ is the Grassmann manifold.
\item $S_{*}$ minimizes the \textit{mass} among all admissible chains locally inside $\Omega$, namely, locally inside $\Omega$ it solves a homotopical Plateau problem in codimension 2. 
\item If $\partial \Omega$ is homeomorphic to $\mathbb{S}^{2}$ and $g \in W^{1,2}(\partial \Omega, \mathcal{N})$, then $S_{*}=\emptyset$ and $u_{*} \in C^{\infty}(\Omega \setminus S_{0}, \mathcal{N})$ is a minimizing harmonic extension of $g$ to $\Omega$, where $S_{0} \subset \Omega$ is a discrete set.
\end{enumerate}
\end{theorem}
The singular energy \(\mathcal{E}^{\mathrm{sg}}_{2} (\gamma)\), appearing in the definition of the varifold $V_{*}$ in Theorem~\ref{Interiorbehavior1th}, is defined for \(\gamma \in C (\mathbb{S}^1, \mathcal{N})\) as the infimum of $\sum_{i = 1}^m \int_{\mathbb{S}^1} \frac{|f_i'|^2}{2}\diff \mathcal{H}^{1}$ for \(f_1, \dotsc, f_m \in C^1 (\mathbb{S}^1, \mathcal{N})\) such that \(\gamma\) has a continuous extension to a domain consisting of the 2-disk \({B}^2_{1}\) from which \(m\) small 2-disks have been removed whose restriction on the boundary of the \(i\)-th disk is a rescaled version of the mapping \(f_i\).

The finite assumption on the fundamental group in Theorem~\ref{Interiorbehavior1th} makes it inapplicable to the case of the unit circle \(\mathcal{N} = \mathbb{S}^1\); the framework of our results is thus completely disjoint from Stern's  \cite{Stern_2020}.
Among the situations covered by our theorem, we point out the case where \(\mathcal{N}\) is the real projective plane \(\mathbb{RP}^2\) motivated by liquid crystals, and the case where \(\mathcal{N}\) is the space of attitudes of the cube \(SO(3)/O\), where \(O\) is the octahedral group, which appears as framefields in meshing problems. We point out that the finiteness assumption on the fundamental group of the compact manifold $\mathcal N$ is equivalent to a compactness assumption on the universal covering \(\widetilde{\mathcal{N}}\) of the manifold $\mathcal N$.

When \(\mathcal{N}\) is the real projective plane \(\mathbb{RP}^2\), Theorem~\ref{Interiorbehavior1th} is a counterpart of Canevari's result for Landau-de Gennes models of nematic liquid crystals \cite{Canevari_2017}.

The strategy of our proof of Theorem~\ref{Interiorbehavior1th} is inspired by Canevari's \(\eta\)-compactness approach for Landau-de Gennes  \cite{Canevari_2017}. 
Whereas Canevari works with a perturbation involving no derivatives and with a domain enlarged to a full linear Sobolev space, in our setting we have a perturbation involving the first-order derivatives defined on a nonlinear space with the same manifold constraint but enlarged by lower integrability requirements. 
This results in different constructions to deal with the different structures and different analytical arguments because the space in which  the minimization is performed depends on the relaxation parameter \(p\). 

We also emphasize that in our results we obtain the uniform interior \(W^{1, q}\)  estimates  and the
local minimization property outside the singular set $S_{*}$ for $u_{*}$, in contrast to a minimization result restricted to balls by Canevari, and we prove that locally inside $\Omega$, $S_{*}$ solves a homotopical Plateau problem in codimension 2.

At a more technical point, we could identify the stationary varifold structure of the limit of the stress-energy tensors directly through Arroyo-Rabasa, De Philippis, Hirsch and Rindler’s work on rectifiability of measures satisfying some differential constraints \cite{ArroyoRabasa_DePhilippis_Hirsch_Rindler_2019}.

It is also worth noting that the smaller the fractional seminorm of the boundary datum $g$, the closer to $\partial \Omega$ the singular set $S_{*}$ of Theorem~\ref{Interiorbehavior1th} is located (see Proposition~\ref{prop empty S_*}).

Under additional assumptions we get global properties of the limit.

\begin{theorem}
\label{theorem_intro_global}
With the assumptions and notations of Theorem~\ref{Interiorbehavior1th}
if we assume moreover that \(\Omega\) is of class $C^{2}$, \(\overline{\Omega}\) is strongly convex at every point of $\partial \Omega$, 
\(g\) is a \(C^{1}\) map outside a finite  set of points \( \{a_1, \dotsc, a_m\} \subset \partial \Omega\) around each of which $g$ is homotopically nontrivial and whose Euclidean norm of the gradient behaves like the distance to the power of \(-1\) close to these points and if for each $n \in \mathbb{N}$, $Du_{n}$ is continuous on $\partial \Omega \setminus \{a_{1}, \dotsc, a_{m}\}$, then the following assertions hold.
\begin{enumerate}[label=(\roman*)]
 \item \(S_{*} \cap \partial \Omega = \set{a_1, \dotsc, a_m}\) and $S_{*}$ lies in the convex hull of \(\set{a_1, \dotsc, a_m}\).
 \item The set \(J\) is finite and $S_{*}=\bigcup_{j \in J} L_{j}$.
 \item The quantity
 \begin{equation*} 
 	\sum_{j \in J} \mathcal{E}^{\mathrm{sg}}_{2} (\gamma_j) \mathcal{H}^1 (L_j) \quad\text{minimizes}\quad \sum_{i \in I} \mathcal{E}^{\mathrm{sg}}_{2} (f_i) \mathcal{H}^1 (K_i)
 \end{equation*}
 among all configurations of closed straight line segments \((K_i)_{i \in I}\) and topological charges \((f_i)_{i \in I} \subset C(\mathbb S^1,\mathcal N)\) satisfying the following conditions: \(I\) is finite; each $K_{i}$ has a positive length; the interiors of the segments (i.e., the segments without their endpoints) are pairwise disjoint; \(K = \bigcup_{i \in I} K_{i}\) is a compact subset of \(\overline{\Omega}\) and \(K \cap \partial \Omega\) is a finite set of points; $K$ has no endpoints in \(\Omega\), namely, if \(K_{i}\) is a segment having an endpoint \(x \in \Omega\), then there exists another segment \(K_{j}\) emanating from \(x\);  there exists a map \(v \in C(\Omega \setminus (K \cup F), \mathcal{N})\), with \(F\) being a locally finite subset of \(\Omega \setminus K\), such that \(\operatorname{tr}_{\partial \Omega}(v)=g\) and \(v\) restricted to small circles transversal to \(K_i\) is, up to orientation, homotopic to \(f_i\). Namely, $S_{*}$ solves a homotopical Plateau problem in codimension 2.
\end{enumerate}
\end{theorem}

The additional ingredient in the proof of Theorem~\ref{theorem_intro_global} is the exploitation of the Pohozaev-type identity combined with White’s maximum principle for minimizing varifolds \cite{White_2010}.

Although our results and proofs are written in the framework of domains of the Euclidean space \(\mathbb{R}^3\), most of the analysis can be transferred to compact 3-dimensional Riemannian manifolds.
The adaptation of Theorem~\ref{Interiorbehavior1th} should be completely straightforward, provided line segments are replaced by geodesics.
Transferring  Theorem~\ref{theorem_intro_global} to manifolds would require a suitable use of the strong convexity of the boundary of submanifolds.

Finally, we point out that if $\overline{\Omega}$ is strongly convex at every point of $\partial \Omega$ and $\pi_{1}(\mathcal{N}) \simeq \mathbb{Z}/2\mathbb{Z}$, then the singular set $S_{*}$ is a minimal connection (see Definition~\ref{defnminconnection} and Proposition~\ref{example 3}). Thus, in this case, we exclude the appearance of closed loops (i.e., homeomorphic images of $\mathbb{S}^{1}$) in $S_{*}$. In fact, the singular sets of minimizing $p_{n}$-harmonic maps $u_{n}$ could have disclination loops \cite[Section~2.4]{Liquidcrystals_1986}, however, in the limit, as $n\to+\infty$, each of these loops, which is contractible in $\Omega$, should become unstable, shrink and disappear. We also suspect that the singular sets of the corresponding minimizing $p_{n}$-harmonic maps converge, up to a subsequence, to the singular set $S_{*}$ in the Hausdorff distance, but this question is beyond the scope of our paper. 
    
  \section{Preliminaries}
    
    \subsection{Conventions and Notation} 
    \emph{Conventions:} in this paper, we say that a value is positive if it is strictly greater than zero, and a value is nonnegative if it is greater than or equal to zero. Euclidean spaces are endowed with the Euclidean inner product $\langle \cdot, \cdot \rangle$ and the induced norm $|\cdot|$. A set will be called a domain whenever it is open and connected. Throughout this paper, $\Omega$, unless otherwise stated, will denote a bounded (locally) Lipschitz domain in $\mathbb{R}^{3}$, and $\mathcal{N}$ will denote a finite-dimensional compact connected smooth Riemannian manifold without boundary and with finite fundamental group, which will be assumed, thanks to Nash's embedding theorem, to be isometrically embedded into the Euclidean space $\mathbb{R}^{\nu}$ for some $\nu\in \mathbb{N}\setminus\{0\}$. The Hausdorff measures,  which we shall use, coincide in terms of normalization with the appropriate outer Lebesgue measures. We shall say that a property holds $\mu$-\textit{almost everywhere} (or simply \textit{almost everywhere}) for a measure $\mu$ (abbreviated $\mu$-a.e.) if it holds outside a set of zero $\mu$-measure. \\

    \noindent\emph{Notation:} if $A,B,M$ are real matrices, we denote the inner product (also known as Frobenius or Hilbert-Schmidt) between $A$ and $B$ by $A:B=\operatorname{tr}(A^{\mathrm{T}}B)$ and the Euclidean norm of $M$ by $|M|=\sqrt{M:M}$. 
    We denote by $B^{l}_{r}(x)$, $\smash{\overline{B}}^{l}_{r}(x)$, and $\partial{B}^{l}_{r}(x)$, respectively, the open ball in $\mathbb{R}^{l}$, the closed ball in $\mathbb{R}^{l}$, and its boundary the $(l-1)$-sphere with center $x$ and radius $r$, where $l \in \mathbb{N}\setminus\{0\}$. 
    If the center is at the origin \(0\), we write $B^{l}_{r}$, $\overline{B}^{l}_{r}$ and $\partial B^{l}_{r}$ the corresponding balls and the $(l-1)$-sphere. 
    We shall sometimes write $\mathbb{S}^{m-1}$ instead of $\partial B^{m}_{1}$. 
    We denote by $\dist(x,A)$ and $|A|$, respectively, the Euclidean distance from $x\in \mathbb{R}^{l}$ to $A\subset \mathbb{R}^{l}$, and the $l$-dimensional Lebesgue measure of $A$. We shall denote by $\mathcal{H}^{l}(A)$ the $l$-dimensional Hausdorff measure of~$A$. If $U\subset \mathbb{R}^{l}$ is Lebesgue measurable, then for $ p \in [1,+\infty)\cup \{+\infty\}$, $L^{p}(U)$ will denote the space consisting of all real measurable functions on $U$ that are $p^{\mathrm{th}}$-power integrable on $U$ if $p \in [1,+\infty)$ and are essentially bounded if $p=+\infty$; $L^{p}(U; \mathbb{R}^{k})$ is the respective space of functions with values in $\mathbb{R}^{k}$.  
    If $p \in [1,+\infty)$ and $E\subset \mathbb{R}^{m}$ is a Borel-measurable set of Hausdorff dimension $l \leq m$, then $L^{p}(E, \mathcal{H}^{l})$ is the  space of real measurable functions on $E$ that are $p^{\mathrm{th}}$-power integrable on $E$ with respect to the $\mathcal{H}^{l}$-measure; $L^{p}(E, \mathbb{R}^{k}, \mathcal{H}^{l})$ is the respective space of functions with values in $\mathbb{R}^{k}$. 
    By $L^{1}_{\loc}(U)$ we denote the space of functions $u$ such that $u\in L^{1}(V)$ for all $V\Subset U$; if $E\subset \mathbb{R}^{m}$ is a Borel-measurable set of Hausdorff dimension $l \leq m$, $L^{1}_{\loc}(E, \mathbb{R}^{k}, \mathcal{H}^{l})$ denote the space of measurable functions on $E$ with values in $\mathbb{R}^{k}$ such that $u \in L^{1}(F,\mathbb{R}^{k}, \mathcal{H}^{l})$ for all $F\Subset E$. The norm on $L^{p}(U)$ ($L^{p}(U;\mathbb{R}^{k})$) is denoted by $\|\cdot\|_{L^{p}(U)}$ ($\|\cdot\|_{L^{p}(U;\mathbb{R}^{k})}$) or   $\|\cdot\|_{p}$ when appropriate. 
    We use the standard notation for Sobolev spaces. 
    If $L:\mathbb{R}^{m}\to \mathbb{R}^{l}$ is a linear map, $L[h] \in \mathbb{R}^{l}$ will denote the value of $L$ at the vector $h \in \mathbb{R}^{m}$. 
    For each $x \in \mathcal{N}$ and for each $r>0$, we shall denote by $B^{\mathcal{N}}_{r}(x)$ the geodesic ball in $\mathcal{N}$ with center $x$ and radius $r$, namely $B^{\mathcal{N}}_{r}(x)=\{y \in \mathcal{N}: d_{\mathcal{N}}(y,x)\leq r\}$, where $d_{\mathcal{N}}:\mathcal{N}\times \mathcal{N}\to [0,+\infty)$ is the geodesic distance in $\mathcal{N}$. 
    If $U\subset \mathbb{R}^{l}$ is a bounded Lipschitz domain, then $\tr_{\partial U}$ will denote the trace operator (the reader may consult, for instance, \cite[Section~4.3]{Evans}). We denote by $|\cdot|_{W^{s,p}(E, \mathbb{R}^{k})}$ (or $|\cdot|_{W^{s,p}(E)}$ when appropriate) the canonical seminorm on $W^{s,p}(E, \mathbb{R}^{k})$: if $E$ is a Borel-measurable set of Hausdorff dimension $l$, then $|u|^{p}_{W^{s,p}(E, \mathbb{R}^{k})}=\int_{E}\int_{E} \frac{|u(x)-u(y)|^{p}}{|x-y|^{l+sp}}\diff \mathcal{H}^{l}(x)\diff \mathcal{H}^{l}(y)$. By $\# X$ we shall denote the cardinality of $X$.

    \subsection{Embedding and nearest point retraction}
    For each $\lambda \in (0,+\infty)$, we define the neighborhood
    \[
    \mathcal{N}_{\lambda}=\{x \in \mathbb{R}^{\nu}: \dist(x, \mathcal{N})<\lambda\}.
    \]
    The next lemma describes the nearest point retraction of a neighborhood of $\mathcal{N}$ onto $\mathcal{N}$.
    \begin{lemma}\label{lemma npr}
        There exists $\lambda_{\mathcal{N}}>0$ such that the nearest point retraction $\Pi_{\mathcal{N}}:\mathcal{N}_{\lambda_{\mathcal{N}}}\to \mathcal{N}$ characterized by 
        \[
        |x-\Pi_{\mathcal{N}}(x)|=\dist(x, \mathcal{N})
        \]
        is a well-defined and smooth map. Moreover, if the mappings $P^{\top}_{\mathcal{N}}:\mathcal{N}\to \mathrm{Lin}(\mathbb{R}^{\nu}, \mathbb{R}^{\nu})$ and $P^{\perp}_{\mathcal{N}}:\mathcal{N}\to \mathrm{Lin}(\mathbb{R}^{\nu}, \mathbb{R}^{\nu})$ are defined for each $x \in \mathcal{N}$ by setting $P^{\top}_{\mathcal{N}}(x)$ and $P^{\perp}_{\mathcal{N}}(x)$ as the orthogonal projections on $T_{x}\mathcal{N}$ and $(T_{x}\mathcal{N})^{\perp}$, identified as linear subspaces of $\mathbb{R}^{\nu}$, then for each $x \in \mathcal{N}_{\lambda_{\mathcal{N}}}$ and $v \in \mathbb{R}^{\nu}$,
        \begin{equation*}
        \left(1-\frac{\dist(x,\mathcal{N})}{\lambda_{\mathcal{N}}}\right)|D\Pi_{\mathcal{N}}(x)[v]|^{2}\leq |P^{\top}_{\mathcal{N}}(\Pi_{\mathcal{N}}(x))[v]|^{2}\leq C|D\Pi_{\mathcal{N}}(x)[v]|^{2},
        \end{equation*}
        where $C=C(\mathcal{N}, \nu)>0$. Furthermore, for each $x \in \mathcal{N}_{\lambda_{\mathcal{N}}}\setminus \mathcal{N}$,
        \begin{equation*}
        |D\dist(x,\mathcal{N})| = 1.
        \end{equation*} 
    \end{lemma}
\begin{proof}
For a proof, we refer to \cite[Lemma~2.1]{Monteil-Rodiac-VanSchaftingen}.
\end{proof}
\subsection{Sobolev mappings into manifolds}
Let $\mathcal{X}$ and $\mathcal{Y}$ be compact Riemannian manifolds. Assume that $\mathcal{X}$ and $\mathcal{Y}$ are isometrically embedded into Euclidean spaces, which is always possible by Nash's isometric embedding theorem \cite{Nash_1956}. Notice that these embeddings are bilipschitz. In fact, if $\mathcal{M}$ is a compact Riemannian manifold, which is isometrically embedded into a Euclidean space, then, clearly, the Euclidean distance between two points of $\mathcal{M}$ is less than or equal to the geodesic distance between these points, which is induced by the Riemannian metric on $\mathcal{M}$. On the other hand, there exists a constant $L=L(\mathcal{M})>0$ such that the geodesic distance on $\mathcal{M}$ is less than or equal to the Euclidean distance on $\mathcal{M}$ times $L$ (the idea behind the proof is that each compact Riemannian manifold $\mathcal{M}$ has a positive normal injectivity radius, thus, one can construct a smooth nearest point retraction from a tubular neighborhood of $\mathcal{M}$ onto $\mathcal{M}$ (see Lemma~\ref{lemma npr}) and appropriately use the fact that for each open cover of $\mathcal{M}$ there is a finite subcover). Thus, identifying $\mathcal{X}$ and $\mathcal{Y}$ to their embedding's images, for each $s \in (0,1)$ and $p \in [1,+\infty)$, we can define the fractional Sobolev space $W^{s,p}(\mathcal{X}, \mathcal{Y})$ as
\[
W^{s,p}(\mathcal{X}, \mathcal{Y})=\{u \in L^{p}(\mathcal{X}, \mathbb{R}^{k}, \mathcal{H}^{l}): u(x) \in \mathcal{Y} \,\ \text{for a.e. $x \in \mathcal{X}$ and}\,\ |u|_{W^{s,p}(\mathcal{X}, \mathbb{R}^{k})}<+\infty\},
\]
where $\mathbb{R}^{k}$ is the Euclidean space into which $\mathcal{Y}$ is embedded and $l \in \mathbb{N}\backslash \{0\}$ is the Hausdorff dimension of $\mathcal{X}$ (the reader may also consult \cite[Section~3.1]{Caselli2024}). This definition does not depend on the choice of the isometric Euclidean embedding of $\mathcal{Y}$. We shall consider in $W^{1,p}(\mathcal{X},\mathbb{R}^{k})$ its strongly closed subset
\[
W^{1,p}(\mathcal{X},\mathcal{Y})=\{u \in W^{1,p}(\mathcal{X}, \mathbb{R}^{k}): u(x)\in \mathcal{Y} \; \text{for a.e.}\; x \in \mathcal{X}\}
\] 
This definition does not depend on the choice of the isometric Euclidean embedding of $\mathcal{Y}$ (see, for instance, \cite[Proposition~2.1]{Bousquet2017} and \cite[Section~3.1]{HELEIN2008417}; for intrinsic definitions of Sobolev spaces for Riemannian manifolds, see \cite[Section~3.2]{HELEIN2008417}). 

For convenience, we recall the definition of weak differentiability of mappings defined on bilipschitz images of open sets. 
\begin{defn}\label{def weak differ}
    Let $U\subset \mathbb{R}^{l}$ be open and  $\Phi: E \to U $ be bilipschitz. A map $u \in L^{1}_{\loc}(E, \mathbb{R}^{k}, \mathcal{H}^{l})$ is said to be weakly differentiable if the map $u\circ \Phi^{-1}$ is weakly differentiable.
\end{defn}
Since the weak differentiability is preserved under bilipschitz maps (see \cite[Theorem~2.2.2]{Ziemer}), the above definition is independent of the bilipschitz parametrization of $E$. 
Taking into account that $E$ is bilipschitz homeomorphic to an open set $U\subset \mathbb{R}^{l}$, through a bilipschitz homeomorphism $\Phi:E\to U$, $E$ admits the approximate tangent plane $T_{x}E$ for $\mathcal{H}^{l}$-a.e.\  $x \in E$ (for the definition of the approximate tangent plane, see, for example, \cite[Section~3.2.16]{Federer} and \cite[Definition~4.3]{Federer_1959}). Thus, we can define the notion of the tangential derivative of $\Phi$ for $\mathcal{H}^{l}$-a.e.\  $x\in E$, as the linear map from the tangent plane $T_{x}E$ to $\mathbb{R}^{l}$. If $u \in L^{1}_{\loc}(E, \mathbb{R}^{k}, \mathcal{H}^{l})$ is weakly differentiable, the weak (tangential) derivative of $u$ at $\mathcal{H}^{l}$-a.e.\  $x \in E$ is defined by $D_{\top}u(x)=D(u\circ \Phi^{-1})(\Phi(x))\circ D_{\top}\Phi(x)$, where $D_{\top}\Phi(x)$ is the tangential derivative of $\Phi$ at $x$ (which is defined for $\mathcal{H}^{l}$-a.e.\  $x \in E$).

We shall denote by $W^{1,1}_{\loc}(E, \mathbb{R}^{k})$ the space of all weakly differentiable mappings on $E$ with values in $\mathbb{R}^{k}$. If $p \in [1, +\infty)$ and $E$ is a bilipschitz image of an open set in $\mathbb{R}^{l}$, then we define 
\[
W^{1,p}(E, \mathbb{R}^{k})=\{u \in W^{1,1}_{\loc}(E, \mathbb{R}^{k}): |u|, |D_{\top}u| \in L^{p}(E, \mathcal{H}^{l})\}. 
\]
Next, we define Sobolev spaces of mappings on bilipschitz images of open sets into a compact Riemannian manifold. 

\begin{defn}\label{def bil Sob man}
    Let $p \in [1, +\infty)$ and $E$ be a bilipschitz image of an open set in $\mathbb{R}^{l}$ with $l \in \mathbb{N}\setminus\{0\}$. Let $\mathcal{Y}$ be a compact Riemannian manifold isometrically embedded into $\mathbb{R}^{k}$. Then we define the Sobolev space $W^{1,p}(E, \mathcal{Y})$ by
    \[
    W^{1,p}(E, \mathcal{Y})=\{u \in W^{1,p}(E, \mathbb{R}^{k}):  u(x) \in \mathcal{Y}\; \text{for a.e.}\; x \in E\}. 
    \]
\end{defn}

\subsection{Topological resolution of the boundary datum}
Since $\mathcal{N}$ is a connected Riemannian manifold, it is path-connected and its fundamental group, namely $\pi_{1}(\mathcal{N}, x_{0})$, where $x_{0}\in \mathcal{N}$, is, up to isomorphism, independent of the choice of a basepoint $x_{0}$ (see \cite[Proposition~1.5]{AT}) and denoted by $\pi_{1}(\mathcal{N})$. 
Moreover, there is a one-to-one correspondence between the set of conjugacy classes in $\pi_{1}(\mathcal{N})$ and the set $[\mathbb{S}^{1},\mathcal{N}]$ of free homotopy classes of maps $\mathbb{S}^{1}\to \mathcal{N}$, that is, homotopy class without any condition on some basepoints (see Exercise~6 in \cite[Section~1.1, p. 38]{AT}). 
In general, there is no group structure on $[\mathbb{S}^{1}, \mathcal{N}]$ (there exist groups $G$ with the functor $L:G\to LG$ sending $G$ to its set of conjugacy classes, which cannot be lifted to a group-valued functor; for instance, one can find an inclusion of groups $H\to G$ implying an inclusion of conjugacy classes $LH\to LG$ such that the order of $LH$ does not divide the order of $LG$). After all, hereinafter, when we talk about ``homotopy'', we are talking about \emph{free} homotopy (with no conditions on base points). We shall denote by $\#[\mathbb{S}^{1}, \mathcal{N}]$ the number of equivalence classes of $[\mathbb{S}^{1}, \mathcal{N}]$.

Next, for convenience, we recall the definition of a bounded Lipschitz domain. 
\begin{defn}\label{def Lip domain}
    A bounded domain $U\subset \mathbb{R}^{l}$, $l\geq 2$ and its boundary $\partial U$ are said to be (locally) Lipschitz if there exist a radius $r_{\partial U}$ and a constant $\delta>0$ such that for every point $x \in \partial U$ and every radius $s \in (0,r_{\partial U})$, up to a rotation of coordinates, it holds
    \[
    U\cap B^{l}_{s}(x) =\{(y^{\prime}, y_{l}) \in \mathbb{R}^{l-1}\times \mathbb{R}: y_{l}>\varphi(y^{\prime})\}\cap B^{l}_{s}(x)
    \]
    for some Lipschitz function $\varphi:\mathbb{R}^{l-1}\to \mathbb{R}$ satisfying $\|D \varphi\|_{L^{\infty}(\mathbb{R}^{l-1})}\leq \delta$.
\end{defn}

Given $k \in \mathbb{N}\backslash \{0\}$ and a bounded Lipschitz domain $U\subset \mathbb{R}^{2}$, we denote by $\mathrm{Conf}_{k}U$ the configuration space of $k$ ordered points of $U$, namely
\begin{equation}\label{configuration}
\mathrm{Conf}_{k}U=\{(a_{1},\dotsc,a_{k}) \in U^{k}: a_{i}\neq a_{j}\,\ \text{if}\,\ i\neq j\}.
\end{equation}
Given $(a_{1},\dotsc,a_{k})\in \mathrm{Conf}_{k}U$, we define the quantity
\begin{equation}\label{varrhobar}
\bar{\varrho}(a_{1},\dotsc,a_{k})=\min\Biggl\{\Biggl\{\frac{|a_{i}-a_{j}|}{2}: i,j \in \{1,\dotsc,k\} \; \text{and}\; i\neq j\Biggr\} \cup \Biggl\{\dist(a_{i}, \partial U): i\in \{1,\dotsc,k\}\Biggr\}\Biggr\},
\end{equation}
so that if $\varrho \in (0, \bar{\varrho}(a_{1},\dotsc,a_{k}))$, we have $\smash{\overline{B}}^{2}_{\varrho}(a_{i})\cap \smash{\overline{B}}^{2}_{\varrho}(a_{j})=\emptyset$ for each $i,j\in\{1,\dotsc,k\}$ such that $i\neq j$ and $\smash{\overline{B}}^{2}_{\varrho}(a_{i})\subset U$ for each $i \in \{1,\dotsc,k\}$, and hence the set $U\setminus \bigcup_{i=1}^{k}\smash{\overline{B}}^{2}_{\varrho}(a_{i})$ is connected and has a Lipschitz boundary $\partial(U\setminus \bigcup_{i=1}^{k}\smash{\overline{B}}^{2}_{\varrho}(a_{i}))=\partial U \cup \bigcup_{i=1}^{k}\partial B^{2}_{\varrho}(a_{i})$. 
\begin{defn}\label{top resol}
    We say that $(\gamma_{1},\dotsc, \gamma_{k})\in C(\mathbb{S}^{1}, \mathcal{N})^{k}$ is a \emph{topological resolution} of $g \in C(\partial U, \mathcal{N})$ whenever there exist $(a_{1},\dotsc, a_{k})\in \mathrm{Conf}_{k}U$, $\varrho \in (0,\bar{\varrho}(a_{1},\dotsc,a_{k}))$ and $u\in C(U\setminus \bigcup_{i=1}^{k} \smash{\overline{B}}^{2}_{\varrho}(a_{i}), \mathcal{N})$ such that $u|_{\partial U}=g$ and for each $i \in \{1,\dotsc,k\}$, $u(a_{i}+\varrho\cdot)|_{\mathbb{S}^{1}}=\gamma_{i}$.
\end{defn}

The definition of topological resolution is independent of the order of the curves in the $k$-tuple $(\gamma_{1},\dotsc,\gamma_{k})$. 
Furthermore, if $(b_{1},\dotsc,b_{k})\in \mathrm{Conf}_{k}U$ and $r\in (0, \bar{\varrho}(b_{1},\dotsc,b_{k}))$, then there exists a homeomorphism between $U\setminus \bigcup_{i=1}^{k} \smash{\overline{B}}^{2}_{\varrho}(a_{i})$ and $U\setminus \bigcup_{i=1}^{k}\smash{\overline{B}}^{2}_{r}(b_{i})$, which shows that the statement of the previous definition is independent of the choice of the points and of the radius. 
Similarly, if for each $i \in \{1,\dotsc,k\}$ the map $\widetilde{\gamma}_{i}$ is freely homotopic to $\gamma_{i}$ and if $\widetilde{g}$ is freely homotopic to $g$, then, by definition of homotopy and the fact that an annulus is homeomorphic to a finite cylinder, we have that $(\widetilde{\gamma}_{1},\dotsc,\widetilde{\gamma}_{k})$ is also a topological resolution of $\widetilde{g}$. 
Thus, the property of being a topological resolution is invariant under homotopies.

Definition~\ref{top resol} can be extended to the case when $g:\partial U\to \mathcal{N}$ is a map of vanishing mean oscillation, namely $g\in \mathrm{VMO}(\partial U, \mathcal{N})$ and $(\gamma_{1},\dotsc,\gamma_{k})\in \mathrm{VMO}(\mathbb{S}^{1}, \mathcal{N})^{k}$ (see \cite{BrezisNirenberg, BN1}). Indeed, if $\Gamma\simeq \mathbb{S}^{1}$ is a closed curve, path-connected components of the space $\mathrm{VMO}(\Gamma,\mathcal{N})$ are known to be the closure in $\mathrm{VMO}(\Gamma, \mathcal{N})$ of path-connected components of $C(\Gamma, \mathcal{N})$ (see \cite[Lemma~A.23]{BrezisNirenberg}). Since $W^{1/2,2}(\partial U,\mathcal{N})\subset \mathrm{VMO}(\partial U, \mathcal{N})$, Definition~\ref{top resol}  extends also  to the case when $g\in W^{1/2,2}(\partial U, \mathcal{N})$.
\begin{defn}\label{topresol vmo}
    We say that $(\gamma_{1},\dotsc,\gamma_{k})\in \mathrm{VMO}(\mathbb{S}^{1}, \mathcal{N})^{k}$ is a \emph{topological resolution} of a map $g \in \mathrm{VMO}(\partial U, \mathcal{N})$ whenever $(\gamma_{1},\dotsc,\gamma_{k})$ is homotopic in $\mathrm{VMO}(\mathbb{S}^{1}, \mathcal{N})^{k}$ to a \emph{topological resolution} $(\widetilde{\gamma}_{1},\dotsc,\widetilde{\gamma}_{k})$ of $\widetilde{g}\in C(\partial U, \mathcal{N})$ which is homotopic to $g$ in $\mathrm{VMO}(\partial U, \mathcal{N})$.
\end{defn}
Since Definition~\ref{top resol} is invariant under homotopies, and continuous maps are homotopic in VMO if and only if they are homotopic through a continuous homotopy, Definition~\ref{topresol vmo} generalizes Definition~\ref{top resol}.
\subsection{Singular \texorpdfstring{$p$}{p}-energy}
First, we define the minimal length in the homotopy class of a map $\gamma \in \text{VMO}(\mathbb{S}^{1}, \mathcal{N})$ by
\begin{equation}\label{minlengthhom}
\lambda(\gamma)=\inf\left\{\int_{\mathbb{S}^{1}}|\widetilde{\gamma}^{\prime}|\diff \mathcal{H}^{1} : \widetilde{\gamma} \in C^{1}(\mathbb{S}^{1}, \mathcal{N})\,\ \text{and}\,\ \gamma \,\ \text{are homotopic in} \,\ \text{VMO}(\mathbb{S}^{1}, \mathcal{N})\right\},
\end{equation}
where $|\cdot|$ is the norm induced by the Riemannian metric $g_{\mathcal{N}}$ of the manifold $\mathcal{N}$ on each fiber of the tangent bundle $T\mathcal{N}$, and since $\mathcal{N}$ is isometrically embedded into $\mathbb{R}^{\nu}$, $|\cdot|$ is the Euclidean norm. 

For each $p \in [1,+\infty)$, using the parametrization by arc length and H\"older's inequality if $p \in (1,+\infty)$, we have
\begin{equation*}
\inf\left\{\int_{\mathbb{S}^{1}}\frac{|\widetilde{\gamma}^{\prime}|^{p}}{p}\diff \mathcal{H}^{1} : \widetilde{\gamma}\in C^{1}(\mathbb{S}^{1}, \mathcal{N})\; \text{and}\; \gamma\; \text{are homotopic in}\; \mathrm{VMO}(\mathbb{S}^{1}, \mathcal{N})\right\}=\frac{\lambda(\gamma)^{p}}{(2\pi)^{p-1}p}.
\end{equation*}
In particular, if $\gamma$ is a minimizing closed geodesic, then 
\begin{equation*}
\lambda(\gamma)=\int_{\mathbb{S}^{1}}|\gamma^{\prime}|\diff \mathcal{H}^{1}=\left((2\pi)^{p-1}p\int_{\mathbb{S}^{1}}\frac{|\gamma^{\prime}|^{p}}{p}\diff \mathcal{H}^{1}\right)^{\frac{1}{p}}=2\pi\|\gamma^{\prime}\|_{\infty},
\end{equation*}
since $|\gamma^{\prime}|$ is constant for a minimizing geodesic. 

Recall that the \emph{systole} of the compact manifold $\mathcal{N}$ is the length of the shortest closed nontrivial geodesic on $\mathcal{N}$, namely,
\begin{equation}\label{systole}
\text{sys}(\mathcal{N})=\inf\{\lambda(\gamma): \gamma \in C^{1}(\mathbb{S}^{1}, \mathcal{N})\,\ \text{is not homotopic to a constant}\}.
\end{equation}
If $[\mathbb{S}^{1}, \mathcal{N}] \simeq \{0\}$, we set $\mathrm{sys}(\mathcal{N})=+\infty$. 
Notice that, if $[\mathbb{S}^{1}, \mathcal{N}]\not \simeq \{0\}$, then for each map $\gamma \in \text{VMO}(\mathbb{S}^{1}, \mathcal{N})$, it holds $\lambda(\gamma) \in \{0\}\cup [\text{sys}(\mathcal{N}), +\infty)$. 

We now define the singular $p$-energy of a map $g \in \mathrm{VMO}(\partial U, \mathcal{N})$, where $U\subset \mathbb{R}^{2}$ is a bounded Lipschitz domain. This notion was previously introduced for $p=2$ in \cite{Renormalmap} and generalized in \cite{VanSchaftingen_VanVaerenbergh} to $p\in [1,+\infty)$.
\begin{defn}\label{def p-singular}
    If $p \in [1,+\infty)$, $U\subset \mathbb{R}^{2}$ is a bounded Lipschitz domain and $g \in \mathrm{VMO}(\partial U, \mathcal{N})$, we define the \emph{singular $p$-energy} of $g$ by
    \begin{equation} \label{eq:pesg}
    \mathcal{E}^{\mathrm{sg}}_{p}(g)=\inf\left\{\sum_{i=1}^{k}\frac{\lambda(\gamma_{i})^{p}}{(2\pi)^{p-1} p}: k \in \mathbb{N}\setminus\{0\}\; \text{and}\; (\gamma_{1},\dotsc,\gamma_{k})\; \text{is a topological resolution of}\; g\right\}.
    \end{equation}
\end{defn}
Observe that $\mathcal{E}^{\mathrm{sg}}_{p}$ is invariant under homotopies and depends continuously on $p$. For each $g \in \mathrm{VMO}(\mathbb{S}^{1}, \mathbb{S}^{1})$, \(\mathcal{E}^{\mathrm{sg}}_{p}(g)= 2 \pi \abs{\deg (g)}^p/p\), where $\pi_{1}(\mathbb{S}^{1}) \simeq \mathbb{Z}$.
If $\mathcal{N}=\mathbb{S}^{2}$, then $\pi_{1}(\mathcal{N})$ is trivial and for each $g \in \mathrm{VMO}(\mathbb{S}^{1}, \mathbb{S}^{2})$, $\mathcal{E}^{\mathrm{sg}}_{p}(g)=0$. 
If $\mathcal{N}=\{n\otimes n -\frac{1}{3}\mathrm{Id}: n \in \mathbb{S}^{2}\}$ and $g(t)=v(t)\otimes v(t)-\frac{1}{3}\mathrm{Id}$, where $v(t)=(\cos(t/2), \sin(t/2), 0) \in \mathbb{S}^{2}$ for each $t \in [0,2\pi]$, then $\mathcal{N}$ is diffeomorphic to $\mathbb{RP}^{2}$, $g \in C^{\infty}(\mathbb{S}^{1}, \mathcal{N})$ and $\mathcal{E}^{\mathrm{sg}}_{p}(g)= 2^{\frac{2-p}{2}}\pi/p$ (see \cite[Lemma~19]{Canevari_2017}). The  reader may consult \cite[Section~9.3]{Renormalmap} for other interesting examples in the case where $p=2$.

\begin{lemma}\label{lem cont singenergy}
    Let $U\subset \mathbb{R}^{2}$ be a bounded Lipschitz domain. Then for each $g\in \mathrm{VMO}(\partial U, \mathcal{N})$,
    	 the function $p \in [1,+\infty) \mapsto \mathcal{E}^{\mathrm{sg}}_{p}(g)$ is locally bounded and locally Lipschitz continuous. In particular, if $\mathcal{N}$ is simply connected, then $\mathcal{E}^{\mathrm{sg}}_{p}(g)=0$ for each $p\in [1,+\infty)$. Furthermore, the infimum in \eqref{eq:pesg} is actually a minimum and the set $\{\mathcal{E}^{\mathrm{sg}}_{p}(\gamma): \gamma \in [\mathbb{S}^{1}, \mathcal{N}]\}$ is finite. 
\end{lemma}
\begin{proof}
    If $\mathcal{N}$ is simply connected, then each $\gamma \in \mathrm{VMO}(\partial U, \mathcal{N})$ is homotopic to a constant map, which yields that $\lambda(\gamma)=0$ and hence $\mathcal{E}^{\mathrm{sg}}_{p}(g)=0$. 
    
    Otherwise, $[\mathbb{S}^{1}, \mathcal{N}]\neq \{0\}$ and then $\mathrm{sys}(\mathcal{N})\in (0,+\infty)$. This comes, since $\mathcal{N}$ is compact, and hence the length spectrum (i.e., the set of nonnegative real numbers $\{\lambda(\gamma): \gamma \in \mathrm{VMO}(\mathbb{S}^{1}, \mathcal{N})\}$) is discrete (for a proof, see \cite[Proposition~3.2]{Monteil-Rodiac-VanSchaftingen}). Then, according to \cite[Lemma~2.10]{VanSchaftingen_VanVaerenbergh}, $p \mapsto \mathcal{E}^{\mathrm{sg}}_{p}(g)$ is locally bounded and locally Lipschitz continuous. The fact that the infimum in \eqref{eq:pesg} is a minimum comes from the discreteness of the length spectrum. Finally, since $\mathcal{E}^{\mathrm{sg}}_{p}$ is invariant under homotopies and $\# [\mathbb{S}^{1}, \mathcal{N}]<+\infty$ by our assumption on $\mathcal{N}$, the set $\{\mathcal{E}^{\textup{sg}}_{p}(\gamma): \gamma \in [\mathbb{S}^{1}, \mathcal{N}]\}$ is finite. This completes our proof of Lemma~\ref{lem cont singenergy}. 
  \end{proof}

It is also worth recalling the notion of a \emph{minimal topological resolution} (when \(p = 2\), see \cite[Definition~2.7]{Renormalmap}).

\begin{defn}\label{mintopresol}
    A topological resolution $(\gamma_{1},\dotsc,\gamma_{k})$ of $g$ is said to be \(p\)-minimal if 
    \[
    \mathcal{E}^{\mathrm{sg}}_{p}(g)=\sum_{i=1}^{k}\frac{\lambda(\gamma_{i})^{p}}{(2\pi)^{p-1} p}
    \]
    and $\lambda(\gamma_{i})>0$ for each $i\in \{1,\dotsc, k\}$.
\end{defn}
\subsection{Renormalized energy of renormalizable maps}
We recall the definition of a renormalizable map (see \cite[Definition~7.1]{Renormalmap}).
\begin{defn}\label{renormalisablemap}
Given an open set $U\subset \mathbb{R}^{2}$, a map $u:U\to \mathcal{N}$ is said to be \emph{renormalizable} whenever there exists a finite set $\mathcal{S}\subset U$ such that for each sufficiently small radius $\varrho>0$, $u\in W^{1,2}(U\setminus\bigcup_{a\in \mathcal{S}}\smash{\overline{B}}^{2}_{\varrho}(a), \mathcal{N})$ and its \emph{renormalized} energy $\mathcal{E}^{\mathrm{ren}}_{2}(u)$ is finite, where
\begin{equation}
\label{eq_te0xaVu2pie1oe1ieziwaira}
\mathcal{E}^{\mathrm{ren}}_{2}(u)=\liminf_{\varrho\to 0+}\int_{U\setminus\bigcup_{a\in \mathcal{S}}\smash{\overline{B}}^{2}_{\varrho}(a)}\frac{|Du|^{2}}{2}\diff x-\sum_{a\in \mathcal{S}}\frac{\lambda(\operatorname{tr}_{\partial B^{2}_{\varrho}(a)}(u))^{2}}{4\pi}\ln\left(\frac{1}{\varrho}\right). 
\end{equation}
We denote by $W^{1,2}_{\mathrm{ren}}(U, \mathcal{N})$ the set of renormalizable maps.
\end{defn}

For each renormalizable map $u:U\to\mathcal{N}$, there exists the \emph{minimal} set $\mathcal{S}_{0}$ for which it holds $u\in W^{1,2}(U\setminus \bigcup_{a \in \mathcal{S}_{0}}\smash{\overline{B}}^{2}_{\varrho}(a), \mathcal{N})$ for each small enough \(\varrho > 0\).
If $a \in U\setminus \mathcal{S}_{0}$, then $\lambda(\operatorname{tr}_{\partial B^{2}_{\varrho}(a)}(u))=0$ for $\varrho>0$ small enough. 
Thus, the $\liminf$ defining $\mathcal{E}^{\mathrm{ren}}_{2}$ in \eqref{eq_te0xaVu2pie1oe1ieziwaira} does not depend on the set $\mathcal{S}$. If $\mathcal{S}_{0}=\emptyset$, then $u \in W^{1,2}(U, \mathcal{N})$ and $\mathcal{E}^{\mathrm{ren}}_{2}(u)=\int_{U}\frac{|Du|^{2}}{2}\diff x$. If $\mathcal{S}_{0}=\{a_{1},\dotsc,a_{k}\}$ with $k \in \mathbb{N}\setminus\{0\}$ and $(a_{1},\dotsc,a_{k}) \in \mathrm{Conf}_{k}(U)$ (see \eqref{configuration}), then, according to \cite[Lemma~2.11]{Renormalmap}, the map
 \[
\varrho\in (0,\bar{\varrho}(a_{1},\dotsc,a_{k})) \mapsto \int_{U\setminus \bigcup_{i=1}^{k}\smash{\overline{B}}^{2}_{\varrho}(a_{i})}\frac{|Du|^{2}}{2}\diff x-\sum_{i=1}^{k}\frac{\lambda(\operatorname{tr}_{\partial B^{2}_{\varrho}(a_{i})}(u))^{2}}{4\pi}\ln\left(\frac{1}{\varrho}\right)
\]
is nonincreasing, where $\bar{\varrho}(a_{1},\dotsc,a_{k}) \in (0,+\infty)$ is defined in \eqref{varrhobar}. The main ingredients of the proof of \cite[Lemma~2.11]{Renormalmap} are the use of the additivity of the integral and the estimate of the  energy on the annulus. The latter comes from the application of the special case of the coarea formula and the comparison of the angular energy taking into account the definition of the minimal length in the homotopy class of a map $\gamma \in \mathrm{VMO}(\mathbb{S}^{1}, \mathcal{N})$ (see \eqref{minlengthhom}).

Let $U \subset \mathbb{R}^{2}$ be a bounded Lipschitz domain, $u \in W^{1,2}_{\mathrm{ren}}(U, \mathcal{N})$ and $S_{0}=\{a_{1},\dotsc,a_{k}\}$ for some $k \in \mathbb{N}\setminus\{0\}$, where $(a_{1},\dotsc,a_{k})\in \mathrm{Conf}_{k} U$. By \cite[Theorem~5.1]{Renormalmap} (see also \cite[Corollary~5.13]{VanSchaftingen_VanVaerenbergh}), $Du \in L^{2,\infty}(U, \mathbb{R}^{\nu}\otimes \mathbb{R}^{2}) \subset L^{p}(U,\mathbb{R}^{\nu}\otimes \mathbb{R}^{2})$ for each $p \in [1,2)$ (see, for instance, \cite[Theorem~5.9]{Castillo_Rafeiro_2016}), where $L^{2,\infty}(U,\mathbb{R}^{\nu}\otimes \mathbb{R}^{2})$ denotes the Marcinkiewicz space or endpoint Lorentz space. In addition, we can define the $p$-renormalized energy of $u$ by
\begin{equation}\label{est ren p}
\mathcal{E}^{\mathrm{ren}}_{p}(u)=
\int_{U}\frac{|Du|^{p}}{p}\diff x - \sum_{i=1}^{k}\frac{\lambda(\operatorname{tr}_{\partial B^{2}_{\varrho}(a_{i})}(u))^{p}}{(2-p)(2\pi)^{p-1}p},
\end{equation}
which is finite and tends to $\mathcal{E}^{\mathrm{ren}}_{2}(u)$ as $p\nearrow 2$.
\begin{prop}\label{prop cont renorm}
    Let $U\subset \mathbb{R}^{2}$ be a bounded Lipschitz domain. If $u\in W^{1,2}_{\mathrm{ren}}(U,\mathcal{N})$, then
    \begin{equation}\label{con ren en}
    \mathcal{E}^{\mathrm{ren}}_{p}(u)\to \mathcal{E}^{\mathrm{ren}}_{2}(u)
    \end{equation}
as $p\nearrow 2$.
\end{prop}
\begin{proof}
    For a proof, see \cite[Proposition~2.17]{VanSchaftingen_VanVaerenbergh}.
\end{proof}

\begin{lemma}\label{lem renenergy ext}
Let $\gamma \in C^{1}(\mathbb{S}^{1}, \mathcal{N})$ be a minimizing geodesic in its (free) homotopy class. Then there exists a renormalizable map $v\in W^{1,2}_{\mathrm{ren}}(B^{2}_{1}, \mathcal{N})$  satisfying the following conditions. 
\begin{enumerate}[label=(\roman*)]
\item $\tr_{\mathbb{S}^{1}}(v)=\gamma$. If $\gamma$ is a constant map, then $v$ is the constant map.
\item
\label{it_ahlaequaih1seo3xu2ooRi3f}
For each $p \in [1,2)$, the following estimate holds
\[
 \int_{B^{2}_{1}}\frac{|Dv|^{p}}{p}\diff x=\mathcal{E}^{\mathrm{ren}}_{p}(v) + \frac{\mathcal{E}^{\mathrm{sg}}_{p}(\gamma)}{2-p}.
\]
\end{enumerate}
\end{lemma}
\begin{proof}
If $\gamma$ is a constant map, then we define $v$ to be the constant map taking the same value in $\mathcal{N}$ as $\gamma$. Clearly, in this case, the assertion \ref{it_ahlaequaih1seo3xu2ooRi3f} holds. 

If $\gamma$ is not homotopic to a constant,  we apply \cite[Proposition~8.3]{Renormalmap}. A key step for the proof of \cite[Proposition~8.3]{Renormalmap} is to use the estimate of \cite[Lemma~2.11]{Renormalmap}, which is based on the comparison of the angular energy. In particular, let $(\gamma_{1},\dotsc,\gamma_{k})$ be a \(2\)-minimal topological resolution of $\gamma$ (see Definition~\ref{mintopresol}) such that for each $j\in \{1,\dotsc,k\}$, $\gamma_{j} \in C^{1}(\mathbb{S}^{1}, \mathcal{N})$ is a minimizing geodesic.  Then, using \cite[Proposition~8.3]{Renormalmap}, together with \eqref{est ren p}, we deduce that there exists $v\in W^{1,2}_{\mathrm{ren}}(B^{2}_{1},\mathcal{N})$ satisfying
\[
\mathcal{E}^{\mathrm{ren}}_{p}(v)
=\int_{B^{2}_{1}}\frac{|Dv|^{p}}{p} \diff x-\sum_{i=1}^{k}\frac{\lambda(\gamma_{i})^{p}}{(2-p)(2\pi)^{p-1}p}
=\int_{B^{2}_{1}}\frac{|Dv|^{p}}{p} \diff x - \frac{\mathcal{E}^{\mathrm{sg}}_{p}(\gamma)}{2-p},
\]
which completes our proof of Lemma~\ref{lem renenergy ext}. 
\end{proof}

\subsection{Several extension results}
We obtain an extension of the boundary datum with a \textit{linear}-type estimate.
    \begin{lemma}\label{lem lifting linear}
        Let $m \in \mathbb{N}\setminus \{0,1\}$, $p \in [1,2]$, $\widetilde{g} \in W^{1,p}(\partial B^{m}_{r}(x_{0}), \widetilde{\mathcal{N}})$ and $g=\pi \circ \widetilde{g}$, 
        where $\pi:\widetilde{\mathcal{N}}\to \mathcal{N}$ is the universal covering of $\mathcal{N}$. Then there exists a map $w \in W^{1,p}(B^{m}_{r}(x_{0}), \mathcal{N})$ such that $\tr_{\partial B^{m}_{r}(x_{0})}(w)=g$ and
        \begin{equation}\label{ineq 2}
        \|D w\|^{p}_{L^{p}(B^{m}_{r}(x_{0}))}\leq C r\|D_{\top} g\|^{p}_{L^{p}(\partial B^{m}_{r}(x_{0}))},
        \end{equation}
        where $C=C(m,\mathcal{N})>0$. Moreover, $w \in W^{1,2}(B^{m}_{r}(x_{0}), \mathcal{N})$.
    \end{lemma}
    
\begin{rem}\label{remliflinear}
If $m\geq 3$, then $w$ can be chosen as the homogeneous extension of $g$. In this case, \eqref{ineq 2} holds with the constant $C=\frac{1}{m-p}\leq 1$ and without assuming that $\pi_{1}(\mathcal{N})$ is finite and $g$ has a lifting. It is also worth mentioning that for each $p \in (1,2)$, the estimate \eqref{ineq 2} holds without assuming that $g$ has a lifting, but the constant $C$ in \eqref{ineq 2} then behaves like $C\sim \frac{1}{2-p}$.
\end{rem}

\begin{proof}[Proof of Lemma \ref{lem lifting linear}]
    Since the manifold $\mathcal{N}$ is compact and its fundamental group $\pi_{1}(\mathcal{N})$ is finite, its covering space $\widetilde{\mathcal{N}}$ is compact. 
    Notice that $g \in W^{1,p}(\partial B^{m}_{r}(x_{0}), \mathcal{N})$, since $\pi:\widetilde{\mathcal{N}}\to \mathcal{N}$ is a local isometry. 
    According to Nash's embedding theorem \cite{Nash_1956}, $\widetilde{\mathcal{N}}$ can be isometrically embedded into $\mathbb{R}^{k}$ for some $k \in \mathbb{N}\setminus\{0\}$. If $p=1$, we define $w(x)\coloneqq g(x_{0}+r(x-x_{0})/|x-x_{0}|)$ for $x \in B^{m}_{r}(x_{0}) \setminus \{x_{0}\}$ (see \cite[Proposition~2.4]{SchoUhl}). Then $w\in W^{1,1}(B^{m}_{r}(x_{0}), \mathcal{N})$ and
    \begin{equation*}
        \int_{B^{m}_{r}(x_{0})}|Dw|\diff x = \frac{r}{m-1} \int_{\partial B^{m}_{r}(x_{0})}|D_{\top}g|\diff \mathcal{H}^{m-1},
    \end{equation*}
which proves \eqref{ineq 2} for $p=1$.
If $ p\in (1,2]$, according to \cite[Section~6.9]{Function}, $\widetilde{g}$ admits the extension $\widetilde{v} \in C^{\infty}(B^{m}_{r}(x_{0}), \mathbb{R}^{k})$ satisfying the estimate
    \begin{equation}\label{est continuous inverse}
    \|D\widetilde{v}\|^{p}_{L^{p}(B^{m}_{r}(x_{0}))}\leq C|\widetilde{g}|^{p}_{W^{1-1/p,p}(\partial B^{m}_{r}(x_{0}))}, 
    \end{equation}
   where $C=C(m)>0$. The construction of the extension $\widetilde{v}$ is standard:   since $\partial B^{m}_{r}(x_{0})$ is smooth, using a partition of unity, the problem reduces to constructing for a given trace $u\in W^{1-1/p,p}(Q)$, where $Q=\{(x_{1}, x_{2}) \in \mathbb{R}^{2}: |x_{i}|<1,\; i=1,2\}$, a real-valued function $U \in W^{1,p}(P)$, where $P=\{(x_{1}, x_{2}, x_{3}) \in \mathbb{R}^{3}:  0<x_{3}<1,\; |x_{i}|<1-x_{3}\}$, by the convolution of $u$ with a kernel $\varphi$ satisfying \cite[(2.5.1.1)]{Function}. Namely, first, extend $u$ on $\mathbb{R}^{2}\setminus Q$ by $u=0$ and next define 
\[   
 U(x)=\frac{1}{x_{3}^{2}}\int_{|x^{\prime}-y^{\prime}|<x_{3}}\varphi\left(\frac{x^{\prime}-y^{\prime}}{x_{3}}\right) u(y^{\prime})\diff y^{\prime},
 \]
 where $x=(x^{\prime}, x^{3}) \in \mathbb{R}^{3}$, $x_{3}>0$. Then, proceeding as in \cite[Theorem~6.2]{H-L}, namely, using Sard's Theorem and choosing an appropriate locally Lipschitz retraction \cite[Lemma~6.1]{H-L}, we define $\widetilde{w} \in W^{1,2}(B^{m}_{r}(x_{0}),\mathcal{N})$ (generally speaking, the statement of \cite[Theorem~6.2]{H-L} provides us with a $W^{1,p}$ extension, however, using the fact that $\pi_{1}(\widetilde{\mathcal{N}})$ is trivial, one observes that the locally Lipschitz retraction coming from \cite[Lemma~6.1]{H-L} belongs to $W^{1,q}_{\loc}(\mathbb{R}^{k}, \mathcal{N})$ for each $q \in [1,3)$, which, since $\widetilde{g}\in W^{1/2,2}(\partial B^{m}_{r}(x_{0}), \mathcal{N})$, implies that the extension coming from \cite[Theorem~6.2]{H-L} is actually a $W^{1,2}$ extension in this particular case) such that $\tr_{\partial B^{m}_{r}(x_{0})}(\widetilde{w})=\widetilde{g}$ and, in view of \eqref{est continuous inverse},
    \begin{equation}\label{ineq 0.14}
    \|D\widetilde{w}\|^{p}_{L^{p}(B^{m}_{r}(x_{0}))} \leq C^{\prime}|\widetilde{g}|^{p}_{W^{1-1/p,p}(\partial B^{m}_{r}(x_{0}))}\leq C^{\prime \prime}r\|D_{\top} \widetilde{g}\|^{p}_{L^{p}(\partial B^{m}_{r}(x_{0}))},
    \end{equation}
    where $C^{\prime}=C^{\prime}(m,\widetilde{\mathcal{N}})>0$, $C^{\prime \prime}=C^{\prime\prime}(m, \widetilde{\mathcal{N}})>0$ and the last estimate comes from \eqref{ineq linear}. 
    Defining $w \coloneqq \pi\circ \widetilde{w}$ and using \eqref{ineq 0.14}, together with the fact that $\pi$ is a local isometry, we deduce \eqref{ineq 2} for $p \in (1,2)$ (since the universal covering is unique up to a diffeomorphism, we can actually assume that $C^{\prime\prime}$ depends only on $m$ and $\mathcal{N}$). This completes our proof of Lemma~\ref{lem lifting linear}.
    \end{proof}
    The next lemma provides us with a \textit{nonlinear}-type estimate on the extension of the boundary datum, in contrast to a \textit{linear}-type estimate provided by Lemma~\ref{lem lifting linear}. 
    \begin{lemma}\label{lem lifting nonlinear}
        Let $m \in \mathbb{N} \setminus \{0,1\}$, $p_{0} \in (1,2)$, $p \in [p_{0}, 2)$, $\widetilde{g} \in W^{1,p}(\partial B^{m}_{r}(x_{0}), \widetilde{\mathcal{N}})$, $\pi:\widetilde{\mathcal{N}}\to \mathcal{N}$ be the universal covering of $\mathcal{N}$ and $g=\pi \circ \widetilde{g}$.
          Then there exists a map $w \in W^{1,p}(B^{m}_{r}(x_{0}), \mathcal{N})$ such that $\tr_{\partial B^{m}_{r}(x_{0})}(w)=g$ and   
        \begin{equation}\label{ineq 1}
        \|D w\|^{p}_{L^{p}(B^{m}_{r}(x_{0}))}\leq C r^{\frac{m-1}{p}}\|D_{\top} g\|^{p-1}_{L^{p}(\partial B^{m}_{r}(x_{0}))},
        \end{equation}
        where $C=C(p_{0}, m, \mathcal{N})>0$. 
    \end{lemma}
\begin{rem}\label{remlifnonlinear}
The estimate \eqref{ineq 1} is based on the fact that $\pi_{1}(\mathcal{N})$ is finite, since in this case $\widetilde{\mathcal{N}}$ is compact (due to the compactness of $\mathcal{N}$, the manifold $\widetilde{\mathcal{N}}$ is compact if and only if $\pi_{1}(\mathcal{N})$ is finite). More precisely, to obtain \eqref{ineq 1}, we have used that the diameter of the manifold $\widetilde{\mathcal{N}}$ (which is identified with its Euclidean embedding) is finite. We have used the constant $p_0$ to avoid the dependence of the constant $C$ in \eqref{ineq 1} on $p$ when $p$ is close to $1$ (see \eqref{main int estimate}). Notice also that if $p$ is close to 2, the estimate \eqref{ineq 1} is valid without assuming that $g$ has a lifting, but the constant $C$ in \eqref{ineq 1} then behaves like $C\sim \frac{1}{2-p}$. 
\end{rem}
\begin{proof}[Proof of Lemma~\ref{lem lifting nonlinear}]
    Let $\widetilde{w} \in W^{1,2}(B^{m}_{r}(x_{0}),\widetilde{\mathcal{N}})$ be the map as in the proof of Lemma~\ref{lem lifting linear}. Using \eqref{est continuous inverse} and \eqref{ineq nonlinear}, we deduce the estimate
    \begin{equation}\label{ineq 0.13}
    \|D\widetilde{w}\|^{p}_{L^{p}(B^{m}_{r}(x_{0}))}\leq Cr^{\frac{m-1}{p}}\|D_{\top}\widetilde{g}\|^{p-1}_{L^{p}(\partial B^{m}_{r}(x_{0}))},
    \end{equation}
    where $C=C(p_{0},m,\mathcal{N})>0$. Define $w=\pi\circ \widetilde{w}$. Since $\pi$ is a local isometry, \eqref{ineq 0.13} implies \eqref{ineq 1}, which concludes our proof of Lemma~\ref{lem lifting nonlinear}. 
    \end{proof}
    
For convenience, we recall the next  proposition.
\begin{lemma}\label{lem rest on T}
    Let $m \in \mathbb{N} \setminus \{0,1\}$, $p \in [1,+\infty)$, $u \in W^{1,p}(B^{m}_{r}(x_{0}), \mathcal{Y})$ and $v \in W^{1,p}(\partial B^{m}_{r}(x_{0}), \mathcal{Y})$, where $\mathcal{Y}$ is either a Euclidean space or a compact Riemannian manifold which is isometrically embedded into $\mathbb{R}^{k}$ for some $k \in \mathbb{N}\setminus\{0\}$. Then the following assertions hold. 
    \begin{enumerate}[label=(\roman*)]
        \item 
        \label{it_aeko0woub7OhLeethohneeth}
        For $\mathcal{H}^{1}$-a.e.\  $\varrho \in (0,r)$, $\operatorname{tr}_{\partial B^{m}_{\varrho}(x_{0})}(u)=u|_{\partial B^{m}_{\varrho}(x_{0})} \in W^{1,p}(\partial B^{m}_{\varrho}(x_{0}), \mathcal{Y})$. 
        \item 
        \label{it_Ohtaap1joLiedei6aiGeipho}
        Assume that $E\subset \partial B^{m}_{r}(x_{0})$ is bilipschitz homeomorphic to an open bounded set in $\mathbb{R}^{l}$ with $l \leq m-1$. Then for $\mathcal{H}^{m(m-1)/2}$-a.e.\  $\omega \in SO(m)$, $\operatorname{tr}_{\omega(E)}(v)=v|_{\omega(E)} \in W^{1,p}(\omega(E), \mathcal{Y})$. 
    \end{enumerate}
\end{lemma}
\begin{proof}
    The proof follows using the appropriate coarea formula (see \cite[Theorem~3.12]{Evans} for the assertion \ref{it_aeko0woub7OhLeethohneeth} and \cite[Theorem~3.2.48]{Federer} for the assertion \ref{it_Ohtaap1joLiedei6aiGeipho}) and the fact that each map $u \in W^{1,p}(B^{m}_{r}(x_{0}),\mathbb{R}^{k})$ and $v \in W^{1,p}(\partial B^{m}_{r}(x_{0}), \mathbb{R}^{k})$ can be approximated by smooth maps in $C^{\infty}(B^{m}_{r}(x_{0}),\mathbb{R}^{k})$ and $C^{\infty}(\partial B^{m}_{r}(x_{0}), \mathbb{R}^{k})$, respectively. 
\end{proof}
\begin{lemma}\label{lem 2.17}
    Let  $p \in [1,2)$ and $g \in W^{1,p}(\partial B^{2}_{r}(x_{0}), \mathcal{N})$. Then there exists $u \in W^{1,p}(B^{2}_{r}(x_{0}),\mathcal{N})$ such that $\tr_{\partial B^{2}_{r}(x_{0})}(u)=g$ and
    \begin{equation}\label{est on circle}
    \int_{B^{2}_{r}(x_{0})}\frac{|Du|^{p}}{p}\diff x \leq Cr\int_{\partial B^{2}_{r}(x_{0})}\frac{|D_{\top}g|^{p}}{p}\diff \mathcal{H}^{1} + \frac{\mathcal{E}^{\mathrm{sg}}_{p}(g)r^{2-p}}{2-p},
    \end{equation}
where $C=C(\mathcal{N})>0$.
\end{lemma}
\begin{proof}
    Up to scaling and translation, we can assume that $r=1$ and $x_{0}=0$. First, if $g$ were homotopic to a constant map, then we would have $\mathcal{E}^{\mathrm{sg}}_{p}(g)=0$ and, according to Lemma~\ref{lem lifting linear}, \eqref{est on circle}  would hold. 
    
    Thus, we can assume that $g$ is not homotopic to a constant. 
    Let $\gamma \in C^{1}(\mathbb{S}^{1}, \mathcal{N})$ be a minimizing geodesic in the (free) homotopy class of $g$. 
    Hereinafter in this proof, $C$ will denote a positive constant that can only depend on $\mathcal{N}$ and can be different from line to line. 
    Let $v \in W^{1,2}_{\mathrm{ren}}(B^{2}_{1},\mathcal{N})$ be the map of Lemma~\ref{lem renenergy ext} such that $\tr_{\mathbb{S}^{1}}(v)=\gamma$. Define $u(x)=v(2x)$ for each $x \in B^{2}_{1/2}$. 
    Using Lemma~\ref{lem renenergy ext}~\ref{it_ahlaequaih1seo3xu2ooRi3f} and the fact that $g$ is homotopic to $\gamma$,  we get
    \begin{equation}\label{est 2.23}
    \int_{B^{2}_{1/2}}\frac{|Du|^{p}}{p}\diff x =2^{p-2} \int_{B^{2}_{1}}\frac{|Dv|^{p}}{p}\diff x \leq \mathcal{E}^{\mathrm{ren}}_{p}(v)+\frac{\mathcal{E}^{\mathrm{sg}}_{p}(g)}{2-p}\leq C +\frac{\mathcal{E}^{\mathrm{sg}}_{p}(g)}{2-p},
    \end{equation}
    where the last estimate comes, since the value $\mathcal{E}^{\mathrm{ren}}_{p}(v)$ depends only on $p$ and $\gamma$ and tends to $\mathcal{E}^{\mathrm{ren}}_{2}(v)$ as $p \nearrow 2$ (see Proposition~\ref{prop cont renorm}), and hence it is bounded by the constant depending only on $\mathcal{N}$, since there are finitely many free homotopy classes and we can fix a geodesic in each one.
    
    Finally, we define $u$ on the annulus $B^{2}_{1}\setminus B^{2}_{1/2}$. Set $f(x)\coloneqq \gamma(2x)$ for each $x \in \partial B^{2}_{1/2}$. In view of the Morrey embedding theorem, we can assume that $g \in C(\mathbb{S}^{1}, \mathcal{N})$. Then $g$ is continuously homotopic to $f$. Restricting a continuous homotopy between the mappings $g$ and $f$ to the segment $[\sigma/2, \sigma] \subset \smash{\overline{B}}^{2}_{1} \backslash B^{2}_{1/2}$ for some $\sigma \in \mathbb{S}^{1}$, we obtain a map $h \in C([\sigma/2, \sigma], \mathcal{N})$ such that $h(\sigma/2)=f(\sigma/2)$ and $h(\sigma)=g(\sigma)$. Next, observe that $B^{2}_{1}\setminus \smash{\overline{B}}^{2}_{1/2}$ is bilipschitz homeomorphic to $(0,1)^{2}/\sim,$ where $\sim$ is the equivalence relation identifying the sides $[(0,0),(0,1)]$ and $[(1,0), (1,1)]$. Thus, we obtain the map $w_{0}\in C(\partial (0,1)^{2},\mathcal{N})$ taking the values of $g$ on $[(0,0),(1,0)]$, of $f$ on $[(0,1), (1,1)]$ and of $h$ on $[(0,0), (0,1)]\cup [(1,0), (1,1)]$. Then $w_{0}$ is continuously homotopic to a constant, and the classical theory of liftings says that there exists $\widetilde{w}_{0} \in C(\partial (0,1)^{2}, \widetilde{\mathcal{N}})$ such that $w_{0}=\pi\circ \widetilde{w}_{0}$, where $\pi: \widetilde{\mathcal{N}}\to \mathcal{N}$ is the universal covering map. Let $\widetilde{c}\in C^{1}([0,1],\widetilde{\mathcal{N}})$ be a minimizing geodesic joining $\widetilde{w}_{0}((0,0))$ and $\widetilde{w}_{0}((0,1))$. Since $\pi_{1}(\mathcal{N})$ is finite and $\mathcal{N}$ is compact, $\widetilde{\mathcal{N}}$ is compact and hence there exists $L=L(\mathcal{N})>0$ such that $\max\{\int_{0}^{1}|\widetilde{c}^{\prime}|^{p}\diff t: p\in [1,2]\}\leq L$ (actually, $L$ depends on $\widetilde{\mathcal{N}}$, but the universal covering is unique up to a diffeomorphism, and we can assume that $L$ depends only on $\mathcal{N}$). Now we define the map $\widetilde{w}:\partial (0,1)^{2}\to \widetilde{\mathcal{N}}$ taking the values of $\widetilde{w}_{0}$ on $[(0,0), (1,0)]\cup [(0,1), (1,1)]$ and  of $\widetilde{c}$ on $[(0,0), (0,1)]\cup [(1,0), (1,1)]$. Define $w=\pi\circ \widetilde{w}$. Then $w$ takes the values of $g$ on $[(0,0),(1,0)]$, of $\pi\circ \widetilde{c}$ on $[(0,0), (0,1)]\cup [(1,0), (1,1)]$ and the values of $f$ on $[(0,1), (1,1)]$. 
       According to Lemma~\ref{lem lifting linear} with $B^{2}_{1}$ replaced by $(0,1)^{2}$ ($(0,1)^{2}$ is bilipschitz homeomorphic to $B^{2}_{1}$ (see \cite{bilipschitz})), there exists $W \in W^{1,p}((0,1)^{2}, \mathcal{N})$ such that 
    \begin{align*}
    \int_{(0,1)^{2}}\frac{|DW|^{p}}{p}\diff x &\leq C \int_{\partial (0,1)^{2}}\frac{|D_{\top} w|^{p}}{p}\diff \mathcal{H}^{1}\\ &\leq C \left(\int_{\mathbb{S}^{1}}\frac{|D_{\top}g|^{p}}{p}\diff \mathcal{H}^{1}+\int_{\mathbb{S}^{1}}\frac{|D_{\top}\gamma|^{p}}{p}\diff\mathcal{H}^{1}+\int_{0}^{1}\frac{|(\pi\circ \widetilde{c})^{\prime}|^{p}}{p}\diff t\right)\\
    &\leq C\int_{\mathbb{S}^{1}}\frac{|D_{\top}g|^{p}}{p}\diff \mathcal{H}^{1} + C, 
    \end{align*}
    where we have used that $\pi$ is a local isometry and that $\int_{0}^{1}|\widetilde{c}^{\prime}|^{p}\diff t \leq C$. Passing to the quotient $(0,1)^{2}\to B^{2}_{1}\setminus \smash{\overline{B}}^{2}_{1/2}$, we obtain the map $\varphi \in W^{1,p}(B^{2}_{1}\setminus \smash{\overline{B}}^{2}_{1/2}, \mathcal{N})$ such that $\tr_{\mathbb{S}^{1}}(\varphi)=g$, $\tr_{\partial B^{2}_{1/2}}(\varphi)=f$ and 
    \begin{equation}\label{est 2.24}
    \int_{B^{2}_{1}\setminus \smash{\overline{B}}^{2}_{1/2}}\frac{|D\varphi|^{p}}{p}\diff x \leq C\int_{\mathbb{S}^{1}}\frac{|D_{\top}g|^{p}}{p}\diff \mathcal{H}^{1} + C.
    \end{equation}
    Next, using \eqref{systole} and if necessary H\"older's inequality, we have
    \[
    \mathrm{sys}(\mathcal{N})\leq \int_{\mathbb{S}^{1}}|D_{\top}g|\diff \mathcal{H}^{1} \leq (2\pi)^{1-\frac{1}{p}}\left(\int_{\mathbb{S}^{1}}|D_{\top}g|^{p}\diff \mathcal{H}^{1} \right)^{\frac{1}{p}}=(2\pi)^{1-\frac{1}{p}}\|D_{\top}g\|_{L^{p}(\mathbb{S}^{1})}.
    \]
    This, together with the facts that $p \in [1,2)$ and $\mathrm{sys}(\mathcal{N}) \in (0,+\infty)$ (since $g$ is not nullhomotopic and $\mathcal{N}$ is compact), implies that
    \begin{equation}\label{est non hom}
    \int_{\mathbb{S}^{1}}\frac{|D_{\top}g|^{p}}{p}\diff \mathcal{H}^{1} \geq \frac{1}{C}.
    \end{equation}
    Define $u(x)= \varphi(x)$ for each $x \in B^{2}_{1}\setminus \smash{\overline{B}}^{2}_{1/2}$. Then $u \in W^{1,p}(B^{2}_{1},\mathcal{N})$, $\tr_{\mathbb{S}^{1}}(u)=g$ and combining \eqref{est 2.23}, \eqref{est 2.24} and \eqref{est non hom}, we deduce the following
    \[
    \int_{B^{2}_{1}}\frac{|Du|^{p}}{p}\diff x \leq C\int_{\mathbb{S}^{1}}\frac{|D_{\top}g|^{p}}{p}\diff \mathcal{H}^{1}+ \frac{\mathcal{E}^{\mathrm{sg}}_{p}(g)}{2-p},
    \]
    which completes our proof of Lemma~\ref{lem 2.17}.
\end{proof}
To lighten the notation, for every $x_{0}\in \mathbb{R}^{3}$, $r>0$ and $L>0$, we define the cylinder and its lateral surface, namely
\begin{align}\label{def cyl} 
\Lambda_{r,L}(x_{0})&\coloneqq B^{2}_{r}(x_{0})\times (-L,L)&
&\text{and}&
\Gamma_{r,L}(x_{0})=\partial B^{2}_{r}(x_{0})\times (-L,L).
\end{align}

\begin{rem}\label{homotopy class} The next construction is very reminiscent the one proposed by Rubinstein and Sternberg (see \cite[Section 2]{rubinstein1996homotopy}) in the case when the target manifold is the circle $\mathbb{S}^{1}$, which was extended by Brezis, Li, Mironescu and Nirenberg (see \cite[Section~1]{brezis1999degree}) to the case where the targed manifold is compact, connected and oriented.
    Let $\Gamma=\Gamma_{r,L}(x_{0})$ for some $x_{0}\in \mathbb{R}^{3}$, $r,L>0$. 
    Then, for each $g \in W^{1,2}(\Gamma, \mathcal{N})$, we can define the so-called ``homotopy class'' of $g$. 
    Indeed, according to \cite[Proposition,~p.~267]{SU}, there exists $(g_{n})_{n \in \mathbb{N}}\subset C^{\infty}(\Gamma,\mathcal{N})$ such that $g_{n}\to g$ strongly in $W^{1,2}(\Gamma,\mathbb{R}^{\nu})$. 
    Then, using the coarea formula and the Sobolev embedding, we observe that for $\mathcal{H}^{1}$-a.e.\  $z \in (-L,L)$, $\tr_{\partial B^{2}_{r}(x_{0})\times \{z\}}(g) = g|_{\partial B^{2}_{r}(x_{0})\times\{z\}}$ is continuous, belongs to $W^{1,2}(\partial B^{2}_{r}(x_{0})\times \{z\}, \mathcal{N})$ and 
    \[g_{n}|_{\partial B^{2}_{r}(x_{0})\times \{z\}}\to g|_{\partial B^{2}_{r}(x_{0})\times \{z\}}\] uniformly. This implies that, for $\mathcal{H}^{1}$-a.e.\  $z_{1},z_{2}\in (-L,L)$, $g|_{\partial B^{2}_{r}(x_{0})\times \{z_{1}\}}$ is continuously homotopic to $g|_{\partial B^{2}_{r}(x_{0})\times \{z_{2}\}}$, and we can define the ``homotopy class'' of $g$, namely $[g]_{\Gamma}$, by setting $[g]_{\Gamma}\coloneqq [g|_{\partial B^{2}_{r}(x_{0})\times \{z_{1}\}}]$. To simplify our notation, hereinafter, we shall always write $[g]_{\Gamma}$ instead of $[g]_{\Gamma_{r,L}(x_{0})}$. 
\end{rem}
\begin{lemma}\label{lem 2.19}
    Let $p\in [1,2)$ and $g \in W^{1,2}(\Gamma_{r,L}(x_{0}), \mathcal{N})$. Then there exists $u \in W^{1,p}(\Lambda_{r,L}(x_{0}), \mathcal{N})$ such that $\tr_{\Gamma_{r,L}(x_{0})}(u)=g$ and the following assertions hold. 
    \begin{enumerate}[label=(\roman*)]
    \item
    \label{it_pha8ieh4Umaejao2vaephei8}
    For some $C=C(\mathcal{N})>0$, the following estimate holds
    \begin{equation*}
    \int_{\Lambda_{r,L}(x_{0})}\frac{|Du|^{p}}{p}\diff x \leq \frac{2L\mathcal{E}^{\mathrm{sg}}_{p}([g]_{\Gamma})r^{2-p}}{2-p}+ C\left(\frac{L^{p}}{r^{p-1}}+r\right)\int_{\Gamma_{r,L}(x_{0})}\frac{|D_{\top}g|^{p}}{p}\diff \mathcal{H}^{2}.
    \end{equation*}
    \item
    \label{it_cang8wooGhodoh3Eegoom7ee}
     For each $z \in \{-L,L\}$, $\tr_{B^{2}_{r}(x_{0})\times \{z\}}(u)\in W^{1,p}(B^{2}_{r}(x_{0})\times \{z\},\mathcal{N})$ and there exists a constant $C=C(\mathcal{N})>0$ such that     
     \begin{equation*}
    \int_{B^{2}_{r}(x_{0})\times \{z\}}\frac{|D\operatorname{tr}_{B^{2}_{r}(x_{0})\times \{z\}}(u)|^{p}}{p}\diff \mathcal{H}^{2}\leq \frac{\mathcal{E}^{\mathrm{sg}}_{p}([g]_{\Gamma})r^{2-p}}{2-p}+C\left(\left(\frac{L}{r}\right)^{p-1}+\frac{r}{L}\right)\int_{\Gamma_{r,L}(x_{0})}\frac{|D_{\top} g|^{p}}{p}\diff \mathcal{H}^{2}.
    \end{equation*}
    \end{enumerate}
\end{lemma}
      \begin{proof}
    Up to translation, we can assume that $x_{0}=0$. Denote $\Gamma=\Gamma_{r,L}$ and  $\Lambda = \Lambda_{r,L}$. Using the coarea formula, we get some  $z_{0}\in (-L/2,L/2)$ such that $g|_{\partial B^{2}_{r}\times \{z_{0}\}}\in W^{1,2}(\partial B^{2}_{r}\times \{z_{0}\}, \mathcal{N})$ and
    \begin{equation}\label{est 2.27}
    \|D_{\top}g\|^{p}_{L^{p}(\partial B^{2}_{r}\times \{z_{0}\})}\leq \frac{2}{L}\|D_{\top}g\|^{p}_{L^{p}(\Gamma)}.
    \end{equation}
    First, we define $u$ on $(B^{2}_{r}\setminus B^{2}_{r/2})\times (-L,L)$. Let $f(\sigma, z)\coloneqq g(r\sigma,z)$ for each $\sigma \in \mathbb{S}^{1}$ and $z \in (-L,L)$. For every $\varrho\in [r/2,r]$, $\sigma \in \mathbb{S}^{1}$ and $z\in (-L,L)$, we set
    \[
    u(\varrho \sigma, z)\coloneqq
    \begin{cases}
    f(\sigma, \xi(\varrho, z))&\text{if}\,\ \varrho\in [\varrho_{0}(z), r], \\
    f(\sigma, z_{0}) &\text{if}\,\ \varrho\in [r/2, \varrho_{0}(z)],
    \end{cases}
    \]
    where 
    \begin{align*}
    \xi(\varrho,z)
    &\coloneqq
    \begin{cases}
    z-\frac{2(r-\varrho)(L-z_{0})}{r} \,\ &\text{if}\,\ z \in [z_{0},L],\\
    z+\frac{2(r-\varrho)(L+z_{0})}{r} \,\ &\text{if}\,\ z \in [-L,z_{0}]
    \end{cases}
    &
    &\text{ and}&
    \varrho_{0}(z)&
    \coloneqq
    \begin{cases}
    r-\frac{r(z-z_{0})}{2(L-z_{0})} \,\ &\text{if}\,\ z \in [z_{0}, L],\\
    r+\frac{r(z-z_{0})}{2(L+z_{0})}\,\ &\text{if}  \,\ z \in [-L,z_{0}].
    \end{cases}
    \end{align*}
    Notice that $z\mapsto \varrho_{0}(z)$ is affine on $[-L,z_{0}]$ and on $[z_{0}, L]$. Also $\varrho_{0}(-L)=\varrho_{0}(L)=r/2$, $\varrho_{0}(z_{0})=r$. Furthermore, $\xi(\varrho,z)\in (-L,L)$ if $z\in (-L,L)$ and $\varrho\in [\varrho_{0}(z),r]$. Performing the changes of variables, we have, firstly, the splitting
    \begin{align*}
     \|Du\|^{p}_{L^{p}((B^{2}_{r}\setminus \smash{\overline{B}}^{2}_{r/2})\times (z_{0},L))} = \mathrm{I} + \mathrm{II}, 
    \end{align*}
    where
    \begin{align*}
    \mathrm{I}
    \coloneqq 
    \int_{r/2}^{\varrho_{0}(z)}\diff \varrho \int_{\mathbb{S}^{1}}\diff \mathcal{H}^{1}(\sigma)\int_{z_{0}}^{L}\varrho^{1-p} \left|\frac{\partial f}{\partial \sigma}\right|^{p}(\sigma,z_{0}) \diff z
    \end{align*}
    and 
   
    \begin{align*}
    \mathrm{II}\coloneqq \int_{\varrho_{0}(z)}^{r} d \varrho\int_{\mathbb{S}^{1}} d \mathcal{H}^{1}(\sigma) \int_{z_{0}}^{L}\varrho\left(\left(\frac{4(L-z_{0})^{2}}{r^{2}}+1\right)\left|\frac{\partial f}{\partial z}\right|^{2}(\sigma,\xi(\varrho,z))+\frac{1}{\varrho^{2}}\left|\frac{\partial f}{\partial \sigma}\right|^{2}(\sigma,\xi(\varrho,z))\right)^{\frac{p}{2}} d z.
    \end{align*}       

    Using the fact that $\frac{1-(1/2)^{2-p}}{2-p}\leq \ln(2)$, the elementary inequality $(a+b)^{\frac{p}{2}}\leq a^{\frac{p}{2}}+b^{\frac{p}{2}}$ for each $a,b\geq 0$ and \eqref{est 2.27}, we estimate \(\mathrm{I}\) by
    \begin{equation}
    \label{firsteint}
    \begin{split}
    \mathrm{I} &\leq \int_{\mathbb{S}^{1}}\ln(2)(L-z_{0})r^{2-p}\left|\frac{\partial f}{\partial \sigma}\right|^{p}(\sigma, z_{0})\diff \mathcal{H}^{1}(\sigma)\\
    & \leq \ln(2)(L-z_{0})r\|D_{\top}g\|^{p}_{L^{p}(\partial B^{2}_{r}\times \{z_{0}\})}  
    \end{split}
    \end{equation}
    and \(\mathrm{II}\) by
    \begin{align*}
    \mathrm{II} &\leq  \frac{r}{2(L-z_{0})}\int_{z_{0}}^{z}\diff \xi\int_{\mathbb{S}^{1}}\diff \mathcal{H}^{1}(\sigma)\int_{z_{0}}^{L}r \left(\left(\frac{4(L-z_{0})^{2}}{r^{2}}+1\right)\left|\frac{\partial f}{\partial z}\right|^{2}(\sigma,\xi)+\frac{4}{r^{2}}\left|\frac{\partial f}{\partial\sigma}\right|^{2}(\sigma,\xi)\right)^{\frac{p}{2}}  \diff z\\
    & \leq \frac{r}{2}\left(\frac{4(L-z_{0})^{2}}{r^{2}}+4\right)^{\frac{p}{2}}\int_{z_{0}}^{L}\diff \xi\int_{\mathbb{S}^{1}}r\left(\left|\frac{\partial f}{\partial z}\right|^{2}(\sigma,\xi)+\frac{1}{r^{2}}\left|\frac{\partial f}{\partial \sigma}\right|^{2}(\sigma,\xi)\right)^{\frac{p}{2}} \diff \mathcal{H}^{1}(\sigma) \\
    & \leq 2\left(\frac{(L-z_{0})^{p}}{r^{p-1}}+r\right)\|D_{\top}g\|^{p}_{L^{p}(\partial B^{2}_{r}\times (z_{0},L))}, 
    \end{align*} 
    which, together with \eqref{firsteint}, implies that
    \begin{equation}\label{est 2.28}
     \|Du\|^{p}_{L^{p}((B^{2}_{r}\setminus \smash{\overline{B}}^{2}_{r/2})\times (z_{0},L))} \leq C\left(\frac{L^{p}}{r^{p-1}}+r\right)\|D_{\top}g\|^{p}_{L^{p}(\Gamma)}, 
     \end{equation}
    where $C>0$ is the constant
    that, hereinafter in this proof, can depend only on $\mathcal{N}$ and can be different from line to line. 
    Similarly, we have
    \begin{equation} \label{est 2.29}
    \|Du\|^{p}_{L^{p}((B^{2}_{r}\setminus \smash{\overline{B}}^{2}_{r/2})\times (-L,z_{0}))} \leq  C\left(\frac{L^{p}}{r^{p-1}}+r\right)\|D_{\top}g\|^{p}_{L^{p}(\Gamma)}. 
    \end{equation}
    Applying Lemma~\ref{lem 2.17}, we define $u$ on $ B^{2}_{r/2}\times \{z_{0}\}$, satisfying
    \begin{equation}
    \label{est 2.30}
    \begin{split}
    \int_{B^{2}_{r/2}\times \{z_{0}\}}\frac{|Du|^{p}}{p}\diff \mathcal{H}^{2}&\leq \frac{\mathcal{E}^{\mathrm{sg}}_{p}([g]_{\Gamma})r^{2-p}}{2-p}+Cr\int_{\partial B^{2}_{r}\times \{z_{0}\}}\frac{|D_{\top}u|^{p}}{p}\diff \mathcal{H}^{1}\\
    &\leq \frac{\mathcal{E}^{\mathrm{sg}}_{p}([g]_{\Gamma})r^{2-p}}{2-p}+ \frac{Cr}{L}\int_{\Gamma}\frac{|D_{\top}g|^{p}}{p}\diff \mathcal{H}^{2},   
    \end{split}
    \end{equation}
    where to obtain the last estimate, we have used \eqref{est 2.27}. Since $u$ restricted to $\partial B^{2}_{r/2}\times (-L,L)$ is independent of the $z$-variable, we can define $u(\varrho \sigma, z)=u(\varrho \sigma, z_{0})$ for every $\varrho \in (0, r/2]$, $\sigma \in \mathbb{S}^{1}$ and $z \in (-L,L)$. At this point, $u$ is defined on $\Lambda$ and $u \in W^{1,p}(\Lambda, \mathcal{N})$. Observe that, integrating both sides of \eqref{est 2.30} with respect to $z \in (-L,L)$, we obtain that
    \begin{equation}\label{est 2.31}
    \int_{B^{2}_{r/2}\times (-L,L)}\frac{|Du|^{p}}{p}\diff x \leq \frac{2L\mathcal{E}^{\mathrm{sg}}_{p}([g]_{\Gamma})r^{2-p}}{2-p} +Cr\int_{\Gamma}\frac{|D_{\top}g|^{p}}{p}\diff \mathcal{H}^{2}.
    \end{equation}
    Combining \eqref{est 2.28} with \eqref{est 2.29} and \eqref{est 2.31}, we complete the proof of the assertion \ref{it_pha8ieh4Umaejao2vaephei8}.
    
    Lastly, looking at the definition of $u$, one observes that $\tr_{B^{2}_{r}\times \{z\}}(u)\in W^{1,p}(B^{2}_{r}\times \{z\}, \mathcal{N})$ for each $z \in \{-L,L\}$. In particular, if $z=L$, then, setting $\mathrm{III}=\|D\operatorname{tr}_{B^{2}_{r}\times \{L\}}(u)\|^{p}_{L^{p}((B^{2}_{r}\setminus \smash{\overline{B}}^{2}_{r/2})\times \{L\})}$, we deduce
    \begin{align*}
    \mathrm{III} &\leq \int_{r/2}^{r}\diff \varrho\int_{\mathbb{S}^{1}}\varrho\left(\frac{4(L-z_{0})^{2}}{r^{2}}\left|\frac{\partial f}{\partial z}\right|^{2}(\sigma,\xi(\varrho,L))+\frac{1}{\varrho^{2}}\left|\frac{\partial f}{\partial \sigma}\right|^{2}(\sigma, \xi(\varrho,L))\right)^{\frac{p}{2}}\diff\mathcal{H}^{1}(\sigma)\\
    & \leq  \frac{r}{2(L-z_{0})}\left(\frac{4(L-z_{0})^{2}}{r^{2}}+4\right)^{\frac{p}{2}}\int_{z_{0}}^{L}\diff \xi\int_{\mathbb{S}^{1}}r\left(\left|\frac{\partial f}{\partial z}\right|^{2}(\sigma,\xi)+\frac{1}{r^{2}}\left|\frac{\partial f}{\partial \sigma}\right|^{2}(\sigma, \xi)\right)^{\frac{p}{2}}\diff \mathcal{H}^{1}(\sigma) \\
    & \leq \frac{r}{2(L-z_{0})}\left(\frac{4(L-z_{0})^{p}}{r^{p}}+4\right)\int_{\Gamma}|D_{\top}g|^{p}\diff \mathcal{H}^{2}\\ 
    & \leq C\left(\left(\frac{L}{r}\right)^{p-1}+\frac{r}{L}\right)\int_{\Gamma}|D_{\top}g|^{p}\diff \mathcal{H}^{2}. 
    \end{align*}
    This, together with \eqref{est 2.30}, proves the inequality of the assertion \ref{it_cang8wooGhodoh3Eegoom7ee} for $z=L$. 
    Similarly, one obtains the same estimate for $\|D\operatorname{tr}_{B^{2}_{r}\times \{-L\}}(u)\|^{p}_{L^{p}(B^{2}_{r}\times \{-L\})}$, which completes our proof of the assertion \ref{it_cang8wooGhodoh3Eegoom7ee} and of Lemma~\ref{lem 2.19}.    
\end{proof}
\subsection{Minimizing $p$-harmonic maps}
Given an open bounded set $U\subset \mathbb{R}^{3}$, 
recall that a mapping $u_{p} \in W^{1,p}(U, \mathcal{N})$ is a $p$-minimizer in $U$ whenever the following holds
\[
\int_{U}\frac{|Du_{p}|^{p}}{p}\diff x = \min\left\{\int_{U} \frac{|Du|^{p}}{p}\diff x: u \in W^{1,p}(U, \mathcal{N}), \,\ u-u_{p} \in W^{1,p}_{0}(U, \mathbb{R}^{\nu}) \right\}.
\]
In particular, if $U$ has a Lipschitz boundary, then $u_{p} \in W^{1,p}(U, \mathcal{N})$ is a $p$-minimizer in $U$ if and only if
\[
\int_{U}\frac{|Du_{p}|^{p}}{p}\diff x = \min\left\{\int_{U} \frac{|Du|^{p}}{p}\diff x: u \in W^{1,p}(U, \mathcal{N}), \,\ \operatorname{tr}_{\partial U} (u)=\operatorname{tr}_{\partial U} (u_{p})\right\}.
\]
When appropriate, we shall simply say that $u_{p}$ is a $p$-minimizer. 

The next proposition says that if $p \in (1,2)$, then for every $p$-minimizer in $\Omega$ which is equal on $\partial \Omega$ to some fixed mapping $g$ lying in $W^{1/2,2}(\partial \Omega, \mathcal{N})$, the $p$-energy is bounded from above by $C((2-p)^{-1})$,  where $C>0$ is a constant depending only on $g$, $\partial \Omega$ and $\mathcal{N}$.
    \begin{prop}
    \label{global energy bound}
        Let $g\in W^{1/2,2}(\partial \Omega, \mathcal{N})$ and $p\in (1,2)$. Then there exists a positive constant $C$ depending only on $\partial \Omega$ and  $\mathcal{N}$ such that if $u_{p}$ is a $p$-minimizer with $\tr_{\partial \Omega}(u_{p})=g$, then  
        \[
        (2-p)\int_{\Omega}\frac{|Du_{p}|^{p}}{p}\diff x\leq C |g|^{p}_{W^{1/2, 2}(\partial \Omega, \mathbb{R}^{\nu})}.
        \]
    \end{prop}
    
    \begin{proof} 
         Inasmuch as $g \in W^{1/2, 2}(\partial \Omega, \mathcal{N})$, it holds $g \in W^{1-1/p,p}(\partial \Omega, \mathbb{R}^{\nu})$ (the main ingredients in the proof are H\"older's inequality, the Sobolev inequality for fractional Sobolev spaces (the reader may consult \cite[Theorem~7.57]{Adams_1975}, \cite[Section~3.3]{Triebel_1983} and \cite[Theorem~1]{Bourgain_Brezis_Mironescu_2002}) and  the fact that $\partial \Omega\times \partial \Omega = \bigcup_{i=1}^{n} \varphi_{i}(B^{2}_{1})\times \varphi_{i}(B^{2}_{1})$ for some $n \in \mathbb{N}\setminus\{0\}$, where $\varphi_{i}$ is a bilipschitz map; see also \cite[Proposition~2.1]{Hitchhiker}) and 
         \begin{equation}
         \label{estimatefracenergyp2}
         |g|_{W^{1-1/p,p}(\partial \Omega, \mathbb{R}^{\nu})} \le C |g|_{W^{1/2, 2}(\partial \Omega, \mathbb{R}^{\nu})}.         
         \end{equation}
         According to \cite[Theorem~6.2]{H-L}, there exists $v \in W^{1,p}(\Omega, \mathcal{N})$ such that $\tr_{\partial \Omega}(v)=g$ and
         \begin{equation}
         \label{eqestimatetraceexthlbd}
         (2-p)\int_{\Omega}\frac{|Dv|^{p}}{p}\diff x \leq C^{\prime} |g|^{p}_{W^{1-1/p,p}(\partial \Omega, \mathbb{R}^{\nu})}.
         \end{equation}
         Next, using the direct method in the calculus of variations (notice that $p \in (1,2)$), we deduce that there exists a $p$-minimizer $u_{p}$ such that $\tr_{\partial \Omega}(u_{p})=g$. In view of the minimality of $u_{p}$ and the estimates \eqref{eqestimatetraceexthlbd}, \eqref{estimatefracenergyp2}, the following holds
        \begin{equation}
        \label{nonlinear est extension}
        (2-p)\int_{\Omega}\frac{|Du_{p}|^{p}}{p}\diff x\leq C^{\prime} |g|^{p}_{W^{1-1/p,p}(\partial \Omega, \mathbb{R}^{\nu})} \leq C^{\prime \prime} |g|_{W^{1/2, 2}(\partial \Omega, \mathbb{R}^{\nu})}^p
        \end{equation}
        for some $C^{\prime \prime}=C^{\prime \prime}(\partial \Omega, \mathcal{N})>0$.
        This completes our proof of Proposition~\ref{global energy bound}.
    \end{proof}
We also recall the following fact. Let $u\in W^{1,p}(\Omega,\mathcal{N})$ be a $p$-minimizer. Fix an arbitrary $\xi \in C^{\infty}_{c}(\Omega, \mathbb{R}^{3})$, which is not identically zero.  Define $\varrho_{0}=\min\{\dist(x, \partial \Omega): x \in \mathrm{supp}(\xi)\}$. If $s \in (0, \varrho_{0}/\|\xi\|_{\infty})$ is small enough, then the mapping $\varphi_{s}(x)= x+s\xi(x)$, $x \in\Omega$ is invertible. For each $s>0$ small enough, define $u_{s}(x)=u\circ \varphi^{-1}_{s}(x)$ for each $x \in \Omega$. Since $u$ minimizes the $p$-energy, $\left.\frac{\diff}{\diff s} \int_{\Omega}|Du_{s}|^{p}\diff x \right\vert_{s=0}=~0$. This, together with the facts that $D\varphi_{s}^{-1}\circ \varphi_{s}=(D\varphi_{s})^{-1}$, $(D\varphi_{s})^{-T}=\mathrm{Id}-s(D\xi)^{T}+O(s^{2})$ and $\det(D\varphi_{s})=1+s\,\mathrm{div}(\xi)+O(s^{2})$, implies that the stress-energy tensor $T^{p}_{u}=\frac{|Du|^{p}}{p}\mathrm{Id}-\frac{Du\otimes Du}{|Du|^{2-p}}$ is (row-wise) divergence free. Namely, we have the following integral identity
\begin{equation}\label{integralidentity}
\int_{\Omega}\sum_{i,j=1}^{3}(|Du|^{p}\delta_{i,j}-p|Du|^{p-2}\langle D_{i}u, D_{j}u\rangle)D_{i}\xi_{j}\diff x = 0 \,\ \,\ \forall \xi \in C^{1}_{c}(\Omega, \mathbb{R}^{3}),
\end{equation}
where $\delta_{i,j}=1$ if $i=j$ and $\delta_{i,j}=0$ otherwise.
A map satisfying \eqref{integralidentity} is called $p$-\textit{stationary} or simply \textit{stationary}.

\section{Extension from triangulation}
    \subsection{Lower bounds for the energy}
We begin with the following simple estimate.
\begin{lemma}
    \label{lem 3.5}
        If $p \in [1,2]$, $\delta \in (0,+\infty)$ and $g \in W^{1,p}(\mathbb{S}^{1}, \mathbb{R}^{k})$, then, defining $v(x)\coloneqq g(x/|x|)$ for each $x \in B^{2}_{1+\delta}\setminus \smash{\overline{B}}^{2}_{1}$, we have $v \in W^{1,p}(B^{2}_{1+\delta}\setminus \smash{\overline{B}}^{2}_{1}, \mathbb{R}^{k})$ and
        \[
        \int_{B^{2}_{1+\delta}\setminus \smash{\overline{B}}^{2}_{1}}|Dv|^{p}\diff x \leq \delta \int_{\mathbb{S}^{1}}|D_{\top} g|^{p}\diff \mathcal{H}^{1}.
        \]
\end{lemma}

\begin{proof}
If $p\in [1,2)$, then $v \in W^{1,p}(B^{2}_{1+\delta}\setminus \smash{\overline{B}}^{2}_{1}, \mathbb{R}^{k})$ and 
    \begin{equation}\label{est from v}
    \int_{B^{2}_{1+\delta}\setminus \smash{\overline{B}}^{2}_{1}}|Dv|^{p}\diff x = \int_{1}^{1+\delta}\varrho^{1-p}\diff \varrho \int_{\mathbb{S}^{1}}|D_{\top}g|^{p}\diff \mathcal{H}^{1}=\frac{(1+\delta)^{2-p}-1}{2-p}\int_{\mathbb{S}^{1}}|D_{\top}g|^{p}\diff \mathcal{H}^{1}.
    \end{equation}
Since \(1 \le p \le 2\), the function \(t \in [0, +\infty) \mapsto t^{2 - p}\) is concave and hence $(1+\delta)^{2-p}\leq 1 +(2-p)\delta$. 
This, together with \eqref{est from v}, completes our proof of Lemma~\ref{lem 3.5}, since the case \(p = 2\) is similar.
\end{proof}

Recall that if $p \in [1,2)$ and $\pi_{1}(\mathcal{N})\not \simeq \{0\}$, then the set $C^{\infty}(B^{2}_{r}(x_{0}), \mathcal{N})$ is not dense in $W^{1,p}(B^{2}_{r}(x_{0}), \mathcal{N})$. In fact, in the latter case there exists a map $\gamma \in W^{1,p}(\partial B^{2}_{r}(x_{0}), \mathcal{N})$, which is not nullhomotopic. Defining $u(x)=\gamma(x_{0}+r(x-x_{0})/|x-x_{0}|)$ for each $x \in B^{2}_{r}(x_{0})\backslash \{x_{0}\}$, we have $u \in W^{1,p}(B^{2}_{r}(x_{0}), \mathcal{N})$. Assuming that there exists a sequence of smooth maps $u_{n}:B^{2}_{r}(x_{0})\to \mathcal{N}$ such that $u_{n} \to u$ in $W^{1,p}(B^{2}_{r}(x_{0}), \mathbb{R}^{\nu})$, and using the coarea formula together with the Sobolev embedding theorem, we could find some $\varrho \in (0,r)$ such that $u_{n}|_{\partial B^{2}_{\varrho}(x_{0})}$ converges uniformly to $\tr_{\partial B^{2}_{\varrho}(x_{0})}(u)=u|_{\partial B^{2}_{\varrho}(x_{0})}$. But this would lead to a contradiction, since $u_{n}|_{\partial B^{2}_{\varrho}(x_{0})}$ is nullhomotopic contrary to $u|_{\partial B^{2}_{\varrho}(x_{0})}$. This type of counterexample has its roots in the seminal work of  Schoen and Uhlenbeck (see \cite[Example]{SU}). Nonetheless, as was first observed by Bethuel and Zheng for the circle $\mathbb{S}^{1}$ (see \cite[Theorem~4]{Bethuel_Zheng_1988}) and generalized by Bethuel (see \cite[Theorem~2]{Approximation}) to general compact Riemannian manifolds, Sobolev maps in $W^{1,p}(B^{2}_{r}(x_{0}),\mathcal{N})$ can be approximated by continuous maps outside a finite set of points in $B^{2}_{r}(x_{0})$. For the reader's convenience, we recall Bethuel's theorem, slightly modified for our purposes. 

\begin{theorem}\label{thmappN}
Let $p \in [1,2)$. Then every map $u \in W^{1,p}(B^{2}_{r}(x_{0}), \mathcal{N}) \cap C(B^{2}_{r}(x_{0})\backslash B^{2}_{r/2}(x_{0}), \mathcal{N})$ can be approximated in $W^{1,p}(B^{2}_{r}(x_{0}), \mathbb{R}^{\nu})$ by maps $u_{n} \in W^{1,p}(B^{2}_{r}(x_{0}), \mathcal{N})$ which are continuous outside some finite set of points $A_{n} \subset B^{2}_{r/2}(x_{0})$ such that $u_{n}=u$ on $B^{2}_{r}(x_{0}) \backslash B^{2}_{3r/4}(x_{0})$ and $u_{n}|_{\partial B^{2}_{\varrho}(a)}$  is not nullhomotopic for each $a \in A_{n}$ and for each sufficiently small $\varrho>0$ depending on $n$. 
\end{theorem}

\begin{proof}[Sketch of the proof]
The proof consists of covering $B^{2}_{r/2}(x_{0})$ with ``good'' and ``bad'' balls with respect to the $p$-energy of $u$, relying on an argument based on the coarea formula, and by changing $u$ inside each good and bad ball from the covering by a smooth function in the case of a good ball and a smooth beyond the center of the ball in the case of a bad ball; the reader may consult \cite[Theorem~2]{Approximation}, where the balls are replaced by cubes. We only need to clarify that $u_{n}|_{\partial B^{2}_{\varrho}(a)}$ is not nullhomotopic for each $a \in A_{n}$ and for each sufficiently small $\varrho>0$. In fact, otherwise, using the coarea formula, we could modify $u_n$ inside a ball $B^{2}_{\delta}(a)$, where $\delta \in (\varrho/2, \varrho)$ for some sufficiently small $\varrho>0$, with the quality of the approximation controlled by the energy of the original map on the ball $B^{2}_{\varrho}(a)$, so that  the modified map $u_{n}$ is continuous on $B^{2}_{\varrho}(a)$.
\end{proof}

    \begin{prop}\label{Sandier lemma} Let $p \in (1,2)$, $u \in W^{1,p}(B^{2}_{r}(x_{0}), \mathcal{N})$ and $\tr_{\partial B^{2}_{r}(x_{0})}(u)\in W^{1,p}(\partial B^{2}_{r}(x_{0}), \mathcal{N})$. 
    Then the following estimate holds
    \begin{equation*}
    \int_{B^{2}_{r}(x_{0})}\frac{|Du|^{p}}{p}\diff x + r\int_{\partial B^{2}_{r}(x_{0})}\frac{|D_{\top}\operatorname{tr}_{\partial B^{2}_{r}(x_{0})}(u)|^{p}}{p}\diff \mathcal{H}^{1}\geq \frac{\mathcal{E}^{\mathrm{sg}}_{p/(p-1)}(\operatorname{tr}_{\partial B^{2}_{r}(x_{0})}(u))r^{2-p}}{2-p}.
    \end{equation*}
\end{prop}
\begin{proof}
    Up to scaling and translation, we can assume that $r=1$ and $x_{0}=0$. 
    Since $p>1$, in view of Morrey's embedding for Sobolev mappings, without loss of generality, we suppose  that $\tr_{\mathbb{S}^{1}} (u) \in C(\mathbb{S}^{1}, \mathcal{N})$. 
    Define the mapping $v: B^{2}_{2}\to \mathcal{N}$ by 
    \[
    v(x)\coloneqq \begin{cases}
    \operatorname{tr}_{\mathbb{S}^{1}}(u) \bigl(\tfrac{x}{|x|}\bigr) \,\ &\text{if \(x \in B^{2}_{2}\setminus B^{2}_{1}\)},\\
    u(x) \,\ &\text{if \(x \in B^{2}_{1}\)}.
    \end{cases}
    \]
    Then $v\in W^{1,p}(B^{2}_{2}, \mathcal{N})\cap C(B^{2}_{2}\setminus B^{2}_{1}, \mathcal{N})$. By  Lemma~\ref{lem 3.5} applied with $\delta=1$, we obtain
    \begin{equation}\label{est 3.36}
    \int_{B^{2}_{2}}|Dv|^{p}\diff x \leq \int_{B^{2}_{1}}|Du|^{p}\diff x + \int_{\mathbb{S}^{1}}|D_{\top}\operatorname{tr}_{\mathbb{S}^{1}} (u)|^{p}\diff \mathcal{H}^{1}.
    \end{equation}
    By Theorem~\ref{thmappN}, there exists $(v_{n})_{n\in \mathbb{N}} \subset W^{1,p}(B^{2}_{2}, \mathcal{N})$ such that $v_{n} \to v$ in $W^{1,p}(B^{2}_{2}, \mathbb{R}^{\nu})$ and for each $n\in \mathbb{N}$, $v_{n}$ is continuous in $B^{2}_{2}\setminus A_{n}$, where $A_{n}\subset B^{2}_{1}$ is a finite set of points; $\tr_{\partial B^{2}_{2}}(v_{n})$ is homotopic to $\tr_{\mathbb{S}^{1}}(u)$; for each point $a \in A_{n}$ and for each sufficiently small radius $\varrho>0$, $v_{n}|_{\partial B^{2}_{\varrho}(a)}$ is homotopically nontrivial. 
    Next, using \cite[Corollary~4.4]{VanSchaftingen_VanVaerenbergh}
    with  \(\delta = 1 < \dist(a,\partial B^{2}_{2})\) (in the statement of \cite[Corollary~4.4]{VanSchaftingen_VanVaerenbergh}, \(\delta\) is assumed to satisfy the same assumptions as in \cite[Proposition~4.3]{VanSchaftingen_VanVaerenbergh}) to $v_{n}\in W^{1,p}(B^{2}_{2}, \mathcal{N})$ with the set of singular points $A_{n}$, we get
    \begin{equation}\label{lowerboundpen}
    \begin{split}
    \mathcal{E}^{\mathrm{sg}}_{p/(p-1)}(\operatorname{tr}_{\mathbb{S}^{1}}(u))=\mathcal{E}^{\mathrm{sg}}_{p/(p-1)}(\operatorname{tr}_{\partial B^{2}_{2}}(v_{n}))
      &\leq (2-p) \frac{(p-1)^{p-1}}{((2\pi)/p)^{2-p}}\int_{B^{2}_{2}} \frac{|Dv_{n}|^{p}}{p}\diff x\\
       &\le (2-p)\int_{B^{2}_{2}} \frac{|Dv_{n}|^{p}}{p}\diff x,
    \end{split}
    \end{equation}
    since the singular $p$-energy is invariant under homotopies and $\tr_{\partial B^{2}_{2}}(v_{n})$ is homotopic to $\tr_{\mathbb{S}^{1}}(u)$ and since for \(1 < p < 2\), $\frac{(p-1)^{p-1}}{((2\pi)/p)^{2-p}}<1$. 
\end{proof}

\begin{rem}\label{good rem}
    Adapting the proof of \cite[Proposition~4.3]{VanSchaftingen_VanVaerenbergh}, yields a variant of the latter, where the estimate depends on $\mathrm{sys}(\mathcal{N})$. Therefore, one can obtain a variant of Proposition~\ref{Sandier lemma} with the estimate depending on $\mathrm{sys}(\mathcal{N})$. More precisely, if 
    \[
    \int_{B^{2}_{r}(x_{0})}|Du|^{p}\diff x + r\int_{\partial B^{2}_{r}(x_{0})}|D_{\top}\operatorname{tr}_{\partial B^{2}_{r}(x_{0})}(u)|^{p}\diff \mathcal{H}^{1}>0,\] then one can obtain the estimate \[\frac{2-p}{p}\left(\int_{B^{2}_{r}(x_{0})}|Du|^{p}\diff x + r\int_{\partial B^{2}_{r}(x_{0})}|D_{\top}\operatorname{tr}_{\partial B^{2}_{r}(x_{0})}(u)|^{p}\diff \mathcal{H}^{1}\right)\geq C\mathcal{E}^{\mathrm{sg}}_{p}(\operatorname{tr}_{\partial B^{2}_{r}(x_{0})}(u))^{p-1}r^{2-p},
    \] 
    where $C=\left(\frac{(\mathrm{sys}(\mathcal{N}))^{p}}{(2\pi)^{p-1}p}\right)^{2-p}$. 
    This is an upper bound of $\mathcal{E}^{\mathrm{sg}}_{p}(\tr_{\mathbb{S}^{1}}(u))$ by the $p$-energy, which is in one instance better than the control in Proposition~\ref{Sandier lemma} of $\mathcal{E}^{\mathrm{sg}}_{p/(p-1)}(\tr_{\mathbb{S}^{1}}(u))$ by the $p$-energy. 
    However, the estimate depends on $\mathrm{sys}(\mathcal{N})$. 
    Since we prefer to avoid the use of estimates depending on the systole, we decided to keep the estimation of Proposition~\ref{Sandier lemma}.
\end{rem}
\subsection{Lipschitz triangulations of \texorpdfstring{$\mathbb{S}^{2}$}{𝕊²}}
For a given integer $n \in \mathbb{N} \setminus \{0\}$, we shall call the uniform 1-dimensional grid of step  $1/n$ in $\partial [-1,1]^{3}$ the set of all points $x\in \partial[-1,1]^{3}$ such that $nx_{i} \in \mathbb{Z}$ for all $i=1,2,3$ except for at most one. 

The uniform 1-dimensional grid of step $1/n$ gives the decomposition of $\partial [-1,1]^{3}$ into mutually disjoint sets (points, edges and two-dimensional cells), namely
\begin{equation}\label{uniformgrid} 
\partial [-1,1]^{3}=\bigcup^{2}_{i=0}\bigcup_{j=1}^{k_{i}}E_{i,j},
\end{equation}
where $\{E_{0,j}, j\in \{1,\dotsc,k_{0}\}\}\coloneqq ((1/n) \mathbb{Z}^{3})\cap \partial [-1,1]^{3}$ and each set $E_{i,j}$ with $i\in \{1,2\}$ is bilipschitz homeomorphic to $B^{i}_{1/n}$. It is worth noting that for each integer $l \in  \mathbb{N}\setminus \{0\}$, the cube $(0,1)^{l}$ is bilipschitz homeomorphic to $B^{l}_{1}$ (see, for instance, \cite{bilipschitz}).

Let $\delta \in (0,1]$ and $n=\floor{1/\delta}$, where $\floor{\cdot}$ denotes the integer part. In view of \eqref{uniformgrid}, we obtain the following  decomposition of $\mathbb{S}^{2}$ into mutually disjoint sets
\begin{equation}\label{uniformtriangulation}
\mathbb{S}^{2}=\bigcup_{i=0}^{2}\bigcup_{j=1}^{k_{i}}F_{i,j}=\bigcup_{i=0}^{2}\bigcup_{j=1}^{k_{i}}\omega \circ \mathrm{P}_{\mathbb{S}^{2}}(E_{i,j}),
\end{equation}
where $\mathrm{P}_{\mathbb{S}^{2}}:\mathbb{R}^{3}\setminus \{0\}\to \mathbb{S}^{2}$ is the radial projection ($\mathrm{P}_{\mathbb{S}^{2}}(x)=x/|x|$, $x \in \mathbb{R}^{3}\setminus \{0\}$), $\omega$ is an element of $SO(3)$, $\{F_{0,1},\dotsc,F_{0,k_{0}}\}$ is the collection of points and for each $F_{i,j}$ with $i \in \{1,2\}$, we have a bilipschitz homeomorphism 
\begin{equation}\label{bilisomtr}
\Phi_{i,j}:F_{i,j}\to B^{i}_{\delta}\,\ \text{with}\,\ \|D_{\top}\Phi_{i,j}\|_{L^{\infty}(F_{i,j})}+\|D(\Phi_{i,j})^{-1}\|_{L^{\infty}(B^{i}_{\delta})}\leq A_{1},
\end{equation}
where $A_{1}>0$ is an absolute constant.  The map $\Phi_{i,j}$ in \eqref{bilisomtr} extends by continuity to a bilipschitz map between $\overline{F}_{i,j}$ and $\smash{\overline{B}}^{i}_{\delta}$ with the same bilipschitz constant $A_{1}$ as in \eqref{bilisomtr}.

If \eqref{uniformtriangulation} and \eqref{bilisomtr} hold for some $\omega \in SO(3)$, we shall call the decomposition \eqref{uniformtriangulation} a uniform Lipschitz \emph{triangulation} (or simply a Lipschitz triangulation) of step $\delta$ of $\mathbb{S}^{2}$. The points $F_{0,1},\dotsc,F_{0,k_{0}}$ will be called the vertices and the set $\bigcup_{i=0}^{1}\bigcup_{j=1}^{k_{i}}F_{i,j}$ the 1-skeleton. 

To lighten the notation, we denote the collection of 1-skeletons of all Lipschitz triangulations of step $\delta$ of $\mathbb{S}^{2}$ by $\mathrm{RT}_{\delta}$, namely
\[
\mathrm{RT}_{\delta}
\coloneqq
\left\{\bigcup_{i=0}^{1}\bigcup_{j=1}^{k_{i}}\omega \circ \mathrm{P}_{\mathbb{S}^{2}}(E_{i,j}): \omega \in SO(3)\right\},
\]
where the $E_{i,j}$'s come from \eqref{uniformgrid} for $n=\floor{1/\delta}$. The importance of the rotation $\omega$ in the definition of $\mathrm{RT}_{\delta}$ will be illustrated a little later, for example in Lemma~\ref{nice res on T}.

Let us now define a Sobolev space on a finite union of bilipschitz images of $d$-cubes (or $d$-dimensional intervals), in particular, on the 1-skeleton of a uniform Lipschitz triangulation of the sphere $\mathbb{S}^{2}$.

\begin{defn}\label{Sobolev on T}
    Let $d, N \in \mathbb{N} \setminus \{0\}$. For each $p \in [1,+\infty)$ and for each set $\Sigma \subset \mathbb{R}^{d}$ being the closure of $\bigcup_{i=1}^{N}\Phi_{i}((a_{i1}, b_{i1})\times...\times(a_{id},b_{id}))=: \bigcup_{i=1}^{N}\Delta_{i}^{d}$, where for each $i\in \{1,\dotsc,N\}$, $\Phi_{i}:(a_{i1}, b_{i1})\times...\times(a_{id},b_{id})\to \Delta_{i}^{d}$  is a bilipschitz homeomorphism and, in addition, $\Delta_{i}^{d}\cap \Delta_{j}^{d}=\emptyset$ if $i\neq j$, $j \in \{1,\dotsc,N\}$, we define 
     $W^{1,p}(\Sigma, \mathbb{R}^{k})$ by
    \begin{multline*}
    W^{1,p}(\Sigma, \mathbb{R}^{k})=\{u : \Sigma \to \mathbb{R}^{k}\,\ \text{is Borel-measurable}: u|_{\Delta^{d}_{i}}\in W^{1,p}(\Delta^{d}_{i}, \mathbb{R}^{k}), \,\ \\
    \operatorname{tr}_{\partial \Delta^{d}_{i}}(u|_{\Delta^{d}_{i}})=u|_{\partial \Delta^{d}_{i}}\,\ \text{for each}\,\ i \in \{1,\dotsc,N\}\},
    \end{multline*}
    where $\partial \Delta^{d}_{i}$ denotes the relative boundary of $\Delta^{d}_{i}$. 
    For $u \in W^{1,p}(\Sigma, \mathbb{R}^{k})$, we set \[
    \|u\|_{W^{1,p}(\Sigma, \mathbb{R}^{k})}=\left(\sum_{i=1}^{N}\|u |_{\Delta^{d}_{i}}\|_{W^{1,p}(\Delta^{d}_{i}, \mathbb{R}^{k})}^p \right)^\frac{1}{p}.
    \] 
    
    If $\mathcal{Y}$ is a compact Riemannian manifold which is isometrically embedded into $\mathbb{R}^{k}$, then $W^{1,p}(\Sigma, \mathcal{Y})$ is defined by 
    \[
    W^{1,p}(\Sigma, \mathcal{Y})=\{u\in W^{1,p}(\Sigma, \mathbb{R}^{k}): u\in \mathcal{Y}\,\ \text{a.e. in}\,\ \Sigma\}.
    \]
\end{defn}
\begin{prop}\label{prop ineq T}
    Let $p \in [1,2]$, $\delta \in (0,1]$, $\Sigma \in \mathrm{RT}_{\delta}$ and $g \in W^{1,p}(\Sigma, \mathcal{N})$ (see Definition~\ref{Sobolev on T}). Assume that for each $2$-cell $F$ from the Lipschitz triangulation of $\mathbb{S}^{2}$ generated by $\Sigma$, the map $g|_{\partial F}$ is homotopic to a constant, where $\partial F$ is the relative boundary of $F$. Then there exists $\widetilde{v}\in W^{1,2}(\mathbb{S}^{2}, \widetilde{\mathcal{N}})$, where $\pi:\mathcal{\widetilde{N}}\to \mathcal{N}$ is the universal covering of $\mathcal{N}$, such that 
    $\operatorname{tr}_{\Sigma} (\pi \circ \widetilde{v}) = g$ and 
    \begin{equation*}\label{ineq T}
        \int_{\mathbb{S}^{2}}|D_{\top}\widetilde{v}|^{p}\diff \mathcal{H}^{2}
        \leq 
        C\delta \int_{\Sigma}|D_{\top} g|^{p}\diff \mathcal{H}^{1},
    \end{equation*}
    where $C=C(\mathcal{N})>0$.
\end{prop}

\begin{proof} We have $\mathbb{S}^{2}=\bigcup F$ and $\Sigma=\bigcup \partial F$, where the  unions are taken over all $2$-cells from the Lipschitz triangulation of $\mathbb{S}^{2}$ generated by $\Sigma$, see \eqref{uniformtriangulation}.   
    \medskip

\noindent    \emph{Step 1.} Fix an arbitrary 2-cell  $F$ from our Lipschitz triangulation of $\mathbb{S}^{2}$.  We know that $g|_{\partial F}$ is homotopic to a constant. Let $\Phi: \overline{F} \to \smash{\overline{B}}^{2}_{\delta}$ be the extension of the corresponding map coming from \eqref{bilisomtr}.
    Then, by the classical theory of continuous liftings, $g|_{\partial F}\circ \Phi^{-1}|_{\partial B^{2}_{\delta}}$ has a continuous lifting, which is also Sobolev; this, together with Lemma~\ref{lem lifting linear} and the estimate \eqref{bilisomtr}, implies that there exists $v \in W^{1,2}(F, \mathcal{N})$ such that $\tr_{\partial F}(v)=g|_{\partial F}$ and 
    \begin{equation}\label{estimate on sperical cell}
    \int_{F}|D_{\top}v|^{p}\diff \mathcal{H}^{2} \leq C\delta \int_{\partial  F} |D_{\top} g|^{p}\diff \mathcal{H}^{1},
    \end{equation}
    where $C=C(\mathcal{N})>0$.     
    \medskip
        
    \noindent \emph{Step 2.} Let $v \in L^{\infty}(\mathbb{S}^{2},\mathcal{N})$ satisfy $v|_{F}=v_{F}$ for each $2$-cell $F$ from our Lipschitz triangulation of $\mathbb{S}^{2}$, where $v_{F}$ is the mapping from \emph{Step 1} corresponding to the cell $F$. Since $v_{F}\in W^{1,2}(F, \mathcal{N})$ and $\tr_{\partial F}(v_{F})=g|_{\partial F}$, we have $v \in W^{1,2}(\mathbb{S}^{2}, \mathcal{N})$. In particular, there exists a map $\widetilde{v} \in W^{1,2}(\mathbb{S}^{2},\widetilde{\mathcal{N}})$ such that $v=\pi\circ \widetilde{v}$ (see, for instance, \cite[Theorem~1]{lifting}).
    Next, using \eqref{estimate on sperical cell} and the fact that each $1$-cell of our Lipschitz triangulation of $\mathbb{S}^{2}$ lies in the relative boundary of exactly two $2$-cells,  since $\pi$ is a local isometry, we get 
    \begin{equation*}
    \begin{split}
    \int_{\mathbb{S}^{2}}|D_{\top}\widetilde{v}|^{p}\diff \mathcal{H}^{2} = 
    \int_{\mathbb{S}^{2}}|D_{\top}v|^{p}\diff \mathcal{H}^{2}=\sum \int_{F}|D_{\top}v|^{p}\diff \mathcal{H}^{2}
    \leq 2C\delta\int_{\Sigma}|D_{\top} g|^{p}\diff \mathcal{H}^{1},
    \end{split}
    \end{equation*}
    where the sum is taken over all $2$-cells $F$. This completes our proof of Proposition~\ref{prop ineq T}.
\end{proof}

\begin{lemma}\label{nice res on T}
    Let $p \in [1,+\infty)$ and $g \in W^{1,p}(\mathbb{S}^{2}, \mathcal{Y})$, where $\mathcal{Y}$ is either a Euclidean space or a compact Riemannian manifold. 
    Let $\delta \in (0,1]$ and $\Sigma \in \mathrm{RT}_{\delta}$. 
    Then for $\mathcal{H}^{3}$-a.e.\  $\omega \in SO(3)$, $\operatorname{tr}_{\omega(\Sigma)}(g)=g|_{\omega(\Sigma)}\in W^{1,p}(\omega(\Sigma), \mathcal{Y})$. Furthermore, there exists a rotation $\omega \in SO(3)$ such that we have $\operatorname{tr}_{\omega(\Sigma)}(g)=g|_{\omega(\Sigma)}\in W^{1,p}(\omega(\Sigma), \mathcal{Y})$ and 
    \begin{equation*}\label{ineq for bilipschitz piece}
    \int_{\omega(\Sigma)}|D_{\top}g|^{p}\diff \mathcal{H}^{1}\leq \frac{A_{2}}{ \delta}\int_{\mathbb{S}^{2}}|D_{\top}g|^{p}\diff \mathcal{H}^{2},
    \end{equation*}
    where $A_{2}>0$ is an absolute constant.
\end{lemma}
\begin{proof} Since $\Sigma\in \mathrm{RT}_{\delta}$, there exists $\omega\in SO(3)$ such that $\Sigma = \bigcup_{i=0}^{1}\bigcup_{j=1}^{k_{i}} \omega(E_{i,j})=\bigcup_{i=0}^{1}\bigcup_{j=1}^{k_{i}} F_{i,j}$ and hence $\Sigma=\bigcup_{j=1}^{k_{1}}\overline{F}_{1,j}$, where the $E_{i,j}$'s come from \eqref{uniformgrid} for $n=\floor{1/\delta}$. 
    Fix an arbitrary $j \in \{1,\dotsc,k_{0}\}$. Let $\overline{F}_{1,j_{1}}$ and $\overline{F}_{1,j_{2}}$ be two edges emanating from $F_{0,j}$. 
    Denote $V=\overline{F}_{1,j_{1}}\cup \overline{F}_{1,j_{2}}$. 
    Then, in view of \eqref{bilisomtr}, there is a bilipschitz mapping  $\Psi:\overline{B}^{1}_{\delta}\to V$. Denote $U=\Psi(B^{1}_{\delta})$. 
    Applying Lemma~\ref{lem rest on T}~\ref{it_Ohtaap1joLiedei6aiGeipho} with $v=g$ and $E=U$, we have $\operatorname{tr}_{\omega(U)}(g)=g|_{\omega(U)} \in W^{1,p}(\omega(U), \mathcal{Y})$ for $\mathcal{H}^{3}$-a.e.\  $\omega \in SO(3)$. 
    Furthermore, for each $\omega \in SO(3)$ for which  $g|_{\omega(U)} \in W^{1,p}(\omega(U), \mathcal{Y})$, there exists a unique mapping $h \in C(\omega(V), \mathcal{Y})$ such that $g|_{\omega(U)}=h$ $\mathcal{H}^{1}$-a.e.\  on $\omega(U)$, since $\omega (U)$ is bilipschitz homeomorphic to $B^{1}_{\delta}$. 
    Thus, successively applying Lemma~\ref{lem rest on T} \ref{it_Ohtaap1joLiedei6aiGeipho}, altogether we obtain that there exists $E\subset SO(3)$ such that $\mathcal{H}^{3}(E)=0$ and the following holds. 
    For each $\omega \in SO(3)\setminus E$ and $j\in \{1,\dotsc,k_{1}\}$ we have $\operatorname{tr}_{\omega(F_{1,j})}(g)=g|_{\omega(F_{1,j})} \in W^{1,p}(\omega(F_{1,j}), \mathcal{Y})$ and there exists a unique mapping $h\in C(\omega(\Sigma), \mathcal{Y})$ such that $\operatorname{tr}_{\omega(\Sigma)}(g)=g|_{\omega(\Sigma)}=h$ $\mathcal{H}^{1}$-a.e.\  on~$\omega(\Sigma)$, namely $g|_{\omega(\Sigma)}\in W^{1,p}(\omega(\Sigma), \mathcal{Y})$ (see Definition~\ref{Sobolev on T}). 
    Next, according to \cite[Theorem~3.2.48]{Federer}, 
    \begin{equation}\label{intgeomlemtr}
    \int_{O(3)}\diff\theta_{3}(\omega) \int_{\omega(\Sigma)}|D_{\top}g|^{p}\diff \mathcal{H}^{1} = \frac{\mathcal{H}^{1}(\Sigma)}{\mathcal{H}^{2}(\mathbb{S}^{2})}\int_{\mathbb{S}^{2}}|D_{\top}g|^{p}\diff\mathcal{H}^{2},
    \end{equation}    
    where $\theta_{3}=(\mathcal{H}^{3}(O(3)))^{-1}\mathcal{H}^{3}\mres O(3)$ is the Haar measure of the orthogonal group $O(3)$ such that $\theta_{3}(O(3))=1$. Denoting 
    \[
    A=\left\{\omega \in O(3): \int_{\omega(\Sigma)}|D_{\top}g|^{p}\diff \mathcal{H}^{1} < \frac{3\mathcal{H}^{1}(\Sigma)}{\mathcal{H}^{2}(\mathbb{S}^{2})}\int_{\mathbb{S}^{2}}|D_{\top}g|^{p}\diff\mathcal{H}^{2}\right\},
    \]
    we observe that $\theta_{3}(A)\geq 2/3$, otherwise \eqref{intgeomlemtr} would not be true. Thus,  $\theta_{3}(A\cap SO(3))\geq 1/6$, since $\theta_{3}(O(3)\setminus SO(3))=1/2$. On the other hand, using \eqref{bilisomtr}, we get $\mathcal{H}^{1}(\Sigma)\leq C/\delta$, where $C>0$ is an absolute constant. 
    This completes our proof of Lemma~\ref{nice res on T}.
    \end{proof}

\begin{lemma}\label{lem lowenergy}
    Let $p_{0} \in (1,2)$, $p\in [p_{0}, 2)$, $\delta \in (0,1]$ and $\Sigma \in \mathrm{RT}_{\delta}$. Let $F$ be a $2$-cell of the Lipschitz triangulation of $\mathbb{S}^{2}$ generated by $\Sigma$. Then there exists a constant $\alpha_{0}=\alpha_{0}(p_{0}, \mathcal{N})>0$ such that the following holds. If $u \in W^{1,p}(F, \mathcal{N})$, $\operatorname{tr}_{\partial F}(u) \in W^{1,p}(\partial F, \mathcal{N})$ and 
    \begin{equation}
    \label{cmcoji3jijn4ignu5nu}
    \int_{F}\frac{|D_{\top}u|^{p}}{p}\diff \mathcal{H}^{2}+\delta\int_{\partial F}\frac{|D_{\top}\operatorname{tr}_{\partial F}(u)|^{p}}{p}\diff \mathcal{H}^{1}\leq \frac{\alpha_{0}\delta^{2-p}}{2-p},
    \end{equation} 
    then $\operatorname{tr}_{\partial F}(u)$ is homotopic to a constant, where $\partial F$ denotes the relative boundary of $F$.
\end{lemma}

\begin{proof}
     Let $\Phi: F \to B^{2}_{\delta}$ be a bilipschitz mapping coming from \eqref{bilisomtr}. 
    To lighten the notation, denote $v\coloneqq u\circ \Phi^{-1}$. 
    According to Definition~\ref{def p-singular} and \eqref{systole}, if a map $\gamma \in W^{1,p}(\partial B^{2}_{\delta}, \mathcal{N})$ is not homotopic to a constant, then     
    \begin{equation}
    \label{kjnfi4h59h49h5g49h689}
   \frac{(\mathrm{sys}(\mathcal{N}))^{\frac{p}{p-1}}(p-1)}{(2\pi)^{\frac{1}{p-1}}p}\leq \mathcal{E}^{\mathrm{sg}}_{p/(p-1)}(\gamma).
    \end{equation}
    Let $\alpha_{0}>0$ be a constant, which will be defined a little later for the proof to work.  Next, using Proposition~\ref{Sandier lemma}, the chain rule for Sobolev mappings, \cite[Theorem~3.2.5]{Federer} and \eqref{cmcoji3jijn4ignu5nu}, we get
    \begin{equation}
      \label{est from (S.1)}
 \frac{\mathcal{E}^{\mathrm{sg}}_{p/(p-1)}(\operatorname{tr}_{\partial B^{2}_{\delta}}(v)) \delta^{2-p}}{2-p}\leq  \int_{B^{2}_{\delta}}\frac{|Dv|^{p}}{p}\diff \mathcal{H}^{2}+ \delta \int_{\partial B^{2}_{\delta}}\frac{|D_{\top}\operatorname{tr}_{\partial B^{2}_{\delta}}(v)|^{p}}{p}\diff \mathcal{H}^{1} \leq \frac{C\alpha_{0} \delta^{2-p}}{2-p},     
     \end{equation}
    where $C>0$ is an absolute constant. Now we define \[\alpha_{0}=\min\left\{ \frac{(\mathrm{sys}(\mathcal{N})
 )^{\frac{p}{p-1}}(p-1)}{2C(2\pi)^{\frac{1}{p-1}}p}: p \in [p_{0},2]\right\}.\] In view of the estimates \eqref{kjnfi4h59h49h5g49h689} and \eqref{est from (S.1)},  $\operatorname{tr}_{\partial B^{2}_{\delta}}(v)$ is  homotopic to a constant, and hence $\operatorname{tr}_{\partial F}(u)$ is homotopic to a constant. This completes our proof of Lemma~\ref{lem lowenergy}.
\end{proof}

The next lemma provides us with an extension to a $3$-cell of given mappings on a $2$-cell with  appropriate energy control.
\begin{lemma}\label{lemma_single_cell}
Let $p \in [1,2]$, $\delta \in (0,1]$ and for each $i \in \{0,1\}$, 
\(g_{i} \in W^{1,p}((0,\delta)^{2} , \mathcal{N})\).
If \(\tr_{\partial (0, \delta)^{2}}(g_{i}) \in W^{1,p}(\partial (0, \delta)^{2}, \mathcal{N})\) and if 
\begin{equation}\label{eq_coincideness_tr}
\operatorname{tr}_{\partial (0, \delta)^{2}}(g_{0})
=\operatorname{tr}_{\partial (0,\delta)^{2} }(g_{1}),
\end{equation}
then there exists $w \in W^{1,p}((0,\delta)^{3}, \mathcal{N})$ such that \(\tr_{(0,\delta)^{2}\times \{i\delta\}}(w)=g_{i}\) for each \(i \in \{0,1\}\) and 
\begin{equation}\label{eq_cell_single}
\begin{split}
\int_{(0,\delta)^3}|Dw|^{p} \diff x &\le A_{3} \delta \sum_{i=0}^{1}\int_{(0,\delta)^2}|D g_{i}|^{p}\diff \mathcal{H}^{2} + A_{3}\delta^2 \int_{\partial (0,\delta)^{2}}|D_{\top} \operatorname{tr}_{\partial (0,\delta)^{2}}(g_{1})|^{p}\diff \mathcal{H}^{1},
\end{split}
\end{equation}
where $A_{3}>0$ is an absolute constant.
\end{lemma}

\begin{proof}
Denote \(E=\partial (0,\delta)^{2} \times (0,\delta)\). We define $w$ on $E$ by
\(
w|_{E}(x)=\operatorname{tr}_{\partial (0,\delta)^{2}}(g_{1})(x')
\),
where \(x=(x',x_3) \in E\). 
Since $\operatorname{tr}_{\partial (0, \delta)^{2}}(g_{1}) \in W^{1,p}(\partial (0, \delta)^{2}, \mathcal{N})$, we have \(w|_{E} \in W^{1,p}(E, \mathcal{N})\). 
Using \cite[Theorem~3.2.22~(3)]{Federer} and changing the variables, one has
\begin{equation}\label{eq_estbeisk}
\begin{split}
\int_{E}|Dw|^{p}\diff \mathcal{H}^{2} = \delta \int_{\partial (0,\delta)^{2}}|D_{\top} \operatorname{tr}_{\partial (0,\delta)^{2}}(g_1)|^{p}\diff \mathcal{H}^{1}.
\end{split}
\end{equation}
Next, we define the map $w|_{\partial (0,\delta)^{3}}: \partial (0,\delta)^{3} \to \mathcal{N}$ by
\[
w|_{\partial (0,\delta)^{3}}(x)=
    \begin{cases}
        g_{1}(x') \,\  &\text{if} \,\ x=(x', x_3) \in (0,\delta)^{2} \times \{\delta\},\\
        \operatorname{tr}_{\partial (0,\delta)^{2}}(g_{1})(x') \,\ &\text{if} \,\ x=(x', x_{3}) \in \partial (0,\delta)^{2}\times \{\delta\},\\
        w|_{E}(x) \,\ & \text{if}\,\ x \in E,\\
        \operatorname{tr}_{\partial (0,\delta)^{2}}(g_{0})(x') \,\ &\text{if}\,\ x=(x',x_{3}) \in \partial (0,\delta)^{2}\times \{0\},\\
        g_{0}(x') \,\  &\text{if}\,\ x=(x', x_3) \in (0,\delta)^{2} \times \{0\}.
    \end{cases}
\]
Then, in view of \eqref{eq_coincideness_tr}, we deduce that $w|_{\partial (0,\delta)^{3}} \in W^{1,p}(\partial (0,\delta)^{3}, \mathcal{N})$ and $\operatorname{tr}_{(0,\delta)^{2}\times \{i\delta\}}(w)=g_{i}$ for each $i \in \{0,1\}$.  Let $\Psi:\smash{\overline{B}}^3_{\delta}\to [0,\delta]^3$ be a bilipschitz homeomorphism with an absolute bilipschitz constant (see \cite[Corollary~3]{bilipschitz}). Notice that $w|_{\partial (0,\delta)^3}\circ \Psi|_{\partial B^{3}_{\delta}} \in W^{1,p}(\partial B^{3}_{\delta}, \mathcal{N})$. Thus, defining $v(x)=w|_{\partial (0,\delta)^{3}}\circ \Psi(\delta x/|x|)$ for each $x \in B^{3}_{\delta}\setminus \{0\}$ and using that $p \in [1,2]$, we have $v\in W^{1,p}(B^{3}_{\delta}, \mathcal{N})$. Next, setting $\zeta(x)=x/\abs{x}$ for $x \in \mathbb{R}^{3}\backslash \{0\}$, using  the special case of the coarea formula, \cite[Theorem~3.2.5]{Federer}, the fact that $p \in [1,2]$ and the chain rule for Sobolev mappings,  we obtain the following chain of estimates
\begin{equation}
\label{homognew}
\begin{split}
\int_{B^{3}_{\delta}}|Dv|^{p}\diff x &= \int_{0}^{\delta}\diff r \int_{\partial B^{3}_{r}} \delta^{p}|D(w\circ \Psi)(\delta\zeta(\sigma))\circ D\zeta(\sigma)|^{p}\diff \mathcal{H}^{2}(\sigma)\\
&= \int_{0}^{\delta}\Bigl(\frac{r}{\delta}\Bigr)^{2-p}\diff r \int_{\partial B^{3}_{\delta}}|D_{\top}(w\circ \Psi)|^{p}\diff \mathcal{H}^{2}\\
&\leq \delta\int_{\partial B^{3}_{\delta}}|D_{\top}(w\circ \Psi)|^{p}\diff \mathcal{H}^{2}\\
&\leq C\delta\int_{\partial (0, \delta)^{3}}|D_{\top}w|^{p}\diff \mathcal{H}^{2},  
\end{split}
\end{equation}
where  $C>0$ is a constant depending only on the bilipschitz constant of $\Psi$.  Then, defining for each $x \in (0,\delta)^{3}\setminus \{\Psi(0)\}$, $w(x)=v(\Psi^{-1}(x))$, we observe that $w\in W^{1,p}((0,\delta)^3, \mathcal{N})$, since $v\in W^{1,p}(B^{3}_{\delta},\mathcal{N})$ and $\Psi$ is bilipschitz. Using \cite[Theorem~3.2.5]{Federer}, the fact that $\Psi$ is bilipschitz and \eqref{homognew}, we have
\begin{equation*}
\int_{(0,\delta)^3}|D w|^{p}\diff x  
\leq C^{\prime} \int_{B^{3}_{\delta}}|Dv|^{p}\diff x \leq C^{\prime\prime} \delta \int_{\partial (0,\delta)^3}|D_{\top}w|^{p}\diff \mathcal{H}^{2},  
\end{equation*}
where $C^{\prime\prime}>0$ depends only on the bilipschitz constant of $\Psi$, which is an absolute constant. This, together with \eqref{eq_estbeisk}, yields \eqref{eq_cell_single} and completes our proof of Lemma~\ref{lemma_single_cell}.
\end{proof}
Using Lemma~\ref{lemma_single_cell}, we obtain the next Luckhaus-type lemma (see \cite[Lemma~1]{Luckhaus_1988}).

\begin{lemma}\label{finding nice v and conmapping}
    Let $p_{0}\in (1,2)$,  $p \in [p_{0}, 2)$, $\delta \in (0,1)$ and $g \in W^{1,p}(\mathbb{S}^{2}, \mathcal{N})$. Then there exists a constant $\alpha=\alpha(p_{0}, \mathcal{N})>0$ such that the following holds. Assume that 
        \begin{equation}\label{lowenergycondition}
        \int_{\mathbb{S}^{2}}\frac{|D_{\top}g|^{p}}{p}\diff \mathcal{H}^{2} \leq \frac{\alpha\delta^{2-p}}{2-p}.
        \end{equation}
        Then there exist mappings $\varphi_{\delta}\in W^{1,p}(B^{3}_{1}\setminus \smash{\overline{B}}^3_{1-\delta}, \mathcal{N})$ 
        and $\widetilde{v} \in W^{1,2}(\mathbb{S}^{2},\widetilde{\mathcal{N}})$, where $\pi:\widetilde{\mathcal{N}}\to \mathcal{N}$ is the universal covering of $\mathcal{N}$, such that 
    \begin{align}
    \label{connmapcon}
    \operatorname{tr}_{\mathbb{S}^{2}}(\varphi_{\delta})&=g,&
    \operatorname{tr}_{\mathbb{S}^{2}}(\varphi_{\delta}((1-\delta )\cdot))&=\pi   \circ \widetilde{v},
    \end{align}
    and the following estimates are satisfied
    \begin{gather}\label{estonconnmap}
    \int_{B^{3}_{1}\setminus \smash{\overline{B}}^3_{1-\delta}}|D\varphi_{\delta}|^{p}\diff x \leq C\delta \int_{\mathbb{S}^{2}}|D_{\top}g|^{p} \diff \mathcal{H}^{2},\\     
    \label{estlifting}
    \int_{\mathbb{S}^{2}}|D_{\top}(\pi \circ \widetilde{v})|^{p}\diff \mathcal{H}^{2}
    = 
        \int_{\mathbb{S}^{2}}|D_{\top}\widetilde{v}|^{p}\diff \mathcal{H}^{2}
        \leq 
        C \int_{\mathbb{S}^{2}}|D_{\top}g|^{p}\diff \mathcal{H}^{2},
    \end{gather}
    where $C=C(\mathcal{N})>0$. 
\end{lemma}
\begin{proof} 
Let $\alpha_{0}=\alpha_{0}(p_{0}, \mathcal{N})>0$ be the constant given by Lemma~\ref{lem lowenergy}. Define $\alpha=\alpha_{0}/2A_{2}$, where $A_{2}$ is the absolute constant of Lemma~\ref{nice res on T}. Without loss of generality, we assume that $A_{2} \geq 1$. Then $\alpha$ depends only on $p_{0}$ and $\mathcal{N}$. In view of Lemma~\ref{nice res on T}, the estimate \eqref{lowenergycondition} and the definition of $\alpha$, there exists $\Sigma \in \mathrm{RT}_{\delta}$ such that $\operatorname{tr}_{\Sigma}(g)=g|_{\Sigma} \in W^{1,p}(\Sigma, \mathcal{N})$, 
\begin{equation}\label{uniformestimateonskeletonofatriangulation}
        \int_{\Sigma}|D_{\top} g|^{p}\diff \mathcal{H}^{1}\leq \frac{A_{2}}{\delta} \int_{\mathbb{S}^{2}}|D_{\top} g|^{p}\diff \mathcal{H}^{2}
        \end{equation}
and for each $2$-cell $F$ from the Lipschitz triangulation of $\mathbb{S}^{2}$ generated by $\Sigma$,
\[
\int_{F}\frac{|D_{\top} g|^{p}}{p}\diff \mathcal{H}^{2}+\delta \int_{\partial F} \frac{|D_{\top} g|^{p}}{p}\diff \mathcal{H}^{1} \leq 2A_{2} \int_{\mathbb{S}^{2}}\frac{|D_{\top} g|^{p}}{p}\diff \mathcal{H}^{2} \leq \frac{\alpha_{0} \delta^{2-p}}{2-p},
\]
where $\partial F$ denotes the relative boundary of $F$. This, according to Lemma~\ref{lem lowenergy}, implies that $g|_{\partial F}$ is homotopically trivial for each $2$-cell $F$ from the Lipschitz triangulation of $\mathbb{S}^{2}$ generated by $\Sigma$. By Proposition~\ref{prop ineq T}, there exists $\widetilde{v} \in W^{1,2}(\mathbb{S}^{2}, \widetilde{\mathcal{N}})$ such that $\operatorname{tr}_{\Sigma} (\pi \circ \widetilde{v}) = \operatorname{tr}_{\Sigma} (g)$ and  
         \begin{equation}\label{estextvtr}
         \int_{\mathbb{S}^{2}}|D_{\top}(\pi\circ \widetilde{v})|^{p}\diff \mathcal{H}^{2}= \int_{\mathbb{S}^{2}}|D_{\top}\widetilde{v}|^{p}\diff \mathcal{H}^{2}\leq C\delta \int_{\Sigma}|D_{\top}g|^{p}\diff \mathcal{H}^{1},
         \end{equation}
        where $C=C(\mathcal{N})>0$. In \eqref{estextvtr} we used that $\pi$ is a local isometry. Using \eqref{uniformestimateonskeletonofatriangulation} and \eqref{estextvtr}, we get    \eqref{estlifting}.     
        
The Lipschitz triangulation of $\mathbb{S}^{2}$ generated by $\Sigma$ yields the decomposition of $B^{3}_{1}\setminus \smash{\overline{B}}^3_{1-\delta}$ into the cells
    \begin{equation}
    \label{decomp123}
    D_{i,j}=\left\{x \in \mathbb{R}^{3}: (1-\delta) < |x| < 1,\,\ \frac{x}{|x|} \in F_{i,j}\right\},
    \end{equation}
    where $i \in \{0,1,2\}$ and $j \in \{1,\dotsc,k_{i}\}$ (see \eqref{uniformtriangulation}). Notice that $D_{i,j}$ is of Hausdorff dimension $(i+1)$. Fix an arbitrary $3$-cell $D:=D_{2,j}$ from the above decomposition. Let $F$ be the $2$-cell of the Lipschitz triangulation of $\mathbb{S}^{2}$ generated by $\Sigma$ such that $D=\{x \in \mathbb{R}^{3}: (1-\delta)<|x|<1, \ \frac{x}{|x|} \in F\}$. Fix a bilipschitz homeomorphism $\Phi: [0,\delta]^{3}\to \smash{\overline{D}}$ with an absolute bilipschitz constant (see \eqref{bilisomtr}) such that $\Phi((0,\delta)^{2} \times \{0\})=(1-\delta) F$ and $\Phi((0,\delta)^{2} \times \{\delta\})=F$. Applying Lemma~\ref{lemma_single_cell} with $g_{0}$ and $g_{1}$ taking the values of $\pi\circ \widetilde{v}\circ (\frac{\Phi}{1-\delta})|_{(0,\delta)^{2} \times \{0\}}$ and $g \circ \Phi|_{(0,\delta)^{2} \times \{1\}}$, respectively, we obtain a mapping $\varphi_{j} \in W^{1,p}(D, \mathcal{N})$ satisfying the estimate
    \begin{equation}
    \label{estm4i05i0jg405ijgitngu56nu}
    \int_{D}|D\varphi_{j}|^{p}\diff x \leq C^{\prime} \delta \left(\int_{F}|D_{\top} \widetilde{v}|^{p}\diff \mathcal{H}^{2}+\int_{F}|D_{\top} g|^{p}\diff \mathcal{H}^{2} \right) + C^{\prime} \delta^{2} \int_{\partial F}|D_{\top} g|^{p} \diff \mathcal{H}^{1},
    \end{equation}
    where $C^{\prime}>0$ is an absolute constant. 
    
    Let $\varphi_{\delta} \in L^{\infty}(B^{3}_{1}\setminus \smash{\overline{B}}^3_{1-\delta}, \mathcal{N})$ satisfy $\varphi_{\delta}|_{D_{2,j}}=\varphi_{j} \in W^{1,p}(D_{2,j}, \mathcal{N})$ for each $3$-cell $D_{2,j}$ from the decomposition \eqref{decomp123}. In view of our construction (see Lemma~\ref{lemma_single_cell}), we can ensure that the traces of the maps $\varphi_{j}$ coincide on the mantles of the cells $D_{2,j}$ and hence $\varphi_{\delta} \in W^{1,p}(B^{3}_{1}\setminus \smash{\overline{B}}^3_{1-\delta}, \mathcal{N})$. Summing \eqref{estm4i05i0jg405ijgitngu56nu}, where $D=D_{2,j}$ and $F=F_{2,j}$, over $j \in \{1,\dotsc, k_{2}\}$, taking into account that each $1$-cell of the Lipschitz triangulation of $\mathbb{S}^{2}$ generated by $\Sigma$ lies in the relative boundary of exactly two $2$-cells from this triangulation, and also using \eqref{estlifting} together with \eqref{uniformestimateonskeletonofatriangulation}, we deduce the following chain of estimates 
\begin{align*}
\int_{B^{3}_{1}\setminus \smash{\overline{B}}^3_{1-\delta}}|D\varphi_{\delta}|^{p}\diff x 
&\leq C^{\prime} \delta\sum_{j=1}^{k_{2}}\left(\int_{F_{2,j}}|D_{\top}\widetilde{v}|^{p}\diff \mathcal{H}^{2}+\int_{F_{2,j}}|D_{\top} g|^{p}\diff \mathcal{H}^{2}\right)+2C^{\prime} \delta^{2}\int_{\Sigma}|D_{\top}g|^{p}\diff \mathcal{H}^{1}\\
&\leq C^{\prime \prime}\delta \int_{\mathbb{S}^{2}}|D_{\top}g|^{p}\diff \mathcal{H}^{2},
\end{align*} 
where $C^{\prime \prime}=C^{\prime \prime}(\mathcal{N})>0$. This completes our proof of Lemma~\ref{finding nice v and conmapping}.
\end{proof}

We state now an \(\eta\)-extension with sublinear growth, which will turn out to be the key tool in the proof of our $\eta$-\emph{compactness} lemma (see Lemma~\ref{heart}).

\begin{lemma}
\label{key tool bilipschitz}
    Let $p_{0} \in (1,2)$ and $p \in [p_{0},2)$.
    There exist constants \(\eta_0, C> 0\), depending only on $p_{0}$ and \(\mathcal{N}\), such that 
    if $g \in W^{1,p}(\partial B^3_r, \mathcal{N})$ satisfies
    \begin{equation}
    \label{estimation key prop bilipschitz}
        \int_{\partial B^3_r}
            \frac
                {|D_{\top} g|^{p}}
                {p}
            \diff \mathcal{H}^{2}
        \leq 
            \frac{\eta_{0} r^{2-p}}{2-p},
    \end{equation}
    then there exists $u \in W^{1,p}(B^3_r, \mathcal{N})$ such that $\operatorname{tr}_{\partial B^{3}_{r}}(u)=g$ and
    \begin{equation}
    \label{beautiful inequality bilipschitz}
        \int_{B^3_r}
            \frac{|Du|^{p}}{p}\diff x 
            \leq 
            Cr^{\frac{2}{p}}
            \biggl(\int_{\partial B^3_r}\frac{|D_{\top} g|^{p}}{p}
            \diff \mathcal{H}^{2}\biggr)^{1 - \frac{1}{p}}.
    \end{equation}
\end{lemma}
\begin{proof}
    By scaling, we can assume that $r=1$. Letting $\delta \in (0,1/2]$, which will be fixed later for the proof to work, we define 
    \begin{equation}
        \label{defeta}
        \eta_0 =\alpha\delta^{2-p},
    \end{equation}
    where $\alpha=\alpha(p_{0}, \mathcal{N})>0$ is the constant of Lemma~\ref{finding nice v and conmapping}. 
    Then, by \eqref{estimation key prop bilipschitz} and \eqref{defeta}, 
    \eqref{lowenergycondition} is satisfied in Lemma~\ref{finding nice v and conmapping} applied with $\delta$ and $g$, which yields 
    $\varphi_{\delta} \in W^{1,p}(B^{3}_{1}\setminus \smash{\overline{B}}^3_{1-\delta}, \mathcal{N})$ and $\widetilde{v}\in W^{1,2}(\mathbb{S}^{2},\widetilde{\mathcal{N}})$, satisfying 
    the condition \eqref{connmapcon}
    together with the estimates \eqref{estonconnmap} and \eqref{estlifting}.
    Using Lemma~\ref{lem lifting nonlinear} with the boundary datum $\pi \circ \widetilde{v}$, we obtain a map $w \in W^{1,p}(B^{3}_{1}, \mathcal{N})$ such that $\tr_{\mathbb{S}^{2}}(w)=\pi \circ \widetilde{v}$ and 
    \begin{equation}
        \label{estimate coming from lifting}
        \int_{B^{3}_{1}}|Dw|^{p}\diff x 
        \leq 
        C\left(\int_{\mathbb{S}^{2}}|D_{\top}\widetilde{v}|^{p}\diff \mathcal{H}^{2}\right)^{1 - \frac{1}{p}} 
        \leq 
        C^{\prime}\left(\int_{\mathbb{S}^{2}}\frac{|D_{\top}g|^{p}}{p}\diff \mathcal{H}^{2}\right)^{1 - \frac{1}{p}},
    \end{equation}
    where $C=C(p_{0}, \mathcal{N})>0$ and  $C^{\prime}=C^{\prime}(p_{0}, \mathcal{N})>0$. In the last estimate we used \eqref{estlifting} under the condition that $p\in [p_{0},2)$. Now we define the mapping $u :B^{3}_{1}\to \mathcal{N}$ by
    \[
    u(x)=
        \begin{cases}
            \varphi_{\delta}(x) \,\ &\text{if \(x \in B^{3}_{1}\setminus \smash{\overline{B}}^3_{1-\delta}\),}\\
            \displaystyle \pi \circ \widetilde{v} \left(\frac{x}{1-\delta}\right) \,\ &\text{if \(x\in \partial B^{3}_{1-\delta}\)},\\
% %             
            w\left(\displaystyle\frac{x}{1-\delta}\right) \,\  &\text{if \(x \in B^{3}_{1-\delta}\)}.
        \end{cases}
    \]
    Observe that $u \in W^{1,p}(B^{3}_{1}, \mathcal{N})$ and $\tr_{\mathbb{S}^{2}}(u)=g$ (see \eqref{connmapcon}). Using this, \eqref{estonconnmap} and \eqref{estimate coming from lifting}, we have
    \begin{equation}\label{twocases}
        \int_{B^{3}_{1}}\frac{|Du|^{p}}{p}\diff x
        \leq  C^{\prime}\left(\int_{\mathbb{S}^{2}}\frac{|D_{\top}g|^{p}}{p}\diff \mathcal{H}^{2}\right)^{1-\frac{1}{p}} + C^{\prime}\delta\int_{\mathbb{S}^{2}}\frac{|D_{\top}g|^{p}}{p}\diff \mathcal{H}^{2},
    \end{equation}
    where we assume that $C^{\prime}$ is large enough so that \eqref{estonconnmap} is valid with $C=C^{\prime}$. Now we need to distinguish between two further
    cases.\\
    \emph{Case~1:} If $\int_{\mathbb{S}^{2}}\frac{|D_{\top}g|^{p}}{p}\diff \mathcal{H}^{2}\leq 1$, then $\int_{\mathbb{S}^{2}}\frac{|D_{\top}g|^{p}}{p}\diff \mathcal{H}^{2}\leq (\int_{\mathbb{S}^{2}}\frac{|D_{\top}g|^{p}}{p}\diff \mathcal{H}^{2})^{1-\frac{1}{p}}$, which, in view of \eqref{twocases}, yields
    \begin{equation}\label{estcase1}
        \int_{B^{3}_{1}}\frac{|Du|^{p}}{p}\diff x \leq  2C^{\prime}\left(\int_{\mathbb{S}^{2}}\frac{|D_{\top}g|^{p}}{p}\diff \mathcal{H}^{2}\right)^{1-\frac{1}{p}}.
    \end{equation}
    \noindent 
    \emph{Case~2:}  When \(\int_{\mathbb{S}^{2}}\frac{|D_{\top}g|^{p}}{p}\diff \mathcal{H}^{2}> 1\), we need to fix $\delta$ so that $\delta^{2-p}=\mathrm{const}$ and $\delta\leq 2-p$. This holds for
    \begin{equation}\label{const&delta}
        \mathrm{const}=\left(\frac{1}{e}\right)^{\frac{1}{e}}\,\  \text{and} \,\ \delta = \left(\frac{1}{e}\right)^{\frac{1}{e(2-p)}},
    \end{equation}
    since $\inf_{x \in (0,+\infty)} x^{x}=(1/e)^{1/e}$. 
    This defines our $\delta$ and hence $\eta_{0}$ (see \eqref{defeta}). 
    Using \eqref{estimation key prop bilipschitz}, \eqref{defeta}, \eqref{twocases}, \eqref{const&delta} and the fact that $\int_{\mathbb{S}^{2}}\frac{|D_{\top}g|^{p}}{p}\diff \mathcal{H}^{2}> 1$, we obtain
    \begin{equation}
    \label{estcase2}
    \begin{split}
        \int_{B^{3}_{1}}\frac{|Du|^{p}}{p}\diff x &\leq C^{\prime}\left(\int_{\mathbb{S}^{2}}\frac{|D_{\top}g|^{p}}{p}\diff \mathcal{H}^{2}\right)^{1 - \frac{1}{p}} + \frac{C^{\prime}\alpha\delta^{3-p}}{2-p} \\&\leq C^{\prime \prime}\left(\int_{\mathbb{S}^{2}}\frac{|D_{\top}g|^{p}}{p}\diff \mathcal{H}^{2}\right)^{1 - \frac{1}{p}}+C^{\prime \prime} \\&\leq 2C^{\prime \prime}\left(\int_{\mathbb{S}^{2}}\frac{|D_{\top}g|^{p}}{p}\diff \mathcal{H}^{2}\right)^{1 - \frac{1}{p}}, 
    \end{split}
    \end{equation}
    where $C^{\prime \prime}=C^{\prime \prime}(p_{0}, \mathcal{N})>0$. 
    The estimates \eqref{estcase1} and \eqref{estcase2} imply then, up to a scaling, \eqref{beautiful inequality bilipschitz}. This completes our proof of Lemma~\ref{key tool bilipschitz}.
    \end{proof}

Now we prove our $\eta$-compactness lemma.

\begin{lemma}
\label{heart} 
Let $\kappa \in (0,1)$, $p_{0}\in (1,2)$, $p \in [p_{0},2)$ and let $E = \Psi(B^3_r) \subset \mathbb{R}^3$ be an open set, where \(\Psi : B^3_r \to \mathbb{R}^3\) is a bilipschitz homeomorphism onto its image with a bilipschitz constant $L\geq 1$. 
There exist $\eta, C>0$, depending only on $\kappa, p_{0}$, $L$ and $\mathcal{N}$, such that if $u_{p} \in W^{1,p}(E, \mathcal{N})$ is a $p$-minimizer and satisfies 
    \begin{equation}
    \label{nice condition}
        \int_{E}\frac{|D u_{p}|^{p}}{p}\diff x \leq \frac{\eta r^{3-p}}{2-p},
    \end{equation}
    then 
    \begin{equation}
    \label{nice decay}
        \int_{\Psi (B^3_{\kappa r})}\frac{|D u_{p}|^{p}}{p}\diff x \leq C r^{3-p}.
    \end{equation}
\end{lemma}
\begin{rem}
Let us point out that the dependence on $p$ in Lemma~\ref{heart} is essential, since it is used in Lemma~\ref{lem concentration}.
\end{rem}
\begin{proof}[Proof of Lemma~\ref{heart}]
    We define \(v_p = u_p \circ \Psi\),
    so that, by the chain rule, we have \(v_p \in W^{1, p} (B^3_r, \mathcal{N})\) and 
    \begin{equation}
    \label{eq_rei2soh6aihahBohyeihie2p}
      \int_{B^3_r} \frac{\vert D v_p\vert^p}{p} \diff x
      \le L^{3 + p}\int_{E} \frac{\vert D u_p\vert^p}{p} \diff x.
    \end{equation}
    Let $\eta_{0}>0$ be the constant of Lemma~\ref{key tool bilipschitz}.  We define
    \[
    G_{p}
    =
        \left\{
            \varrho \in (\kappa r, r): 
            \operatorname{tr}_{\partial B^3_\varrho}(v_p)=v_p|_{\partial B^3_\varrho} \in W^{1,p}(\partial B^3_{\varrho}, \mathcal{N})
            \text{ and }\int_{\partial B^3_\varrho}\frac{|D v_p|^{p}}{p}\diff \mathcal{H}^{2} \leq \frac{ \eta_0 (\kappa r)^{2-p}}{2-p}\right\}.
    \]
    In view of \eqref{eq_rei2soh6aihahBohyeihie2p} and \eqref{nice condition},
    \begin{equation}
    \label{eq_ib2soquohGheicoogohp0im7}
     \mathcal{H}^1 ((\kappa r, r) \setminus G_p)
     \le \frac{2 - p}{\eta_0 (\kappa r)^{2- p}} \int_{B^3_r} \frac{\vert D v_p\vert ^{p}}{p}  \diff x
     \le \frac{L^{3+p} \eta r}{\eta_{0} \kappa^{2-p}} \le \frac{L^{5} \eta r}{\eta_{0} \kappa}.
    \end{equation}
    Choosing 
    \[
     \eta \coloneqq \frac{\eta_0 \kappa (1 - \kappa)}{2 L^{5}}
    \]
    and using \eqref{eq_ib2soquohGheicoogohp0im7}, we obtain that 
    \begin{equation}
    \label{nu9n3459h4598h489ht}
        \mathcal{H}^1 (G_p) \ge \frac{(1 - \kappa)r}{2}.
    \end{equation}
    If $\varrho \in G_{p}$, then 
    \[
    \int_{\partial B^3_{\varrho}}\frac{|D v_p|^{p}}{p}\diff \mathcal{H}^{2}\leq \frac{\eta_{0} (\kappa r)^{2-p}}{2-p}
    \le \frac{\eta_{0} \varrho^{2-p}}{2-p}.
    \]
    Thus, letting \(v_p^\varrho : B^3_\varrho\to \mathcal{N}\) be the extension of \(v_p \vert_{\partial B^3_\varrho}\) given by Lemma~\ref{key tool bilipschitz}, 
    we have by the \(p\)-minimizing property and the bilipschitz invariance,
    \begin{equation}\label{penergydecay}
    \begin{split}
    L^{-3 - p} \int_{B^3_\varrho} \frac{\lvert D v_p \rvert^p}{p} \diff x 
    & \le 
       \int_{\Psi(B^3_\varrho)} \frac{\lvert D u_p \rvert^p}{p} \diff x 
       \le 
       \int_{\Psi(B^3_\varrho)} \frac{\lvert D (v_p^\varrho \circ \Psi^{-1}) \rvert^p}{p} \diff x\\ 
       &\le 
         L^{3 + p} 
         \int_{B^3_\varrho} \frac{\lvert D v^\varrho_p \rvert^p}{p} \diff x 
         \le C L^{3 + p} \varrho^{\frac{2}{p}}\biggl(\int_{\partial B^3_\varrho}\frac{|Dv_p|^{p}}{p}\diff \mathcal{H}^{2}\biggr)^{1 - \frac{1}{p}}
    \end{split}
    \end{equation}
    for each $\varrho\in G_{p}$, where $C=C(p_{0},\mathcal{N})>0$. 
    We define now the function \(F : [\kappa r, r] \to [0, \infty)\) by
    \[
     F (\varrho) = \int_{B^3_{\varrho}}\frac{|Dv_p|^{p}}{p}\diff x.
    \]
    Without loss of generality, we assume that $F(\kappa r)>0$, because otherwise \eqref{nice decay} follows.
    The function $F$ is absolutely continuous and for every \(\varrho \in G_p\), we have by \eqref{penergydecay},
    \[
      F'(\varrho) = \int_{\partial B^3_\varrho}\frac{|D v_p|^{p}}{p}\diff \mathcal{H}^{2}
      \ge c \frac{F (\varrho)^\frac{p}{p - 1}}{\varrho^{\frac{2}{p - 1}}},
    \]
    and thus 
    \[
    \begin{split}
      \frac{p - 1}{F(\kappa r)^\frac{1}{p - 1}}
      \ge \frac{p - 1}{F(\kappa r)^\frac{1}{p - 1}} - 
      \frac{p - 1}{F(r)^\frac{1}{p - 1}}
      = \int_{\kappa r}^{r} \frac{F'(\varrho)}{F (\varrho)^\frac{p}{p - 1}} \diff \varrho 
      &\ge c \int_{G_\varrho} \frac{\diff \varrho}{\varrho^\frac{2}{p - 1}} \\
      &\ge c \frac{\mathcal{H}^1 (G_\varrho)}{r^\frac{2}{p-1}}
      \ge \frac{c (1 - \kappa)}{2 r^{\frac{3 - p}{p - 1}}},
    \end{split}
    \]
    where $c=c(p_{0}, L, \mathcal{N})>0$ and the last estimate comes from \eqref{nu9n3459h4598h489ht}. Then the estimate \eqref{nice decay} follows through a bilipschitz change of variable, completing the proof of Lemma~\ref{heart}.
\end{proof}

\section{Convergence to a harmonic map}

The following variant (Lemma~\ref{lemma_Luckhaus}) of the Luckhaus lemma (see \cite[Lemma~1]{Luckhaus_1988}, \cite[Lemma~3]{Luckhaus_1993}) is our key tool in the proof of Proposition~\ref{prop conv to harmonic map} of the fact that a weakly convergent sequence of minimizing $p$-harmonic maps converges to a minimizing 2-harmonic map when $p\nearrow 2$. 
It says that if two $W^{1,p}$ mappings from a compact 2-dimensional Riemannian manifold  $\mathcal{M}$ without boundary into $\mathcal{N}$ are close in $L^{p}$, then we can find a map $w$ in $\mathcal{M}\times (0,T)$, having these two mappings as boundary values, such that $w$ does not leave a tubular neighborhood of $\mathcal{N}$ and such that its $W^{1,p}$ norm is of order $T$.

\begin{lemma}\label{lemma_Luckhaus}
 Let $\mathcal{M}$ be a compact Riemannian manifold of dimension~2 without boundary and $p\in (1,2]$. Then there exist positive constants $C$, $T_{0}$, depending only on $\mathcal{M}$, such that for every $u,v \in W^{1,p}(\mathcal{M}, \mathcal{N})$ and $T\in (0,T_{0})$ there exists a map $w\in W^{1,p}(\mathcal{M}\times (0,T), \mathbb{R}^{\nu})$ satisfying 
 \begin{enumerate}[label=(\roman*)]
     \item $\operatorname{tr}_{\mathcal{M}\times \{0\}}(w) = u, \,\ \operatorname{tr}_{\mathcal{M}\times \{T\}}(w)=v$;
     \item 
     \label{it_eepogaechuwaeghee3UlooPo}
     $ \displaystyle \int_{\mathcal{M}\times (0,T)}|Dw|^{p}\diff \mathcal{H}^{2}\diff t \leq  CT\int_{\mathcal{M}}\left(|D_{\top}u|^{p}+|D_{\top}v|^{p}+\frac{|u-v|^{p}}{T^{p}}\right)\diff \mathcal{H}^{2}$;
     \item 
     \label{it_Ahohtev2do6yee5bohk9Aw5O}
     \begin{align*}\displaystyle \ess_{\mathcal{M}\times (0,T)}& \dist(w,\mathcal{N}) 
      \leq C \norm{u - v}_{L^\infty (\mathcal{M})}\\
     \\ \leq \,\ & CT^{1-\frac{2}{p}}\left(\int_{\mathcal{M}}\left(|D_{\top}u|^{p}+|D_{\top}v|^{p}+\frac{|u-v|^{p}}{T^{p}}\right)\diff \mathcal{H}^{2}\right)^{\frac{1}{p}\left(1-\frac{1}{2p}\right)}\left(\int_{\mathcal{M}}\frac{|u-v|^{p}}{T^{p}}\diff \mathcal{H}^{2}\right)^{\frac{1}{p}\frac{1}{2p}}.
         \end{align*}
 \end{enumerate}
\end{lemma}
\begin{proof}
  For a proof, we refer the reader to \cite[Lemma~3]{Luckhaus_1993} (see also \cite[Lemma~1]{Luckhaus_1988} for the original Luckhaus lemma, where $\mathcal{M}=\mathbb{S}^{2}$). Notice that the inductive assumption in \cite[Lemma~3]{Luckhaus_1993} may not be satisfied. However, if one slightly modifies the Luckhaus proof, namely, one defines $w(x,t)=\frac{t}{T}v(x)+\left(1-\frac{t}{T}\right)u(x)$ for $x$ belonging to the closure of the $l$-skeleton of $\mathcal{M}$ in the case when either $[p]<p$ and $l=[p]$ or $[p]=p$ and $l=[p-1]$, then the inductive assumption in \cite[Lemma~3]{Luckhaus_1993} is satisfied, and the proof works as written (observe that Luckhaus defines $w$ as above for $x$ belonging to the closure of the $l$-skeleton of $\mathcal{M}$ in the case when $l=[p-1]$).  It is also possible to reduce the dependence of the constant $C$ in \cite[Lemma~3]{Luckhaus_1993} on $\dim(\mathcal{N})$. 
\end{proof}

\begin{prop}\label{prop conv to harmonic map}
    Let $U\subset \mathbb{R}^{3}$ be an open set. Let $(p_{n})_{n\in \mathbb{N}} \subset [1,2)$, $u_{n} \in W^{1,p_{n}}(U, \mathcal{N})$ and $p_{n}\nearrow 2$ as $n\to +\infty$. 
    If for each open set $K\csubset U$ and each \(n \in \mathbb{N}\), $u_{n}$ is a $p_{n}$-minimizer in $K$ and if 
    \begin{equation}\label{C1}
    \sup_{n \in \mathbb{N}}
    \int_{K}\frac{|Du_{n}|^{p_{n}}}{p_{n}}\diff x\leq C,
    \end{equation}
    where $C=C(K)>0$, then there exists $u_{*} \in W^{1,2}_{\loc}(U, \mathcal{N})$ such that, up to a subsequence (not relabeled), the following assertions hold.
    \begin{enumerate}[label=(\roman*)]
    \item 
    \label{item:weakLp} 
    $u_{n}$ converges to $u_{*}$ weakly in $W^{1,p}_{\loc}(U, \mathbb{R}^{\nu})$ for all $p\in (1,2)$ and a.e.\  in $U$.
    \item \label{item2convtoharmproposition} For each open set $K\csubset U$,    
    \[
    \int_{K}\frac{|Du_{n}|^{p_{n}}}{p_{n}}\diff x\to \int_{K}\frac{|Du_{*}|^{2}}{2}\diff x
    \]
    as $n\to +\infty$. Furthermore, $\int_{K}|Du_{n}-Du_{*}|^{p_{n}}\diff x \to 0$ and $Du_{n}\to Du_{*}$ a.e.\  in $U$ as $n\to +\infty$. 
    \item  \label{item3convharm} If $K\csubset U$ is open, then $u_{*}$ is a $2$-minimizer in $K$. Furthermore, there exists at most a countable and locally finite subset $S_{0}$ of $U$ such that $u_{*}\in~C^{\infty}(U\setminus S_{0}, \mathcal{N})$. 
    \end{enumerate} 
\end{prop}

\begin{proof}
    Fix an open set $K\csubset U$. Since $\overline{K}\subset U$ is compact, for any cover by open balls, $\overline{K}$ admits a finite subcover. Then there exists an open set $V$ such that $\overline{K} \subset V$, $\overline{V}\subset U$ and $\overline{V}$ is a compact smooth Riemannian manifold of dimension~3 with boundary \(\partial V\) (possibly disconnected, but with a finite number of connected components). 
    Then $\partial V$ is a 2-dimensional compact smooth Riemannian manifold without boundary. According to Lemma~\ref{lemma npr} with $\mathcal{N}$ replaced by $\partial V$, there exists $\lambda_{\partial V}>0$ such that the nearest point retraction $\Pi_{\partial V}: \partial V_{\lambda_{\partial V}}\to \partial V$ is a  well-defined and smooth mapping, where $\partial V_{\lambda_{\partial V}}=\{x\in \overline{V}: \dist(x, \partial V)\leq \lambda_{\partial V}\}$. Choose $\lambda \in (0,\lambda_{\partial V})$ so that $\overline{K}\subset V\setminus \overline{A}_{\lambda}$, where $A_{\lambda}=\{x \in V: \dist(x,\partial V)<\lambda\}$. 
    
    If $p_{n}\in (p,2)$ is large enough, then, applying Young's inequality, we have
    \begin{equation}\label{eq_kogashohzoolej2LeeTh9Iel}
        \int_{V}\frac{|Du_{n}|^{p}}{p}\diff x
        \leq 
            \int_{V}\frac{|Du_{n}|^{p_{n}}}{p_{n}} \diff x 
            + 
            \Bigl(\frac{1}{p} - \frac{1}{p_n}\Bigr)|V| 
    \end{equation}
    so that 
    \begin{equation}
    \label{estimboundpenr}
     \limsup_{n \to +\infty} \int_{V}\frac{|Du_{n}|^{p}}{p}\diff x < \infty
    \end{equation}
    (see \eqref{C1}). Since $\mathcal{N}$ is compact, $(u_{n})_{n\in \mathbb{N}}$ is bounded in $L^{\infty}(U, \mathbb{R}^{\nu})$, which, together with \eqref{estimboundpenr}, implies that for each $p\in (1,2)$, $(u_{n})_{n\in \mathbb{N}}$ is bounded in $W^{1,p}(V, \mathbb{R}^{\nu})$. Thus, there exists a map $u_{*}: V \to \mathbb{R}^{\nu}$ such that, up to a subsequence (still denoted by $n$), for each $p \in (1,2)$,
    \begin{align}
        \label{conv for seq}
        u_{n} & \rightharpoonup u_{*}\text{ weakly in $W^{1,p}(V, \mathbb{R}^{\nu})$}&
        &\text{ and } &
        u_{n}&\to u_{*}\text{ a.e.\ in}\,\ V.
    \end{align}
    In particular, \(u_{*} \in \mathcal{N}\) a.e. in \(V\).
    
    By the weak convergence, \eqref{eq_kogashohzoolej2LeeTh9Iel} and the fact that $p_{n}\nearrow 2$ as $n\to +\infty$, for each $p \in (1,2)$ we have 
    \begin{equation}
    \label{estimate p-energy weak limit}
    \int_{V}\frac{|Du_{*}|^{p}}{p}\diff x \leq \liminf_{n\to+\infty}\int_{V} \frac{|Du_{n}|^{p}}{p}\diff x  \leq \liminf_{n\to +\infty}  \int_{V}\frac{|Du_{n}|^{p_{n}}}{p_{n}} \diff x 
    + \Bigl(\frac{1}{p} - \frac{1}{2}\Bigr)|V|.
    \end{equation}
    Letting \(p\nearrow 2\) in \eqref{estimate p-energy weak limit} and using Fatou's lemma, we obtain the estimate
    \begin{equation}
    \label{eq_Dahshie8EedeeGhae4ohdiec}
      \int_{V}\frac{|Du_{*}|^{2}}{2}\diff x \leq\liminf_{n\to +\infty}  \int_{V}\frac{|Du_{n}|^{p_{n}}}{p_{n}} \diff x,
    \end{equation}
    which, together with \eqref{C1}, implies that \(u_{*}\in W^{1,2}(V, \mathcal{N})\). We prove that $u_{*}$ is a 2-minimizer in $K$  and that, up to a subsequence (not relabeled),   
    \begin{equation}
    \label{contoh}
    \int_{K}\frac{|Du_{n}|^{p_{n}}}{p_{n}}\diff x \to \int_{K} \frac{|Du_{*}|^{2}}{2}\diff x
    \end{equation} 
    and $\|Du_{n}-Du_{*}\|^{p_{n}}_{L^{p_{n}}(K, \mathbb{R}^{\nu} \otimes \mathbb{R}^{3})}\to 0$ as $n\to +\infty$. 

    Applying \cite[Theorem~3.2.22~(3)]{Federer} with $W=A_{\lambda}$, $Z=[0,\lambda]$ and $f:W\to Z$ defined by $f(x)=|x-\Pi_{\partial V}(x)|=\dist(x,\partial V)$, and using also the fact that $|Df(x)| = 1$ for each  $x \in A_{\lambda}$ (see Lemma~\ref{lemma npr}), Fatou's lemma and \eqref{C1}, we obtain
    \begin{equation*}
    \int^{\lambda}_{0}\diff \varrho \;\liminf_{n\to +\infty} \int_{f^{-1}(\{\varrho\})}\frac{|Du_{n}|^{p_{n}}}{p_{n}}\diff \mathcal{H}^{2} \leq \liminf_{n\to +\infty} \int_{A_{\lambda}}\frac{|Du_{n}|^{p_{n}}}{p_{n}} \diff x\leq C.
    \end{equation*}
    This implies that 
    \[
        \mathcal{H}^{1}
            \Bigl(\Bigr\{\varrho \in (0,\lambda): \liminf_{n\to +\infty} \int_{f^{-1}(\{\varrho\})}\frac{|Du_{n}|^{p_{n}}}{p_{n}}\diff \mathcal{H}^{2} \geq \frac{4C}{\lambda}\Bigr\}\Bigr)\leq \frac{\lambda}{4}.
    \]
    So there exists $\varrho \in (0,\lambda/2)$ such that, up to a subsequence (not relabeled), we know that for each sufficiently large $n \in \mathbb{N}$,
    \begin{equation}
        \label{estimfortr}
        \operatorname{tr}_{f^{-1}(\{\varrho\})}(u_{n})=u_{n}|_{f^{-1}(\{\varrho\})} \in W^{1, p_{n}}(f^{-1}(\{\varrho\}), \mathcal{N}), \; \int_{f^{-1}(\{\varrho\})}\frac{|Du_{n}|^{p_{n}}}{p_{n}}\diff \mathcal{H}^{2}<\frac{4C}{\lambda}
    \end{equation}
    and $\operatorname{tr}_{f^{-1}(\{\varrho\})}(u_{*})=u_{*}|_{f^{-1}(\{\varrho\})}$. Fix such $\varrho \in (0,\lambda/2)$ and an arbitrary $p\in (1,2)$. Now we define $\mathcal{M}=f^{-1}(\{\varrho\})$. Notice that $\mathcal{M}$ is a $2$-dimensional compact smooth Riemannian manifold without boundary. In view of \eqref{conv for seq} and the trace theorem, we have $u_{n}|_{\mathcal{M}} \rightharpoonup u_{*}|_{\mathcal{M}}$ weakly in $W^{1-1/p,p}(\mathcal{M}, \mathbb{R}^{\nu})$ (we use that a bounded linear operator maps a weakly convergent sequence to a weakly convergent one). Then, by the compact embedding, up to a subsequence (not relabeled),
    \begin{equation}\label{comp for tr}
     u_{n}|_{\mathcal{M}}\to u_{*}|_{\mathcal{M}} \,\ \text{strongly in} \,\ L^{q}(\mathcal{M}, \mathbb{R}^{\nu})
     \end{equation}
     for each $q \in \left[1,\frac{2p}{2-p}\right)$. Since $u_{n}|_{\mathcal{M}}\in W^{1,p_{n}}(\mathcal{M}, \mathcal{N})$, \eqref{comp for tr} implies that $u_{*}(x) \in \mathcal{N}$ for $\mathcal{H}^{2}$-a.e.\  $x \in \mathcal{M}$. Using \eqref{estimfortr} and \eqref{comp for tr},  we deduce that, up to a subsequence (not relabeled), 
    \[
    u_{n}|_{\mathcal{M}}\rightharpoonup u_{*}|_{\mathcal{M}}\,\ \text{weakly in}\,\ W^{1,p}(\mathcal{M}, \mathbb{R}^{\nu}). 
    \]
    Then, by the weak convergence, Young's inequality and \eqref{estimfortr}, we have
    \begin{equation}
      \label{(0.51)}
      \begin{split}
    \int_{\mathcal{M}}\frac{|D_{\top}u_{*}|^{p}}{p}\diff \mathcal{H}^{2} &\leq \liminf_{n\to +\infty} \int_{\mathcal{M}}\frac{|D_{\top} u_{n}|^{p}}{p}\diff \mathcal{H}^{2} \\ & \leq \liminf_{n\to+ \infty}\left (\int_{\mathcal{M}}\frac{|D u_{n}|^{p_{n}}}{p_{n}}\diff \mathcal{H}^{2} + \Bigl(\frac{1}{p}- \frac{1}{p_{n}}\Bigr) \mathcal{H}^{2}(\mathcal{M})\right) \\ 
    & \leq  \frac{4C}{\lambda}+\Bigl(\frac{1}{p}-\frac{1}{2}\Bigr)\mathcal{H}^{2}(\mathcal{M}), 
    \end{split}
\end{equation}
where we have also used that $|D_{\top}u_{n}|\leq |Du_{n}|$. 
This implies that $u_{*}|_{\mathcal{M} }\in W^{1,2}(\mathcal{M}, \mathcal{N})$, because $p \in (1,2)$ was arbitrarily chosen and we already know that $u_{*}(x) \in \mathcal{N}$ for $\mathcal{H}^{2}$-a.e.\  $x \in \mathcal{M}$. Moreover, $u_{*} \in L^{\infty}(\mathcal{M}, \mathbb{R}^{\nu})$.
Also, \eqref{(0.51)} yields
    \begin{equation}\label{E2energybound}
    \int_{\mathcal{M}}\frac{|D_{\top}u_{*}|^{2}}{2}\diff \mathcal{H}^{2}\leq 
    \lim_{p \nearrow 2} 
    \int_{\mathcal{M}}\frac{|D_{\top}u_{*}|^{p}}{p} \diff \mathcal{H}^{2}
    \leq
    \frac{4C}{\lambda}.
    \end{equation}  
        Next, we want to apply Lemma~\ref{lemma_Luckhaus}. Let $T_{0}=T_{0}(\mathcal{M})>0$ be the constant of Lemma~\ref{lemma_Luckhaus}. Choose $T\in (0,T_{0})$ sufficiently small and let $\Phi: A_{\lambda} \setminus A_\varrho \to \mathcal{M}\times [0, \lambda-\varrho)$ be defined by 
    \[
     \Phi (x) = (\Pi_{\mathcal{M}}(x), \dist (x, \mathcal{M})). 
    \]
    By Lemma~\ref{lemma npr}, there exists $\lambda_{\mathcal{N}}>0$ such that the nearest point retraction $\Pi_{\mathcal{N}}:\mathcal{N}_{\lambda_{\mathcal{N}}} \to \mathcal{N}$, where $
    \mathcal{N}_{\lambda_{\mathcal{N}}}=\{x \in \mathbb{R}^{\nu}: \dist(x, \mathcal{N})<\lambda_{\mathcal{N}}\},
    $ is well defined and smooth.  Setting $s_{n}=\|u_{n}-u_{*}\|_{L^{p_{n}}(\mathcal{M}, \mathbb{R}^{\nu})}$ (without loss of generality, we assume that for each $n \in \mathbb{N}$ large enough, $s_{n}>0$) and using \eqref{estimfortr}, \eqref{comp for tr} and \eqref{E2energybound}, we observe that $s_{n}\to 0$ and for each sufficiently large $n \in \mathbb{N}$,
    \begin{equation}\label{estimluckh}
    \int_{\mathcal{M}}\left(|D_{\top} u_{n}|^{p_{n}}+|D_{\top}u_{*}|^{p_{n}}+\frac{|u_{n}-u_{*}|^{p_{n}}}{s^{p_{n}/2}_{n}}\right) \diff \mathcal{H}^{2} \leq C^{\prime},
    \end{equation}
    where $C^{\prime}$ is a positive constant independent of $n$. Applying Lemma~\ref{lemma_Luckhaus} with $u=u_{n}|_{\mathcal{M}}$, $v=u_{*}|_{\mathcal{M}}$ and $T=s_{n}^{1/2}$, we obtain a map $w_{n}\in W^{1,p_{n}}(\mathcal{M}\times (0,s_{n}^{1/2}), \mathbb{R}^{\nu})$ interpolating between $u_{n}|_{\mathcal{M}}$ and $u_{*}|_{\mathcal{M}}$. 
    According to Lemma~\ref{lemma_Luckhaus}~\ref{it_Ahohtev2do6yee5bohk9Aw5O}, \eqref{estimluckh} and the definition of $s_{n}$, if $n \in \mathbb{N}$ is large enough, we have $\dist(w_{n}(x, t), \mathcal{N})<\lambda_{\mathcal{N}}$ for a.e.\  $(x, t) \in \mathcal{M}\times (0,s_{n}^{1/2})$. 
    Thus, for each $n \in \mathbb{N}$ large enough, we can define the map 
    \[
        \varphi_{n}(y)=\Pi_{\mathcal{N}}(w_{n}(\Phi(y))) 
    \]
    for each $y\in \Phi^{-1}(\mathcal{M}\times (0,s_{n}^{1/2}))$. 
    Then we have the following: $\varphi_{n}\in W^{1,p_{n}}(\Phi^{-1}(\mathcal{M}\times (0,s_{n}^{1/2})), \mathcal{N})$;  $\tr_{\mathcal{M}}(\varphi_{n})=u_{n}|_{\mathcal{M}}$, 
    $\tr_{\Phi^{-1}(\mathcal{M}\times\{s_{n}^{1/2}\})}(\varphi_{n})=u_{*}|_{\mathcal{M}}\circ \Phi_{n}\circ  \Phi|_{\Phi^{-1}(\mathcal{M}\times \{s_{n}^{1/2}\})}$, where $\Phi_{n}: \mathcal{M}\times [s^{1/2}_{n},T]$ is defined by $\Phi_{n}(x,t)=\Phi^{-1}(x, t-s^{1/2}_{n})$ and we have used that $\Phi^{-1}(x,0)=x$; in view of Lemma~\ref{lemma_Luckhaus}~\ref{it_eepogaechuwaeghee3UlooPo} and \eqref{estimluckh},
    \begin{equation}\label{estiml2}
    \int_{\Phi^{-1}(\mathcal{M}\times (0,s_{n}^{1/2}))}\frac{|D\varphi_{n}|^{p_{n}}}{p_{n}}\diff y \leq C^{\prime\prime} s_{n}^{1/2},
    \end{equation}
    where $C^{\prime\prime}>0$ is a constant independent of $n$, but depending on the bilipschitz constant of $\Phi$. 
    
    Now let $E = V \setminus \overline{A}_\varrho$. According to our construction, $E$ is a finite union of bounded Lipschitz domains, and $E$ contains $K$. 
    Fix an arbitrary $w_{*} \in W^{1,2}(E, \mathcal{N})$ with $\tr_{\mathcal{M}}(w_{*})=\tr_{\mathcal{M}}(u_{*})$. Let $\zeta \in C_{c}([0,+\infty))$ be the Lipschitz map such that $\zeta=1$ on $[0,1]$, $\zeta(s)=2-s$ for $s\in [1,2]$, $\zeta=0$ on $[2,+\infty]$. Hence $\|\zeta^{\prime}\|_{L^{\infty}([0, +\infty))}\leq 1$. Next, for each sufficiently large $n \in \mathbb{N}$, define $w_{n}:E\to \mathcal{N}$ by
    \begin{equation}
    \label{defwn}
    w_{n}(y)
    =\begin{cases}
    \varphi_{n}(y) \,\  &\text{if \(y \in \Phi^{-1}(\mathcal{M}\times (0,s_{n}^{1/2}])\),}\\
    \displaystyle w_{*}(\Psi_{n}(y)) \,\ &\text{if \(y \in E\setminus \Phi^{-1}(\mathcal{M}\times (0,s_{n}^{1/2}])\),}
    \end{cases}
    \end{equation}
    where $\Psi_{n}: E\setminus \Phi^{-1}(\mathcal{M}\times (0,s_{n}^{1/2}]) \to E$ is the bilipschitz map defined by 
    \[
    \Psi_{n}(y)=\left(1-\zeta\left(\frac{\dist(y,\mathcal{M})}{s_{n}^{1/3}}\right)\right) y + \zeta\left(\frac{\dist(y,\mathcal{M})}{s_{n}^{1/3}}\right)\Phi_{n}(\Phi(y)).
    \]
    Notice that: $\tr_{\Phi^{-1}(\mathcal{M}\times \{s^{1/2}_{n}\})}(w_{*} \circ \Psi_{n})= u_{*}|_{\mathcal{M}}\circ \Phi_{n}\circ \Phi|_{\Phi^{-1}(\mathcal{M}\times \{s_{n}^{1/2}\})}$; if $y\in E\setminus \Phi^{-1}(\mathcal{M}\times (0,s_{n}^{1/2}])$ and $\dist(y,\mathcal{M})\leq s_{n}^{1/3}$, then we have $\Psi_{n}(y)=\Phi_{n}(\Phi(y))$; if $y\in E\setminus \Phi^{-1}(\mathcal{M}\times (0,s_{n}^{1/2}])$ and $\dist(y,\mathcal{M})\geq 2s_{n}^{1/3}$, then $\Psi_{n}(y)=y$; for each $y\in E$, if $n \in \mathbb{N}$ is large enough, then $\Psi_{n}(y)$ is defined and $\Psi_{n}(y)\to y$ as $n\to +\infty$; for a.e.\  $y \in E\setminus \Phi(\mathcal{M}\times (0,s_{n}^{1/2}])$ and for each $n \in \mathbb{N}$ large enough,
    \begin{equation}
    \label{LipestPsin}
    \begin{split}
    |D\Psi_{n}(y) - \operatorname{Id}|&=\biggl|\zeta^{\prime}\biggl(\frac{\dist(y,\mathcal{M})}{s_{n}^{1/3}}\biggr) \frac{D\dist(y,\mathcal{M})}{s_{n}^{1/3}}\otimes (\Phi_{n}(\Phi(y))-y) \\ & \,\ \,\  \,\ \,\ \,\ \,\ + \zeta\biggl(\frac{\dist(y,\mathcal{M})}{s_{n}^{1/3}}\biggr)(D\Phi_{n}(\Phi (y)) \circ D\Phi(y) -\mathrm{Id})\biggr|\\
    &\leq  Ls_n^{1/6},
    \end{split}
    \end{equation}
    where $L>0$ is a constant independent of $n$, but depending on the bilipschitz constant of $\Phi$. In the above computations we have used that $|\zeta^{\prime}|\leq 1$, $|D\dist(y,\mathcal{M})|= 1$ and $|a\otimes b|\leq |a| |b|$ for $a,b \in \mathbb{R}^{3}$. 
    In particular, when \(n \in \mathbb{N}\) is large enough, \(|D\Psi_{n}(y) - \operatorname{Id}|<1/2\). It is worth noting that, according to the definition of $\Psi_{n}$ and \cite[4.8~(5)]{Federer_1959}, $\Psi_{n}$ is continuously differentiable on the three regions of  $E\setminus \Phi^{-1}(\mathcal{M}\times (0,s_{n}^{1/2}])$, where $\dist(\cdot, \mathcal{M})< s^{1/3}_{n}$, $\dist(\cdot, \mathcal{M}) \in (s^{1/3}_{n}, 2s^{1/3}_{n})$, and $\dist(\cdot, \mathcal{M})> 2s^{1/3}_{n}$, respectively. Then, using the inverse function theorem for $\Psi_{n}$ on each of these three regions (namely, on compact subsets of these regions), and taking into account the definition of $\Psi_{n}$ (namely, the injectivity of $\Psi_{n}$), one observes that $\Psi_{n}$ admits the inverse $\Psi^{-1}_{n}$ such that $\Psi^{-1}_{n}$ is a.e. differentiable on $E$ and $D\Psi^{-1}_{n}$ is sufficiently close to the identity a.e. on $E$. Furthermore, $\Psi^{-1}_{n}$ is absolutely continuous on each segment lying in the closure of one of the images of $\Psi_{n}$ of the three  regions of $E\setminus \Phi^{-1}(\mathcal{M}\times (0,s_{n}^{1/2}])$ considered above. Thus, $\Psi_{n}$ is bilipschitz on  $E\setminus \Phi^{-1}(\mathcal{M}\times (0,s_{n}^{1/2}])$, as we mentioned earlier. 
    It is also worth noting that for a.e. $y \in E$, if $n$ is large enough, then $D\Psi_{n}(y)$ is defined, $D(w_{*}\circ\Psi_{n})(y)=Dw_{*}(\Psi_{n}(y))\circ D\Psi_{n}(y)$ and $D\Psi_{n}(y)\to \mathrm{Id}$, $D(w_{*}\circ \Psi_{n})(y) \to Dw_{*}(y)$  as $n\to +\infty$. In addition, if $n$ is large enough, then for a.e.\  $y\in E$,
    \begin{align*}
        \frac{|D(w_{*}\circ \Psi_{n})(y)1_{E\setminus \Phi^{-1}(\mathcal{M}\times (0,s_{n}^{1/2}])}(y)|^{p_{n}}}{p_{n}} &\leq 2^{p_{n}}\frac{|Dw_{*}(\Psi_{n}(y)) 1_{E\setminus \Phi^{-1}(\mathcal{M}\times (0,s_{n}^{1/2}])}(y)|^{p_{n}}}{p_{n}}\\
        &\leq 2^{p_{n}}\Biggl(\frac{|Dw_{*}(\Psi_{n}(y))1_{E\setminus \Phi^{-1}(\mathcal{M}\times (0,s_{n}^{1/2}])}(y)|^{2}}{2}+\frac{1}{p_{n}}- \frac{1}{2}\Biggr)\\
        &\leq h(y), 
    \end{align*}
    where we have used \eqref{LipestPsin}, the fact that $|\mathrm{Id}|=\sqrt{3}$, Young's inequality and since $|Dw_{*}|^{2}\in L^{1}(E)$, there exists $h \in L^{1}(E)$ such that the last estimate holds.
    
    Taking into account the above observations and using the Lebesgue dominated convergence theorem, we deduce that 
    \begin{equation}\label{Ldomwn}
        \int_{E\setminus \Phi^{-1}(\mathcal{M}\times (0,s_{n}^{1/2}])}\frac{|Dw_{n}|^{p_{n}}}{p_{n}}\diff y \to \int_{E}\frac{|Dw_{*}|^{2}}{2}\diff y
    \end{equation}
as $n\to+\infty$. The mapping $w_{n}$ is a competitor for $u_{n}$ in $E$, and hence, in view of \eqref{defwn},
    \begin{equation}\label{passing to limit}
    \int_{E}\frac{|Du_{n}|^{p_{n}}}{p_{n}}\diff y \le
    \int_{\Phi^{-1}(\mathcal{M}\times (0,s_{n}^{1/2}))}\frac{|D\varphi_{n}|^{p_{n}}}{p_{n}}\diff y+\int_{E\setminus \Phi^{-1}(\mathcal{M}\times (0,s_{n}^{1/2}])}\frac{|D w_{n}|^{p_{n}}}{p_{n}}\diff y.
    \end{equation}
    Using  \eqref{eq_Dahshie8EedeeGhae4ohdiec} with $V$ replaced by $E$ (which is possible thanks to \eqref{conv for seq} and proceeding as in \eqref{estimate p-energy weak limit}), together with \eqref{estiml2}, \eqref{Ldomwn} and \eqref{passing to limit}, we obtain
    \begin{align*}
    \int_{E}\frac{|Du_{*}|^{2}}{2}\diff y=\lim_{p\nearrow 2} \int_{E}\frac{|Du_{*}|^{p}}{p} \diff y\leq  \liminf_{n\to +\infty} \int_{E}\frac{|Du_{n}|^{p_{n}}}{p_{n}}\diff y&\leq  \limsup_{n\to +\infty} \int_{E}\frac{|Du_{n}|^{p_{n}}}{p_{n}}\diff y\\&\leq \int_{E}\frac{|Dw_{*}|^{2}}{2} \diff y.
    \end{align*}
    This implies that $\int_{E}\frac{|Du_{n}|^{p_{n}}}{p_{n}}\diff y \to \int_{E}\frac{|Du_{*}|^{2}}{2}\diff y$ and $u_{*}$ is a 2-minimizer in $E$. In particular, according to Lemma~\ref{rem norm convergence} below, it holds \[\|Du_{n}-Du_{*}\|^{{p_{n}}}_{L^{p_{n}}(E, \mathbb{R}^{\nu}\otimes \mathbb{R}^{3})}\to 0,\] 
    which proves \eqref{contoh}. 
    
    Since $u_{*}$ is a $2$-minimizer in $E$ and $K\csubset E$ is open, $u_{*}$ is a $2$-minimizer in $K$ and, according to \cite[Theorem II]{SchoUhl},  there exists a finite set $S_{K} \subset K$ such that $u_{*}\in C^{\infty}(K\setminus S_{K}, \mathcal{N})$. 
    
    In order to reach the conclusion, we use a diagonal argument and rely on the fact that one can write \(U = \bigcup_{m \in \mathbb{N}} K_m\), with \(K_m \subset K_{m + 1}\) so that if \(K \csubset U\), then \(K \subset K_m\) for some \(m \in \mathbb{N}\). This completes our proof of Proposition~\ref{prop conv to harmonic map}.
\end{proof}
\begin{lemma}\label{rem norm convergence}
    Let $U\subset \mathbb{R}^{3}$ be open and bounded, $u_{*} \in W^{1,2}(U, \mathbb{R}^{\nu})$, $(p_{n})_{n\in \mathbb{N}}\subset [1,2)$, $p_{n}\nearrow 2$ as $n\to +\infty$ and  $(u_{n})_{n\in \mathbb{N}}\subset W^{1,p_{n}}(U, \mathbb{R}^{\nu})$. If $u_{n} \rightharpoonup u_{*}$ weakly in $W^{1,p}(U, \mathbb{R}^{\nu})$ for all $p\in  (1,2)$ and $\int_{U}|Du_{n}|^{p_{n}}\diff x \to \int_{U}|Du_{*}|^{2}\diff x$ as $n\to +\infty$, then  $\int_{U}|Du_{n}-Du_{*}|^{p_{n}}\diff x \to 0$ as $n\to +\infty$.
\end{lemma}

\begin{proof}
     The desired convergence comes by using Hanner's inequality, see \cite[Lemma~6.3]{VanSchaftingen_VanVaerenbergh}. To apply \cite[Lemma~6.3]{VanSchaftingen_VanVaerenbergh}, we only need to prove that
     \begin{equation}\label{condforlemma6.3} 
         4\|Du_{*}\|^{2}_{L^{2}(U, \mathbb{R}^{\nu}\otimes \mathbb{R}^{3})}\leq \liminf_{n\to +\infty}\|Du_{n}+Du_{*}\|^{p_{n}}_{L^{p_{n}}(U, \mathbb{R}^{\nu}\otimes \mathbb{R}^{3})}.
     \end{equation}
 Let $p\in (1,2)$ be arbitrary. Notice that $Du_{n}+Du_{*}\rightharpoonup 2Du_{*}$ weakly in $L^{p}(U,\mathbb{R}^{\nu}\otimes \mathbb{R}^{3})$ and hence
 \begin{equation}\label{estimconvexweak}
  \liminf_{n\to+\infty}\|Du_{n}+Du_{*}\|_{L^{p}(U, \mathbb{R}^{\nu} \otimes \mathbb{R}^{3})}\geq 2 \|Du_{*}\|_{L^{p}(U,\mathbb{R}^{\nu}\otimes \mathbb{R}^{3})}.
 \end{equation}
 Using H\"older's inequality and \eqref{estimconvexweak}, we obtain 
 \begin{equation}
   \label{estimfrmwconv}
   \begin{split}
     \liminf_{n\to +\infty}\|Du_{n}+Du_{*}\|^{p_{n}}_{L^{p_{n}}(U, \mathbb{R}^{\nu}\otimes \mathbb{R}^{3})} &\geq  \liminf_{n\to +\infty} |U|^{1-\frac{p_{n}}{p}} \|Du_{n}+Du_{*}\|^{p_{n}}_{L^{p}(U, \mathbb{R}^{\nu}\otimes \mathbb{R}^{3})}\\
     &\geq 4|U|^{1-\frac{2}{p}}\|Du_{*}\|^{2}_{L^{p}(U, \mathbb{R}^{\nu}\otimes \mathbb{R}^{3})}.
     \end{split}
\end{equation}    
Letting $p \nearrow 2$ and using the continuity of the $L^{p}$-norm, in view of \eqref{estimfrmwconv}, we obtain \eqref{condforlemma6.3}, which completes our proof of Lemma~\ref{rem norm convergence}.
\end{proof}

 \section{The singular set}
    \label{section_singular_set}
    
    Let $p \in (1,2)$ and $u_{p} \in W^{1,p}(\Omega, \mathcal{N})$ be a $p$-minimizer in $\Omega$. For each Borel-measurable set $E\in \mathcal{B}(\overline{\Omega})$, we define the positive Radon measure $\mu_{p}$ at $E$ by
    \[
    \mu_{p}(E)\coloneqq (2-p)\int_{E}\frac{|Du_{p}|^{p}}{p}\diff x.  
    \]

Let $(p_{n})_{n\in \mathbb{N}} \subset [1,2)$, $p_{n}\nearrow 2$ as $n \to +\infty$ and $(u_{n})_{n\in \mathbb{N}}$ be a sequence of $p_{n}$-minimizers in $\Omega$.  We assume that for each $n \in \mathbb{N}$ large enough, setting $\mu_{n}=\mu_{p_{n}}$, it holds
\begin{equation}\label{C2}
\mu_{n}(\overline{\Omega}) \leq C_{0},
\end{equation}
where $C_{0}>0$ is a constant independent of $n$. According to Proposition~\ref{global energy bound}, the condition \eqref{C2} is satisfied whenever for each $n \in \mathbb{N}$ large enough, $\operatorname{tr}_{\partial \Omega}(u_{n})=g \in W^{1/2,2}(\partial \Omega, \mathcal{N})$.  Next, in view of \eqref{C2}, using the Banach-Alaoglu theorem, we obtain a positive Radon measure $\mu_{*}\in (C(\overline{\Omega}))^{\prime}$ such that, up to a subsequence (not relabeled),
\begin{equation}\label{wcm}
\mu_{n} \overset{*}{\rightharpoonup} \mu_{*} \,\ \text{weakly* in}\,\ (C(\overline{\Omega}))^{\prime}.
\end{equation}
We denote the support of $\mu_{*}$ by $S_{*} \subset \overline{\Omega}$. 

\begin{rem}\label{rem global convergence}
If the measure $\mu_{*}$ is independent of the subsequence, the entire  sequence $(\mu_{n})_{n\in \mathbb{N}}$ weakly* converges to $\mu_{*}$ as $n \to +\infty$, since in the latter case each subsequence of $(\mu_{n})_{n\in \mathbb{N}}$ contains a subsequence weakly* converging to $\mu_{*}$.
In general, the limit measure $\mu_{*}$ does depend on the subsequence.
To see this, we consider the projective plane $\mathcal{N}=\mathbb{RP}^{2}$, whose universal covering \(\pi \colon \mathbb{S}^2 \to \mathbb{RP}^2\) identifies antipodal points  (see e.g. \cite[Example~0.4]{AT}), so that, in particular, $\pi_{1}(\mathbb{RP}^2)\simeq \mathbb{Z}/2\mathbb{Z}$.
Defining first
\(
 \Hat{g} \colon \mathbb{S}^2 \to \mathbb{S}^1
\)
by
\[
 \Hat{g} \brk{x} =
 \frac{\brk{x_1 x_2, x_3}}{\brk{x_1^2 + x_3^2}^{1/2}\brk{x_2^2 + x_3^2}^{1/2}},
\]
(this is well defined since \(x_1^2 x_2^2 + x_3^2 = x_1^2x_2^2 + 1 - x_1^2-x^{2}_{2} = \brk{1 - x_1^2} \brk{1 - x_2^2} = \brk{x_2^2 + x_3^2} \brk{x_1^2 + x_3^2}\))
one checks that  \(\Hat{g} \brk{x_1, x_2, x_3} = \Hat{g} \brk{x_2, x_1, x_3} = -\Hat{g} \brk{x_1, -x_2, -x_3}\), that \(\Hat{g}\) has degree \(1\) around small circles around \(\brk{1, 0, 0}\) and \(\brk{-1, 0, 0}\) and degree \(-1\) around small circles around \(\brk{0, 1, 0}\) and \(\brk{0, -1, 0}\), and that \(\Hat{g} \in W^{1/2, 2} \brk{\mathbb{S}^2, \mathbb{S}^1}\).
We define
\(g = \tau \circ \Hat{g}\),
where the map \(\tau \colon \mathbb{S}^1 \to \mathbb{RP}^2\) defined by the condition that
\(\tau \brk{\cos  \brk{\theta}, \sin  \brk{\theta}} = \pi\brk{0, \cos \brk{\frac{\theta}{2}}, \sin \brk{\frac{\theta}{2}}}\).
We have \(g \in W^{1/2, 2} \brk{\mathbb{S}^2, \mathbb{RP}^2}\); since \(\tau\) is not homotopic to a constant, \(g\) is topologically nontrivial in small circles around the same points as \(\Hat{g}\); we have \(g \brk{x_1, x_2, x_3} = g \brk{x_2, x_1, x_3} =
\sigma \brk{g \brk{x_1, -x_2, -x_3}}\),
where \(\sigma \colon \mathbb{RP}^2 \to \mathbb{RP}^2\) is the isometry characterized
by \(\sigma \brk{\pi \brk{y_1, y_2, y_3}} = \pi \brk{y_1, y_3, -y_2}\), so that \(\sigma\brk{\tau \brk{z}} = \tau \brk{-z}\).
In particular, if \(\rho \brk{x_1, x_2, x_3} = \brk{x_1, -x_2, - x_3}\) and $(u_{n})_{n\in \mathbb{N}}$ is a sequence of $p_{n}$-minimizers in $B^{3}_{1}$ having $g$ as a trace on $\mathbb{S}^{2}$, $(\sigma \circ u_{n} \circ \rho)$ is also. Define a sequence $(v_{k})_{k \in \mathbb{N}}$ of $p_{\floor{\frac{k}{2}}}$-minimizers in $B^{3}_{1}$ having $g$ as a trace on $\mathbb{S}^{2}$ by $v_{2k}=u_{k}$ and $v_{2k+1}=\sigma \circ u_{k} \circ \rho$ for each $k \in \mathbb{N}$. Letting $k \to +\infty$, by Proposition~\ref{example 3}, the corresponding limit measure \(\mu_*\) to the subsequence $(v_{2k})_{k \in \mathbb{N}}$ has as support either \([(1,0,0), (0, 1, 0)] \cup [(-1,0,0), (0,-1,0)]\) or  \([(1,0,0), (0,-1,0)] \cup [(-1,0,0), (0,1,0)]\). Assume that the support of $\mu_{*}$ is \([(1,0,0), (0, 1, 0)] \cup [(-1,0,0), (0,-1,0)]\). Then the corresponding limit measure to the subsequence $(v_{2k+1})_{k \in \mathbb{N}}$ has as support \([(1,0,0), (0,-1,0)] \cup [(-1,0,0), (0,1,0)]\). Altogether, $(v_{k})_{k\in \mathbb{N}}$ has two subsequences yielding two different limit measures.
\end{rem}

\begin{lemma}\label{lem concentration}
    Let $\eta>0$ be the constant of Lemma~\ref{heart}, where $\kappa=1/2$, $p_{0}=3/2$, $\Psi(x)=x+x_{0}$. Assume that $B^{3}_{r}(x_{0}) \subset \Omega$ and $\mu_{*}(\smash{\overline{B}}^{3}_{r}(x_{0}))<\eta r$. Then $\mu_{*}(B^{3}_{r/2}(x_{0}))=0$. 
    \end{lemma}
\begin{proof}
    In this proof, every ball is centered at point $x_{0}$. It is a well-known fact (see, for instance, \cite[Section~1.9]{Evans}) that the weak* convergence
    of the Borel measures $\mu_{n}$ is equivalent to the following two inequalities
    \begin{equation}\label{two cond of weak*}
    \mu_{*}(F)\geq \limsup_{n\to +\infty}\mu_{n}(F), \,\ \,\  \mu_{*}(G)\leq \liminf_{n\to +\infty} \mu_{n}(G)
    \end{equation}
    whenever $F\subset \overline{\Omega}$ is closed and $G \subset \overline{\Omega}$ is open. Notice that this is where we take advantage of working with $(C(\overline{\Omega}))^{\prime}$ instead of $(C_{0}(\Omega))^{\prime}$. Using \eqref{two cond of weak*} and the facts that $\mu_{*}(\smash{\overline{B}}^{3}_{r})<\eta r$ and $\mu_{n}(\smash{\overline{B}}^{3}_{r})=\mu_{n}(B^{3}_{r})$, we obtain  the following estimate
    \begin{align*}
    \eta r> \limsup_{n\to+\infty} \mu_{n}(B^{3}_{r}).
    \end{align*}
    This estimate, together with the fact that $\eta r^{3-p_{n}} \to \eta r$ as $n \to +\infty$, implies that for all sufficiently large $n \in \mathbb{N}$, it holds
    \[
    \mu_{n}(B^{3}_{r}) < \eta r^{3-p_{n}}.
    \]
    But then Lemma~\ref{heart} says that for all sufficiently large $n \in \mathbb{N}$,
    \begin{equation}\label{nice concentration estimate}
    \mu_{n}(B^{3}_{r/2})\leq (2-p_{n})C r^{3-p_{n}},
    \end{equation}
    where $C>0$ is a constant independent of $n$. Letting $n$ tend to $+\infty$ in \eqref{nice concentration estimate} and using \eqref{two cond of weak*}, we complete our proof of Lemma~\ref{lem concentration}.
\end{proof}

The following monotonicity of the $p$-energy holds, which will be used later.
\begin{lemma}\label{lem mon of p-energy} Let $p \in [1,3)$, $x_{0}\in \Omega$ and $0<\varrho<r<\dist(x_{0}, \partial \Omega)$. Let $u_{p}\in W^{1,p}(\Omega, \mathcal{N})$ be a $p$-minimizer. Then 
    \[
    \varrho^{p-3}\int_{B^{3}_{\varrho}(x_{0})}\frac{|Du_{p}|^{p}}{p}\diff x \leq r^{p-3}\int_{B^{3}_{r}(x_{0})}\frac{|Du_{p}|^{p}}{p}\diff x.
    \]
\end{lemma}
\begin{proof} For a proof, the reader may consult Section~4 in \cite{H-L}.
\end{proof}
\begin{cor}\label{mon for density}
    Let $x_{0} \in \Omega$. Then the function 
    \[
    r \in (0, \dist(x_{0}, \partial \Omega)) \mapsto \frac{\mu_{*}(\smash{\overline{B}}^3_{r}(x_{0}))}{r}
    \]
    is nondecreasing.
\end{cor}
    \begin{proof}[Proof of Corollary~\ref{mon for density}]
        Let $0<\varrho<r<\dist(x_{0}, \partial \Omega)$. In view of Lemma~\ref{lem mon of p-energy}, for each $n \in \mathbb{N}$ and for an arbitrary $\varepsilon \in (0, r-\varrho)$, 
        \[
        (\varrho +\varepsilon)^{p_{n}-3}\mu_{n}(B^{3}_{\varrho+\varepsilon}(x_{0}))\leq r^{p_{n}-3}\mu_{n}(\smash{\overline{B}}^3_{r}(x_{0})).
        \]
        Letting $n$ tend to $+\infty$ and using \eqref{two cond of weak*}, we obtain
        \[
        (\varrho+\varepsilon)^{-1}\mu_{*}(B^{3}_{\varrho+\varepsilon}(x_{0}))\leq r^{-1}\mu_{*}(\smash{\overline{B}}^3_{r}(x_{0})).
        \]
        Then, letting $\varepsilon \searrow 0$ and using the fact that $\mu_{*}(\overline{\Omega})<+\infty$, we conclude our proof of Corollary~\ref{mon for density}.
    \end{proof}
Thanks to Corollary~\ref{mon for density}, for each $x_{0} \in \Omega$, the 1-dimensional density
\begin{equation}\label{defofdensity}
\Theta_{1}(\mu_{*}, x_{0})\coloneqq \lim_{r\to 0+} \frac{\mu_{*}(\smash{\overline{B}}^3_{r}(x_{0}))}{2r}=\lim_{r\to 0+}\frac{\mu_{*}(B^{3}_{r}(x_{0}))}{2r}
\end{equation}
exists and finite. In order to lighten the notation, we shall simply write $\Theta_{1}(x_{0})$ instead of $\Theta_{1}(\mu_{*}, x_{0})$ when appropriate.

\begin{prop}\label{nice characterization for singular set}
    Let $\eta>0$ be the constant of Lemma~\ref{heart}, where $\kappa=1/2$, $p_{0}=3/2$ and $\Psi(x)=x+x_{0}$ whenever $E=B^{3}_{r}(x_{0})$. Then it holds
    \[
    \Omega\cap S_{*}=\left\{x \in \Omega: \Theta_{1}(x)>0\right\}=\Bigl\{x \in \Omega: \Theta_{1}(x)\geq \frac{\eta}{2}\Bigr\}.
    \]
\end{prop}
\begin{proof}
    First, assume by contradiction that for some $x \in \Omega$, $\Theta_{1}(x)=a \in (0, \eta/2)$. Fix 
    \[
    \varepsilon \in \Bigl(0,\min\Bigl\{\frac{a}{2},\frac{\eta}{2}-a\Bigr\}\Bigr).
    \] 
   Then there exists $\delta>0$ such that for all $r \in (0,\delta)$,
    \[
    0<a-\varepsilon<\frac{\mu_{*}(\smash{\overline{B}}^3_{r}(x))}{2r}<\frac{\eta}{2}.
    \]
    However, Lemma~\ref{lem concentration} says that $\mu_{*}(B^{3}_{r/2}(x))=0$, which leads to a contradiction and proves that \[\{x \in \Omega: \Theta_{1}(x)>0\}=\{x \in \Omega: \Theta_{1}(x)\geq \eta/2\}.\] 
    On the other hand, by the definition of the support of a measure, $x \in \Omega \cap S_{*}$ if and only if for all sufficiently small $r>0$, $\mu_{*}(\smash{\overline{B}}^3_{r}(x))>0$, which holds, according to Lemma~\ref{lem concentration}, if and only if for all sufficiently small $r>0$, $\mu_{*}(\smash{\overline{B}}^3_{r}(x))\geq \eta r$. The last holds if and only if $\Theta_{1}(x)\geq \eta/2$. This proves that $\Omega\cap S_{*}=\{x \in \Omega: \Theta_{1}(x)\geq \eta/2\}$ and completes our proof of Proposition~\ref{nice characterization for singular set}.
\end{proof}

\begin{prop}\label{conv to harm} Let $(p_{n})_{n \in \mathbb{N}}\subset [1,2)$, $p_{n} \nearrow 2$ as $n \to +\infty$ and $(u_{n})_{n\in \mathbb{N}}$ be a sequence of $p_{n}$-minimizers in $\Omega$ satisfying \eqref{C2}. Then there exists a mapping $u_{*} \in W^{1,2}_{\loc}(\Omega \setminus S_{*}, \mathcal{N})$ such that, up to a subsequence (not relabeled), $u_{n} \to u_{*}$ a.e. in $\Omega$ and for which the conclusions of Proposition~\ref{prop conv to harmonic map} apply with $U = \Omega \setminus S_*$.
\end{prop}
\begin{rem}
For interior quantitative estimates of the weak gradient of $u_{*}$, the reader may consult Proposition~\ref{prop est for the differential inside domain}.
\end{rem}
\begin{proof} [Proof of Proposition~\ref{conv to harm}]
Let $K\csubset \Omega \setminus S_{*}$ be open. Using \eqref{wcm}, a standard covering argument and Lemma~\ref{heart}, we conclude that there exists a positive constant $C$, possibly depending only on $K$, $\Omega$ and $\mathcal{N}$, such that, up to a subsequence (not relabeled), for each $n \in \mathbb{N}$, 
\[
\int_{K}\frac{|Du_{n}|^{p_{n}}}{p_{n}} \diff x \le C.
\]
Applying a diagonal argument, we observe that, up to a subsequence (not relabeled), the condition \eqref{C1} of Proposition~\ref{prop conv to harmonic map} is satisfied, where \(U = \Omega \setminus S_*\), for our sequence of $p_{n}$-minimizers. Thus, the conclusions of Proposition~\ref{prop conv to harmonic map} apply, yielding the desired mapping $u_{*}$. This completes our proof of Proposition~\ref{conv to harm}. 
  \end{proof}

\subsection{The limit measure is the weight measure of the stationary varifold}
Recall that a set $E\subset \mathbb{R}^{3}$ is said to be countably $\mathcal{H}^{1}$-rectifiable if there exist countably many Lipschitz functions $f_{i}:\mathbb{R}^{1}\to \mathbb{R}^{3}$ such that
\[
\mathcal{H}^{1}\Bigl(E\setminus \bigcup_{i=0}^{+\infty}f_{i}(\mathbb{R}^{1})\Bigr)=0.
\]
Let $\mathrm{G}(3,1)$ be the Grassmann manifold, which is the space of 1-dimensional subspaces of $\mathbb{R}^{3}$. Following Allard (see \cite{Allard}), for each $A\in \mathrm{G}(3,1)$, we shall also  use ``$A$'' to denote the orthogonal projection of $\mathbb{R}^{3}$ onto $A$. 
That is, $\mathrm{G}(3,1)=\{A\in M_{3}(\mathbb{R}): A\circ A=A,\, A^{\mathrm{T}}=A\,\ \text{and}\,\ \operatorname{Im}(A)=A\}$, where $M_{3}(\mathbb{R})$ denotes the space of real $3\times 3$ matrices. We also define $\mathrm{G}_{3}(\Omega)= \Omega\times \mathrm{G}(3,1)$. Recall that $V$ is said to be a 1-dimensional varifold in $\Omega$ if $V$ is a positive Radon measure on $\mathrm{G}_{3}(\Omega)$. 

Let $S\subset \Omega$ be countably $\mathcal{H}^{1}$-rectifiable and $\mu\in (C(\overline{\Omega}))^{\prime}$ be a positive Radon measure satisfying $\mu\mres \Omega= \theta \mathcal{H}^{1}\mres S$ for some nonnegative $\theta \in L^{1}(\Omega)$. Rectifiability gives that for $\mathcal{H}^{1}$-a.e.\  $x \in S$ (or, equivalently, for $\mu$-a.e.\  $x \in \Omega$), there exists a unique approximate tangent plane $T_{x}S\in \mathrm{G}(3,1)$ to $S$ at $x$ (see, for instance, \cite[Theorem~2.83~(i)]{APD}). For $\mu$-a.e. $x \in \Omega$, we denote by $A(x) \in \mathrm{G}(3,1)$ the orthogonal projection onto $T_{x}S$.
We define the varifold $V$, which is naturally associated to $\mu$, as the pushforward measure $V=(\mathrm{Id},A)_{\#}\mu\mres\Omega$. Thus, for each Borel-measurable set $E \in \mathcal{B}(\mathrm{G}_{3}(\Omega))$, $V(E)=\mu(\{x \in \Omega: (x,A(x))\in E\})$. The varifold $V$ is said to be stationary  (see Section~4.2 in \cite{Allard}) if its first variation vanishes, namely
\begin{equation*}
\int_{\Omega}A(x):D\xi(x)\diff \mu(x) = 0 \,\ \,\ \forall \xi \in C^{1}_{c}(\Omega, \mathbb{R}^{3}).
\end{equation*}
\begin{prop}\label{prop 1-var}
The set $S_{*}\cap \Omega$ is countably $\mathcal{H}^{1}$-rectifiable and $\mathcal{H}^{1}(S_{*}\cap \Omega)<+\infty$. The measure $V_{*}=(\mathrm{Id},A_{*})_{\#}\mu_{*}\mres\Omega$ is a stationary 1-dimensional rectifiable varifold in $\Omega$, where for $\mu_{*}$-a.e. $x \in \Omega$, $A_{*}(x)$ represents the orthogonal projection onto the approximate tangent plane $T_{x}S_{*}$ to $S_{*}$ at $x$. Furthermore, it holds $\mu_{*}\mres \Omega (d x) = \Theta_{1}(x) \mathcal{H}^{1}\mres (S_{*}\cap \Omega)(d x)$, where the density $\Theta_{1}$ is defined in \eqref{defofdensity}. 
\end{prop}
\begin{proof}
    Let $(u_{n})_{n \in \mathbb{N}}\subset W^{1,p_{n}}(\Omega, \mathcal{N})$, where $(p_{n})_{n\in \mathbb{N}} \subset (3/2,2)$, be a sequence of $p_{n}$-minimizers satisfying \eqref{wcm}. For each $n\in \mathbb{N}$, define the map $A^{p_{n}}:\Omega \to M_{3}(\mathbb{R})$ by
    \begin{equation*}
    A^{p_{n}}_{i,j}=\frac{2-p_{n}}{p_{n}}(|Du_{n}|^{p_{n}}\delta_{i,j}-p_{n}|Du_{n}|^{p_{n}-2}\langle D_{i}u_{n}, D_{j}u_{n}\rangle).
    \end{equation*}
    Then observe that  
    \begin{equation}\label{est trace 5.8}
    (A^{p_{n}})^{\mathrm{T}}=A^{p_{n}}, \,\ \,\ \tr(A^{p_{n}})=(3-p_{n}) \frac{\diff\mu_{n}}{\diff x} \geq \frac{\diff \mu_{n}}{\diff x}
    \end{equation}
    and
    \begin{equation}\label{est tv 5.9}
    |A^{p_{n}}|=\sqrt{A^{p_{n}}:A^{p_{n}}}= \sqrt{3-2p_{n}+p^{2}_{n}}\frac{\diff \mu_{n}}{\diff x} \leq  \sqrt{3}\frac{\diff \mu_{n}}{\diff x},
    \end{equation}
    where $\frac{\diff \mu_{n}}{\diff x}=(2-p_{n})\frac{|Du_{n}|^{p_{n}}}{p_{n}}$. For each $v \in \mathbb{S}^{2}$, 
    \begin{equation}\label{est ev 5.10}
    \sum_{i,j=1}^{3}\langle A^{p_{n}}_{i,j}v_{i},v_{j}\rangle=\frac{2-p_{n}}{p_{n}}\left(|Du_{n}|^{p_{n}}-p_{n}|Du_{n}|^{p_{n}-2}\left|\sum_{i=1}^{3}v_{i}D_{i}u_{n}\right|^{2}\right)\leq \frac{\diff\mu_{n}}{\diff x},
    \end{equation}
    which implies that all eigenvalues  of $A^{p_{n}}$ are less than or equal to $\frac{\diff \mu_{n}}{\diff x}$. Furthermore, by \eqref{integralidentity}, 
    \begin{equation}\label{integid p-harm}
    \int_{\Omega}A^{p_{n}}(x):D\xi(x) \diff x = 0 \,\ \,\ \forall \xi \in C^{1}_{c}(\Omega, \mathbb{R}^{3}).
    \end{equation} 
    In view of \eqref{C2} and \eqref{est tv 5.9}, there exists $\widetilde{A}\in (C_{0}(\Omega, M_{3}(\mathbb{R})))^{\prime}$ such that, up to a subsequence (not relabeled),  $A^{p_{n}}\diff x\overset{*}{\rightharpoonup} \widetilde{A}$ weakly* in $(C_{0}(\Omega, M_{3}(\mathbb{R})))^{\prime}$. By \eqref{wcm} and \eqref{est tv 5.9}, 
    \begin{equation}\label{abscont}
    |\widetilde{A}|\leq \sqrt{3} \mu_{*}\mres \Omega,
    \end{equation}
    which implies that $\widetilde{A}$ is absolutely continuous with respect to $\mu_{*}\mres \Omega$. Then, by the Radon-Nikod\'ym theorem, there exists $F\in L^{1}(\Omega, M_{3}(\mathbb{R}), \mu_{*})$ such that $\widetilde{A}=F \mu_{*}\mres \Omega$ in $(C_{0}(\Omega, M_{3}(\mathbb{R})))^{\prime}$. Letting $n$ tend to $+\infty$ and taking into account \eqref{est trace 5.8}, \eqref{est ev 5.10} and \eqref{integid p-harm}, we observe that for $\mu_{*}$-a.e.\  $x \in \Omega$, $(F(x))^{\mathrm{T}}=F(x)$, $\tr(F(x))\geq 1$, the eigenvalues of $F(x)$ are less than or equal to \(1\) and
    \begin{equation}\label{fvarF}
    \int_{\Omega}F(x):D\xi(x) \diff \mu_{*}(x) =0 \,\ \,\ \forall \xi \in C^{1}_{c}(\Omega, \mathbb{R}^{3}).
    \end{equation}
    According to \cite[Proposition~3.1]{ArroyoRabasa_DePhilippis_Hirsch_Rindler_2019} and \cite[Lemma~2.3]{ArroyoRabasa_DePhilippis_Hirsch_Rindler_2019}, the identity \eqref{fvarF} implies that there exists an $\mathcal{H}^{1}$-rectifiable set $R\subset \Omega$ and a map $\lambda \in L^{\infty}(R, M_{3}(\mathbb{R}))$ satisfying the following conditions. For $\mathcal{H}^{1}$-a.e.\  $x_{0} \in R$, $|\lambda(x_{0})|=1$, $\lambda(x_{0}) \in \mathrm{G}(3,1)$ is the orthogonal projection matrix onto $T_{x_{0}}R$  such that $\widetilde{A}\mres\{\Theta^{*}_{1}(|\widetilde{A}|)>0\}(d x)= \Theta^{*}_{1}(|\widetilde{A}|,x)\lambda(x)\mathcal{H}^{1}\mres R (d x)$, where $T_{x_{0}}R$ is the approximate tangent plane to $R$ at $x_{0}$ and $\Theta^{*}_{1}(|\widetilde{A}|,x_{0})=\limsup_{r\to 0+}\frac{|\widetilde{A}|(B^{3}_{r}(x_{0}))}{2r}$. Then, for $\mu_{*}$-a.e.\  $x \in \Omega$, two eigenvalues of $F(x)$ are zeros. Since for $\mu_{*}$-a.e.\  $x \in \Omega$, $\tr(F(x))\geq 1$ and the eigenvalues of $F(x)$ are less than or equal to 1, we deduce that the eigenvalues of $F(x)$ (up to a permutation) are $(1,0,0)$ and $|F(x)|=1$ for $\mu_{*}$-a.e.\  $x\in \Omega$.  Using this, together with \eqref{defofdensity} and Proposition~\ref{nice characterization for singular set}, we observe that $|\widetilde{A}|=\mu_{*}\mres \Omega$, $\Theta^{*}_{1}(|\widetilde{A}|)=\Theta_{1}(\mu_{*})$ and $R=S_{*}\cap \Omega$. In particular, for $\mu_{*}$-a.e.\  $x \in \Omega$, $F(x)=A_{*}(x)$ and $V_{*}(d A, d x)=\delta_{A_{*}(x)}(d A)\otimes \mu_{*}(d x)=\delta_{A_{*}(x)}(d A)\otimes \Theta_{1}(\mu_{*},x)\mathcal{H}^{1}\mres (S_{*}\cap \Omega)(d x)$, which, together with \eqref{fvarF}, implies that $V_{*}$ is a 1-dimensional stationary rectifiable varifold (see \cite{Allard, DPDRG}). This completes our proof of Proposition~\ref{prop 1-var}.
\end{proof}

\begin{rem}
\label{rem_zieM6dahgh7uecie8ohTh5Za}
Our proof of Proposition~\ref{prop 1-var}, based on \cite[Proposition~3.1]{ArroyoRabasa_DePhilippis_Hirsch_Rindler_2019}, differs from Canevari's proof for the three-dimensional Landau-de Gennes model \cite[Proposition~55]{Canevari_2017}.
His strategy consists in first showing that the set $S_{*}\cap \Omega$ is countably $\mathcal{H}^{1}$-rectifiable using the Moore-Preiss  rectifiability theorem (see \cite[Theorem~3.5]{Moore}, \cite[Theorem~5.3]{Preiss}); next, using the Ambrosio-Soner construction (see \cite[Lemma~3.9]{Ambrosio-Soner}), he proves that the Radon-Nikod\'ym derivative of the limit tensor has the eigenvalues (up to a permutation) $(1,0,0)$, and concludes that the naturally associated one-dimensional varifold is stationary; the rectifiability of the latter then follows from Rectifiability Theorem in \cite[Section~5.5]{Allard}.
Our more direct strategy of proof could be adapted to work in Canevari's setting, and his strategy of proof would provide another pathway to Proposition~\ref{prop 1-var}. 
\end{rem}

\subsection{Energy bounds on a cylinder}
Let us define the notion of a uniform Lipschitz triangulation on the lateral surface of a (finite) cylinder $\Lambda_{1,L}\coloneqq B^{2}_{1}\times (-L,L)$. 
Let $L\geq 1$ and $h \in (0,1)$. 
Set $n \coloneqq \floor{1/h}\geq 1$. 
For each $k \in [0,n-1]$, we define the line segment $S_{k}=[(e^{\frac{2\pi i k}{n}}, -L), (e^{\frac{2\pi i k}{n}}, L)] \subset \mathbb{C} \times \mathbb{R} \simeq \mathbb{R}^3$ and, for \(j \in [0,2\floor{L}n]\) the points 
\[
 a^{k}_{j} \coloneqq \left(e^{\frac{2\pi i k}{n}}, -L+ \frac{Lj}{\floor{L}n}\right)
 \in \mathbb{C} \times \mathbb{R} \simeq \mathbb{R}^3.
\]
For each $k \in [0,n-1]$ and $j \in [0,2\floor{L}n-1]$, let $\gamma^{k}_{j}$ be the union of two geodesics (\emph{a geodesic cross}) on $\overline{\Gamma}_{1,L}$, where one of which connecting $a^{k}_{j}$ with $a^{k+1}_{j+1}$ and the other connecting $a^{k}_{j+1}$ with $a^{k+1}_{j}$. By $c_{k,z}$, where $z\in \{-L,L\}$, we denote the arc on $\mathbb{S}^{1}\times \{z\}$ connecting $a^{k}_{2\floor{L}n}$ with $a^{k+1}_{2\floor{L}n}$ if $z=L$ and $a^{k}_{0}$ with $a^{k+1}_{0}$ if $z=-L$. Observe that the union of all the line segments $S_{k}$, all the geodesic crosses $\gamma^{k}_{j}$ and all the arcs $c_{k,z}$ provides us with the decomposition of $\overline{\Gamma}_{1,L}$ into mutually disjoint $0$, $1$ and $2$-dimensional cells. 
Hereinafter, we denote by $SO(2)$ the subgroup of $SO(3)$ of rotations with respect to the axis \(\{0\} \times \mathbb{R}\), and whose action on the plane $\mathbb{R}^{2}\times \{0\}$ coincides with the action of the orthogonal group $SO(2)$ on the plane. 
Thus, for each $\omega \in SO(2)$, we obtain a decomposition of $\overline{\Gamma}_{1,L}$ into mutually disjoint \(0\), \(1\) and \(2\)-dimensional cells, namely
\begin{equation}\label{grid on cyl}
\overline{\Gamma}_{1,L}= \bigcup_{i=0}^{2}\bigcup_{j=1}^{k_{i}}\omega(E_{i,j})=\bigcup_{i=0}^{2}\bigcup_{j=1}^{k_{i}}F_{i,j},
\end{equation}
where $\{F_{0,1},\dotsc, F_{0,k_{0}}\}$ is the set of points and for each $i \in \{1,2\}$, there exists a bilipschitz homeomorphism 
\begin{equation}\label{bilip on cyl}
\Phi_{i,j}: F_{i,j}\to B^{i}_{h} \,\ \text{with} \,\ \|D_{\top} \Phi_{i,j}\|_{L^{\infty}(F_{i,j})}+\|D_{\top} (\Phi_{i,j})^{-1}\|_{L^{\infty}(B^{i}_{h})} \leq A_{4},
\end{equation}
where $A_{4}>0$ is an absolute constant; see Figure~\ref{figure grid cylinder}.  The map $\Phi_{i,j}$ in \eqref{bilip on cyl} extends by continuity to a bilipschitz map between $\overline{F}_{i,j}$ and $\smash{\overline{B}}^{i}_{h}$ with the same bilipschitz constant $A_{4}$ as in \eqref{bilip on cyl}.
\begin{figure}
    \centering
    \includegraphics[width=.5\textwidth]{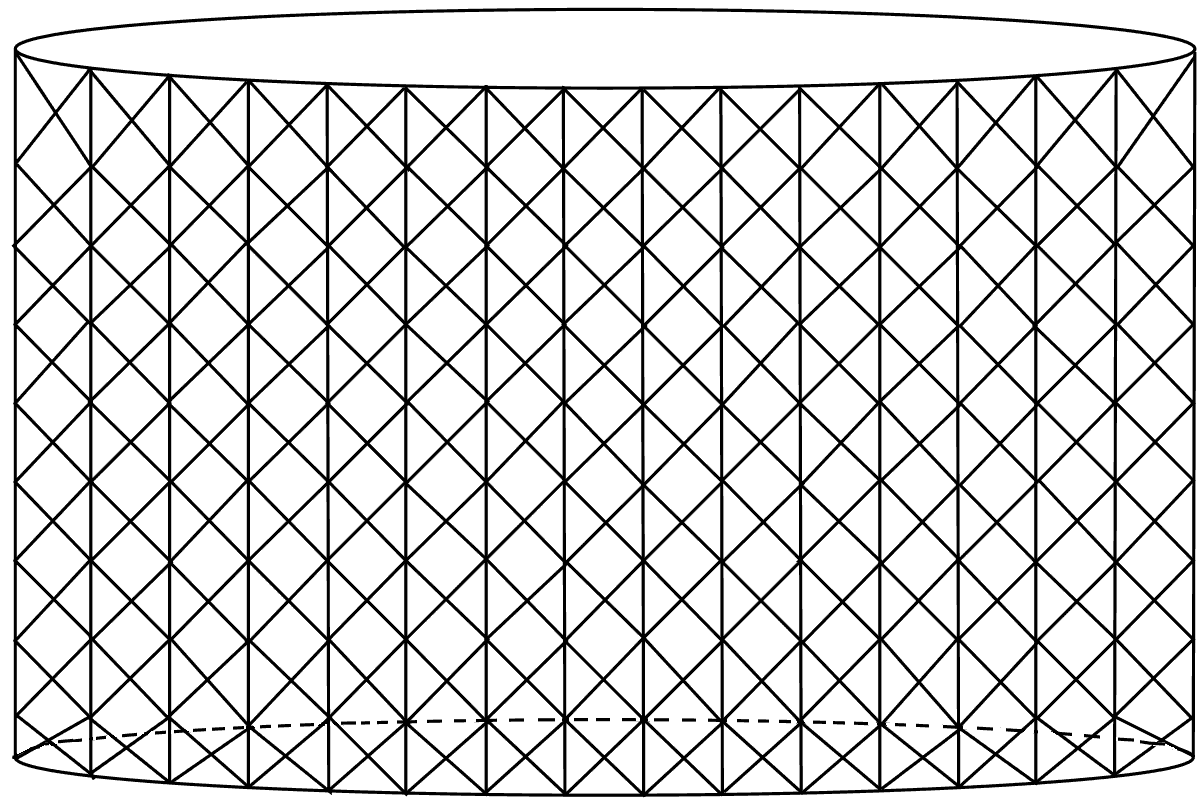}
    \caption{A uniform Lipschitz triangulation of the lateral surface of a cylinder.}
    \label{figure grid cylinder}
\end{figure}

 The decomposition \eqref{grid on cyl} we shall call a uniform Lipschitz (or simply a Lipschitz) triangulation of $\overline{\Gamma}_{1,L}$ of step $h$ with the 1-skeleton $\Sigma= \bigcup_{i=0}^{1} \bigcup_{j=0}^{k_{i}} F_{i, j}$.
 
 To lighten the notation, we denote the collection of 1-skeletons of all Lipschitz triangulations of $\overline{\Gamma}_{1,L}$ of step $h$ by $\mathrm{RC}_{L,h}$, namely
 \[
\mathrm{RC}_{L,h}:=\left\{\bigcup_{i=0}^{1}\bigcup_{j=1}^{k_{i}}\omega(E_{i,j}): \omega \in SO(2)\right\},
\]
where $\bigcup_{i=0}^{1}\bigcup_{j=1}^{k_{i}}E_{i,j}$ is the union of  all the line segments $S_{k}$, all the geodesic crosses $\gamma^{k}_{j}$ and all the arcs $c_{k,z}$. 

It is worth noting that Lemmas~\ref{lem triangulation on cylinder},~\ref{lem lowenergy cylinder},~\ref{finding nice v and conmapping cylinder} are the counterparts of Lemmas~\ref{nice res on T},~\ref{lem lowenergy},~\ref{finding nice v and conmapping} on the lateral surface of a cylinder. A consequence of Lemma~\ref{finding nice v and conmapping cylinder} is Corollary~\ref{cor on cylinder}, which provides us with the key arguments that will be used in the proof of Proposition~\ref{prop energy bounds on cylinders}. The latter gives us the estimates of the $p$-energy on a cylinder, which will be used to prove the fact that the limit measure has a finite set of values.
\begin{lemma}\label{lem triangulation on cylinder}
    Let $p \in [1,+\infty)$, $L\geq 1$, $h \in (0,1)$ and $\mathcal{Y}$ be a compact Riemannian manifold or a Euclidean space. Let $g \in W^{1,p}(\Gamma_{1,L}, \mathcal{Y})$ and $\Sigma \in \mathrm{RC}_{L, h}$. Then for $\mathcal{H}^{1}$-a.e.\  $\omega \in SO(2)$ we have $\operatorname{tr}_{\omega(\Sigma \cap \Gamma_{1,L})}(g)=g|_{\omega(\Sigma\cap \Gamma_{1,L})} \in W^{1,p}(\omega(\Sigma\cap \Gamma_{1,L}), \mathcal{Y})$. Furthermore, there exists $\omega \in SO(2)$ such that $\operatorname{tr}_{\omega(\Sigma \cap \Gamma_{1,L})}(g)=g|_{\omega(\Sigma\cap \Gamma_{1,L})} \in W^{1,p}(\omega(\Sigma\cap \Gamma_{1,L}), \mathcal{Y})$ and 
    \[
    \int_{\omega(\Sigma \cap \Gamma_{1,L})}|D_{\top}g|^{p}\diff \mathcal{H}^{1}\leq \frac{A_{5}}{h}\int_{\Gamma_{1,L}}|D_{\top}g|^{p}\diff \mathcal{H}^{2},
    \]
 where $A_{5}>0$ is an absolute constant.
\end{lemma}
\begin{proof}
For a proof, we refer to \cite[Lemma~3.4]{Top}. The main ingredient is the integralgeometric formula
\[
\int_{SO(2)}\diff \omega \int_{\omega(\Sigma \cap \Gamma_{1,L})}|D_{\top} g|^{p}\diff \mathcal{H}^{1}=\frac{\mathcal{H}^{1}(\Sigma\cap \Gamma_{1,L})}{\mathcal{H}^{2}(\Gamma_{1,L})}\int_{\Gamma_{1,L}}|D_{\top}g|^{p}\diff \mathcal{H}^{2}
\]
(for the integralgeometric formulas, the reader may consult \cite{Federer, IGF}; see also \cite[Lemma~3.3]{Top}).
\end{proof}

\begin{lemma}\label{lem lowenergy cylinder}
    Let $p_{0} \in (1,2)$, $p \in [p_{0},2)$, $L\geq 1$ and $h \in (0,1)$. 
    Then there exists a constant $\kappa_{0}=\kappa_{0}(p_{0}, \mathcal{N})>0$ such that the following holds.
        Let $\Sigma \in \mathrm{RC}_{L, h}$ and $F$ be a $2$-cell of the Lipschitz triangulation of $\overline{\Gamma}_{1, L}$ generated by $\Sigma$. Let $g \in W^{1,p}(F, \mathcal{N})$ and $\operatorname{tr}_{\partial F}(g) \in W^{1,p}(\partial F, \mathcal{N})$, where $\partial F$ denotes the relative boundary of $F$.  If
    \begin{equation*}
    \int_{F}\frac{|D_{\top}g|^{p}}{p}\diff \mathcal{H}^{2} + h\int_{\partial F}\frac{|D_{\top}\operatorname{tr}_{\partial F}(g)|^{p}}{p}\diff \mathcal{H}^{1}\leq \frac{\kappa_{0}h^{2-p}}{2-p},
    \end{equation*} 
    then $\operatorname{tr}_{\partial F}(g)$ is nullhomotopic.
\end{lemma}
\begin{proof}
The proof follows by reproducing the proof of Lemma~\ref{lem lowenergy} with minor modifications. Essentially, it is enough to replace the 2-dimensional cell in the proof of Lemma~\ref{lem lowenergy} by the $2$-cell $F$ of Lemma~\ref{lem lowenergy cylinder}. 
\end{proof}

\begin{lemma}\label{finding nice v and conmapping cylinder}
    Let $p_{0} \in (1,2)$, $p \in [p_{0},2)$, $L_{0}\geq 2$, $h\in (0,1/2]$ and $g \in~W^{1,p}(\Gamma_{1,L_{0}}, \mathcal{N})$, where $\Gamma_{1,L_{0}}$ is defined in \eqref{def cyl}.  Then there exists $\kappa=\kappa(p_{0}, \mathcal{N})>0$ such that the following holds. Assume that 
    \begin{equation}\label{lowenergyconditioncylinder}
        \int_{\Gamma_{1,L_{0}}}\frac{|D_{\top}g|^{p}}{p}\diff \mathcal{H}^{2} \leq \frac{\kappa h^{2-p}}{2-p}.
    \end{equation}
    There exist $z_{-} \in (-L_{0}+h,-L_{0}+2h)$, $z_{+}\in (L_{0}-2h,L_{0}-h)$, $\varphi_{h}\in W^{1,p}((B^{2}_{1}\setminus \smash{\overline{B}}^{2}_{1-h}) \times (z_{-},z_{+}), \mathcal{N})$ such that the following assertions hold.
    \begin{enumerate}[label=(\roman*)]
        \item  \label{it_AhxeeCeeth8ieghooloboL4a}
        \(
       \operatorname{tr}_{\mathbb{S}^{1}\times (z_{-},z_{+})}(\varphi_{h})=g|_{\mathbb{S}^{1}\times (z_{-}, z_{+})}\), \(v_h \coloneqq 
            \operatorname{tr}_{\partial B^{2}_{1-h}\times (z_{-},z_{+})}(\varphi_{h}) \in W^{1,2}(\partial B^{2}_{1-h}\times (z_{-},z_{+}), \mathcal{N})\).
        \item 
        \label{it_ChahghohLail8jah6Aibohng}
        For each $z \in \{z_{-},z_{+}\}$, $\operatorname{tr}_{(B^{2}_{1}\setminus \smash{\overline{B}}^{2}_{1-h})\times \{z\}}(\varphi_{h})\in W^{1,p}((B^{2}_{1}\setminus \smash{\overline{B}}^{2}_{1-h})\times \{z\}, \mathcal{N})$ 
        and
        \[
        \int_{(B^{2}_{1}\setminus \smash{\overline{B}}^{2}_{1-h})\times \{z\}}|D_{\top} \operatorname{tr}_{(B^{2}_{1}\setminus \smash{\overline{B}}^{2}_{1-h})\times \{z\}}(\varphi_{h})|^{p}\diff \mathcal{H}^{2}\leq A_{6} \int_{\Gamma_{1, L_{0}}}|D_{\top} g|^{p}\diff \mathcal{H}^{2},
        \]
        where $A_{6}>0$ is an absolute constant.
    \item  \label{it_Bit7ooRoo3ceithiegeej0xi} There exists a constant $C=C(\mathcal{N})>0$ such that
    \begin{equation}\label{estliftingcylinder}
        \int_{\partial B^{2}_{1-h}\times (z_{-},z_{+})}|D_{\top}v_{h}|^{p}\diff \mathcal{H}^{2}\leq C\int_{\Gamma_{1,L_{0}}}|D_{\top}g|^{p}\diff \mathcal{H}^{2}
    \end{equation}
    and
    \begin{equation}\label{estonconnmapcylinder}
        \int_{(B^{2}_{1}\setminus \smash{\overline{B}}^{2}_{1-h})\times (z_{-},z_{+})}|D\varphi_{h}|^{p}\diff x \leq Ch \int_{\Gamma_{1,L_{0}}}|D_{\top}g|^{p} \diff \mathcal{H}^{2}.
    \end{equation}
    \end{enumerate}
 \end{lemma}
\begin{proof} Using the coarea formula, we find $z_{-}\in (-L_{0}+h,-L_{0}+2h)$ and $z_{+}\in (L_{0}-2h,L_{0}-h)$ such that for each $z \in \{z_{-}, z_{+}\}$, $\tr_{\mathbb{S}^{1}\times \{z\}}(g)=g|_{\mathbb{S}^{1}\times \{z\}} \in W^{1,p}(\mathbb{S}^{1}\times \{z\}, \mathcal{N})$ and
    \begin{equation}\label{est 5.30}
    \int_{\mathbb{S}^{1}\times \{z\}}|D_{\top}g|^{p}\diff \mathcal{H}^{1}\leq \frac{2}{h}\int_{\Gamma_{1,L_{0}}}|D_{\top}g|^{p}\diff \mathcal{H}^{2}.
    \end{equation} 
    Define $L\coloneqq (z_{+}-z_{-})/2$, $\mathrm{e}\coloneqq (0,0,(z_{+}+z_{-})/2)$ and $u(y)\coloneqq g(\mathrm{e}+y)$ for each $y \in \Gamma_{1,L}$. 
  Let $\kappa_{0}=\kappa_{0}(p_{0}, \mathcal{N})>0$ be the constant given by Lemma~\ref{lem lowenergy cylinder}. Define $\kappa=\kappa_{0}/2A_{5}$, where $A_{5}>0$ is the absolute constant of Lemma~\ref{lem triangulation on cylinder}. Without loss of generality, we assume that $A_{5}\geq 2$. Then $\kappa$ depends only on $p_{0}$ and $\mathcal{N}$. Using Lemma~\ref{lem triangulation on cylinder}, the estimates \eqref{lowenergyconditioncylinder}, \eqref{est 5.30} and taking into account the definition of $\kappa$, we obtain $\Sigma \in \mathrm{RC}_{L, h}$ such that $\operatorname{tr}_{\Sigma}(u)\in W^{1,p}(\Sigma, \mathcal{N})$, $\operatorname{tr}_{\Sigma \cap \Gamma_{1, L}}(u)=u|_{\Sigma\cap \Gamma_{1,L}} \in W^{1,p}(\Sigma \cap \Gamma_{1,L}, \mathcal{N})$,
\begin{equation}\label{uniformfactskeleton}
    \int_{\Sigma \cap \Gamma_{1,L}}|D_{\top} u|^{p}\diff \mathcal{H}^{1} \leq \frac{A_{5}}{h} \int_{\Gamma_{1,L}} |D_{\top} u|^{p}\diff \mathcal{H}^{2}
    \end{equation}
and for each $2$-cell $F$ from the Lipschitz triangulation of $\overline{\Gamma}_{1,L}$ generated by $\Sigma$,
\[
\int_{F}\frac{|D_{\top} u|^{p}}{p}\diff \mathcal{H}^{2}+h \int_{\partial F} \frac{|D_{\top} \operatorname{tr}_{\partial F}(u)|^{p}}{p}\diff \mathcal{H}^{1} \leq 2A_{5} \int_{\Gamma_{1,L_{0}}}\frac{|D_{\top} g|^{p}}{p}\diff \mathcal{H}^{2} \leq \frac{\kappa_{0} h^{2-p}}{2-p},
\]
where $\partial F$ denotes the relative boundary of $F$. This, according to Lemma~\ref{lem lowenergy cylinder}, implies that $\operatorname{tr}_{\partial F}(u)$ is nullhomotopic for each $2$-cell $F$ from the Lipschitz triangulation of $\overline{\Gamma}_{1,L}$ generated by $\Sigma$. Proceeding similarly as in Proposition~\ref{prop ineq T} (changing $\mathbb{S}^{2}$ to $\overline{\Gamma}_{1,L}$ and the geometry of cells, namely changing ``squares'' to ``squares or triangles'' which are all bilipschitz homeomorphic to $B^{2}_{h}$), we obtain $w_{h} \in W^{1,2}(\Gamma_{1,L}, \mathcal{N})$ such that $\tr_{\Sigma}(w_{h})=\tr_{\Sigma}(u)$, $\tr_{\Sigma \cap \Gamma_{1,L}}(w_{h})=w_{h}|_{\Sigma\cap \Gamma_{1,L}}=u|_{\Sigma \cap \Gamma_{1,L}}$ and  
    \begin{equation}\label{est 5.28}
    \int_{\Gamma_{1,L}}|D_{\top}w_{h}|^{p}\diff \mathcal{H}^{2}\leq C h \int_{\Sigma}|D_{\top}u|^{p}\diff \mathcal{H}^{1},
    \end{equation}
    where $C=C(\mathcal{N})>0$. Using \eqref{est 5.30}, \eqref{uniformfactskeleton} and \eqref{est 5.28}, we obtain the estimate
    \[
    \int_{\Gamma_{1,L}}|D_{\top} w_{h}|^{p}\diff \mathcal{H}^{2}\leq C^{\prime}\int_{\Gamma_{1,L_{0}}}|D_{\top} g|^{p}\diff \mathcal{H}^{2},
    \]
    where $C^{\prime}=C^{\prime}(\mathcal{N})>0$. Now we define $v_{h}(x)=w_{h}((x_{1}/(1-h),x_{2}/(1-h),x_{3})-\mathrm{e})$ for $x \in \partial B^{2}_{1-h}\times (z_{-},z_{+})$. In view of the last estimate, we get    \eqref{estliftingcylinder}.

    From the Lipschitz triangulation of $\overline{\Gamma}_{1,L}$ generated by $\Sigma$ we obtain the decomposition of $(B^{2}_{1}\setminus \smash{\overline{B}}^{2}_{1-h}) \times (z_{-},z_{+})$ into the cells
    \begin{equation}
    \label{decompositioninto123cellscylinder}
    D_{i,j}=\left\{x=(x^{\prime},x_{3}) \in \mathbb{R}^{3}: (1-h) < |x^{\prime}| < 1,\,\ \left(\frac{x^{\prime}}{|x^{\prime}|},x_{3}\right) \in (F_{i,j}+\mathrm{e})\right\},
    \end{equation}
    where $D_{i,j}$ is of Hausdorff dimension $(i+1)$. Fix an arbitrary $3$-cell $D_{2,j}$ from the above decomposition. Let $F_{2,j}$ be the $2$-cell of the Lipschitz triangulation of $\overline{\Gamma}_{1,L}$ generated by $\Sigma$ such that $D_{2,j}=\{x \in \mathbb{R}^{3}: (1-h)< |x^{\prime}|<1,\,\ (\frac{x^{\prime}}{|x^{\prime}|},x_{3}) \in (F_{2,j}+\mathrm{e})\}$. Let $\Phi_{j}: [0,h]^{3}\to \smash{\overline{D}_{2,j}}$ be a bilipschitz homeomorphism with an absolute bilipschitz constant (see \eqref{bilip on cyl}) such that $\Phi_{j}((0,h)^{2} \times \{0\})=\{x \in \mathbb{R}^{3}: |x^{\prime}|=1-h,\,\ (\frac{x^{\prime}}{|x^{\prime}|}, x_{3}) \in (F_{2,j} + \mathrm{e})\}$ and $\Phi_{j}((0,h)^{2} \times \{1\})=F_{2,j}+\mathrm{e}$. Applying Lemma~\ref{lemma_single_cell} with $g_{0}$ and $g_{1}$ taking the values of $v_{h}\circ \Phi_{j}|_{(0,h)^{2} \times \{0\}}$ and $g \circ \Phi_{j}|_{(0,h)^{2} \times \{1\}}$, respectively, and changing the variables, we obtain a map $\varphi_{j} \in W^{1,p}(D_{2,j}, \mathcal{N})$ satisfying the estimate
    \begin{equation}
    \label{nu549j54j945jgu95jgu65ugj}
    \begin{split}
    \int_{D_{2,j}}|D\varphi_{j}|^{p}\diff x &\leq C^{\prime \prime} h \biggl(\int_{\Phi_{j}((0,h)^{2} \times \{0\})}|D_{\top} v_{h}|^{p}\diff \mathcal{H}^{2}+\int_{F_{2,j}+\mathrm{e}}|D_{\top} g|^{p}\diff \mathcal{H}^{2} \biggr)\\  & \qquad \qquad + C^{\prime \prime} h^{2} \int_{\partial F_{2,j}+\mathrm{e}}|D_{\top} g|^{p} \diff \mathcal{H}^{1},
    \end{split}
    \end{equation}
    where $C^{\prime \prime}>0$ is an absolute constant. 
    
    Let $\varphi_{h} \in L^{\infty}((B^{2}_{1}\setminus \smash{\overline{B}}^{2}_{1-h}) \times (z_{-},z_{+}), \mathcal{N})$ satisfy $\varphi_{h}|_{D_{2,j}}=\varphi_{j} \in W^{1,p}(D_{2,j}, \mathcal{N})$ for each $3$-cell $D_{2,j}$ from the decomposition \eqref{decompositioninto123cellscylinder}. In view of our construction (see Lemma~\ref{lemma_single_cell}), we can ensure that the traces of the maps $\varphi_{j}$ coincide on the mantles of the $3$-cells $D_{2,j}$. Thus, we deduce that $\varphi_{h} \in W^{1,p}((B^{2}_{1}\setminus \smash{\overline{B}}^{2}_{1-h}) \times (z_{-},z_{+}), \mathcal{N})$.  Observe that \ref{it_AhxeeCeeth8ieghooloboL4a} and \ref{it_ChahghohLail8jah6Aibohng} are satisfied, where \ref{it_ChahghohLail8jah6Aibohng} comes from the fact that
    \begin{equation*}
    \begin{split}
    \int_{(B^{2}_{1} \setminus \smash{\overline{B}}^{2}_{1-h})\times \{z\}}|D_{\top} \operatorname{tr}_{(B^{2}_{1} \setminus \smash{\overline{B}}^{2}_{1-h})\times \{z\}}(\varphi_{h})|^{p}\diff \mathcal{H}^{2} &=\int_{1-h}^{1} r^{1-p} \diff r \int_{\mathbb{S}^{1}\times \{z\}}|D_{\top} g|^{p}\diff \mathcal{H}^{1}\\ & \leq 2h \int_{\mathbb{S}^{1} \times \{z\}} |D_{\top} g|^{p}\diff \mathcal{H}^{1}
    \end{split}
    \end{equation*}
    for each $z \in \{z_{-}, z_{+}\}$ and \eqref{est 5.30}. Furthermore, summing \eqref{nu549j54j945jgu95jgu65ugj} over $j \in \{1,\dotsc, k_{2}\}$, taking into account that each $1$-cell of the Lipschitz triangulation of $\overline{\Gamma}_{1, L}$ generated by $\Sigma$ lies in the relative boundary of at most two $2$-cells from this triangulation, and also using \eqref{estliftingcylinder} together with \eqref{est 5.30} and \eqref{uniformfactskeleton}, we deduce the following chain of estimates 
\begin{equation*}
\begin{split}
\int_{(B^{2}_{1}\setminus \smash{\overline{B}}^2_{1-h})\times (z_{-}, z_{+})}|D\varphi_{h}|^{p}\diff x 
&\leq C^{\prime \prime} h \biggl(\int_{\partial B^{2}_{1-h} \times (z_{-}, z_{+})}|D_{\top}v_{h}|^{p}\diff \mathcal{H}^{2}+\int_{\mathbb{S}^{1}\times (z_{-},z_{+})}|D_{\top} g|^{p}\diff \mathcal{H}^{2}\biggr) \\ & \qquad \qquad +2C^{\prime \prime} h^{2}\int_{\Sigma+\mathrm{e}}|D_{\top}g|^{p}\diff \mathcal{H}^{1}\\
&\leq C^{\prime \prime \prime}h \int_{\Gamma_{1,L_{0}}}|D_{\top}g|^{p}\diff \mathcal{H}^{2},
\end{split}
\end{equation*} 
where $C^{\prime \prime \prime}=C^{\prime \prime \prime}(\mathcal{N})>0$. This completes our proof of \ref{it_Bit7ooRoo3ceithiegeej0xi} and hence of Lemma~\ref{finding nice v and conmapping cylinder}.
\end{proof}

\begin{cor}\label{cor on cylinder}
        Let $p_{0} \in (1,2)$, $p \in [p_{0}, 2)$, $\delta \in (0,1/2]$ and $g \in~W^{1,p}(\Gamma_{\delta r,r}, \mathcal{N})$, where $\Gamma_{\delta r,r}$ is defined in \eqref{def cyl}. Then there exists $\tau=\tau(p_{0}, \mathcal{N})>0$ such that the following holds. Assume that 
        \begin{equation}\label{penerboundsurf}
            \int_{\Gamma_{\delta r, r}}\frac{|D_{\top}g|^{p}}{p}\diff \mathcal{H}^{2} \leq \frac{\tau (\delta r)^{2-p}}{2-p}.
        \end{equation}
        Then there exist points $z_{-} \in (-r+\delta r/2,-r+\delta r)$, $z_{+}\in (r-\delta r, r-\delta r/2)$ and a mapping $\varphi_{\delta}\in W^{1,p}((B^{2}_{\delta r }\setminus \smash{\overline{B}}^{2}_{\delta r/2}) \times (z_{-},z_{+}), \mathcal{N})$ such that the following assertions hold.
        \begin{enumerate}[label=(\roman*)]
            \item    
            \label{tracepropertyonthelateralsurfaces12}
           \(\operatorname{tr}_{\partial B^{2}_{\delta r}\times I}(\varphi_{\delta})=g|_{\partial B^{2}_{\delta r}\times I},\) \(
             v_{\delta} \coloneqq
                \operatorname{tr}_{\partial B^{2}_{\delta r/2}\times I}(\varphi_{\delta}) \in W^{1,2}(\partial B^{2}_{\delta r/2}\times I, \mathcal{N}),
            \)
            where $I=(z_{-},z_{+})$.
            \item 
            \label{it_Ahk5aboQu8vaighoo6ashaev}
            For each $z \in \{z_{-},z_{+}\}$, $\operatorname{tr}_{(B^{2}_{\delta r}\setminus \smash{\overline{B}}^{2}_{\delta r/2})\times \{z\}}(\varphi_{\delta})\in W^{1,p}((B^{2}_{\delta r}\setminus \smash{\overline{B}}^{2}_{\delta r/2})\times \{z\}, \mathcal{N})$ 
            and 
            \[
            \int_{(B^{2}_{\delta r}\setminus \smash{\overline{B}}^{2}_{\delta r/2})\times \{z\}}\frac{|D_{\top} \operatorname{tr}_{(B^{2}_{\delta r}\setminus \smash{\overline{B}}^{2}_{\delta r/2})\times \{z\}}(\varphi_{\delta})|^{p}}{p}\diff \mathcal{H}^{2}\leq A_{6} \int_{\Gamma_{\delta r, r}}\frac{|D_{\top} g|^{p}}{p}\diff \mathcal{H}^{2}.
            \]
            \item 
            \label{it_aisheraesaenai1ohCoh2Eer}
            There exists a constant $C=C(\mathcal{N})>0$ such that
            \begin{equation*}
                \int_{\partial B^{2}_{\delta r/2}\times (z_{-},z_{+})}\frac{|D_{\top}v_{\delta}|^{p}}{p}\diff \mathcal{H}^{2}\leq C\int_{\Gamma_{\delta r,r}}\frac{|D_{\top}g|^{p}}{p}\diff \mathcal{H}^{2}
            \end{equation*}
            and
            \begin{equation*}
                \int_{(B^{2}_{\delta r}\setminus \smash{\overline{B}}^{2}_{\delta r/2})\times (z_{-},z_{+})}\frac{|D\varphi_{\delta}|^{p}}{p}\diff x \leq C\delta r \int_{\Gamma_{\delta r, r}}\frac{|D_{\top}g|^{p}}{p} \diff \mathcal{H}^{2}.
            \end{equation*}
        \end{enumerate}
\end{cor}

\begin{proof}[Proof of Corollary~\ref{cor on cylinder}]
    Define $f(x)\coloneqq g(\delta r x)$ for each $x \in \Gamma_{1,1/\delta}$. Let $\kappa=\kappa(p_{0}, \mathcal{N})>0$ be the constant of Lemma~\ref{finding nice v and conmapping cylinder}. Fix $\tau\in (0, \kappa/2]$. Then, in view of \eqref{penerboundsurf}, it holds 
    \[
    \int_{\Gamma_{1,1/\delta}}\frac{|D_{\top}f|^{p}}{p}\diff \mathcal{H}^{2}\leq \frac{\tau}{2-p}\leq \frac{\kappa}{2(2-p)}<\frac{\kappa (1/2)^{2-p}}{2-p}.
    \]
    This allows us to apply Lemma~\ref{finding nice v and conmapping cylinder} with $L_{0}=~1/\delta$, $h=1/2$ for the map $f \in W^{1,p}(\Gamma_{1,1/\delta}, \mathcal{N})$. 
    Thus, we obtain points $\xi_{-}\in (-1/\delta+1/2,-1/\delta+1)$, $\xi_{+}\in (1/\delta-1,1/\delta-1/2)$ and maps $w\in W^{1,2}(\partial B^{2}_{1/2}\times (\xi_{-},\xi_{+}), \mathcal{N})$, $\psi\in W^{1,p}((B^{2}_{1}\setminus \smash{\overline{B}}^{2}_{1/2})\times (\xi_{-},\xi_{+}), \mathcal{N})$ such that the assertions \ref{it_AhxeeCeeth8ieghooloboL4a}-\ref{it_Bit7ooRoo3ceithiegeej0xi} of Lemma~\ref{finding nice v and conmapping cylinder} hold for $z_{-},z_{+},v_{h}, \varphi_{h}$ replaced by $\xi_{-}, \xi_{+}, w, \psi$, respectively. 
    Scaling back to $\Gamma_{\delta r,r}$, namely, defining $z_{-}=\delta r\xi_{-}$, $z_{+}=\delta r\xi_{+}$, $v_{\delta}(\cdot)=w(\cdot/\delta r)$ and $\varphi_{\delta}(\cdot)=\psi(\cdot/\delta r)$, we complete our proof of Corollary~\ref{cor on cylinder}.
\end{proof}

\begin{prop}\label{prop energy bounds on cylinders}
    Let $p_{0} \in (1,2)$, $p\in [p_{0},2)$, $\delta \in (0,1/2]$ and $g \in W^{1,p}(\partial \Lambda_{\delta r,r}, \mathcal{N})$, where $\Lambda_{\delta r,r}$ and $\Gamma_{\delta r,r}$ are defined in \eqref{def cyl}. Let $\tau=\tau(p_{0}, \mathcal{N})>0$ be the constant of Corollary~\ref{cor on cylinder}. Assume that 
    \begin{equation*}
    \int_{\Gamma_{\delta r, r}}\frac{|D_{\top}g|^{p}}{p}\diff \mathcal{H}^{2}\leq \frac{\tau (\delta r)^{2-p}}{2-p}. 
    \end{equation*} 
    Let $u$ be a $p$-minimizer in $\Lambda_{\delta r,r}$ with $\tr_{\partial \Lambda_{\delta r,r}}(u)=g$. Then there exists $\gamma \in [\mathbb{S}^{1}, \mathcal{N}]$ such that the following estimates hold
    \begin{equation}
    \label{upper bound cylinder}
        \begin{split}
    \int_{\Lambda_{\delta r, r}}\frac{|Du|^{p}}{p}\diff x& \leq \frac{2\mathcal{E}^{\mathrm{sg}}_{p}(\gamma)r^{3-p}}{2-p} +  C\delta^{1-p}r\int_{\Gamma_{\delta r,r}}\frac{|D_{\top}g|^{p}}{p}\diff \mathcal{H}^{2} \\
    & \,\ +C\delta r\biggl(\int_{B^{2}_{\delta r}\times \{-r\}}\frac{|D_{\top}g|^{p}}{p}\diff \mathcal{H}^{2}+\int_{B^{2}_{\delta r}\times \{r\}}\frac{|D_{\top}g|^{p}}{p}\diff \mathcal{H}^{2}\biggr)
    +\frac{C(\delta r)^{3-p}}{2-p}
\end{split}
    \end{equation}
    and
    \begin{equation} 
    \label{lower bound cylinder}
    (2-2\delta)\frac{\mathcal{E}^{\mathrm{sg}}_{p/(p-1)}(\gamma)\delta^{2-p}r^{3-p}}{2-p} - C\delta r\int_{\Gamma_{\delta r,r}}\frac{|D_{\top}g|^{p}}{p}\diff \mathcal{H}^{2} \leq \int_{\Lambda_{\delta r, r}}\frac{|Du|^{p}}{p}\diff x
    \end{equation}
    where $C=C(\mathcal{N})>0$.
\end{prop}
\begin{rem}
The singular $p$-energies $\mathcal{E}^{\mathrm{sg}}_{p}(\gamma)$ and $\mathcal{E}^{\mathrm{sg}}_{p/(p-1)}(\gamma)$ appearing in \eqref{upper bound cylinder} and \eqref{lower bound cylinder}, respectively, are not the same in general (see Definition~\ref{def p-singular}).
\end{rem}

Before proving Proposition~\ref{prop energy bounds on cylinders}, let us consider the case where $r=1$, $\gamma \in C^{1}(\mathbb{S}^{1}, \mathcal{N})$ and the map $g \in W^{1,p}(\partial \Lambda_{\delta,1}, \mathcal{N})$ is defined on the closure of the lateral surface $\Gamma_{\delta, 1}$ (see \eqref{def cyl}) as $g(\sigma, z)=\gamma(\sigma/\delta)$ for each point $(\sigma, z) \in \partial B^{2}_{\delta}\times [-1,1]$ and on the ball $B^{2}_{\delta}\times \{z\}$ with $z \in \{-1, 1\}$, as the extension of the mapping $g(\cdot, z) \in C^{1}(\partial B^{2}_{\delta} \times \{z\}, \mathcal{N})$ coming from Lemma~\ref{lem 2.17}, so that $g$ is of class $C^{1}$ on $\Gamma_{\delta, 1}$ and its homotopy class on $\Gamma_{\delta, 1}$ (see Remark~\ref{homotopy class}) is given by $[\gamma] \in [\mathbb{S}^{1}, \mathcal{N}]$. 
The upper bound \eqref{upper bound cylinder} for the $p$-energy on $\Lambda_{\delta,1}$ of a minimizer $u \in W^{1,p}(\Lambda_{\delta, 1}, \mathcal{N})$ with $\tr_{\partial \Lambda_{\delta, 1}}(u)=g$ comes from Lemma~\ref{lem 2.19}~\ref{it_pha8ieh4Umaejao2vaephei8}, while the lower bound \eqref{lower bound cylinder} for the $p$-energy of $u$ on $\Lambda_{\delta,1}$ is a consequence of the application of the coarea formula and Proposition~\ref{Sandier lemma}.

\begin{proof}[Proof of Proposition~\ref{prop energy bounds on cylinders}]
 By scaling, it is enough to prove Proposition~\ref{prop energy bounds on cylinders} in the case when $r=1$. So we assume that $r=1$. In view of the conditions of Proposition~\ref{prop energy bounds on cylinders}, we can apply Corollary~\ref{cor on cylinder} to $g$ in $\Gamma_{\delta,1}$. Let points $z_{-} \in (-1+\delta/2, -1+\delta)$, $z_{+}\in (1-\delta, 1-\delta/2)$ and maps $v_{\delta} \in W^{1,2}(\partial B^{2}_{\delta/2}\times (z_{-}, z_{+}), \mathcal{N})$, $\varphi_{\delta} \in W^{1,p}((B^{2}_{\delta}\setminus \smash{\overline{B}}^{2}_{\delta/2})\times (z_{-},z_{+}), \mathcal{N})$ be given by Corollary~\ref{cor on cylinder} applied to our $g\in W^{1,p}(\Gamma_{\delta,1}, \mathcal{N})$. According to Remark~\ref{homotopy class}, $v_{\delta}$ has a well-defined ``homotopy class'' on $\partial B^{2}_{\delta/2}\times (z_{-}, z_{+})$. Thus, there exists $\gamma \in [\mathbb{S}^{1}, \mathcal{N}]$ such that for $\mathcal{H}^{1}$-a.e.\  $z \in (z_{-}, z_{+})$, $[v_{\delta}|_{\partial B^{2}_{\delta/2}\times \{z\}}]=[\gamma]$. This defines our $\gamma$. Hereinafter in this proof, $C$ will denote a positive constant that can depend only on $\mathcal{N}$ and can be different from line to line. 
 \medbreak
 \noindent
 \emph{Step~1.} We prove the  upper bound \eqref{upper bound cylinder}. For convenience, in order to construct a competitor for $u$, we define the following subdomains of $\Lambda_{\delta,1}$:
 \begin{align*}
 \Lambda^{-}_{\delta}&\coloneqq B^{2}_{\delta}\times (-1,z_{-}),&
 \Lambda^{0}_{\delta}&\coloneqq B^{2}_{\delta/2}\times (z_{-},z_{+}),\\
 E_{\delta} &\coloneqq (B^{2}_{\delta}\setminus \smash{\overline{B}}^{2}_{\delta/2})\times (z_{-},z_{+}),& \Lambda^{+}_{\delta}&\coloneqq B^{2}_{\delta}\times (z_{+},1).
 \end{align*}
 Let us construct a competitor $w$ for $u$, namely $w\in W^{1,p}(\Lambda_{\delta,1}, \mathcal{N})$ with $\tr_{\partial \Lambda_{\delta,1}}(w)=g$. 
 We define $w$ in $\Lambda^{0}_{\delta}$ by using Lemma~\ref{lem 2.19}~\ref{it_pha8ieh4Umaejao2vaephei8}, namely we obtain a map $w|_{\Lambda^{0}_{\delta}} \in W^{1,p}(\Lambda^{0}_{\delta},\mathcal{N})$ such that 
 \begin{equation}
  \label{est 5.38}
  \begin{split}
 \int_{\Lambda^{0}_{\delta}}\frac{|Dw|^{p}}{p}\diff x
 &\leq \frac{(z_{+}-z_{-})\mathcal{E}^{\mathrm{sg}}_{p}(\gamma)\delta^{2-p}}{(2-p)2^{2-p}}+C\biggl(\frac{(z_{+}-z_{-})^{p}}{2\delta^{p-1}}+\frac{\delta}{2}\biggr)\int_{\partial B^{2}_{\delta/2}\times (z_{-},z_{+})}\frac{|D_{\top}v_{\delta}|^{p}}{p}\diff \mathcal{H}^{2}\\
 & \leq \frac{2\mathcal{E}^{\mathrm{sg}}_{p}(\gamma)}{2-p}+\frac{C}{\delta^{p-1}}\int_{\Gamma_{\delta,1}}\frac{|D_{\top}g|^{p}}{p}\diff \mathcal{H}^{2},
 \end{split}
 \end{equation}
 where we have also used that $(z_{+}-z_{-})\leq 2$, $\delta \in (0,1/2]$, $p\in (1,2)$ and Corollary~\ref{cor on cylinder}~\ref{it_aisheraesaenai1ohCoh2Eer}.
 Next, we define the mapping $g^{+}:\partial \Lambda^{+}_{\delta}\to \mathcal{N}$ by
 \[
 g^{+}(x)\coloneqq 
 \begin{cases}
 g(x) \,\ &\text{if} \,\ x \in \partial \Lambda^{+}_{\delta} \cap \partial \Lambda_{\delta,1},\\ 
 \operatorname{tr}_{(B^{2}_{\delta}\setminus \smash{\overline{B}}^{2}_{\delta/2})\times \{z_{+}\}}(\varphi_{\delta})(x) \,\ &\text{if} \,\ x \in (B^{2}_{\delta}\setminus \smash{\overline{B}}^{2}_{\delta/2})\times \{z_{+}\},\\
 \operatorname{tr}_{\smash{\overline{B}}^{2}_{\delta/2}\times \{z_{+}\}}(w)(x) \,\ &\text{if} \,\ x \in \smash{\overline{B}}^{2}_{\delta/2}\times \{z_{+}\}. 
 \end{cases}
 \]
 By Lemma~\ref{lem 2.19}~\ref{it_cang8wooGhodoh3Eegoom7ee}, $g^{+}|_{B^{2}_{\delta/2}\times \{z_{+}\}} \in W^{1,p}(B^{2}_{\delta/2}\times \{z_{+}\}, \mathcal{N})$. 
 On the other hand, according to Corollary~\ref{cor on cylinder}~\ref{it_Ahk5aboQu8vaighoo6ashaev},  $g^{+}|_{(B^{2}_{\delta}\setminus \smash{\overline{B}}^{2}_{\delta/2})\times \{z_{+}\}} \in W^{1,p}((B^{2}_{\delta}\setminus \smash{\overline{B}}^{2}_{\delta/2})\times \{z_{+}\}, \mathcal{N})$.
 Thus, taking into account that the traces of $g$ and $\varphi_{\delta}$ coincide on $\partial B^{2}_{\delta} \times \{z_{+}\}$, and the traces of $\varphi_{\delta}$ and $w$ coincide on $\partial B^{2}_{\delta/2}\times \{z_{+}\}$, we conclude that $g^{+}\in W^{1,p}(\partial \Lambda^{+}_{\delta}, \mathcal{N})$. 
 In particular, using Lemma~\ref{lem 2.19}~\ref{it_cang8wooGhodoh3Eegoom7ee}, we obtain
 \begin{equation}
  \label{est 5.39}
  \begin{split}
  \int_{\partial \Lambda^{+}_{\delta}}\frac{|D_{\top} g^{+}|^{p}}{p}\diff \mathcal{H}^{2}&\leq  \int_{ \Gamma_{\delta,1}}\frac{|D_{\top}g|^{p}}{p}\diff \mathcal{H}^{2}+ \int_{B^{2}_{\delta}\times \{1\}}\frac{|D_{\top}g|^{p}}{p}\diff \mathcal{H}^{2}+\frac{\mathcal{E}^{\mathrm{sg}}_{p}(\gamma) \delta^{2-p}}{(2-p)2^{2-p}}\\  
 &\,\ \,\ +C\biggl(\biggl(\frac{z_{+}-z_{-}}{\delta}\biggr)^{p-1}+\frac{\delta}{z_{+}-z_{-}}\biggr)\int_{\Gamma_{\delta,1}}\frac{|D_{\top}g|^{p}}{p}\diff \mathcal{H}^{2}\\
 & \leq \frac{C}{\delta^{p-1}}\int_{\Gamma_{\delta,1}}\frac{|D_{\top}g|^{p}}{p}\diff \mathcal{H}^{2}+ \int_{B^{2}_{\delta}\times \{1\}}\frac{|D_{\top}g|^{p}}{p}\diff \mathcal{H}^{2}+\frac{C\delta^{2-p}}{2-p},  
 \end{split}
 \end{equation}
where the last estimate comes since $(z_{+}-z_{-})\leq 2$, $\delta \in (0,1/2]$, $p\in (1,2)$ and $\#[\mathbb{S}^{1}, \mathcal{N}]<+\infty$ (hence $\mathcal{E}^{\mathrm{sg}}_{p}(\gamma)\leq C$).  Notice that $\smash{\overline{\Lambda}}^{+}_{\delta}$ is bilipschitz homeomorphic to $\smash{\overline{B}}^3_{\delta}$, namely, there exists a bilipschitz mapping $\Phi:\smash{\overline{\Lambda}}^{+}_{\delta} \to \smash{\overline{B}}^3_{\delta}$ with an absolute bilipschitz constant. We define
  \begin{align*}
  w|_{\Lambda^{+}_{\delta}}(x)
  & \coloneqq g^{+}\left(\Phi^{-1}\left(\frac{\delta \Phi(x)}{|\Phi(x)|}\right)\right)\,\ && \text{for}\,\ x \in \Lambda^{+}_{\delta}\setminus \{\Phi^{-1}(0)\}.
 \end{align*}
 Then, using \eqref{est 5.39}, we obtain 
 \begin{equation}
 \begin{split}
 \label{est 5.42}
  \int_{\Lambda^{+}_{\delta}}\frac{|Dw|^{p}}{p}\diff x &\leq C\delta \int_{\partial \Lambda^{+}_{\delta}}\frac{|D_{\top}g^{+}|^{p}}{p}\diff \mathcal{H}^{2} \\ &\leq C \delta^{2-p}\int_{\Gamma_{\delta,1}}\frac{|D_{\top}g|^{p}}{p}\diff \mathcal{H}^{2}+C\delta \int_{B^{2}_{\delta}\times \{1\}}\frac{|D_{\top}g|^{p}}{p}\diff \mathcal{H}^{2}+\frac{C\delta^{3-p}}{2-p}.  
  \end{split}
 \end{equation}
We carry out a similar construction on the set $\Lambda^{-}_{\delta}$ and obtain a map $w|_{\Lambda^{-}_{\delta}} \in W^{1,p}(\Lambda^{-}_{\delta}, \mathcal{N})$ satisfying the same type of energy control as in \eqref{est 5.42}.
 In $E_{\delta}$, we define $w|_{E_{\delta}}=\varphi_{\delta}$. By Corollary~\ref{cor on cylinder}~\ref{it_aisheraesaenai1ohCoh2Eer},
 \begin{equation}\label{est 5.43}
 \int_{E_{\delta}}\frac{|Dw|^{p}}{p}\diff x \leq C\delta \int_{\Gamma_{\delta,1}}\frac{|D_{\top}g|^{p}}{p}\diff \mathcal{H}^{2}.
 \end{equation}
 Altogether, we have constructed the map $w \in W^{1,p}(\Lambda_{\delta,1}, \mathcal{N})$ with $\tr_{\partial \Lambda_{\delta,1}}(w)=g$ (we refer to Corollary~\ref{cor on cylinder}~\ref{tracepropertyonthelateralsurfaces12}). Since $u$ is a $p$-minimizer in $\Lambda_{\delta,1}$ with $\tr_{\partial \Lambda_{\delta,1}}(u)=g$ and, in view of \eqref{est 5.38}, \eqref{est 5.42}, \eqref{est 5.43}, the facts that $\delta\in (0,1/2]$ and $p\in (1,2)$, we have
 \begin{align*}
 \int_{\Lambda_{\delta,1}}\frac{|Du|^{p}}{p}\diff x &\leq \int_{\Lambda_{\delta,1}}\frac{|Dw|^{p}}{p}\diff x \leq \frac{2\mathcal{E}^{\mathrm{sg}}_{p}(\gamma)}{2-p} +  \frac{C}{\delta^{p-1}}\int_{\Gamma_{\delta,1}}\frac{|D_{\top}g|^{p}}{p}\diff \mathcal{H}^{2} \\ & \,\ +C\delta\left(\int_{B^{2}_{\delta}\times \{-1\}}\frac{|D_{\top}g|^{p}}{p}\diff \mathcal{H}^{2}+\int_{B^{2}_{\delta}\times \{1\}}\frac{|D_{\top}g|^{p}}{p}\diff \mathcal{H}^{2}\right)+\frac{C\delta^{3-p}}{2-p},
\end{align*}
which yields \eqref{upper bound cylinder}.
\medbreak
\noindent \emph{Step~2.} We prove the lower bound \eqref{lower bound cylinder}. For each $(\varrho, \sigma, z) \in [0,\delta)\times \mathbb{S}^{1} \times (z_{-}, z_{+})$, define the map $\psi: B^{2}_{\delta}\times (z_{-},z_{+})$ by
\[
\psi(\varrho \sigma, z) 
\coloneqq
    \begin{cases}
        u(2\varrho \sigma, z) \,\ &\text{if \(\varrho \in [0, \delta/2)\) and \(z \in (z_{-},z_{+})\)},\\
        \varphi_{\delta}((3\delta/2-\varrho)\sigma,z) \,\ &\text{if \(\varrho \in [\delta/2, \delta)\) and \(z \in (z_{-},z_{+})\)}.
    \end{cases}
\]
Then $\psi \in W^{1,p}(B^{2}_{\delta}\times (z_{-},z_{+}), \mathcal{N})$ and $\tr_{\partial B^{2}_{\delta}\times (z_{-},z_{+})}(\psi)=v_{\delta}(\cdot/2, \cdot) \in W^{1,2}(\partial B^{2}_{\delta} \times (z_{-},z_{+}), \mathcal{N})$. By the definition of $\psi$ and Corollary~\ref{cor on cylinder}~\ref{it_aisheraesaenai1ohCoh2Eer}, 
\begin{equation}
\label{est 5.45}
\begin{split}
\int_{B^{2}_{\delta}\times (z_{-},z_{+})}\frac{|D\psi|^{p}}{p}\diff x &\leq \int_{B^{2}_{\delta}\times (z_{-},z_{+})}\frac{|Du|^{p}}{p}\diff x +  \int_{E_{\delta}}\frac{|D\varphi_{\delta}|^{p}}{p}\diff x\\ &\leq \int_{B^{2}_{\delta}\times (z_{-},z_{+})}\frac{|Du|^{p}}{p}\diff x + C\delta\int_{\Gamma_{\delta,1}}\frac{|D_{\top} g|^{p}}{p}\diff \mathcal{H}^{2}. 
\end{split}
\end{equation}
Using the coarea formula and that $\tr_{\partial B^{2}_{\delta}\times (z_{-},z_{+})}(\psi)=v_{\delta}(\cdot/2, \cdot) \in W^{1,2}(\partial B^{2}_{\delta} \times (z_{-},z_{+}), \mathcal{N})$, we observe that $\tr_{\partial B^{2}_{\delta}\times \{z\}}(\psi)=\psi|_{\partial B^{2}_{\delta}\times \{z\}} \in W^{1,2}(\partial B^{2}_{\delta}\times \{z\}, \mathcal{N})$ and $[\psi|_{\partial B^{2}_{\delta}\times \{z\}}]=[\gamma]$ for $\mathcal{H}^{1}$-a.e.\  $z \in (z_{-}, z_{+})$. Then, applying Proposition~\ref{Sandier lemma}, for $\mathcal{H}^{1}$-a.e.\  $z\in (z_{-},z_{+})$, we obtain
\begin{align*}
\int_{B^{2}_{\delta}\times \{z\}}\frac{|D\psi|^{p}}{p}\diff \mathcal{H}^{2}+\delta \int_{\partial B^{2}_{\delta}\times \{z\}}\frac{|D_{\top}\psi|^{p}}{p}\diff \mathcal{H}^{1}\geq \frac{\mathcal{E}^{\mathrm{sg}}_{p/(p-1)}(\gamma)\delta^{2-p}}{2-p}.
\end{align*}
Integrating both sides of the above inequality over $(z_{-},z_{+})$ with respect to $d z$, we get
\begin{equation}
\label{est 5.46}
\begin{split}
\int_{B^{2}_{\delta} \times (z_{-},z_{+})}\frac{|D\psi|^{p}}{p}\diff x &\geq \frac{(z_{+}-z_{-})\mathcal{E}^{\mathrm{sg}}_{p/(p-1)}(\gamma)\delta^{2-p}}{2-p} - \frac{\delta}{2^{p-1}}  \int_{\partial B^{2}_{\delta/2} \times (z_{-},z_{+})}\frac{|Dv_{\delta}|^{p}}{p}\diff \mathcal{H}^{2}\\
&\geq (2-2\delta)\frac{\mathcal{E}^{\mathrm{sg}}_{p/(p-1)}(\gamma)\delta^{2-p}}{2-p} - C\delta\int_{\Gamma_{\delta,1}}\frac{|D_{\top}g|^{p}}{p}\diff \mathcal{H}^{2}, 
\end{split}
\end{equation}
where we have also used Corollary~\ref{cor on cylinder}~\ref{it_aisheraesaenai1ohCoh2Eer}. Combining \eqref{est 5.45} and \eqref{est 5.46}, we obtain \eqref{lower bound cylinder}. This completes our proof of Proposition~\ref{prop energy bounds on cylinders}.
\end{proof}
\subsection{The limit measure has the density with a finite set of values}    
Recall that $\mu_{*}$ is the positive Radon measure on $\overline{\Omega}$ defined at the beginning of Section~\ref{section_singular_set} and its support (\emph{i.e.}, $\mathrm{supp}(\mu_{*})$) is the set $S_{*}$. We prove the following 
\begin{prop}
\label{prop density dv}
    For $\mathcal{H}^{1}$-a.e.\  $x \in S_{*}\cap \Omega$, we have 
    \(\Theta_{1}(\mu_{*},x) \in \{\mathcal{E}^{\mathrm{sg}}_{2}(\gamma): \gamma \in [\mathbb{S}^{1}, \mathcal{N}]\} \setminus \{0\}
    \).
\end{prop}
\begin{proof}
    According to Proposition~\ref{prop 1-var} and \cite[Theorem~2.83~(i)]{APD}, $\mu_{*}$ admits an approximate tangent space with multiplicity $\Theta_{1}(x_{0})$ for $\mathcal{H}^{1}$-a.e.\  $x_{0} \in S_{*}\cap \Omega$. Namely, for $\mathcal{H}^{1}$-a.e.\  $x_{0} \in S_{*}\cap \Omega$ (or, equivalently, for $\mu_{*}$-a.e.\  $x_{0} \in \Omega$), there exists a unique approximate tangent plane $T_{x_{0}}S_{*} \in \mathrm{G}(3,1)$ to $S_{*}$ at $x_{0}$ and
    \begin{align}
    \label{weak* conv 5.47}
    \mu_{x_{0},r}=T^{x_{0},r}_{\#} \mu_{*}\mres \Omega & \overset{*}{\rightharpoonup} \mu_{0}=\Theta_{1}(x_{0})\mathcal{H}^{1}\mres T_{x_{0}}S_{*} && \text{weakly* in}\,\ (C_{0}(\mathbb{R}^{3}))^{\prime}\,\ \text{as}\,\ r\searrow 0,
    \end{align}
    where $T^{x_{0},r}(x)\coloneqq (x-x_{0})/r$. 
    Up to rotation and translation, to lighten the notation, we shall assume that $x_{0}=0$ and $T_{x_{0}}S_{*}=\{(0,0)\} \times \mathbb{R}$. Recall that $(p_{n})_{n\in \mathbb{N}}\subset [1,2)$,  $p_{n}\nearrow 2$ as $n\to +\infty$ and for each $n\in \mathbb{N}$ large enough,
    \begin{equation}\label{gl energy}
    \int_{\Omega}\frac{|Du_{n}|^{p_{n}}}{p_{n}}\diff x \leq \frac{C_{0}}{2-p_{n}},
    \end{equation}
    where $C_{0}>0$ is a constant independent of $n$ (see \eqref{C2}). 
    
    Let $\delta \in (0,1/2)$ be fairly small, which we shall carefully choose later for the proof to work. Let $\tau>0$ be the constant of Corollary~\ref{cor on cylinder}, where $p_{0}=3/2$. Let $\eta>0$ be the constant of Lemma~\ref{heart}, where $\kappa=1/2$, $p_{0}=3/2$ and $\Psi$ is a bilipschitz homeomorphism between a ball and a cube with an absolute bilipschitz constant (see, for instance, \cite[Corollary~3]{bilipschitz}). We fix some $\varepsilon \in (0,\min\{\eta,\tau\})$. Next, since $\mu_{0}((\smash{\overline{B}}^{2}_{3\delta/2}\setminus B^{2}_{\delta/2})\times [-2,2])=0$, in view of the weak* convergence (see \eqref{weak* conv 5.47}), there exists $r_{0}=r_{0}(\varepsilon,\delta) \in (0, \min\{1,\dist(x_{0}, \partial \Omega)/4\})$ such that 
    \begin{equation}\label{est 5.48}
    \mu_{*}((\smash{\overline{B}}^{2}_{3\delta r/2}\setminus B^{2}_{\delta r/2})\times [-2r,2r])< \frac{\varepsilon \delta r}{100} \,\ \text{for each}\,\ r \in (0,r_{0}].
    \end{equation}
    Taking into account \eqref{wcm} and \eqref{two cond of weak*}, the estimate \eqref{est 5.48}, together with the Fatou lemma, implies that for a fixed $r \in (0,r_{0}]$, 
    \begin{equation}\label{est by Fatou 1}
    \int_{\delta r/2}^{3\delta r/2}\diff \varrho \,\liminf_{n\to +\infty}(2-p_{n})\int_{\Gamma_{\varrho, r}} \frac{|Du_{n}|^{p_{n}}}{p_{n}}\diff \mathcal{H}^{2} < \frac{\varepsilon \delta r}{100}.
    \end{equation}
    Due to \eqref{est by Fatou 1}, there exist $\varrho \in (3\delta r/4, 5\delta r/4)$ and $n_{0}=n_{0}(\delta, r, \varepsilon) \in \mathbb{N}$ such that, up to a subsequence (not relabeled), for each $n \geq n_{0}$, $\tr_{\Gamma_{\varrho,r}}(u_{n})=u_{n}|_{\Gamma_{\varrho,r}} \in W^{1,p_{n}}(\Gamma_{\varrho,r}, \mathcal{N})$  (see \eqref{def cyl}) and 
    \begin{equation}\label{est 5.51}
    \int_{\Gamma_{\varrho,r}}\frac{|Du_{n}|^{p_{n}}}{p_{n}}\diff \mathcal{H}^{2}<\frac{\varepsilon \varrho^{2-p_{n}}}{2-p_{n}},
    \end{equation}
    where we have used that $\varrho^{2-p_{n}}\to 1$ as $n \to +\infty$. To simplify the notation, without loss of generality, we assume that $\varrho=\delta r$. Notice that $\Lambda_{\delta r, 5r/4} \subset B^{3}_{\sqrt{2}r}$, where $\sqrt{2}r<d_{0}\coloneqq \dist(x_{0}, \partial \Omega)/2$ (since it holds $r \leq r_{0}<d_{0}/2)$. Thus, using the monotonicity of the $p$-energy (see Lemma~\ref{lem mon of p-energy}) and \eqref{gl energy}, for each $n \in \mathbb{N}$ large enough, we get
    \[
    \int_{\Lambda_{\delta r, 5r/4}}\frac{|Du_{n}|^{p_{n}}}{p_{n}}\diff x \leq \left(\frac{\sqrt{2}r}{d_{0}}\right)^{3-p_{n}}\int_{B^{3}_{d_{0}}}\frac{|Du_{n}|^{p_{n}}}{p_{n}}\diff x \leq \left(\frac{\sqrt{2}r}{d_{0}}\right)^{3-p_{n}}\frac{C_{0}}{2-p_{n}}.
    \]
    Using again the Fatou lemma, it holds
    \[
    \int_{-5r/4}^{5r/4}\diff z \, \liminf_{n\to+\infty}(2-p_{n})\int_{B^{2}_{\delta r}\times \{z\}}\frac{|Du_{n}|^{p_{n}}}{p_{n}}\diff \mathcal{H}^{2} \leq \frac{C_{0}\sqrt{2}r}{d_{0}}.
    \]
    This implies that there exist $z^{r}_{-} \in (-5r/4, -3r/4)$ and $z^{r}_{+}\in (3r/4, 5r/4)$ such that, up to a subsequence (still denoted by the same index), for each $z\in \{z^{r}_{-}, z^{r}_{+}\}$ we have the following $\tr_{B^{2}_{\delta r}\times \{z\}}(u_{n})=u_{n}|_{B^{2}_{\delta r}\times \{z\}} \in W^{1,p_{n}}(B^{2}_{\delta r}\times \{z\}, \mathcal{N})$ and for each $n \in \mathbb{N}$ large enough,
    \begin{equation}\label{est 5.52}
    \int_{B^{2}_{\delta r}\times \{z\}}\frac{|Du_{n}|^{p_{n}}}{p_{n}}\diff \mathcal{H}^{2}\leq \frac{C_{1} }{2-p_{n}},
    \end{equation}
    where $C_{1}=C_{1}(C_{0}, d_{0})>0$. Without loss of generality, to lighten the notation, we shall assume that $z^{r}_{-}=-r$ and $z^{r}_{+}=r$. Thus,  $\tr_{\partial \Lambda_{\delta r,r}}(u_{n})=u_{n}|_{\partial \Lambda_{\delta r,r}}\in W^{1,p_{n}}(\partial \Lambda_{\delta r, r}, \mathcal{N})$. From \eqref{est 5.51} and \eqref{est 5.52} we obtain
    \begin{equation}\label{estimationsenersurf}
    \int_{\Gamma_{\delta r,r}}\frac{|D_{\top} \operatorname{tr}_{\Gamma_{\delta r,r}}(u_{n})|^{p_{n}}}{p_{n}}\diff \mathcal{H}^{2} \leq \frac{\varepsilon (\delta r)^{2-p_{n}}}{2-p_{n}} \,\ \text{and}\,\ \int_{B^{2}_{\delta r}\times \{z\}}\frac{|D_{\top} \operatorname{tr}_{B^{2}_{\delta r} \times \{z\}}(u_{n})|^{p_{n}}}{p_{n}}\diff \mathcal{H}^{2}\leq \frac{C_{1}}{2-p_{n}}
    \end{equation}
    for each $z \in \{-r,r\}$ and  $n\in\mathbb{N}$ large enough. Thus, for each sufficiently large $n \in \mathbb{N}$, $u_{n}$ fulfills the conditions of Proposition~\ref{prop energy bounds on cylinders} (recall that $0<\varepsilon<\tau$). Using Proposition~\ref{prop energy bounds on cylinders} and \eqref{estimationsenersurf}, we obtain $\gamma_{n} \in [\mathbb{S}^{1}, \mathcal{N}]$ such that the following holds. Firstly,
    \begin{equation}\label{est 5.53}
            \int_{\Lambda_{\delta r, r}}\frac{|Du_{n}|^{p_{n}}}{p_{n}}\diff x \leq \frac{2\mathcal{E}^{\mathrm{sg}}_{p_{n}}(\gamma_{n})r^{3-p_{n}}}{2-p_{n}} +  \frac{C\varepsilon\delta^{3-2p_{n}} r^{3-p_{n}}}{2-p_{n}} + \frac{C\delta r}{2-p_{n}}+\frac{C(\delta r)^{3-p_{n}}}{2-p_{n}}
        \end{equation}
    for each sufficiently large $n \in \mathbb{N}$, where $C=C(C_{1}, \mathcal{N})>0$. Secondly,
    \begin{equation} \label{est 5.54}
        (2-2\delta)\frac{\mathcal{E}^{\mathrm{sg}}_{p_{n}/(p_{n}-1)}(\gamma_{n})\delta^{2-p_{n}}r^{3-p_{n}}}{2-p_{n}} - \frac{C^{\prime}\varepsilon(\delta r)^{3-p_{n}}}{2-p_{n}}\leq \int_{\Lambda_{\delta r, r}}\frac{|Du_{n}|^{p_{n}}}{p_{n}}\diff x
    \end{equation}
       for each sufficiently large $n \in \mathbb{N}$, where $C^{\prime}=C^{\prime}(\mathcal{N})>0$. 
       
       Now, applying Lemma~\ref{lem concentration}, due to \eqref{est 5.48},  we can cover the compact set $\partial B^{2}_{\delta r}\times [-r, r]$ by a finite number of small cubes (which are bilipschitz homeomorphic to balls) being subsets of $\Omega\setminus S_{*}$ and obtain that the set \[E_{\delta, r}= (\smash{\overline{B}}^{2}_{101\delta r/100}\setminus B^{2}_{99\delta r/100})\times [-r-\delta r/100, r+\delta r /100]\] is  a subset of $\Omega \setminus S_{*}$. Denote by $E^{0}_{\delta, r}$ the interior of $E_{\delta,r}$. Then, by Proposition~\ref{conv to harm}, $u_{n} \rightharpoonup u_{*}$ weakly in $W^{1,p}(E^{0}_{\delta,r}, \mathbb{R}^{\nu})$ for all $p\in (1,2)$. Furthermore, there exists a finite set $\{x_{1},\dotsc,x_{l}\}\subset E^{0}_{\delta,r}$ such that $u_{*} \in C^{\infty}(E^{0}_{\delta, r}\setminus \{x_{1},\dotsc,x_{l}\}, \mathcal{N})$. By the weak convergence and the Sobolev embedding, without loss of generality, $u_{n}|_{\partial B^{2}_{\delta r}\times \{z\}}$ converges uniformly to $u_{*}|_{\partial B^{2}_{\delta r} \times \{z\}} \in C(\partial B^{2}_{\delta r}\times \{z\}, \mathcal{N})$ for each $z \in \{-r,r\}$. 
       
       We claim that for each $n \in \mathbb{N}$ large enough,  $[\gamma_{n}]=[u_{*}|_{\partial B^{2}_{\delta r}\times \{-r\}}]=[u_{*}|_{\partial B^{2}_{\delta r}\times \{r\}}]$. Indeed, 
    a close inspection of the proofs of Lemma~\ref{finding nice v and conmapping cylinder} and Corollary~\ref{cor on cylinder} shows that the mapping $v_{\delta,n} \in W^{1,2}(\partial B^{2}_{\delta r/2}\times~(z^{n}_{-}, z^{n}_{+}), \mathcal{N})$, coming from Corollary~\ref{cor on cylinder}, has a well-defined homotopy class on the surface $\partial B^{2}_{\delta r/2}\times (z^{n}_{-}, z^{n}_{+})$ and for each $z \in \{z^{n}_{-}, z^{n}_{+}\}$, $[v_{\delta,n}|_{\partial B^{2}_{\delta r/2} \times \{z\}}]=[\gamma_{n}]$, where $z^{n}_{-}\in (-r+\delta r/2, -r+\delta r)$ and $z^{n}_{+}\in (r-\delta r, r-\delta r/2)$ (the reader may consult the arguments given immediately above the estimate \eqref{est 5.28}, take into account Remark~\ref{homotopy class}, and observe that the trace on $\mathbb{S}^{1}\times \{L\}$ of the map $w_{h}\in W^{1,2}(\Gamma_{1,L}, \mathcal{N})$ constructed in the proof of Proposition~\ref{prop energy bounds on cylinders} admits a continuous representative). We can actually extend $\varphi_{\delta, n}$ on $(B^{2}_{\delta r}\setminus \smash{\overline{B}}^{2}_{\delta r/2})\times (-r,r)$ and $v_{\delta, n}$ on $\partial B^{2}_{\delta r/2}\times (-r,r)$, where $\varphi_{\delta, n}$ comes from Corollary~\ref{cor on cylinder}, so that these extensions have the following properties: $\varphi_{\delta,n} \in W^{1,p}((B^{2}_{\delta r}\setminus \smash{\overline{B}}^{2}_{\delta r/2})\times (-r,r), \mathcal{N})$, $v_{\delta, n} \in W^{1,2}(\partial B^{2}_{\delta r/2}\times (-r,r), \mathcal{N})$, $\tr_{\partial B^{2}_{\delta r/2}\times (-r,r)}(\varphi_{\delta,n})=v_{\delta,n}$ (the proof consists in constructing the map on the union of surfaces $(\partial B^{2}_{\delta r/2}\times (-r, z^{n}_{-})) \cup (\partial B^{2}_{\delta r/2} \times (z^{n}_{+}, r))$ by similar way as was constructed $v_{\delta,n}$ on $\partial B^{2}_{\delta r/2}\times (z^{n}_{-}, z^{n}_{+})$, namely, by choosing appropriate triangulations on  the sets $\partial B^{2}_{\delta r}\times (-r,z^{n}_{-})$,  $\partial B^{2}_{\delta r} \times (z^{n}_{+}, r)$ and projecting them onto $\partial B^{2}_{\delta r/2}\times (-r, z^{n}_{-})$, $\partial B^{2}_{\delta r/2}\times (z^{n}_{+},r)$, respectively, and then proceed as in the construction of the maps $v_{\delta,n}$, $\varphi_{\delta,n}$; the traces of the obtained extensions of $\varphi_{\delta, n}$ and $v_{\delta,n}$ will coincide with the traces of $\varphi_{\delta, n}$ and $v_{\delta,n}$ on $\partial B^{2}_{\delta r}\times \{z^{n}_{-}\}$, $\partial B^{2}_{\delta r}\times \{z^{n}_{+}\}$ and $\partial B^{2}_{\delta r/2}\times \{z^{n}_{-}\}$, $\partial B^{2}_{\delta r/2}\times \{z^{n}_{+}\}$, respectively, and hence we obtain a $W^{1,p}$ extension of $\varphi_{\delta,n}$ on $(B^{2}_{\delta r}\setminus \smash{\overline{B}}^{2}_{\delta r/2})\times (-r,r)$ and a $W^{1,2}$ extension of $v_{\delta,n}$ on $\partial B^{2}_{\delta r/2} \times (-r,r)$). Since for each $z \in \{-r,r\}$, $\tr_{\partial B^{2}_{\delta r/2}\times \{z\}}(v_{\delta, n})$ is homotopic to $u_{n}|_{\partial B^{2}_{\delta r}\times \{z\}}$, we deduce that for each $n \in \mathbb{N}$ large enough, $[\gamma_{n}]=[u_{*}|_{\partial B^{2}_{\delta r}\times \{z\}}]$. This proves our claim.
    
   Thus, there exists $\gamma \in [\mathbb{S}^{1},\mathcal{N}]$ such that for each $n\in \mathbb{N}$ large enough, $[\gamma_{n}]=[\gamma]$. Next, multiplying both sides of \eqref{est 5.53} and \eqref{est 5.54} by $(2-p_{n})$ and then passing to the limit when $n$ tend to $+\infty$, we obtain
   \begin{equation}\label{est 5.55}
   r(2-2\delta)\mathcal{E}_{2}^{\mathrm{sg}}(\gamma)-C^{\prime}\varepsilon \delta r\leq \mu_{*}(\Lambda_{\delta r,r})\leq 2r \mathcal{E}^{\mathrm{sg}}_{2}(\gamma)+ \frac{C\varepsilon r}{\delta} + 2C\delta r,
   \end{equation}
   where we have used Lemma~\ref{lem cont singenergy}. Now we choose $\delta=\varepsilon^{1/2}$ and hence $\delta\searrow 0$ as $\varepsilon \searrow 0$. Dividing both sides of \eqref{est 5.55} by $r$ and letting $r \searrow 0$ and then $\varepsilon \searrow 0$, we obtain
   \[
   2\mathcal{E}^{\mathrm{sg}}_{2}(\gamma)\leq \Theta_{1}(x_{0})\mathcal{H}^{1}(\{(0,0)\} \times [-1, 1])\leq 2\mathcal{E}^{\mathrm{sg}}_{2}(\gamma).
   \]
   This implies that $\Theta_{1}(x_{0})=\mathcal{E}^{\mathrm{sg}}_{2}(\gamma)$. Since $x_{0} \in S_{*}\cap \Omega$, by Proposition~\ref{nice characterization for singular set}, $\Theta_{1}(x_{0})>0$ and hence $\mathcal{E}^{\mathrm{sg}}_{2}(\gamma)> 0$. This completes our proof of Proposition~\ref{prop density dv}.
\end{proof}
According to Lemma~\ref{lem cont singenergy} and Proposition~\ref{prop density dv}, the density of the one-dimensional varifold $V_{*}$ coming from Proposition~\ref{prop 1-var}, which was \emph{naturally} associated with $\mu_{*}\mres\Omega$, is bounded from below $\|V_{*}\|$-a.e.\  in $\Omega$ (\emph{i.e.}, $\mu_{*}$-a.e.\  in $\Omega$ or, equivalently, $\mathcal{H}^{1}$-a.e.\  in $S_{*}\cap \Omega$) and has a finite set of values. Thus, due to Allard and Almgren's theorem \cite[Theorem, p. 89]{Allard-Almgren}, we obtain the following structure for $S_{*}\cap \Omega$.

\begin{prop}\label{prop structure S_*}
    Any compact set $K\subset \Omega$ such that $S_{*}\cap K\neq \emptyset$ has an open neighborhood $U\csubset \Omega$ with a Lipschitz boundary and a finite number of connected components such that $S_{*}\cap \overline{U}$ is a union of a finite number of closed straight line segments, each of which has a positive length, and the interiors of these segments (\textit{i.e.}, the segments without their endpoints) are pairwise disjoint. If $K$ is connected, then $U$ is connected; otherwise, $U$ may have several connected components, but their closures are pairwise disjoint. Furthermore, $S_{*}\cap \partial U$ is a finite set of points such that for each point $x \in S_{*} \cap \partial U$ there is exactly one segment lying in $S_{*}\cap \overline{U}$ and emanating from $x$.
\end{prop}
\begin{proof}
    If $\pi_{1}(\mathcal{N})=\{0\}$, then by Proposition~\ref{prop density dv} and Proposition~\ref{nice characterization for singular set}, it holds $S_{*}\cap \Omega=\emptyset$ and the proof is trivial. 
    
    Assume that $\pi_{1}(\mathcal{N})\neq \{0\}$. According to the definition of the varifold $V_{*}$ (see Proposition~\ref{prop 1-var}), the weight of $V_{*}$ (or, also the projection of $V_{*}$ on $\Omega$), defined by $\|V_{*}\|(E)=V_{*}(E\times \mathrm{G}(3,1))$ for each $E \in \mathcal{B}(\Omega)$, is equal to $\mu_{*}\mres\Omega=\Theta_{1}\mathcal{H}^{1}\mres(S_{*}\cap \Omega)$. Then, by Proposition~\ref{nice characterization for singular set}, there exists $c=c(\mathcal{N})>0$ such that for each $x \in S_{*}\cap \Omega$, 
    \begin{equation}\label{positive dn est}
    \Theta_{1}(\|V_{*}\|,x)\coloneqq \lim_{r\to 0+}\frac{\|V_{*}\|(B^{3}_{r}(x))}{2r}=\lim_{r\to 0+}\frac{\mu_{*}(B^{3}_{r}(x))}{2r}=\Theta_{1}(x)\geq c.
    \end{equation}
In view of Proposition~\ref{prop 1-var} and \eqref{positive dn est}, we can apply \cite[Theorem, p. 89]{Allard-Almgren} and obtain that there exists a countable family $\mathcal{P}$ of bounded open intervals (\emph{i.e.}, segments without endpoints) contained in $S_{*}\cap \Omega$ such that $V_{*}=\sum\{\Theta_{1}(I)|I|: I\in \mathcal{P}\}$, where $\Theta_{1}(I)$ is the unique member of the range of $\Theta_{1}$ restricted to the interval $I$ and $|I| \in (C_{0}(\mathrm{G}_{3}(\Omega)))^{\prime}$ is a finite positive Radon measure defined by $|I|=\delta_{\mathrm{e}\otimes \mathrm{e}}(d T)\otimes \mathcal{H}^{1}\mres I(d x)$, where $\mathrm{e}$ is the unit direction vector of $I$. Taking into account Lemma~\ref{lem cont singenergy} and Proposition~\ref{prop density dv}, we have that $\{\Theta_{1}(I): I\in \mathcal{P}\}\subset \{\mathcal{E}^{\mathrm{sg}}_{2}(\gamma): \gamma \in [\mathbb{S}^{1}, \mathcal{N}]\}$ is a finite set of values. 
Furthermore, the proof of \cite[Theorem, p. 89]{Allard-Almgren} says that for each $x \in S_{*}\cap \Omega$, there exists a radius $r>0$ such that $S_{*}\cap \Omega\cap \smash{\overline{B}}^3_{r}(x)$ is a finite union of radii of $\smash{\overline{B}}^3_{r}(x)$, that is, closed straight line segments joining $x$ with a point on $\partial B^{3}_{r}(x)$. Let $K\subset \Omega$ be compact and satisfy $K\cap S_{*} \neq \emptyset$. Then there exists a finite collection of open balls $B^{3}_{2r_{1}}(x_{1}) ,\dotsc, B^{3}_{2r_{l}}(x_{l})$ whose closures are contained in $\Omega$, centered at $S_{*}$ and the following holds: $K\cap S_{*} \subset \bigcup_{i=1}^{l}B^{3}_{r_{i}}(x_{i})$; $S_{*}\cap \smash{\overline{B}}^3_{2r_{i}}(x_{i})$ is a finite union of radii of $\smash{\overline{B}}^3_{2r_{i}}(x_{i})$ for each $i \in \{1,\dotsc,l\}$ (at least two radii are contained in each of these covering balls, because otherwise we would obtain a contradiction with the stationarity of $V_{*}$). Now fix an arbitrary $i \in \{1,\dotsc,l\}$. If for some $j \in\{1,\dotsc,l\}$, $j\neq i$ a radius $L_{j}$ of $\smash{\overline{B}}^3_{r_{j}}(x_{j})$ (\emph{i.e.}, a closed segment joining $x_{j}$ with a point on $\partial B^{3}_{r_{j}}(x_{j})$) lying in $S_{*}\cap \smash{\overline{B}}^3_{r_{j}}(x_{j})$ intersects $\smash{\overline{B}}^3_{r_{i}}(x_{i})$, then, by construction of our covering, there exists a radius $L_{i}$ of $\smash{\overline{B}}^3_{r_{i}}(x_{i})$ and a straight line containing both $L_{i}$ and $L_{j}$ (the union of these radii is a straight line segment). Covering the rest of $K$ with a finite number of chains of sufficiently small balls, one can appropriately choose a neighborhood $U$ of $K$ satisfying the conditions of Proposition~\ref{prop structure S_*} by passing from ball to ball in each chain from the resulting cover and successively constructing $\partial U$. This completes our proof of Proposition~\ref{prop structure S_*}.
\end{proof}
As a byproduct of the proof of Proposition~\ref{prop density dv}, taking into account Proposition~\ref{prop structure S_*}, we obtain the following 
\begin{cor}\label{cor 5.17}
    Let $L \subset S_{*}$ be a closed straight line segment. Let $x$ be a point lying in the relative interior of $L$ and  $D\subset \Omega$ be a closed $2$-disk (which is an embedded submanifold of $\mathbb{R}^{3}$) with center $x$ such that $D\cap S_{*}=\{x\}$ and $\partial D\cap S_{0}=\emptyset$, where $S_{0}$ is at most a countable and locally finite subset of $\Omega \backslash S_{*}$ coming from Proposition~\ref{prop conv to harmonic map}~\ref{item3convharm}, which applies to $u_{*}$ in $\Omega \backslash S_{*}$ (see Proposition~\ref{conv to harm}). Then the homotopy class of $u_{*}|_{\partial D}$ is nontrivial and  for each point $y$ in the relative interior of $L$, $\Theta_{1}(\mu_{*}, y)=\mathcal{E}^{\mathrm{sg}}_{2}(u_{*}|_{\partial D})$. 
    \end{cor}
\begin{proof}[Proof of Corollary~\ref{cor 5.17}]
    Fix an arbitrary point $y$ lying in the relative interior of $L$. According to Proposition~\ref{conv to harm} and Proposition~\ref{prop structure S_*}, we can find an open set $U \csubset \Omega$ such that $x, y \in U$, $D \subset U$ and  $u_{*} \in W^{1,2}_{\loc}(U\setminus L, \mathcal{N})\cap  C^{\infty}(U\setminus (L\cup S_{0}), \mathcal{N})$. Since $S_{0}$ is locally finite, there exists a closed 2-disk $D_{y}\subset U$ such that $D_{y}\cap S_{*}=\{y\}$, $\partial D_{y}\cap S_{0}=\emptyset$ and $u_{*}|_{\partial D}$ is continuously homotopic to $u_{*}|_{\partial D_{y}}$. Then $\mathcal{E}^{\mathrm{sg}}_{2}(u_{*}|_{\partial D}) =\mathcal{E}^{\mathrm{sg}}_{2}(u_{*}|_{\partial D_{y}})$. On the other hand, as a byproduct of the proof of Proposition~\ref{prop density dv},  $\Theta_{1}(y)=\mathcal{E}^{\mathrm{sg}}_{2}(u_{*}|_{\partial D})$, which completes our proof of Corollary~\ref{cor 5.17}. 
\end{proof}

The stationarity of $V_{*}$ and Proposition~\ref{prop structure S_*}  yield the following observation.
\begin{cor}\label{cor 5.18}
    Let $x_{0} \in S_{*}\cap \Omega$. Then, according to Proposition~\ref{prop structure S_*}, there exists a radius $r>0$ such that $\smash{\overline{B}}^3_{r}(x_{0})\subset \Omega$ and $S_{*}\cap \smash{\overline{B}}^3_{r}(x_{0})=\bigcup_{i=1}^{l}L_{i}$, where $l\in \mathbb{N}\setminus\{0\}$ and $L_{i}$ is a closed straight line segment emanating from $x_{0}$. Let $v_{i}\in\mathbb{S}^{2}$ be the direction vector of $L_{i}$ pointing outward from $x_{0}$. Then the stationarity of $V_{*}$ yields that the weighted sum of the $v_{i}$'s is zero, namely
    \begin{equation}\label{weight sum=0}
    \sum_{i=1}^{l}\lambda_{i}v_{i}=0,
    \end{equation}
    where $\lambda_{i}=\mathcal{E}^{\mathrm{sg}}_{2}(u_{*}|_{\partial D_{i}})>0$ with $D_{i}$ being a closed 2-disk (which is an embedded submanifold of $\mathbb{R}^{3}$) centered inside $L_{i}$ and lying in $B^{3}_{r}(x_{0})$ on the 2-plane orthogonal to $L_{i}$ such that $D_{i}\cap S_{*}$ is the singleton, $\partial D_{i} \cap S_{0}=\emptyset$, where $S_{0}$ is at most a countable and locally finite subset of $\Omega \backslash S_{*}$ coming from Proposition~\ref{prop conv to harmonic map}~\ref{item3convharm}, which applies to $u_{*}$ in $\Omega \backslash S_{*}$ (see Proposition~\ref{conv to harm}). 
    \end{cor}    
\begin{proof}[Proof of Corollary~\ref{cor 5.18}] Fix an arbitrary $\chi \in C^{1}_{c}(B^{3}_{r}(x_{0}), \mathbb{R}^{3})$. Since the first variation of $V_{*}$ vanishes, where $V_{*}(d\,T,\diff x)=\delta_{A_{*}(x)}(dT)\otimes \Theta_{1}(x)\mathcal{H}^{1}\mres (S_{*}\cap \Omega)(dx)$ with $A_{*}(x)$ being the orthogonal projection matrix onto the approximate tangent plane $T_{x}S_{*}$ to $S_{*}$ at $x$, we have
    \begin{equation}\label{stationary sum 0}
        0=\int_{\Omega}A_{*}(x):D\chi(x) \diff \mu_{*}(x)=\int_{S_{*}\cap B^{3}_{r}(x_{0})}\Theta_{1}(x) A_{*}(x):D\chi(x) \diff \mathcal{H}^{1}(x).
    \end{equation}
Next, using that $S_{*}\cap \smash{\overline{B}}^3_{r}(x_{0})=\bigcup_{i=1}^{l}L_{i}$, $\mu_{*}\mres B^{3}_{r}(x_{0})=\sum_{i=1}^{l}\lambda_{i}\mathcal{H}^{1}\mres L_{i}$ and  $A_{*}|_{L_{i}}=v_{i}\otimes v_{i}$, in view of \eqref{stationary sum 0}, we get
\begin{equation}\label{stationary sum}
\begin{split}
0&=\sum_{i=1}^{l}\lambda_{i}\int_{L_{i}}v_{i}\otimes v_{i}:D\chi(x)\diff \mathcal{H}^{1}(x)\\
&=\sum_{i=1}^{l}\lambda_{i}\int_{0}^{\mathcal{H}^{1}(L_{i})}\frac{\diff}{\diff t}\langle \chi(x_{0}+t v_{i}), v_{i}\rangle \diff t\\
&= -\sum_{i=1}^{l}\langle \chi(x_{0}), \lambda_{i} v_{i} \rangle. 
\end{split}
\end{equation}
Since $\chi \in C^{1}_{c}(B^{3}_{r}(x_{0}), \mathbb{R}^{3})$ was arbitrarily chosen, \eqref{stationary sum} yields \eqref{weight sum=0}, which completes our proof of Corollary~\ref{cor 5.18}.
\end{proof}

\begin{rem}\label{rem sing points}
    We obtain several direct consequences of Corollary~\ref{cor 5.18}. 
    \begin{enumerate}[label=(\roman*)]
        \item For each $x_{0} \in S_{*}\cap \Omega$, at least two segments emanate from $x_{0}$, because otherwise $x_{0}$ would be an endpoint for some segment lying in $S_{*}\cap \Omega$, but this would contradict \eqref{weight sum=0}. 
        \item If locally around $x_{0}$ the set $S_{*}\cap \Omega$ is a union of exactly two segments emanating from $x_{0}$, then the angle between them is $\pi$, since a smaller angle  would contradict \eqref{weight sum=0}. 
        \item If exactly three segments emanate from $x_{0}$, then their direction vectors lie in the same two-dimensional plane. 
        \item 
        \label{it_uk8ewikae5Quaoth3eiJaeth}
        If exactly four segments emanate from $x_{0}$, then $\lambda_{1}v_{1}+\lambda_{2}v_{2} +\lambda_{3}v_{3}+\lambda_{4}v_{4}=0$, which yields a system of equations. Firstly, 
        $|\lambda_{1} v_{1} +\lambda_{2} v_{2}|^{2}=|\lambda_{3}v_{3}+\lambda_{4}v_{4}|^{2}$ and hence 
        \[
         \lambda_1^2 + \lambda_2^2 + 2 \lambda_1 \lambda_2 \langle v_1,  v_2\rangle
         = \lambda_3^2 + \lambda_4^2 + 2 \lambda_3 \lambda_4 \langle v_3, v_4 \rangle.
        \]
        Secondly, for each $j \in \{1, \dotsc, 4\}$, 
        \[
        \lambda_{j}+\sum_{i=1, \, i\neq j}^{4} \lambda_{i}\langle v_{i},  v_{j} \rangle=0,
        \]
        since $|v_{j}|=1$. Thirdly, for each $j \in \{1, \dotsc, 4\}$, $|\lambda_{j} v_{j}|^{2}=|\sum_{i=1,\, i\neq j}^{4} \lambda_{i}v_{i}|^{2}$ and hence
        \[
        \lambda_{j}^{2}=\sum_{i=1, \, i\neq j}^{4} \lambda_{i}^{2}+2 \sum_{i,\, k\in (\{1,\dotsc, 4\} \setminus \{j\}) \cap \{i<k\}} \lambda_{i} \lambda_{k} \langle v_{i}, v_{k} \rangle.
        \]
        If $\lambda_{1}=\lambda_{2}=\lambda_{3}=\lambda_{4}$ in \eqref{weight sum=0}, then, up to relabeling,  $v_{1}+v_{2}=-v_{3}-v_{4}$, and there exists a double cone with apex $0$ such that $v_{1}$, $v_{2}$ lie on the upper nappe and $v_{3}$, $v_{4}$ lie on the lower nappe of the cone, respectively. Furthermore, the angle between $v_{1}$ and $v_{2}$ is the same as between $v_{3}$ and $v_{4}$, which is twice the angle of the double cone.
        \end{enumerate} 
        Moreover, in \cite[Proposition~2]{Canevari_2017} it was proved that in the case when $\pi_{1}(\mathcal{N}) \simeq \mathbb{Z}/2\mathbb{Z}$ (for instance, $\mathcal{N}=SO(n)$ or $\mathcal{N}=\mathbb{R}\mathbb{P}^{n-1}$ for $n\in \mathbb{N}$, $n\geq 3$), then only an even number of segments can emanate from $x\in S_{*}\cap \Omega$. We prove that in the case when $\pi_{1}(\mathcal{N})=\mathbb{Z}/2\mathbb{Z}$, there cannot be branching points inside $\Omega$ (see Corollary~\ref{cor branching points}). 
\end{rem}

\subsection{Interior estimates and \texorpdfstring{$W^{1,q}_{\loc}(\Omega)$}{W1,qloc(Ω)} compactness for $p$-minimizers when $q<2$}
We prove that the $q$-energy of $p_{n}$-minimizers with $(p_{n})_{n\in \mathbb{N}} \subset [1,2)$ and $p_{n} \nearrow 2$ as $n\to +\infty$ is locally bounded inside $\Omega$ for each $q\in [1,2)$ under the condition \eqref{C2}.

For the reader's convenience, let us recall the construction of dyadic cubes in $\mathbb{R}^{3}$ (see, for instance, \cite{Auscher_2012}).
\begin{defn}\label{defofdyadiccubes}
Let $[0,1)^{3}$ be the reference cube, $j \in \mathbb{Z}$ and $x \in \mathbb{Z}^{3}$. Then define the dyadic cube of generation $j$ with lower left corner $2^{-j} x$ by 
\[
Q_{j,x}
\coloneqq 
2^{-j} (x + [0, 1)^3) = \{y \in \mathbb{R}^{3}: 2^{j}y-x \in [0,1)^{3}\},
\]
the set of generation $j$ dyadic cubes by $\mathscr{D}_{j}=\{Q_{j,x}: x \in \mathbb{Z}^{3}\}$ and the set of all dyadic cubes by 
\[
\mathscr{D}
\coloneqq 
\bigcup_{j\in \mathbb{Z}}\mathscr{D}_{j}=\{Q_{j, x}: j \in \mathbb{Z}\,\ \text{and}\,\ x \in \mathbb{Z}^{3}\}.  
\]
For each $k\in \mathbb{N}\setminus\{0\}$, we also define the enlarged cubes
\[
kQ_{j,x}
\coloneqq 
\biggl\{y \in \mathbb{R}^{3}: 2^{j}y-x \in \Bigl[-\frac{k-1}{2}, \frac{k+1}{2}\Bigr)^{3}\biggr\}.
\]
\end{defn}
According to \cite[Theorem~2.2.3~(ii)]{Auscher_2012}, for each $j\in \mathbb{Z}$ and $x \in \mathbb{Z}^{3}$, there exists a unique $y\in \mathbb{Z}^{3}$ such that $Q_{j,x}\subset Q_{j-1, y}$. The cube $Q_{j-1, y}$ will be called the \emph{parent} of $Q_{j, x}$. For each $j \in \mathbb{Z}$ and $y \in \mathbb{Z}^{3}$, the cubes $Q_{j+1, x}$ for which $Q_{j, y}$ is the parent are called the \emph{children} of $Q_{j, y}$.

The distance from $U \subset \mathbb{R}^{3}$ to $V \subset \mathbb{R}^{3}$ is given by
\[
\dist(U, V)\coloneqq \inf\{|x-y|: x \in U, \,\ y \in V\}.
\]

\begin{prop}\label{prop goodest dyadic}
	Let $V\subset \mathbb{R}^{3}$ be a bounded open set. Let $C_{1}, C_{2}$ be positive constants, $p \in (1,2]$ and $f \in L^{p}(V)$. Assume that for each dyadic cube $Q_{j, x_{0}}$ such that $2Q_{j, x_{0}} \subset V$ and
			\begin{equation}\label{goodcubeest}
				C_{1}\int_{2Q_{j,x_{0}}} |f|^{p}\diff x \leq 2^{-j(3-p)}
			\end{equation}
			it holds
			\begin{equation}\label{goodcubeestcons}
				\int_{Q_{j, x_{0}}}|f|^{p}\diff x \leq C_{2} 2^{-j(3-p)}.
			\end{equation}
			Then for each $q \in [1,p)$, $f \in L^{q}(V)$ and 
			\begin{equation}\label{goodestq}
				\int_{V}(\dist(x,\partial V) |f(x)|)^{q}\diff x \leq 10^{4}C_{2}^{\frac{q}{p}}|V|+ 10^{4}\frac{C_{1}C_{2}^{\frac{q}{p}}}{p-q}\int_{V}(\dist(x,\partial V)|f(x)|)^{p}\diff x.
			\end{equation}
\end{prop}
\begin{proof}
According to the Whitney covering theorem for dyadic cubes, namely \cite[Theorem~2.3.1]{Auscher_2012}, there exists a collection of mutually disjoint dyadic cubes $\mathcal{F}=\{Q_{i}\}_{i\in I}$, where $I$ is countable, such that
\begin{equation}\label{estdiamdist}
\frac{1}{30}\dist(Q_{i}, \partial V)\leq \diam(Q_{i})\leq \frac{1}{10}\dist(Q_{i}, \partial V)
\end{equation}
and 
\begin{equation}\label{dyadicdecompV}
V=\bigcup_{i\in I}Q_{i}.
\end{equation}
Since $V$ is bounded, there exists (the minimal) $j_{0} \in \mathbb{Z}$ such that for some $x_{0} \in \mathbb{Z}^{3}$, $Q_{j_{0}, x_{0}} \in \mathcal{F}$ and for each $x \in \mathbb{Z}^{3}$ and $j \in (-\infty, j_{0})\cap \mathbb{Z}$, $Q_{j, x}  \not \in \mathcal{F}$. 
For each $j \in [j_{0}, +\infty) \cap \mathbb{Z}$, define $\mathcal{F}_{j}\coloneqq \mathscr{D}_{j}\cap \mathcal{F}$. 
Since $V$ is bounded, $\mathcal{F}_{j}$ is finite. 
Notice also that for each $Q \in \mathcal{F}$, $3Q \subset V$ (this comes from \eqref{estdiamdist}, see \cite[Proposition~2.3.3~(i)]{Auscher_2012}). 
We say that a dyadic cube $Q \in \mathcal{F}_{j}$ is \emph{good} if \eqref{goodcubeest} holds with $Q_{j,x_{0}}$ replaced by $Q$, and it is \emph{bad} otherwise. 
Let us now define the collection of all good cubes of $\mathcal{F}_{j}$ by $\mathcal{F}_{j,0}$ and their union by $G_{j,0}$. 
For each bad cube $Q \in \mathcal{F}_{j}$, we consider the children of $Q$. 
Similarly, we say that a child $Q^{\prime}$ of $Q$ is good if \eqref{goodcubeest} holds with $Q_{j, x_{0}}$ replaced by $Q^{\prime}$, and it is bad otherwise. 
We define the collection of all good children of bad cubes of $\mathcal{F}_{j}$ by $\mathcal{F}_{j,1}$ and their union by $G_{j,1}$. 
Next, we select good children of bad children of bad cubes in $\mathcal{F}_{j}$ and denote this collection by $\mathcal{F}_{j,2}$ and their union by $G_{j, 2}$. 
Denote the number of bad cubes in $\mathcal{F}_{j}$ by $N_{j,0}$. 
The number of bad children of bad cubes of $\mathcal{F}_{j}$ denote by $N_{j,1}$ and the number of bad children of bad children of bad cubes of $\mathcal{F}_{j}$ denote by $N_{j,2}$. 
Repeating this procedure, we obtain sequences $(\mathcal{F}_{j,n})_{n\in \mathbb{N}}$, $(G_{j,n})_{n\in \mathbb{N}}$ and $(N_{j,n})_{n\in \mathbb{N}}$. 
Define $G_{j}\coloneqq \bigcup_{n\in \mathbb{N}} G_{j,n}$. 
We claim that $G_{j}$ has full measure in $\bigcup_{Q \in \mathcal{F}_{j}}Q$. 
Indeed, by \cite[Theorem~2.2.3]{Auscher_2012}, for each $y \in \mathbb{R}^{3}$, there exists a sequence of dyadic cubes $(Q_{k, y_{k}})_{k\in \mathbb{N}}$ such that $\{y\}=\bigcap_{k\in \mathbb{N}}Q_{k, y_{k}}$. 
If $y \in \bigcup_{Q\in \mathcal{F}_{j}}Q\setminus G_{j}$, then, by construction, for each $k \in \mathbb{N}$ large enough, $2Q_{k, y_{k}} \subset B^{3}_{2^{-k+1}} (y) \subset V$ and 
\[
C_{1}\int_{B^{3}_{2^{-k+1}}(y)}|f|^{p}\diff x \geq 2^{-k(3-p)},
\]
which implies that 
\[
\liminf_{k\to +\infty} \frac{1}{|B^{3}_{2^{-k+1}}|}\int_{B^{3}_{2^{-k+1}}(y)}|f|^{p}\diff x = +\infty. 
\]
Thus, in view of the Lebesgue differentiation theorem (see \cite[Theorem~1.32]{Evans}) and the fact that $f \in L^{p}(V)$, $|\bigcup_{Q \in \mathcal{F}_{j}} Q \setminus G_{j}|=0$. 
This proves our claim. On the other hand, observe that for each $n\in \mathbb{N}$,
\begin{equation}\label{trivialcubdyad}
\# \mathcal{F}_{j,n+1}\leq 2^{3} N_{j, n}
\end{equation}
and, in view of \eqref{goodcubeest}, 
\begin{equation}\label{ebbals}
N_{j,n} \leq 3^{3} 2^{(j+n)(3-p)} C_{1} \int_{G_{j}}|f|^{p}\diff x,
\end{equation}
where we have also used that  if $Q\in \mathscr{D}_{j}$, then there exists exactly $3^{3}-1$ other dyadic cubes $Q^{
\prime} \in \mathscr{D}_{j}$ such that $2Q \cap 2Q^{\prime} \not = \emptyset$, and that $G_{j}$ has full measure in $\bigcup_{Q \in \mathcal{F}_{j}}Q$. Combining \eqref{trivialcubdyad} with \eqref{ebbals}, for each $n\in \mathbb{N}$, we get
\begin{equation}\label{numberofgoodballsi}
\#\mathcal{F}_{j,n+1} \leq 6^{3} 2^{(j+n)(3-p)} C_{1}\int_{G_{j}}|f|^{p}\diff x.
\end{equation}
It is also worth noting that 
\begin{equation}\label{step0esti}
\#\mathcal{F}_{j,0}\leq 2^{3j}|G_{j}|.
\end{equation}
Applying H\"older's inequality and \eqref{goodcubeestcons}, for each $Q \in \mathcal{F}_{j,n}$, we obtain
\begin{equation*}
\int_{Q}|f|^{q}\diff x  \leq |Q|^{1-\frac{q}{p}} \left(\int_{Q}|f|^{p} \diff x\right)^{\frac{q}{p}}\leq C_{2}^{\frac{q}{p}}2^{-3(j+n)\left(1-\frac{q}{p}\right)} 2^{-(j+n)\left(\frac{3q}{p}-q\right)}\leq C_{2}^{\frac{q}{p}}2^{-(j+n)(3-q)}.
\end{equation*}
 Using this, \eqref{numberofgoodballsi} and \eqref{step0esti}, we deduce that
\begin{equation*}
\begin{split}
\int_{G_{j}}|f|^{q}\diff x = \sum_{n=0}^{+\infty} \int_{G_{j,n}}|f|^{q}\diff x &= \sum_{Q \in \mathcal{F}_{j,0}}\int_{Q}|f|^{q}\diff x + \sum_{n=0}^{+\infty} \sum_{Q \in \mathcal{F}_{j, n+1}} \int_{Q}|f|^{q}\diff x \\
& \leq \# \mathcal{F}_{j,0} C_{2}^{\frac{q}{p}} 2^{-j(3-q)}+\sum_{n=0}^{+\infty} \#\mathcal{F}_{j, n+1} C_{2}^{\frac{q}{p}}2^{-(j+n+1)(3-q)} \\ 
& \leq C_{2}^{\frac{q}{p}}2^{jq}|G_{j}|+ 3^{3}2^{q}C_{1}C_{2}^{\frac{q}{p}}\sum_{n=0}^{+\infty} 2^{-(j+n)(p-q)}\int_{G_{j}}|f|^{p}\diff x\\
& \leq C_{2}^{\frac{q}{p}}2^{jq}|G_{j}|+3^{3}2^{jq+p}\frac{C_{1}C_{2}^{\frac{q}{p}}}{\ln(2)(p-q)}\int_{G_{j}}2^{-jp}|f|^{p}\diff x, 
\end{split}
\end{equation*}
where, in the last estimate, we have used that $2^{p-q}-1\geq \ln(2)(p-q)$. Hence
\begin{equation}\label{niceestgj}
\int_{G_{j}}2^{-jq}|f|^{q}\diff x \leq C_{2}^{\frac{q}{p}}|G_{j}|+3^{3}2^{p}\frac{C_{1}C_{2}^{\frac{q}{p}}}{\ln(2)(p-q)}\int_{G_{j}}2^{-jp}|f|^{p}\diff x.
\end{equation}
Since $1\leq q< p \leq 2$, from \eqref{estdiamdist}, \eqref{niceestgj} and the fact that $G_{j}$ has full measure in $\bigcup_{Q \in \mathcal{F}_{j}}Q$, we obtain
\[
\int_{G_{j}}(\dist(x, \partial V) |f(x)|)^{q}\diff x \leq 10^{4}C_{2}^{\frac{q}{p}}|G_{j}|+10^{4}\frac{C_{1}C_{2}^{\frac{q}{p}}}{p-q}\int_{G_{j}}(\dist(x, \partial V)|f(x)|)^{p}\diff x.
\]
Summing both sides of the above inequality over $j \in [j_{0}, +\infty)\cap \mathbb{Z}$ and taking into account that $\bigcup_{j=j_{0}}^{+\infty}G_{j}$ has full measure in $V$ (this comes from \eqref{dyadicdecompV} and the fact that $G_{j}$ has full measure in $\bigcup_{Q \in \mathcal{F}_{j}}Q$), yields
\[
\int_{V}(\dist(x, \partial V)|f(x)|)^{q}\diff x \leq 10^{4}C_{2}^{\frac{q}{p}}|V|+10^{4}\frac{C_{1}C_{2}^{\frac{q}{p}}}{p-q}\int_{V}(\dist(x, \partial V)|f(x)|)^{p}\diff x. 
\]
This completes our proof of Proposition~\ref{prop goodest dyadic}.
\end{proof}

As a consequence of Proposition~\ref{prop goodest dyadic}, we obtain an improvement of Proposition~\ref{conv to harm}.

\begin{prop}\label{prop est for the differential inside domain} 
	Let $(u_{n})_{n\in \mathbb{N}}$ be a sequence of $p_{n}$-minimizers in $\Omega$. Then the following assertions hold. 
	\begin{enumerate}[label=(\roman*)]
		\item 
		\label{it_so3EPh9zoev2Gae0OoQuieg9}
		There exists $C=C(\diam(\Omega), \mathcal{N})>0$ such that for each $q \in [1,2)$, 
		\begin{equation*}
			\limsup_{n \to +\infty}\int_{\Omega}(\dist(x,\partial \Omega)|Du_{n}(x)|)^{q}\diff x \leq C |\Omega|+\frac{C}{2 - q}\limsup_{n\to+\infty}(2-p_{n})\int_{\Omega} \frac{|Du_{n}|^{p_{n}}}{p_{n}}\diff x. 
		\end{equation*}
		\item 
		\label{it_Aex0Aijushoetha0looKo6ei}
		Let $(u_{n})_{n\in \mathbb{N}}$ satisfy \eqref{C2}  and let $u_{*}$ be given by Proposition~\ref{conv to harm}. Then for each $q \in [1,2)$,  $u_{*} \in W^{1,q}_{\loc}(\Omega, \mathcal{N})$ and, up to a subsequence (not relabeled), $u_{n} \to u_{*}$  in $W^{1,q}_{\loc}(\Omega, \mathbb{R}^{\nu})$ and
		\begin{equation}\label{estonweakgradient0ustar}
		\int_{\Omega} (\dist(x,\partial \Omega)|Du_{*}(x)|)^{q} \diff x \leq C |\Omega| + \frac{C}{2-q} \limsup_{n \to +\infty} (2-p_{n}) \int_{\Omega} \frac{|Du_{n}|^{p_{n}}}{p_{n}} \diff x,
		\end{equation}
		where $C=C(\diam(\Omega),  \mathcal{N})>0$.
	\end{enumerate}
\end{prop}
\begin{proof}
Let $\eta, C>0$ be the constants of Lemma~\ref{heart}, where $\kappa=1/2$, $p_{0}=3/2$ and $\Psi$ comes from \cite[Corollary~3]{bilipschitz}. Observe that if $Q$ is a dyadic cube with barycenter $x_{Q}$ and sidelength $r>0$, then, denoting its interior by $Q^{\circ}$, we have $Q^{\circ}=\Psi(B^{3}_{r/2})+x_{Q}$, $2Q^{\circ}= \Psi(B^{3}_{r})+x_{Q}$ and the bilipschitz constant of  the map $\Psi+x_{Q}$ is the bilipschitz constant of $\Psi$. Thus, applying Proposition~\ref{prop goodest dyadic} with $V=\Omega$, $p=p_{n}$, $f=|Du_{n}|$, $C_{2}=C$, $C_{1}=\frac{2-p_{n}}{\eta}$ and taking into account that $\dist(\cdot, \partial \Omega)\leq \diam(\Omega)$ everywhere in $\Omega$, for each $q \in [1,2)$ we get
\begin{equation*}
\begin{split}
\limsup_{n\to+\infty}\int_{\Omega}(\dist(x, \partial \Omega)|Du_{n}(x)|)^{q}\diff x &\leq C^{\prime} |\Omega|+ \frac{C^{\prime}}{2-q}\limsup_{n\to +\infty} (2-p_{n})\int_{\Omega}\frac{|Du_{n}|^{p_{n}}}{p_{n}}\diff x,
\end{split}
\end{equation*}
where $C^{\prime}=C^{\prime}(\diam(\Omega), \mathcal{N})>0$. This yields \ref{it_so3EPh9zoev2Gae0OoQuieg9}. 

Now we prove \ref{it_Aex0Aijushoetha0looKo6ei}. Let $q \in [1,2)$ and $\mathcal{F}=\{Q_{i}\}_{i\in I}$ be a Whitney covering of $\Omega\backslash S_{*}$ by dyadic cubes such that $\Omega\backslash S_{*}=\bigcup_{i\in I}Q_{i}$ and $2Q_{i}\subset \Omega\backslash S_{*}$ for each $i \in I$ (see, for example, \cite[Theorem~2.3.1]{Auscher_2012}). 
Notice that $I$ is countable. 
Using the fact that $|S_{*}|=0$ (see Proposition~\ref{prop 1-var}), the weak convergence (see Proposition~\ref{prop conv to harmonic map}~\ref{item:weakLp}), together with the convexity, and also using Fatou's lemma, up to a subsequence (not relabeled), one has
\begin{equation*}
 \label{estbyliminfweaklimit}
\begin{split}
\int_{\Omega}(\dist(x, \partial \Omega)|Du_{*}(x)|)^{q}\diff x &= \sum_{i\in I} \int_{Q_{i}}(\dist(x,\partial \Omega)|Du_{*}(x)|)^{q}\diff x \\
&\leq \sum_{i\in I} \liminf_{n\to +\infty} \int_{Q_{i}}(\dist(x,\partial \Omega)|Du_{n}(x)|)^{q}\diff x \\
& \leq \liminf_{n\to +\infty} \sum_{i \in I} \int_{Q_{i}}(\dist(x,\partial \Omega)|Du_{n}(x)|)^{q}\diff x \\
&=\liminf_{n\to +\infty} \int_{\Omega} (\dist(x,\partial \Omega) |Du_{n}(x)|)^{q}\diff x.  
\end{split}
\end{equation*}
This, in view of \ref{it_so3EPh9zoev2Gae0OoQuieg9} and \eqref{C2}, implies \eqref{estonweakgradient0ustar} and the fact that $u_{*}\in W^{1,q}_{\loc}(\Omega, \mathcal{N})$ for all $q \in [1,2)$. Next, let $q_{0} \in (q,2)$ and $U \Subset \Omega$ be an open set. Fix a sequence $(K_{n})_{n\in \mathbb{N}}$ of compact sets such that $K_{n} \Subset U \backslash S_{*}$, $K_{n}\subset K_{n+1}$ and $\bigcup_{n \in \mathbb{N}} K_{n}=U\setminus S_{*}$. Applying H\"{o}lder's inequality and using the fact that $|Du_{n}-Du_{*}|^{q_{0}} \leq 2^{q_{0}-1}(|Du_{n}|^{q_{0}}+|Du_{*}|^{q_{0}})$ (by convexity), for each $n \in \mathbb{N}$ large enough we obtain 
\begin{equation*}
\begin{split}
\int_{U}|Du_{n}-Du_{*}|^{q} \diff x &\leq |U\setminus K_{n}|^{1-\frac{q}{q_{0}}} 2^{q-\frac{q}{q_{0}}}\left(\int_{U\setminus K_{n}}(|Du_{n}|^{q_{0}}+|Du_{*}|^{q_{0}})\diff x\right)^{\frac{q}{q_{0}}} \\
& \qquad \qquad + |K_{n}|^{1 - \frac{q}{p_{n}}}\left(\int_{K_{n}}|Du_{n}-Du_{*}|^{p_{n}}\diff x \right)^{\frac{q}{p_{n}}}.
\end{split}
\end{equation*}
By carefully choosing $K_{n}$ and letting $n$ tend to $+\infty$, in view of Proposition~\ref{prop conv to harmonic map}~\ref{item2convtoharmproposition}, up to a subsequence (not relabeled), we can ensure that the last two terms tend to $0$ as $n \to +\infty$, and hence $u_{n} \to u_{*}$ in $W^{1,q}(U, \mathbb{R}^{\nu})$ as $n \to +\infty$. Using a diagonal argument, we deduce \ref{it_Aex0Aijushoetha0looKo6ei}, which completes our proof of Proposition~\ref{prop est for the differential inside domain}. 
\end{proof}

 \section{Interior minimality of the singular set}
 In this section, we prove some properties of local minimality inside $\Omega$ of the positive Radon measure $\mu_{*} \in (C(\overline{\Omega}))^{\prime}$ and its support $S_{*}$, defined at the beginning of Section~5.
 
 First, we define the notion of an \emph{admissible chain} for a bounded Lipschitz domain $U\subset \mathbb{R}^{3}$ and a boundary datum $g \in W^{1/2,2}(\partial U, \mathcal{N})$.
 \begin{defn}\label{def admissible chain}
     Given a bounded Lipschitz domain $U\subset \mathbb{R}^{3}$ and a map $g \in W^{1/2,2}(\partial U, \mathcal{N})$, we say that a closed set $S \subset \overline{U}$ is an \emph{admissible chain} for $U$ and $g$ if the following properties are satisfied.
  \begin{enumerate}[label=(P.\arabic*)]
      \item
      \label{item_P1}
      $S$ is the union of a finite number of closed straight line segments, each of which has a positive length, and the interiors of these segments (\emph{i.e.}, the segments without their endpoints) are pairwise disjoint.  
      \item 
      \label{item_P2}
      $S\cap \partial U$ is a finite set of points. 
      \item 
      \label{item_P3}
      $S$ has no endpoints in $U$, namely, if $L\subset S$ is a segment having an endpoint $x \in U$, then there exists another segment $L^{\prime} \subset S$ emanating  from $x$.
      \item
      \label{item_P4}
      There exists a map $u \in W^{1,2}_{\loc}(U\setminus S, \mathcal{N})$ such that $\tr_{\partial U}(u)=g$.
  \end{enumerate}
The set of all admissible chains for $U$ and $g$ will be denoted by $\mathscr{C}(U,g)$.
 \end{defn}
 
\begin{rem}\label{remark chain} 
Assume that there exists $u \in W^{1,2}_{\loc}(U\setminus S, \mathcal{N})$ satisfying \ref{item_P4}. Then there exists $v\in W^{1,2}_{\loc}(U\setminus S, \mathcal{N}) \cap C(U\setminus (S\cup S_{v}), \mathcal{N})$ such that $\tr_{\partial U}(v)=g$, where $S_{v}$ is at most a countable and locally finite subset of $U\setminus S$ (see \cite[Theorem~II]{SchoUhl}). 
Let $L$ be a segment in the chain $S$ and $x_{1}, x_{2}$ be interior points of $L$ such that $x_{1}\neq x_{2}$. Since $v \in W^{1,2}_{\loc}(U\setminus S, \mathcal{N})$, using the coarea formula, we can find a sufficiently small radius $r>0$ such that the following holds. 
There exists $\gamma \in [\mathbb{S}^{1}, \mathcal{N}]$ such that  for $\mathcal{H}^{1}$-a.e.\  $x \in [x_{1},x_{2}]$, letting $\partial D$ to be the circle with center $x$ and radius $r$ lying in the 2-plane orthogonal to $L$, we have $\tr_{\partial D}(v)=v|_{\partial D} \in W^{1,2}(\partial D, \mathcal{N})$, $v|_{\partial D}$ is homotopic to $\gamma$ and  $\mathcal{E}^{\mathrm{sg}}_{2}(v|_{\partial D})=\mathcal{E}^{\mathrm{sg}}_{2}(\gamma)$. 
On the other hand, since $v \in C(U\setminus (S\cup S_{v}), \mathcal{N})$ and $S_{v}$ is locally finite, we can choose $x_{1}$ and $x_{2}$ to be arbitrarily close to the endpoints of $L$ and repeat the previous procedure for a smaller radius $\varrho \in (0,r)$, but without changing $\gamma \in [\mathbb{S}^{1}, \mathcal{N}]$. Thus, we can define the \emph{mass} of $S$ with respect to the extension $v$ of $g$ as the sum $\sum_{i=1}^{l}\lambda_{i}\mathcal{H}^{1}(L_{i})$, where $\bigcup_{i=1}^{l}L_{i}=S$ is the union of the segments lying in $S$ and $\lambda_{i}=\mathcal{E}^{\mathrm{sg}}_{2}(\gamma_{i})$, where $\gamma_{i} \in [\mathbb{S}^{1}, \mathcal{N}]$ is defined for each $L_{i}$ with respect to $v$.
Denote this mass by $\mathbb{M}(v, g, S)$. 
It is worth noting that for different extensions $u_{1} \in W^{1,2}_{\loc}(U\setminus S, \mathcal{N}) \cap C(U\setminus (S\cup S_{u_{1}}), \mathcal{N})$ and $u_{2} \in W^{1,2}_{\loc}(U\setminus S, \mathcal{N}) \cap C(U\setminus (S\cup S_{u_{2}}), \mathcal{N})$ of $g$ satisfying the property \ref{item_P4} of Definition~\ref{def admissible chain}, we can obtain different masses of $S$ with respect to $u_{1}$ and $u_{2}$, however, the following
    \begin{equation}\label{mass inf}
    \inf\{\mathbb{M}(u, g, S): u \in W^{1,2}_{\loc}(U\setminus S, \mathcal{N}) \cap C(U \setminus (S\cup S_{u}), \mathcal{N}),\,\ \operatorname{tr}_{\partial U}(u)=g\}
    \end{equation}
is achieved. Indeed, if $u_{n} \in W^{1,2}_{\loc}(U\setminus S, \mathcal{N})$ is a minimizing sequence for \eqref{mass inf}, then, since $\# [\mathbb{S}^{1}, \mathcal{N}]<+\infty$, the set $\{\mathcal{E}^{\mathrm{sg}}_{2}(\gamma): \gamma \in [\mathbb{S}^{1}, \mathcal{N}]\}$ is finite, and hence $\{\mathbb{M}(u_{n}, g, S) : n\in \mathbb{N}\}$ is finite. Thus, for each sufficiently large $n\in \mathbb{N}$, $\mathbb{M}(u_{n}, g, S)$ is equal to the infimum in \eqref{mass inf}, and each $u_{n}$, for sufficiently large $n \in \mathbb{N}$, realizes the infimum in \eqref{mass inf}.
\end{rem}
   
 Next, using Remark~\ref{remark chain}, we define the notion of the \emph{mass} of an admissible chain. 
 \begin{defn}\label{defnofmass}
     Given a bounded Lipschitz domain $U \subset \mathbb{R}^{3}$ and a map $g \in W^{1/2,2}(\partial U, \mathcal{N})$, if $S \in \mathscr{C}(U,g)$, we define its mass by
     \[
    \mathbb{M}(g, S)=\mathbb{M}(u, g, S),
     \]
     where $u \in W^{1,2}_{\loc}(U\setminus S, \mathcal{N}) \cap C(U\setminus (S \cup S_{u}), \mathcal{N})$ realizes the infimum in \eqref{mass inf} for $g$ and $S$.
 \end{defn}
 
 When appropriate, we shall simply write $\mathbb{M}(S)$ instead of $\mathbb{M}(g, S)$. Notice that the definition of $\mathbb{M}(g, S)$ also depends on $\mathcal{N}$, but we decided not to mention it explicitly to simplify the notation.
 
 Now we prove the interior minimality of $S_{*}$.
 \begin{prop}\label{min of sing}
     Let $u_{*} \in W^{1,2}_{\loc}(\Omega\setminus S_{*}, \mathcal{N})$ and $(p_{n})_{n\in \mathbb{N}} \subset [1,2)$ be given by Proposition~\ref{conv to harm}. Then for any domain $K \csubset \Omega$ such that $\overline{K}\cap S_{*}\neq \emptyset$, there exists a Lipschitz domain $U \csubset \Omega$ containing $K$ such that $\operatorname{tr}_{\partial U}(u_{*})=u_{*}|_{\partial U} \in W^{1/2,2}(\partial U, \mathcal{N})$, $S_{*} \cap \overline{U} \in \mathscr{C}(U,u_{*}|_{\partial U})$ and for each $S \in \mathscr{C}(U, u_{*}|_{\partial U})$,
     \begin{equation}\label{HomologicalPlateau}
     \mu_{*}(S_{*}\cap U)=\mathbb{M}(u_{*}|_{\partial U}, S_{*}\cap \overline{U})\leq \mathbb{M}(u_{*}|_{\partial U}, S).
     \end{equation}
 \end{prop}
 
 \begin{rem}
     Proposition~\ref{min of sing} allows the choice of the domain $U$ to be arbitrarily close to $K$. 
      \end{rem}
 \begin{proof}[Proof of Proposition~\ref{min of sing}]
     According to Proposition~\ref{prop structure S_*}, since $\overline{K}$ is connected, there exists a Lipschitz domain $V\csubset \Omega$ containing $\overline{K}$ such that $S_{*}\cap \overline{V}$ is a union of a finite number of closed straight line segments, each of which has a positive length, and the interiors of these segments are pairwise disjoint. Also $S_{*}\cap \partial V$ is a finite set of points such that for each point $x \in S_{*}\cap \partial V$ there is exactly one segment lying in $S_{*}\cap \overline{V}$ and emanating from $x$. Changing, if necessary, $\partial V$ locally around each point of $S_{*}\cap \partial V$, we can assume that for each $x \in S_{*}\cap \partial V$, $\partial V$ coincides with its tangent plane at $x$ in a small neighborhood of $x$, and the segment lying in $S_{*}\cap \overline{V}$ and emanating from $x$ is orthogonal to this tangent plane.
     
     \medskip
     
     \noindent \emph{Step~1}. 
     We define $U$. 
     Fix $\lambda>0$ small enough such that for each $t \in (0,\lambda]$, all the above properties prescribed for $V$ concerning $S_{*}$ and $K$ hold also for $V_{t}=\{y \in V: \dist(y, \partial V)>t\}$, which is a Lipschitz domain. 
     Applying \cite[Theorem~3.2.22~(3)]{Federer} with $W=V\setminus \overline{V}_{\lambda}$, $Z=(0,\lambda)$ and $f:W\to Z$ defined by $f(y)=\dist(y,\partial V)$, and using also the fact that $|Df(y)| = 1$ for a.e.\  $y \in W$, Fatou's lemma, \eqref{C2} and Proposition~\ref{prop est for the differential inside domain}~\ref{it_Aex0Aijushoetha0looKo6ei}, we obtain
     \begin{multline}\label{cosobcon7ok54kk5kg0k}
         \int^{\lambda}_{0}\diff \varrho \;\liminf_{n\to +\infty}\int_{f^{-1}(\{\varrho\})}(|u_{n}-u_{*}|^{q}+|Du_{n}-Du_{*}|^{q})\diff \mathcal{H}^{2}\\ \leq \liminf_{n\to +\infty} \int_{W}(|u_{n}-u_{*}|^{q}+|Du_{n}-Du_{*}|^{q}) \diff y=0
     \end{multline}
for each $q \in (1,2)$ and
     \begin{multline*}
         \int^{\lambda}_{0}\diff \varrho \;\liminf_{n\to +\infty} (2-p_{n})\int_{f^{-1}(\{\varrho\})}(|Du_{n}|^{p_{n}}+|Du_{*}|^{p_{n}})\diff \mathcal{H}^{2}\\ \leq \liminf_{n\to +\infty} (2-p_{n})\int_{W}(|Du_{n}|^{p_{n}}+|Du_{*}|^{p_{n}}) \diff y\leq C^{\prime},
     \end{multline*}
     where $C^{\prime}=C^{\prime}(\dist(V, \Omega), \diam(\Omega),  C_{0}, \mathcal{N})>0$ and $C_{0}>0$ is the constant coming from \eqref{C2}. On the other hand, using Proposition~\ref{conv to harm} (thanks to which we know that, up to a subsequence (not relabeled), $u_{n}\to u_{*}$ in $L^{p}(\Omega, \mathbb{R}^{\nu})$ for all $p \in [1,+\infty)$ (see Proposition~\ref{prop conv to harmonic map}~\ref{item:weakLp})) and the continuity of the $L^{p}$-norm, we deduce that $\|u_{n}-u_{*}\|^{p_{n}}_{L^{p_{n}}(W, \mathbb{R}^{\nu})}\to 0$  as $n\to +\infty$ and hence 
     \begin{equation*}
         \lim_{n\to +\infty}\int_{0}^{\lambda}\diff \varrho \int_{f^{-1}(\{\varrho\})}|u_{n}-u_{*}|^{p_{n}}\diff \mathcal{H}^{2}=0.
     \end{equation*}
     Then there exists $t \in (0, \lambda)$ such that for $U=V_{t}$, up to a subsequence (not relabeled), for each $n\in \mathbb{N}$ large enough: $\tr_{\partial U}(u_{n})=u_{n}|_{\partial U}, \tr_{\partial U}(u_{*})=u_{*}|_{\partial U}\in W^{1,p_{n}}(\partial U, \mathcal{N}) \subset W^{1/2,2}(\partial U, \mathcal{N})$ and
     \begin{equation}\label{ineq for p-energy limit map boundary}
         \int_{\partial U} (|Du_{n}|^{p_{n}}+|Du_{*}|^{p_{n}})\diff \mathcal{H}^{2}\leq \frac{C^{\prime \prime}}{2-p_{n}},
     \end{equation}
     where $C^{\prime \prime}=C^{\prime \prime}(t,\dist(V,\Omega), \diam(\Omega), C_{0}, \mathcal{N})>0$; 
     \begin{equation}\label{lpnconvfortraces} 
     \|u_{n}|_{\partial U}-u_{*}|_{\partial U}\|^{p_{n}}_{L^{p_{n}}(\partial U, \mathbb{R}^{\nu})}\to 0 
 \end{equation}
     as $n\to +\infty$;
    \begin{equation}\label{uniformconvinterunustar1}
    \|u_{n}|_{\partial U}-u_{*}|_{\partial U}\|_{L^{\infty}(\partial U; \mathbb{R}^{\nu})} \to 0
    \end{equation}
    as $n\to +\infty$, in view of \eqref{cosobcon7ok54kk5kg0k} and Morrey's embedding for Sobolev mappings.
     Since $t$ and $V$ uniquely determine $U$, we can assume that $C^{\prime \prime}$ actually depends on $U, \diam(\Omega)$, $C_{0}$ and $\mathcal{N}$. 
     Thus, we have defined $U$. According to our construction and Proposition~\ref{conv to harm},  the properties \ref{item_P1}-\ref{item_P4} of Definition~\ref{def admissible chain} are fulfilled for $U$, $u_{*}|_{\partial U}$ and $S_{*}\cap \overline{U}$. Therefore, $S_{*}\cap \overline{U} \in \mathscr{C}(U,u_{*}|_{\partial U})$.
     
     \medskip
     
     \noindent \emph{Step~2.} Given $S \in \mathscr{C}(U, u_{*}|_{\partial U})$, we construct a sequence of mappings $w_{n} \in W^{1,p_{n}}(U, \mathcal{N})$ such that $\tr_{\partial U}(w_{n})=u_{*}|_{\partial U}$ and
     \begin{equation}\label{ener comp by mass}
         \limsup_{n\to +\infty}(2-p_{n})\int_{U}\frac{|Dw_{n}|^{p_{n}}}{p_{n}}\diff y \leq \mathbb{M}(S).
     \end{equation}
     Consider a map $u \in W^{1,2}_{\loc}(U\setminus S, \mathcal{N})$ minimizing the expression in \eqref{mass inf}, where $g=u_{*}|_{\partial U}$. 
     Let $\varepsilon \in (0,1)$ and $\delta \in (0,1)$ be small enough and $\delta$ be fairly small with respect to $\varepsilon$. 
     Then, using the coarea formula and taking into account that $u \in W^{1,2}_{\loc}(U\setminus S, \mathcal{N})$, we obtain $\delta^{\prime} \in (3\delta/4, 5\delta/4)$ and $\varepsilon^{\prime} \in (3\varepsilon/4, 5\varepsilon/4)$ such that, defining $S_{\delta^{\prime}}=\{y \in U: \dist(y, S)\leq \delta^{\prime}\}$, we have the following. Firstly, $\tr_{\partial S_{\delta^{\prime}} \cap U}(u)=u|_{\partial S_{\delta^{\prime}}\cap U} \in W^{1,2}(\partial S_{\delta^{\prime}}\cap U, \mathcal{N})$.  Secondly, for each branching point $x \in S\cap \overline{U}$ of $S$ (if such a point exists), $\tr_{(\partial B^{3}_{\varepsilon^{\prime}}(x)\setminus S_{\delta^{\prime}})\cap U}(u)=u|_{(\partial B^{3}_{\varepsilon^{\prime}}(x)\setminus S_{\delta^{\prime}})\cap U} \in W^{1,2}((\partial B^{3}_{\varepsilon^{\prime}}(x)\setminus S_{\delta^{\prime}})\cap U, \mathcal{N})$, where $(\partial B^{3}_{\varepsilon^{\prime}}(x)\setminus S_{\delta^{\prime}})\cap U=\partial B^{3}_{\varepsilon^{\prime}}(x)\setminus S_{\delta^{\prime}}$ if $x \in U$ (we recall that $\delta$ is fairly small with respect to $\varepsilon$ and the structure of $S$ is given by Definition~\ref{def admissible chain}).
     To lighten the notation, we assume that $\delta^{\prime}=\delta$ and $\varepsilon^{\prime}=\varepsilon$. 
     Let $L$ be a segment of $S$. 
     Up to rotation and translation, we can assume that $L=\{(0,0)\} \times [-h,h]$, where $h=\mathcal{H}^{1}(L)/2>0$. 
     Let $\Lambda \coloneqq B^{2}_{\delta}\times (z(0,0,-h), z(0,0,h))$, where $z(0,0,-h)\coloneqq -h+\sqrt{\varepsilon^{2}-\delta^{2}}$ if $(0,0,-h)$ is a branching point of $S$ (\emph{i.e.}, several segments of $S$ emanate from this point), $z(0,0,h)\coloneqq h-\sqrt{\varepsilon^{2}-\delta^{2}}$ if $(0,0,h)$ is a branching point of $S$, $z(0,0,-h)\coloneqq -h+\kappa \delta$ if $(0,0,-h) \in \partial U$ is an endpoint of $S$ (\emph{i.e.}, only one segment of $S$ emanates from this point, namely $L$), $z(0,0,h)=h-\kappa \delta$ if $(0,0,h) \in \partial U$ is an endpoint of $S$, where $\kappa=\kappa(S)>0$ will be defined later for the proof to work. 
     Hereinafter in this proof, $C$ will denote a positive constant that is independent of $n$, $\varepsilon$ and $\delta$ and can be different from line to line. 
     We define $w_{n}$ in $\Lambda$ by using Lemma~\ref{lem 2.19}~\ref{it_pha8ieh4Umaejao2vaephei8}, namely we obtain a map $w_{n}|_{\Lambda} \in W^{1,p_{n}}(\Lambda, \mathcal{N})$ such that 
     \begin{equation} \label{est in cylinder min}
         \int_{\Lambda}\frac{|Dw_{n}|^{p_{n}}}{p_{n}}\diff y \leq \frac{\mathcal{H}^{1}(L)\mathcal{E}^{\mathrm{sg}}_{p_{n}}(\gamma)\delta^{2-p_{n}}}{2-p_{n}}+\frac{C}{\delta^{p_{n}-1}}\int_{\Gamma}\frac{|D_{\top}u|^{p_{n}}}{p_{n}}\diff \mathcal{H}^{2},          
     \end{equation}
   where $\Gamma$ is the lateral surface of $\Lambda$ and  $\gamma \in [\mathbb{S}^{1}, \mathcal{N}]$ is homotopic to $u|_{\partial D} \in C(\partial D, \mathcal{N})$, where $\partial D$ is the circle being the boundary of a $2$-disk orthogonal to $L$ with fairly small radius and whose center belongs to the interior of $L$. 
   We have used that $\Gamma\subset (\partial S_{\delta} \cap U)$ and $u|_{\partial S_{\delta}\cap U} \in W^{1,2}(\partial S_{\delta}\cap U, \mathcal{N})$. 

   We need to distinguish between two further cases. 
   
   \noindent \emph{Case~1.} 
   Let $x \in \partial U$ be an endpoint of $S$ (if it exists). 
   Then, taking into account the properties \ref{item_P1} and \ref{item_P2} of Definition~\ref{def admissible chain} and the fact that for each $x \in S\cap \partial U$, $\partial U$ coincides with its tangent plane at $x$ in a small neighborhood of $x$,  we can find $\kappa=\kappa(S)>0$ such that for each endpoint $y \in \partial U$ of $S$, if $L$ is a segment of $S$ emanating from $y$, then the 2-disk orthogonal to $L$ with center $\xi \in L$ such that $|\xi-y|=\kappa \delta$ lies entirely in $U$. 
   Now let $L$ denote the segment of $S$ emanating from $x$. Up to rotation and translation, $L=\{(0,0)\} \times  [-h,h]$ and $x=(0,0,h)$, where $h\coloneqq \mathcal{H}^{1}(L)/2>0$. 
   Let $\Lambda_{\delta} \coloneqq \{y=(y_{1}, y_{2}, t) \in U: \dist(y, L)<\delta\,\ \text{and}\,\ t \in (h-\kappa \delta, h)\}$. 
   Then $\overline{\Lambda}_{\delta}$ is bilipschitz homeomorphic to $\smash{\overline{B}}^3_{\delta}$ with a bilipschitz constant depending on $\kappa$. 
   Thus, there exists a bilipschitz map $\Psi: \overline{\Lambda}_{\delta}\to \smash{\overline{B}}^3_{\delta}$. We have constructed $w_{n}$ on $\Lambda=B^{2}_{\delta}\times (z(0,0,-h), h-\kappa \delta)$. 
   By Lemma~\ref{lem 2.19}~\ref{it_cang8wooGhodoh3Eegoom7ee}, $\tr_{B^{2}_{\delta}\times \{h-\kappa \delta\}}(w_{n}) \in W^{1,p_{n}}(B^{2}_{\delta}\times \{h-\kappa \delta\}, \mathcal{N})$ and
   \begin{equation}
      \label{est 2-disk cyl}
      \begin{split}
       \int_{B^{2}_{\delta}\times \{h-\kappa \delta\}}\frac{|D\operatorname{tr}_{B^{2}_{\delta}\times \{h-\kappa\delta\}}(w_{n})|^{p_{n}}}{p_{n}}\diff \mathcal{H}^{2} &\leq \frac{\mathcal{E}^{\mathrm{sg}}_{p_{n}}(\gamma)\delta^{2-p_{n}}}{2-p_{n}} + \frac{C}{\delta^{p_{n}-1}}\int_{\Gamma}\frac{|D_{\top}u|^{p_{n}}}{p_{n}}\diff \mathcal{H}^{2}\\
       &\leq \frac{C\delta^{2-p_{n}}}{2-p_{n}}+\frac{C}{\delta^{p_{n}-1}}\int_{\Gamma}\frac{|D_{\top}u|^{p_{n}}}{p_{n}}\diff \mathcal{H}^{2},
       \end{split}
       \end{equation}
       where $\gamma \in [\mathbb{S}^{1}, \mathcal{N}]$ comes from \eqref{est in cylinder min}, $\Gamma$ is the lateral surface of $\Lambda$, and we have used that $\#[\mathbb{S}^{1}, \mathcal{N}]<+\infty$. Now we define the mapping $g_{n}:\partial \Lambda_{\delta}\to \mathcal{N}$ by
       \begin{equation*}
       g_{n}(y)=
        \begin{cases}
           u_{*}|_{\partial U} (y)&\text{if \(y \in \partial \Lambda_{\delta}\cap \partial U\),}\\
           u|_{\partial S_{\delta}\cap U}(y) &\text{if \(y \in \partial S_{\delta}\cap U\),}\\
           \operatorname{tr}_{B^{2}_{\delta}\times \{h-\kappa \delta\}}(w_{n}) & \text{if \(y \in B^{2}_{\delta}\times \{h-\kappa \delta\}\).} 
       \end{cases}
   \end{equation*}
    Then, since the traces of $u_{*}|_{\partial U}$ and $u|_{\partial S_{\delta}\cap U}$ coincide on $\partial S_{\delta}\cap \partial U$, and the traces of $u|_{\partial S_{\delta}\cap U}$ and $\tr_{B^{2}_{\delta}\times \{h-\kappa\delta\}}(w_{n})$ coincide on $\partial B^{2}_{\delta}\times \{h-\kappa\delta\}$, we deduce that $g_{n}\in W^{1,p_{n}}(\partial \Lambda_{\delta},\mathcal{N})$. Using \eqref{ineq for p-energy limit map boundary}, \eqref{est 2-disk cyl} and the fact that $\delta \in (0,1)$ is fairly small, we obtain
    \begin{equation}\label{est en bound datum}
     \int_{\partial \Lambda_{\delta}}\frac{|D_{\top}g_{n}|^{p_{n}}}{p_{n}}\diff \mathcal{H}^{2} \leq \frac{C}{2-p_{n}}+\frac{C\delta^{2-p_{n}}}{2-p_{n}}+\frac{C}{\delta^{p_{n}-1}}\int_{\partial S_{\delta}\cap U}\frac{|D_{\top}u|^{p_{n}}}{p_{n}}\diff \mathcal{H}^{2}.
    \end{equation}
    Next, we define $w_{n}|_{\Lambda_{\delta}}(y)=g_{n}(\Psi^{-1}(\delta \Psi(y)/|\Psi(y)|))$ for $y \in \Lambda_{\delta}\setminus \{\Psi^{-1}(0)\}$. Then, using \eqref{est en bound datum}, we deduce the following
    \begin{equation}
     \label{est wn endpoint}
     \begin{split}
    \int_{\Lambda_{\delta}}\frac{|Dw_{n}|^{p_{n}}}{p_{n}}\diff y 
    &\leq C\delta \int_{\partial \Lambda_{\delta}}\frac{|D_{\top}g_{n}|^{p_{n}}}{p_{n}}\diff \mathcal{H}^{2}\\
    &\leq \frac{C\delta}{2-p_{n}}+\frac{C\delta^{3-p_{n}}}{2-p_{n}}+C\delta^{2-p_{n}}\int_{\partial S_{\delta}\cap U}\frac{|D_{\top}u|^{p_{n}}}{p_{n}}\diff \mathcal{H}^{2}.
    \end{split}
    \end{equation}
    
    \medskip 
    
    \noindent \emph{Case~2.} Let $x$ be a branching point of $S$ (if it exists). Then, according to our construction, $u|_{(\partial B^{3}_{\varepsilon}(x)\setminus S_{\delta})\cap U} \in W^{1,2}((\partial B^{3}_{\varepsilon}(x)\setminus S_{\delta})\cap U, \mathcal{N})$. Furthermore, $(\partial B^{3}_{\varepsilon}(x)\setminus S_{\delta})\cap U=\partial B^{3}_{\varepsilon}(x)\setminus S_{\delta}$ if $x\in U$. Let $L_{i}$, $i \in \{1,\dotsc,s\}$ be the segments of $S$ emanating from $x$. We fill the holes in $(\partial B^{3}_{\varepsilon}(x)\setminus S_{\delta})\cap U$ with the closures of the open 2-disks $D_{i}$ such that $D_{i}$ is orthogonal to $L_{i}$, has radius $\delta$, and its center lies on $L_{i}$ at a distance of $\sqrt{\varepsilon^{2}-\delta^{2}}$ from the point $x$. Thus, $\mathcal{M}=(\partial U \cap \smash{\overline{B}}^{3}_{\varepsilon}(x)) \cup ((\partial B^{3}_{\varepsilon}(x)\setminus S_{\delta})\cap U) \cup \bigcup_{i=1}^{l}\overline{D}_{i}$ is a compact 2-manifold without boundary, which is bilipschitz homeomorphic to $\partial B^{3}_{\varepsilon}$. Furthermore, according to our construction, for each $i \in \{1,\dotsc,l\}$, $\tr_{D_{i}}(w_{n})\in W^{1,p_{n}}(D_{i}, \mathcal{N})$. Since the traces of $u_{*}|_{\partial {U}}$ and $u|_{(\partial B^{3}_{\varepsilon}(x)\setminus S_{\delta})\cap U}$ coincide on $\partial U \cap \partial B^{3}_{\varepsilon}(x)$, and the traces of $u|_{(\partial B^{3}_{\varepsilon}(x)\setminus S_{\delta})\cap U}$ and $\tr_{D_{i}}(w_{n})$ coincide on $\partial D_{i}$, defining $v_{n}:\mathcal{M}\to \mathcal{N}$ by 
    \[
    v_{n}(y)=
    \begin{cases}
    u_{*}|_{\partial U}(y) \,\ &\text{if}\,\ y \in \partial U \cap \smash{\overline{B}}^{3}_{\varepsilon}(x),\\
    u|_{(\partial B^{3}_{\varepsilon}(x)\setminus S_{\delta})\cap U}(y) \,\ &\text{if}\,\ y \in (\partial B^{3}_{\varepsilon}(x)\setminus S_{\delta})\cap U,\\
    \operatorname{tr}_{\overline{D}_{i}}(w_{n})(y) \,\ &\text{if} \,\ y \in \overline{D}_{i},
    \end{cases}
    \]
    we observe that $v_{n}\in W^{1,p_{n}}(\mathcal{M}, \mathcal{N})$. Next, let $E$ be an open bounded set whose boundary is $\mathcal{M}$. Let $\Phi: \overline{E}\to \smash{\overline{B}}^3_{\varepsilon}$ be a bilipschitz homeomorphism. We define $w_{n}|_{E}(y)=v_{n}(\Phi^{-1}(\varepsilon \Phi(y)/|\Phi(y)|))$ for $y \in E\setminus \{\Phi^{-1}(0)\}$. Then, using \eqref{ineq for p-energy limit map boundary}, \eqref{est 2-disk cyl} and Definition~\ref{def admissible chain}, we obtain
    \begin{equation}
     \label{Est branching point}
     \begin{split}
        \int_{E}\frac{|Dw_{n}|^{p_{n}}}{p_{n}}\diff y &\leq C\varepsilon \int_{\mathcal{M}} \frac{|D_{\top}v_{n}|^{p_{n}}}{p_{n}}\diff \mathcal{H}^{2}\\
        &\leq C\varepsilon\int_{\partial U\cap B^{3}_{\varepsilon}(x)}\frac{|D_{\top}u_{*}|^{p_{n}}}{p_{n}}\diff \mathcal{H}^{2}+ C\varepsilon\int_{(\partial B^{3}_{\varepsilon}(x)\setminus S_{\delta})\cap U}\frac{|D_{\top}u|^{p_{n}}}{p_{n}}\diff \mathcal{H}^{2}\\ & \qquad \qquad +C\varepsilon \sum_{i=1}^{l}\int_{D_{i}}\frac{|D_{\top}\operatorname{tr}_{D_{i}}(w_{n})|^{p_{n}}}{p_{n}}\diff \mathcal{H}^{2}\\
        &\leq \frac{C\varepsilon}{2-p_{n}} + C\varepsilon\int_{(\partial B^{3}_{\varepsilon}(x)\setminus S_{\delta})\cap U}\frac{|D_{\top}u|^{p_{n}}}{p_{n}}\diff \mathcal{H}^{2} +\frac{C\varepsilon\delta^{2-p_{n}}}{2-p_{n}} \\ & \qquad \qquad+\frac{C\varepsilon}{\delta^{p_{n}-1}}\int_{\partial S_{\delta}\cap U}\frac{|D_{\top}u|^{p_{n}}}{p_{n}}\diff \mathcal{H}^{2}. 
        \end{split}
        \end{equation}
    Finally, we define $w_{n}$ for each $y$ in $U$ lying outside the set $E_{\varepsilon, \delta}$, which is the union of $S_{\delta}$ with the union of $\smash{\overline{B}}^3_{\varepsilon}(x)\cap U$ over all branching points $x$ of $S$, by $w_{n}(y)=u(y)$. Since the traces of the piecewise definitions of $w_{n}$ coincide, $w_{n}\in W^{1,p_{n}}(U, \mathcal{N})$. Next, gathering together \eqref{est in cylinder min}, \eqref{est wn endpoint}, \eqref{Est branching point} and using the facts that $\varepsilon, \delta \in (0,1)$ are sufficiently small, $\delta$ is fairly small with respect to $\varepsilon$,  $(p_{n})_{n\in \mathbb{N}}\subset [1,2)$ and $p_{n}\nearrow 2$ as $n\to +\infty$, $S$ consists of a finite union of segments (and hence it has a finite number of branching points and endpoints), for all sufficiently large $n \in \mathbb{N}$, we deduce the following
    \begin{align*}
    \int_{U}\frac{|Dw_{n}|^{p_{n}}}{p_{n}}\diff y &\leq \frac{\mathbb{M}(S)\delta^{2-p_{n}}}{2-p_{n}} + \frac{C\varepsilon}{2-p_{n}}+C\varepsilon\sum_{x\in P}\int_{(\partial B^{3}_{\varepsilon}(x)\setminus S_{\delta})\cap U}\frac{|D_{\top}u|^{p_{n}}}{p_{n}}\diff \mathcal{H}^{2}\\ & \qquad \qquad + \frac{C\varepsilon\delta^{2-p_{n}}}{2-p_{n}}+ \frac{C}{\delta^{p_{n}-1}}\int_{\partial S_{\delta}\cap U}\frac{|D_{\top}u|^{p_{n}}}{p_{n}}\diff \mathcal{H}^{2}+\int_{U\setminus E_{\varepsilon,\delta}}\frac{|Du|^{p_{n}}}{p_{n}}\diff y, 
    \end{align*}
    where $P$ is the set of branching points of $S$. Multiplying both sides of the above inequality by $(2-p_{n})$, then letting $n$ tend to $+\infty$ and using the continuity of the $L^{p}$-norm and the facts that $u \in W^{1,2}(\partial S_{\delta}\cap U, \mathcal{N})$, $u \in W^{1,2}((\partial B^{3}_{\varepsilon}(x)\setminus S_{\delta})\cap U, \mathcal{N})$ for each branching point $x$ of $S$, $u \in W^{1,2}(U\setminus E_{\varepsilon, \delta}, \mathcal{N})$, we deduce that
    \begin{equation*}
        \limsup_{n\to+\infty}(2-p_{n})\int_{U}\frac{|Dw_{n}|^{p_{n}}}{p_{n}}\diff y \leq \mathbb{M}(S)+C\varepsilon.
    \end{equation*}
    Since $\varepsilon>0$ was chosen arbitrarily small and $C$ is independent of $\varepsilon$, the last estimate implies \eqref{ener comp by mass}.
    
    \medskip
    
    \noindent \emph{Step~3.} We prove \eqref{HomologicalPlateau}. Let the sequence of maps $w_{n} \in W^{1,p_{n}}(U, \mathcal{N})$ be as in \emph{Step~2}. Using Lemma~\ref{lemma_Luckhaus}, we construct a competitor for the minimizer $u_{n}$ by interpolating between $u_{n}$ and $w_{n}$. In view of \eqref{lpnconvfortraces} and the fact that $\tr_{\partial U}(w_{n})=u_{*}|_{\partial U}$,
    \begin{equation}\label{tracelpnconv}
        \|u_{n}|_{\partial U}-\operatorname{tr}_{\partial U}(w_{n})\|^{p_{n}}_{L^{p_{n}}(\partial U,\mathbb{R}^{\nu})}\to 0
    \end{equation}
    as $n\to+\infty$.  Let $T_{0}=T_{0}(\partial U)>0$ be the constant of Lemma~\ref{lemma_Luckhaus}. Choose $T^{\prime}\in (0,T_{0})$ sufficiently small and $\Phi: \partial U\times [0,T^{\prime}]\to \overline{E}_{r_{0}}$ such that $\Phi$ is bilipschitz onto its image, $\Phi(\sigma,0)=\sigma$, $D_{\top}\Phi(\sigma,0)$ is an isometry for $\sigma\in \partial U$, where $E_{r_{0}}=\{x \in U: \dist(x, \partial U)<r_{0}\}$ for some fairly small constant $r_{0}>0$ depending on $T^{\prime}$. According to Lemma~\ref{lemma npr}, there exists a constant $\lambda_{\mathcal{N}}>0$ such that the nearest point retraction $\Pi_{\mathcal{N}}:\mathcal{N}_{\lambda_{\mathcal{N}}} \to \mathcal{N}$ is well defined and smooth, where $
    \mathcal{N}_{\lambda_{\mathcal{N}}}=\{x \in \mathbb{R}^{\nu}: \dist(x, \mathcal{N})<\lambda_{\mathcal{N}}\}
    $.  Setting $s_{n}=\|u_{n}|_{\partial U}-\tr_{\partial U}(w_{n})\|_{L^{p_{n}}(\partial U, \mathbb{R}^{\nu})}$ and using \eqref{ineq for p-energy limit map boundary} and \eqref{tracelpnconv}, we observe that $s_{n}\to 0$ and for each sufficiently large $n \in \mathbb{N}$,
    \begin{equation}\label{estimluckhmin}
        \int_{\partial U}\left(|D_{\top} u_{n}|^{p_{n}}+|D_{\top}\operatorname{tr}_{\partial U}(w_{n})|^{p_{n}}+\frac{|u_{n}-u_{*}|^{p_{n}}}{s^{p_{n}/2}_{n}}\right) \diff \mathcal{H}^{2} \leq \frac{C}{2-p_{n}}.
    \end{equation}
      Then, by Lemma~\ref{lemma_Luckhaus}, applied with $u=u_{n}|_{\partial U}$, $v=\tr_{\partial U}(w_{n})=u_{*}|_{\partial U}$ and $T=s_{n}^{1/2}$, there exists a mapping $\xi_{n}\in W^{1,p_{n}}(\partial U\times (0,s_{n}^{1/2}), \mathbb{R}^{\nu})$ interpolating between $u_{n}|_{\partial U}$ and $\tr_{\partial U}(w_{n})$. By \eqref{uniformconvinterunustar1}, we deduce that
     \[
      \ess_{(\sigma, t) \in \partial U \times (0,s^{1/2}_{n})}  \dist(\xi_{n}(\sigma,t), \mathcal{N}) \leq \|u_{n}|_{\partial U}-u_{*}|_{\partial U}\|_{L^{\infty}(\partial U, \mathbb{R}^{\nu})} \to 0
      \]
      as $n \to +\infty$. This implies that for each $n\in \mathbb{N}$ large enough, it holds 
       \begin{equation}\label{unifesttomanifoldnforxinonasmoothfiber}
      \dist(\xi_{n}(\sigma,t), \mathcal{N})<\frac{\lambda_{\mathcal{N}}}{10}
       \end{equation}
       for a.e.\  $(\sigma, t) \in \partial U\times (0,s_{n}^{1/2})$.  
     
      In view of \eqref{unifesttomanifoldnforxinonasmoothfiber} and Lemma~\ref{lemma npr}, for each $n \in \mathbb{N}$ large enough we can define the mapping $\varphi_{n}(y)\coloneqq \Pi_{\mathcal{N}}(\xi_{n}(\Phi^{-1}(y)))$ for each $y\in \Phi(\partial U\times (0,s_{n}^{1/2}))$. 
    Then the following assertions hold. First, $\varphi_{n}\in W^{1,p_{n}}(\Phi(\partial U\times (0,s_{n}^{1/2})), \mathcal{N})$. Second,  \[\operatorname{tr}_{\partial U}(\varphi_{n})=u_{n}|_{\partial U},\,\ 
    \operatorname{tr}_{\Phi(\partial U\times\{s_{n}^{1/2}\})}(\varphi_{n})=u_{*}|_{\partial U}\circ \Phi_{n}\circ  \Phi^{-1}|_{\Phi(\partial U\times \{s_{n}^{1/2}\})},\] where $\Phi_{n}: \partial U\times [s^{1/2}_{n},T^{\prime}]$ is defined by $\Phi_{n}(\sigma,t)\coloneqq \Phi(\sigma, t-s^{1/2}_{n})$ and we have used the fact that $\Phi(\sigma,0)=\sigma$. Thirdly, using \cite[Theorem 3.2.5]{Federer}, the fact that $\Psi$ is bilipschitz, 
    Lemma~\ref{lemma_Luckhaus}~\ref{it_eepogaechuwaeghee3UlooPo}, where $T=s^{1/2}_{n}$, and \eqref{estimluckhmin}, we have the following  
    \begin{equation}\label{estiml2min}
    \begin{split}
        (2-p_{n})&\int_{\Phi(\partial U\times (0,s_{n}^{1/2}))}\frac{|D\varphi_{n}|^{p_{n}}}{p_{n}}\diff y \leq C(2-p_{n})\int_{\partial U\times (0,s_{n}^{1/2})}\frac{|D\xi_{n}|^{p_{n}}}{p_{n}}\diff \mathcal{H}^{2}\diff t\\
        &\leq  C(2-p_{n})s^{1/2}_{n} \int_{\partial U}\left(|D_{\top} u_{n}|^{p_{n}}+|D_{\top}\operatorname{tr}_{\partial U}(w_{n})|^{p_{n}}+\frac{|u_{n}-u_{*}|^{p_{n}}}{s^{p_{n}/2}_{n}}\right) \diff \mathcal{H}^{2} \leq C s_{n}^{1/2}.
       \end{split}
    \end{equation}
    Let $\zeta \in C_{c}([0,+\infty))$ be the Lipschitz map such that $\zeta=1$ on $[0,1]$, $\zeta(s)=2-s$ for $s\in [1,2]$, $\zeta=0$ on $[2,+\infty]$. Hence $\|\zeta^{\prime}\|_{L^{\infty}([0, +\infty))}\leq 1$. Next, for each sufficiently large $n \in \mathbb{N}$, define $\psi_{n}:U\to \mathcal{N}$ by
    \begin{equation}\label{defwnmin}
        \psi_{n}(y)\coloneqq \begin{cases}
            \varphi_{n}(y) &\text{if \(y \in \Phi(\partial U\times (0,s_{n}^{1/2}))\),}\\
            u_{*}|_{\partial U}\circ \Phi_{n}\circ \Phi^{-1}|_{\Phi(\partial U\times \{s_{n}^{1/2}\})}(y) &\text{if \(y \in \Phi(\partial U\times \{s_{n}^{1/2}\})\),}\\
            \displaystyle w_{n}(\Psi_{n}(y)) &\text{if \(y \in U\setminus \Phi(\partial U\times (0,s_{n}^{1/2}])\),}
        \end{cases}
    \end{equation}
    where $\Psi_{n}: U\setminus \Phi(\partial U\times (0,s_{n}^{1/2}]) \to U$ is the map defined by 
    \[
    \Psi_{n}(y)\coloneqq \biggl(1-\zeta\biggl(\frac{\dist(y,\partial U)}{s_{n}^{1/3}}\biggr)\biggr) y + \zeta\biggl(\frac{\dist(y,\partial U)}{s_{n}^{1/3}}\biggr)\Phi_{n}(\Phi^{-1}(y)).
    \]
    We note that $\psi_{n}\in W^{1,p_{n}}(U, \mathcal{N})$, $\tr_{\partial U}(\psi_{n})=u_{n}|_{\partial U}$, $\Psi_{n}$ is bilipschitz, $|D\Psi_{n} - \operatorname{Id} |=O(s_{n}^{1/6})$ a.e.\  in $U\setminus \Phi(\partial U\times (0,s^{1/2}_{n}])$ (for more details, see the proof of Proposition~\ref{prop conv to harmonic map}). Thus, $\psi_{n}$ is a competitor for $u_{n}$. Then, using \eqref{estiml2min}, \eqref{defwnmin}, \cite[Theorem~3.2.5]{Federer} and \eqref{ener comp by mass}, we obtain
    \begin{equation}
    \label{upp est by mass}
    \begin{split}
     \liminf_{n\to +\infty}(2-p_{n})\int_{U}\frac{|Du_{n}|^{p_{n}}}{p_{n}}\diff y &\leq \liminf_{n\to+\infty}(2-p_{n})\int_{U}\frac{|D\psi_{n}|^{p_{n}}}{p_{n}}\diff y \\ &\leq \limsup_{n\to +\infty}\biggl(Cs^{1/2}_{n}+(2-p_{n})(1+o(s^{1/10}_{n}))\int_{U}\frac{|Dw_{n}|^{p_{n}}}{p_{n}}\diff y\biggr)\\
     &\leq \mathbb{M}(S).       
    \end{split}
    \end{equation}
    On the other hand, using \eqref{two cond of weak*}, the fact that $U$ is open, Corollary~\ref{cor 5.17} and the geometry of $S_{*}\cap \overline{U}$ (see the beginning of the proof and \emph{Step~1}), we deduce that
    \begin{equation}\label{estmasssingset}
        \mu_{*}(S_{*}\cap U)=\mathbb{M}(S_{*}\cap \overline{U})\leq \liminf_{n\to +\infty}(2-p_{n})\int_{U}\frac{|Du_{n}|^{p_{n}}}{p_{n}}\diff y.
    \end{equation} 
    Combining \eqref{upp est by mass} with \eqref{estmasssingset}, yields \eqref{HomologicalPlateau}. This completes our proof of Proposition~\ref{min of sing}.
 \end{proof}
 
 We shall apply Proposition~\ref{min of sing} to prove that in the case when $\pi_{1}(\mathcal{N})\simeq \mathbb{Z}/2\mathbb{Z}$ the set $S_{*}$ has no branching points inside $\Omega$. To this end, we need the next extension result. The following statement is similar to \cite[Lemma~3]{Canevari_Orlandi_2021}.
 
\begin{prop}\label{propextoutsideplans}
Let \(\Omega \subset \mathbb{R}^{3}\) be a bounded strictly convex domain, \(\pi_{1} (\mathcal{N})\) be Abelian (and endowed with the law \(+\)) and \(a_1, \dotsc, a_k \subset \overline{\Omega}\) be distinct points.
Let \(g \in C(\partial \Omega \setminus \{a_1, \dotsc, a_k\}, \mathcal{N})\) and \(\alpha :  \{a_1, \dotsc, a_k\}^2 \to \pi_1 (\mathcal{N})\) satisfy the following conditions:
\begin{enumerate}[label=(\roman*)]
 \item 
 \label{it_ai1Soh0oNaejoith8faichee}
 \(\alpha (a_i, a_j) = -\alpha(a_j, a_i)\), so that, in particular, \(\alpha (a_i, a_i) = 0\);
 \item
 \label{it_chae9la2eej1Agh3oofae0th}
 if \(a_i \in \partial \Omega\), then \(\sum_{j = 1}^k \alpha (a_i, a_j)\) is the homotopy class of \(g\) around \(a_i\);
 \item 
 \label{it_kaiPhaiyah2pheech2ooz0in}
 if \(a_i \in \Omega\), then \(\sum_{j = 1}^k \alpha (a_i, a_j) = 0\).
\end{enumerate}
Then there exists \(u \in C (\Omega \setminus (\Sigma \cup F), \mathcal{N})\), where \(F\) is finite and \(\Sigma = \bigcup \{[a_i, a_j]: \alpha(a_i, a_j) \ne 0\}\),
and for every \(i, j \in \{1, \dotsc, k\}\), 
\(u\) is homotopic to \(\alpha(a_i, a_j)\) on small circles transversal to \([a_i, a_j]\).
\end{prop}

The proof of Proposition~\ref{propextoutsideplans} will use the following two-dimensional construction.

\begin{prop}\label{propextoutsideplans2d}
 Let \(U \subset \mathbb{R}^2\) be a bounded strictly convex domain, \(\pi_{1} (\mathcal{N})\) be Abelian and \(a_1, \dotsc, a_k \subset U\) be distinct points.
Assume that \(h \in C(\partial U, \mathcal{N})\) and \(\gamma \colon  \{a_1, \dotsc, a_k\} \to \pi_1 (\mathcal{N})\) satisfy the condition that 
\(\sum_{i = 1}^k \gamma (a_i)\) is the homotopy class of \(h\) on \(\partial U\).
Then there exists \(u \in C (\overline{U} \setminus \{a_1, \dotsc, a_k\},  \mathcal{N})\), such that \(u|_{\partial U}=h\) and \(u\) is homotopic to \(\gamma(a_i)\) on small circles transversal around \(a_i\).
\end{prop}
\begin{proof}
We consider a set \(L\) of segments in \(\mathbb{R}^2\) such that \(\{a_1, \dotsc, a_k\} \subset L\) and \(U \setminus L\) is simply connected. 
For \(r > 0\) small enough, the disks \(B^{2}_{2r} \brk{a_1}, \dotsc, B^{2}_{2 r} \brk{a_k}\) are pairwise disjoint. 
We fix \(u\) on \(\partial B^{2}_r \brk{a_i}\) so that \(u \vert_{\partial B^{2}_r \brk{a_i}}\) has the prescribed homotopy class \(\gamma \brk{a_i}\).
We extend then this map \(u\) to \(\brk{U \cap L} \setminus \bigcup_{i = 1}^k \smash{\overline{B}}^{2}_r \brk{a_i}\) by connectedness of \(\mathcal{N}\).
Since \(V = U \setminus \brk{L \cup \bigcup_{i = 1}^k \smash{\overline{B}}^{2}_r \brk{a_i}}\) is simply connected and by our assumption on the homotopy classes, \(u\) has a trivial homotopy on \(\partial V\). Thus, we can extend \(u\) to \(\overline{V}\) in such a way that $u=h$ on $\partial U$.
Finally, \(u\) is extended homogeneously on the set \(B^{2}_r \brk{a_i} \setminus \set{a_i}\) for each $i\in \{1, \dotsc, k\}$.
\end{proof}

\begin{proof}[Proof of Proposition~\ref{propextoutsideplans}]
As a first step, given a plane \(P\subset \mathbb{R}^3\) that is transversal to \(\Sigma\) and does not contain \(F\), the map \(g\) can be extended to \((\partial \Omega \cup (P \cap \Omega))\setminus \Sigma\).
Indeed, the homotopy class of \(g\) on \(\partial \Omega\) is equal to the sum of the homotopy classes on small circles around \(\{a_1, \dotsc, a_k\} \cap \partial \Omega\) on one side of \(P\). By the conditions \ref{it_chae9la2eej1Agh3oofae0th} and \ref{it_kaiPhaiyah2pheech2ooz0in},
the homotopy class of \(g\) on \(\partial \Omega\) is equal to the sum of the homotopy classes on small circles around \(\Sigma \cap P\).
By Proposition~\ref{propextoutsideplans2d}, \(g\) can be extended to \(P \cap \Omega \setminus \Sigma \) with the prescribed homotopy on small circles around points of  \(\Sigma \cap P\).
Reiterating this construction finitely many times, the problem is reduced to the case where either \(k = 0\) and one can perform a homogeneous extension with respect to any interior point, or the set \(\Sigma\) is starshaped with respect to some point in \(\overline{\Omega}\) and we perform then a homogeneous extension with respect to that point.
\end{proof}

  \begin{cor}\label{cor branching points}
     Let $\pi_{1}(\mathcal{N})=\mathbb{Z}/2\mathbb{Z}$. Then $S_{*}$ has no branching points inside $\Omega$.
 \end{cor}

 \begin{proof}[Proof of Corollary~\ref{cor branching points}]
     Assume by contradiction that $S_{*}$ has a branching point $x_{0} \in \Omega$. Then, according to Remark~\ref{rem sing points}, locally around $x_{0}$, $S_{*}$ is a finite union of an even number of segments emanating from $x_{0}$. Fix two segments from this union that form the smallest angle among all possible pairs. Denote these segments by $L_{1}$ and $L_{2}$. Then there exists a truncated cone $K \subset \Omega$ such that $\overline{K}\cap S_{*}=\overline{K}\cap (L_{1} \cup L_{2})$ and $\partial K\cap (L_{1}\cup L_{2})=\{x_{1}, x_{2}, y_{1}, y_{2}\}$, where $x_{1}$, $y_{1}$ and $x_{2}$, $y_{2}$ are the endpoints of $L_{1}\cap \overline{K}$, $L_{2} \cap \overline{K}$, respectively. Furthermore, $x_{1}, x_{2}$ belong to the smaller circular base of $K$ and $y_{1}$, $y_{2}$ belong to the bigger circular base of $K$. Choosing $K$ so that its smaller circular base is fairly close to $x_{0}$ and using the triangle inequality, we can ensure that $|x_{1}-x_{2}|+|y_{1}-y_{2}|<|x_{1}-y_{1}|+|x_{2}-y_{2}|$. 
     This, together with Propositions~\ref{min of sing},~\ref{propextoutsideplans}, yields a contradiction, since the set $U$ of Proposition~\ref{min of sing}, which contains $K$, can be chosen to be convex and arbitrarily close to $K$ (for instance, we can choose $U=\{x \in \mathbb{R}^{3}: \dist(x, K)<\varepsilon\}$ for some sufficiently small $\varepsilon \in (0,1)$). 
     Therefore, $S_{*}$ has no branching points inside $\Omega$, which completes our proof of Corollary~\ref{cor branching points}.
 \end{proof}
\section{Analysis up to the boundary}
\subsection{Preliminary results}
    If the fractional seminorm of the boundary datum is fairly small, we may have a concentration of the $p$-energy as $p \nearrow 2$ only near the boundary of the domain. Namely, the smaller the fractional seminorm of the boundary datum, the closer to $\partial \Omega$ the singular set is located.
\begin{prop}\label{prop empty S_*}
    Let $g \in W^{1/2,2}(\partial \Omega, \mathcal{N})$, $r_{0}\in (0,1)$, $\Omega_{r_{0}}:=\{x \in \Omega: \dist(x, \partial \Omega)>r_{0}\}\neq \emptyset$, $(p_{n})_{n\in \mathbb{N}}\subset [1,2)$, $p_{n}\nearrow 2$ as $n\to+\infty$ and $(u_{n})_{n\in \mathbb{N}}\subset W^{1,p_{n}}(\Omega, \mathcal{N})$ be a sequence of $p_{n}$-minimizers with $\tr_{\partial \Omega}(u_{n})=g$. Then there exists $\varepsilon=\varepsilon(r_{0},\partial \Omega, \mathcal{N})>0$ such that the following holds. Assume that  $|g|_{W^{1/2,2}(\partial \Omega, \mathbb{R}^{\nu})}\leq \varepsilon.$ Then $S_{*}\cap \Omega_{r_{0}}=\emptyset$. 
\end{prop}

\begin{proof}
    Fix arbitrary $x_{0} \in \Omega_{r_{0}}$ and $r\in (0,r_{0})$. Let $\eta>0$ be the constant of Lemma~\ref{heart}, where $\kappa=1/2$, $p_{0}=3/2$ and $\Psi(x)=x+x_{0}$ whenever $E=B^{3}_{r}(x_{0})$. Let $\varepsilon>0$ be small enough so that for each $n \in \mathbb{N}$, 
    \begin{equation}\label{estimglobalpenergy02}
    (2-p_{n})\int_{\Omega} \frac{|Du_{n}|^{p_{n}}}{p_{n}}\diff x \leq \eta r^{3/2}_{0},
    \end{equation}
    which is possible in view of Proposition~\ref{global energy bound} and the assumption $|g|_{W^{1/2,2}(\partial \Omega, \mathbb{R}^{\nu})}\leq \varepsilon$.
    Then, applying Lemma~\ref{lem mon of p-energy} and using \eqref{estimglobalpenergy02}, for each fairly small $\delta>0$ we have
    \begin{align*}
    (2-p_{n})\int_{B^{3}_{r+\delta}(x_{0})}\frac{|Du_{n}|^{p_{n}}}{p_{n}}\diff x &\leq (2-p_{n})\left(\frac{r+\delta}{r_{0}}\right)^{3-p_{n}}\int_{B^{3}_{r_{0}}(x_{0})}\frac{|Du_{n}|^{p_{n}}}{p_{n}}\diff x\\ &\leq (2-p_{n})\left(\frac{r+\delta}{r_{0}}\right)^{3-p_{n}}\int_{\Omega}\frac{|Du_{n}|^{p_{n}}}{p_{n}}\diff x \\ &\leq \eta r^{3/2-3+p_{n}}_{0}(r+\delta)^{3-p_{n}}.
    \end{align*} 
    Letting $n$ tend to $+\infty$, we have $\mu_{*}(\smash{\overline{B}}^3_{r}(x_{0}))\leq \mu_{*}(B^{3}_{r+\delta}(x_{0}))\leq \eta r_{0}^{1/2}(r+\delta)$ (see \eqref{two cond of weak*}). Letting $\delta$ tend to $0+$, yields $\mu_{*}(\smash{\overline{B}}^3_{r}(x_{0}))<\eta r$, since $r_{0}\in (0,1)$. Then, according to Lemma~\ref{lem concentration}, $\mu_{*}(B^{3}_{r/2}(x_{0}))=0$ and hence $B^{3}_{r/2}(x_{0})\subset \Omega\setminus S_{*}$. Since $x_{0} \in \Omega_{r_{0}}$ was arbitrary, $S_{*}\cap \Omega_{r_{0}}=\emptyset$, which completes our proof of Proposition~\ref{prop empty S_*}.
\end{proof}

The next proposition says that if the boundary of the domain is homeomorphic to $\mathbb{S}^{2}$ and the boundary datum is regular enough, then no singular set arises in the limit.
\begin{prop}\label{example 1}
    Let $\partial \Omega$ be homeomorphic to $\mathbb{S}^{2}$, $g \in W^{1,2}(\partial \Omega, \mathcal{N})$, $(p_{n})_{n\in \mathbb{N}}\subset [1,2)$, $p_{n}\nearrow 2$ as $n\to +\infty$ and $(u_{n})_{n\in \mathbb{N}}\subset W^{1,p_{n}}(\Omega, \mathcal{N})$ be a sequence of $p_{n}$-minimizers such that for each $n\in \mathbb{N}$, $\tr_{\partial \Omega}(u_{n})=g$. Then $S_{*}=\emptyset$ and the map $u_{*}$ coming from Proposition~\ref{conv to harm} is a minimizing harmonic extension of $g$ to $\Omega$. Furthermore, $u_{*} \in C^{\infty}(\Omega \setminus S_{0}, \mathcal{N})$, where $S_{0} \subset \Omega$ is a discrete set.
    \end{prop}
    \begin{rem}
    In Proposition~\ref{example 1}, the topology of $\Omega$ is simple, namely $\Omega \simeq B^{3}_{1}$. However, the singular set $S_{*}$ does not occur in the limit in some cases even if the topology of $\Omega$ is not simple, for example when \(g\) is the restriction of a smooth map on \(\Bar{\Omega}\).
    \end{rem}
\begin{proof}[Proof of Proposition~\ref{example 1}]
    Since $\Omega\subset \mathbb{R}^{3}$ is a bounded (locally) Lipschitz domain and $\partial \Omega$ is homeomorphic to $\mathbb{S}^{2}$, there exists a bilipschitz homeomorphism $\Phi: \overline{\Omega} \to \smash{\overline{B}}^{3}_{1}$.  Let $h=g\circ \Phi^{-1}|_{\mathbb{S}^{2}}$. Then $h \in W^{1,2}(\mathbb{S}^{2}, \mathcal{N})$.
    Next, defining, as it was done in \cite[Proposition~2.4]{SchoUhl}, $u(x)=h(x/|x|)$ for $x \in B^{3}_{1} \setminus \{0\}$, we have $u \in W^{1,2}(B^{3}_{1}, \mathcal{N})$ and $\tr_{\mathbb{S}^{2}}(u)=h$. Since $\Phi$ is bilipschitz and $u \circ \Phi$ is a competitor for $u_{n}$, we deduce that
    \begin{equation}\label{n4u39jj5tj5jtj54j58uj}
    \mu_{n}(\overline{\Omega})\leq C(2-p_{n})\int_{B^{3}_{1}}\frac{|Du|^{p_{n}}}{p_{n}}\diff x=\frac{C(2-p_{n})}{3-p_{n}}\int_{\mathbb{S}^{2}}\frac{|D_{\top}h|^{p_{n}}}{p_{n}}\diff \mathcal{H}^{2} \to 0
    \end{equation}
    as $n \to +\infty$, where $C>0$ is a constant depending only on the bilipschitz constant of $\Phi$, and we have used that $\mu_{n}(\Omega)=\mu_{n}(\overline{\Omega})$. Thus, $\mu_{*}(\overline{\Omega})=0$ (here we again take advantage of considering the $\mu_{n}$'s as elements of $(C(\overline{\Omega}))^{\prime}$ instead of $(C_{0}(\Omega))^{\prime}$, because any constant belongs to $C(\overline{\Omega})$, and hence the weak* convergence preserves mass, namely, since $\mu_{n} \overset{*}{\rightharpoonup} \mu_{*}$ weakly* in $(C(\overline{\Omega}))^{\prime}$, it holds $\mu_{n}(\overline{\Omega})\to \mu_{*}(\overline{\Omega})$). Therefore, the support of $\mu_{*}$ is the empty set, that is $S_{*}=\emptyset$. On the other hand, in view of \eqref{n4u39jj5tj5jtj54j58uj}, the condition \eqref{C1} of Proposition~\ref{prop conv to harmonic map} is valid for the sequence $(u_{n})_{n \in \mathbb{N}}$ and the set $K=\Omega$. Thus, the conclusions of Proposition~\ref{prop conv to harmonic map} apply for some mapping $u_{*} \in W^{1,2}(\Omega, \mathcal{N})$, which comes from Proposition~\ref{conv to harm}. In particular, $u_{*}$ is a minimizing harmonic extension of $g$ to $\Omega$.
Then, according to \cite[Theorem II]{SchoUhl}, $u_{*} \in C^{\infty}(\Omega \setminus S_{0}, \mathcal{N})$, where $S_{0}\subset \Omega$ is a discrete set, which completes our proof of Proposition~\ref{example 1}.
\end{proof}

As an immediate consequence of Lemma~\ref{key tool bilipschitz}, we have its counterpart on a half-ball. Recall that $\mathbb{R}^{3}_{+}=\{(x',x_{3}) \in \mathbb{R}^{3}: x_{3} \in (0, +\infty)\}$.

\begin{lemma}
\label{label tool bilipschitz boundary}
    Let $p_{0}\in (1,2)$, $p \in [p_{0},2)$.
    There exist constants \(\eta_0, C > 0\), depending only on \(p_0\) and \(\mathcal{N}\), such that 
    if $g \in W^{1,p}(\partial (B^3_r \cap \mathbb{R}^3_+), \mathcal{N})$ satisfies
    \begin{equation*}
        \int_{\partial (B^3_r \cap \mathbb{R}^3_+)}
            \frac
                {|D_{\top} g|^{p}}
                {p}
            \diff \mathcal{H}^{2}
        \leq 
            \frac{\eta_{0} r^{2-p}}{2-p},
    \end{equation*}
    then there exists $u \in W^{1,p}(B^3_r \cap \mathbb{R}^3_+, \mathcal{N})$ such that $\operatorname{tr}_{\partial (B^{3}_{r} \cap \mathbb{R}^{3}_{+})}(u) =g$ and
    \begin{equation*}
        \int_{B^3_r\cap \mathbb{R}^3_+}
            \frac{|Du|^{p}}{p}\diff x 
            \leq 
            Cr^{\frac{2}{p}}
            \biggl(\int_{\partial (B^{3}_{r} \cap \mathbb{R}^{3}_{+})}\frac{|D_{\top} g|^{p}}{p}
            \diff \mathcal{H}^{2}\biggr)^{1-\frac{1}{p}}.
    \end{equation*}
\end{lemma}

Next, we prove the $\eta$-compactness lemma up to the boundary of $\Omega$, which is the counterpart of Lemma~\ref{heart} on the boundary.

\begin{prop}\label{etacompboundary}
Let \(\kappa \in (0, 1)\), \(p_0 \in (1, 2)\), \(p\in [p_{0}, 2)\) and let \(U = \Psi (B^3_r \cap \mathbb{R}^3_+, \mathcal{N})\) be an open set, where \(\Psi : \smash{\overline{B}}^3_r \cap \smash{\overline{\mathbb{R}}}^3_+ \to \mathbb{R}^3\) is a bilipschitz homeomorphism onto its image $\overline{U}$ with a bilipschitz constant \(L \ge 1\).
There exist $\eta, C>0$, depending only on $\kappa, p_{0}$, $L$ and $\mathcal{N}$, such that if $u_{p} \in W^{1,p}(U, \mathcal{N})$ is a $p$-minimizer and if 
\begin{equation}
\label{etacompcondb}
    \int_{U}\frac{|Du_{p}|^{p}}{p}\diff x 
    + r \int_{\Psi(B^3_r \cap \mathbb{R}^2 \times \{0\})} \frac{|D_{\top} \operatorname{tr}_{\Psi(B^3_r \cap \mathbb{R}^2 \times \{0\})}(u_{p})|^{p}}{p} \diff \mathcal{H}^{2}
    \leq \frac{\eta r^{3-p}}{2-p},
\end{equation}
then
\begin{equation}\label{etaconsb}
    \int_{\Psi (B^3_{\kappa r} \cap \mathbb{R}^3_+) }\frac{|Du_{p}|^{p}}{p}\diff x \leq Cr^{3-p}.
\end{equation}
\end{prop}
\begin{proof}
Given the mapping \(v_p \coloneqq u_p \circ \Psi \in W^{1, p} (B^3_r \cap \mathbb{R}^3_+, \mathcal{N})\) satisfying
\begin{equation}
\label{eq_Chahd3voa1fiec0sheoSha2e}
\int_{B^3_r \cap \mathbb{R}^3_+} \frac{\lvert D v_p\rvert^p}{p} \diff x
\le L^{3+p} \int_{U} \frac{\lvert D u_p\rvert^p}{p} \diff x
\end{equation}
and the constant $\eta_{0}=\eta_{0}(p_{0}, \mathcal{N})>0$ of Lemma~\ref{label tool bilipschitz boundary}, we define the set 
\begin{multline}
\label{eq_aicaez9Aibuiwuxagahvaung}
    G_{p}
    =
        \biggl\{
            \varrho \in (\kappa r, r): 
            \operatorname{tr}_{\partial (B^3_\varrho \cap \mathbb{R}^3_+)}(v_p)=v_p|_{\partial (B^3_\varrho  \cap \mathbb{R}^3_+)} \in W^{1,p}(\partial (B^3_{\varrho}  \cap \mathbb{R}^3_+), \mathcal{N})\\
            \text{ and }\int_{\partial B^3_\varrho \cap \mathbb{R}^{3}_{+}}\frac{|D v_p|^{p}}{p}\diff \mathcal{H}^{2} \leq \frac{ \eta_0 (\kappa r)^{2-p}}{2(2-p)}\biggr\}.
\end{multline}
As a consequence of \eqref{etacompcondb}, \eqref{eq_Chahd3voa1fiec0sheoSha2e} and \eqref{eq_aicaez9Aibuiwuxagahvaung},
\begin{equation*}
     \mathcal{H}^1 ((\kappa r, r) \setminus G_p)
     \le \frac{2 (2 - p)}{\eta_0 (\kappa r)^{2- p}} \int_{B^{3}_{r}\cap \mathbb{R}^{3}_{+}} \frac{\vert D v_p\vert ^{p}}{p} \diff x
     \le \frac{2(2 - p)L^{3 + p}}{\eta_0 (\kappa r)^{2- p}} \int_{U} \frac{\vert D u_p\vert ^{p}}{p} \diff x
    \le \frac{(1-\kappa) r}{2},
\end{equation*}
provided we choose \(\eta\) satisfying
\begin{equation}
\label{eq_UH2koh2phiu9aitohTee2mie}
     \eta \le \frac{\eta_0 \kappa(1 - \kappa)}{4 L^{5}};
\end{equation}
equivalently we have then 
 \begin{equation}
        \mathcal{H}^1 (G_p) \ge \frac{(1 - \kappa)r}{2}.
  \end{equation}
Fix $\eta$ satisfying \eqref{eq_UH2koh2phiu9aitohTee2mie}. For every \(\varrho \in G_p\) the following holds
\begin{equation*}  
\begin{split}
    \int_{\partial (B^3_{\varrho} \cap \mathbb{R}^3_+)}
    \frac{|D_{\top}v_p|^{p}}{p}\diff \mathcal{H}^{2}
    &= \int_{\partial B^3_{\varrho} \cap \mathbb{R}^3_+}
    \frac{|D_{\top}v_p|^{p}}{p}\diff \mathcal{H}^{2}
    + \int_{B^3_{\varrho} \cap \mathbb{R}^{2}\times \{0\}}
    \frac{|D_{\top}v_p|^{p}}{p}\diff \mathcal{H}^{2}\\
    &\leq \frac{\eta_{0} (\kappa r)^{2-p}}{2(2-p)}
    + \frac{\eta r^{2 - p}L^{2+p}}{2 - p}
    \le \frac{\eta_{0} \varrho^{2-p}}{2-p},
\end{split}
\end{equation*}
where we have used \eqref{etacompcondb}, the fact that $|D_{\top}v_{p}|\leq |Dv_{p}|$ and \eqref{eq_UH2koh2phiu9aitohTee2mie}. Under this condition, according to Lemma~\ref{label tool bilipschitz boundary},    
\begin{equation}
     \int_{B^{3}_{\varrho} \cap \mathbb{R}^{3}_{+}} \frac{\lvert D v_p \rvert^p}{p} \diff x 
   \le C \varrho^{\frac{2}{p}}\biggl(\int_{\partial (B^{3}_{\varrho} \cap \mathbb{R}^{3}_{+})}\frac{|D_{\top}v_p|^{p}}{p}\diff \mathcal{H}^{2}\biggr)^{1-\frac{1}{p}}.
\end{equation}
To obtain \eqref{etaconsb}, we conclude as in the proof of Lemma~\ref{heart}.
\end{proof}

The next corollary is a direct consequence of Proposition~\ref{etacompboundary} and Definition~\ref{def Lip domain}, so we omit its proof.
\begin{cor}\label{cor etacompbound}
    Let $\kappa \in (0,1)$ and $g \in W^{1,2}(E, \mathcal{N})$, where $E$ is a relatively open subset of $\partial \Omega$. Let $r_{\partial \Omega}>0$ be the number coming from Definition~\ref{def Lip domain}, which is applied with $U=\Omega$. Then there exist constants $p_{1}=p_{1}(\partial \Omega, \mathcal{N}, \|D_{\top} g\|_{L^{2}(E, \mathbb{R}^{\nu}\otimes \mathbb{R}^{3})}) \in [3/2,2)$, $\eta=\eta(\kappa, \partial \Omega, \mathcal{N})>0$ and $C=C(\kappa, \partial \Omega,  \mathcal{N})>0$ such that the following holds. Let $p \in [p_{1},2)$ and $u_{p} \in W^{1,p}(\Omega, \mathcal{N})$ be a $p$-minimizer such that $\operatorname{tr}_{E}(u_{p})=g$. Then for each $r \in (0,r_{\partial \Omega})$ and $x_{0} \in E$ such that $B^{3}_{r}(x_{0})\cap \partial \Omega \subset E$ and 
    \begin{equation*}
        \int_{B^{3}_{r}(x_{0})\cap \Omega}\frac{|Du_{p}|^{p}}{p}\diff x \leq \frac{\eta r^{3-p}}{2-p},
    \end{equation*}
    it holds
    \begin{equation*}
        \int_{B^{3}_{\kappa r}(x_{0})\cap \Omega}\frac{|Du_{p}|^{p}}{p}\diff x \leq Cr^{3-p}.
    \end{equation*}
\end{cor}
\begin{cor}\label{cor conv uptotheboundary}
    Let $E$ be a relatively open subset of $\partial \Omega$. Assume that \eqref{C2} holds and that for all $n \in \mathbb{N}$ large enough, $\tr_{E}(u_{n})=g \in W^{1,2}(E, \mathcal{N})$. Then there exists a closed set $S_{*} \subset \overline{\Omega}$ and a mapping $u_{*} \in W^{1,2}_{\loc}((\Omega \cup E) \setminus S_{*}, \mathcal{N})$ satisfying the assertions of Proposition~\ref{prop conv to harmonic map} and the assertion \ref{it_Aex0Aijushoetha0looKo6ei} of Proposition~\ref{prop est for the differential inside domain} such that $\tr_{E}(u_{*})=g$ and, up to a subsequence (not relabeled), 
    \[
    u_{n} \rightharpoonup u_{*} \,\ \text{weakly in} \,\ W^{1,p}_{\loc}((\Omega \cup E) \setminus S_{*}, \mathbb{R}^{\nu}) \,\ \text{for all}\,\ p \in (1,2).
    \]
\end{cor}
\begin{proof}[Proof of Corollary~\ref{cor conv uptotheboundary}]
The proof of the assertion \ref{it_Aex0Aijushoetha0looKo6ei} of Proposition~\ref{prop est for the differential inside domain} remains unchanged. One only needs to take Corollary~\ref{cor etacompbound} into account in the proof of Proposition~\ref{prop conv to harmonic map}. 
\end{proof}

\subsection{Global estimates and \texorpdfstring{$W^{1,q}$}{W1q} compactness}
 
It is well known that the set $C^{\infty}(\partial \Omega, \mathcal{N})$ is dense in $W^{1/2,2}(\partial \Omega, \mathcal{N})$ if and only if $\pi_{1}(\mathcal{N}) \simeq \{0\}$ (we refer the reader to \cite[Theorem~4]{Brezis-Mironescu}). Although $C^{\infty}(\partial \Omega, \mathcal{N})$ is not dense in $W^{1/2,2}(\partial \Omega, \mathcal{N})$ if $\pi_{1}(\mathcal{N})$ is nontrivial, the set of smooth maps outside a finite  set of points whose Euclidean norm of the gradient behaves like the distance to the power of -1 close to these points is dense in $W^{1/2,2}(\partial \Omega, \mathcal{N})$ (see \cite[Theorem~1]{Mucci_2009} and \cite[Theorem~3]{Brezis-Mironescu}). Namely, letting  $\mathcal{R}^{1}_{0}(\partial \Omega, \mathcal{N})$ to be the class of maps $\varphi\in C^1(\partial \Omega\setminus \{a_{1},\dotsc, a_{l}\}, \mathcal{N})$ for some $a_{1},\dotsc, a_{l}\in \partial \Omega$ such that 
\begin{equation}\label{degen der}
  |D_{\top}\varphi(x)|\leq \frac{C}{\dist(x,\bigcup_{i=1}^{l}a_{i})},
\end{equation}
where $C$ is a positive constant independent of $x$, we have 
\[
\overline{\mathcal{R}^{1}_{0}(\partial \Omega, \mathcal{N})}^{W^{1/2,2}}=W^{1/2,2}(\partial \Omega, \mathcal{N}).
\] 
Hereinafter, for each $g \in \mathcal{R}^{1}_{0}(\partial \Omega, \mathcal{N})$, we shall denote by $S(g)$ the singular set of $g$, namely the set of points $a \in \partial \Omega$ such that for each sufficiently small radius $\varrho>0$, $g|_{\partial B^{\partial \Omega}_{\varrho}(a)}$ is not nullhomotopic, where $B^{\partial \Omega}_{\varrho}(a)$ is the geodesic ball in $\partial \Omega$ with center $a$ and radius $\varrho$ such that $B^{\partial \Omega}_{\varrho}(a)$ is far enough from $S(g)\backslash \{a\}$. 

It is worth noting that for each $g \in \mathcal{R}^{1}_{0}(\partial \Omega, \mathcal{N})$, there exists the \emph{minimal} set of points $\{a_{1}, \dotsc, a_{l}\} \subset \partial \Omega$ such that $g \in C^{1}(\partial \Omega \setminus \{a_{1}, \dotsc, a_{l}\}, \mathcal{N})$. Moreover, $S(g) \subset \{a_{1}, \dotsc, a_{l}\}$ and the inclusion can be strict. 
\begin{lemma}\label{lemma energ boundatum}
Let $g \in \mathcal{R}^{1}_{0}(\partial \Omega, \mathcal{N})$ and $\{a_{1}, \dotsc, a_{l}\} \subset \partial \Omega$ be the minimal set of different points such that $g \in C^{1}(\partial \Omega \setminus \{a_{1}, \dotsc, a_{l}\}, \mathcal{N})$. Then for each $q \in [1,2)$, $g \in W^{1,q}(\partial \Omega, \mathcal{N})$ and 
\begin{equation*}
\int_{\partial \Omega} |D_{\top} g|^{q} \diff \mathcal{H}^{2} \leq \frac{C}{2-q},
\end{equation*}
where $C>0$ is a constant depending only on $\partial \Omega$, the set $\{a_{1}, \dotsc, a_{l}\}$ and the constant of the estimate \eqref{degen der} applied with $\varphi=g$.
\end{lemma}
\begin{proof} 
   Fix an arbitrary $q \in [1,2)$. Define
    \[
    \bar{\varrho}_{\partial \Omega}(a_{1}, \dotsc, a_{l})=\min\left\{\frac{|a_{i}-a_{j}|}{2}: i, j \in \{1, \dotsc, l\} \,\ \text{and} \,\ i \not = j\right\}
    \]
    so that if $\varrho \in (0, \bar{\varrho}_{\partial \Omega}(a_{1}, \dotsc, a_{l}))$, then $\smash{\overline{B}}^{3}_{\varrho}(a_{i}) \cap \smash{\overline{B}}^{3}_{\varrho}(a_{j}) = \emptyset$ for each $i,j \in \{1, \dotsc, l\}$ such that  $i \not = j$. Then for each $i \in \{1, \dotsc, l\}$ and for each sufficiently small $\varrho \in  (0, \bar{\varrho}_{\partial \Omega}(a_{1}, \dotsc, a_{l}))$, 
    \begin{equation}\label{estimbiliphomeomdegenderpointai}
    \int_{B^{3}_{\varrho}(a_{i}) \cap \partial \Omega}|D_{\top}g|^{q}\diff \mathcal{H}^{2}\leq C \int_{0}^{\varrho}\frac{\diff t}{t^{1-q}} \leq \frac{C}{2-q},
    \end{equation}
    where $C>0$ is a constant depending only on $\partial \Omega$, $\bar{\varrho}_{\partial \Omega}(a_{1}, \dotsc, a_{l})$ and the constant of the estimate \eqref{degen der} applied with $\varphi=g$. On the other hand, in view of \eqref{degen der}, if $\varrho>0$ and $x \in \partial \Omega \setminus \bigcup_{i=1}^{l} \smash{\overline{B}}^{3}_{\varrho}(a_{i})$, then  $|D_{\top} g(x)| \le C/\varrho$. Taking this and \eqref{estimbiliphomeomdegenderpointai} into account, we complete our proof of Lemma~\ref{lemma energ boundatum}.
\end{proof}

Assuming that $g \in \mathcal{R}^{1}_{0}(\partial \Omega, \mathcal{N})$, we obtain the next improvement of Proposition~\ref{prop est for the differential inside domain}.
\begin{prop}\label{prop quant behuptotheb}
Let $g \in \mathcal{R}^{1}_{0}(\partial \Omega, \mathcal{N})$ be such that $g \in C^{1}(\partial \Omega \setminus S(g), \mathcal{N})$. Let $(p_{n})_{n\in \mathbb{N}} \subset [1,2)$, $p_{n} \nearrow 2$ as $n\to +\infty$ and $(u_{n})_{n\in \mathbb{N}}$ be a sequence of $p_{n}$-minimizers such that $\tr_{\partial \Omega}(u_{n}) = g$ for each $n \in \mathbb{N}$.  Let $u_{*}$ be a map given by Proposition~\ref{conv to harm}. Assume that there exists $r_{0}>0$ such that if $x \in \mathbb{R}^{3} \setminus \Omega$ and $\dist(x, \partial \Omega)\leq r_{0}$, there exists a unique point $P(x) \in \partial \Omega$ such that $\dist(x, \partial \Omega)=|x-P(x)|$.  
Then the following assertions hold.
\begin{enumerate}[label=(\roman*)]
\item 
\label{it_Eev4ahtip9aleetho0Avoo6i}
There exists $C=C(g, \Omega, \mathcal{N})>0$ such that for each $q \in [1,2)$, 
\begin{equation*}
\limsup_{n \to +\infty}\int_{\Omega}|Du_{n}|^{q}\diff x \leq C+\frac{C}{2 - q}\limsup_{n\to+\infty}(2-p_{n})\int_{\Omega} \frac{|Du_{n}|^{p_{n}}}{p_{n}}\diff x. 
\end{equation*}
\item 
\label{it_ZoV0keeweev6aiboal0dudie}
For each $q \in [1,2)$,  $u_{*} \in W^{1,q}(\Omega, \mathcal{N})$ and, up to a subsequence (not relabeled), $u_{n} \to u_{*}$  in $W^{1,q}(\Omega, \mathbb{R}^{\nu})$ and for some $C=C(g, \Omega, \mathcal{N})>0$,
		\begin{equation*}
		\int_{\Omega}|Du_{*}|^{q} \diff x \leq C+ \frac{C}{2-q} \limsup_{n \to +\infty} (2-p_{n}) \int_{\Omega} \frac{|Du_{n}|^{p_{n}}}{p_{n}} \diff x.
		\end{equation*}
\end{enumerate}
\end{prop}
\begin{rem}
If $\Omega$ is of class $C^{2}$ or is convex, then it has the \emph{unique nearest point property}, namely, the assumption of Proposition~\ref{prop quant behuptotheb} applies (for more details, we refer to \cite{Federer_1959}, \cite{Krantz_Parks_1981}). It is also worth noting that the constants in the assertions \ref{it_Eev4ahtip9aleetho0Avoo6i}, \ref{it_ZoV0keeweev6aiboal0dudie} depend on $S(g)$ and the constant coming from \eqref{degen der}, where $\varphi=g$. Since the latter constant depends only on $g$, we can assume that the constants in \ref{it_Eev4ahtip9aleetho0Avoo6i}, \ref{it_ZoV0keeweev6aiboal0dudie} depend only on $g$, $\Omega$ and $\mathcal{N}$.
\end{rem}
\begin{proof}[Proof of Proposition~\ref{prop quant behuptotheb}]
Fix $R>0$ large enough so that $\Omega \Subset B^{3}_{R}$ and for each $x \in \Omega$ the distance from $x$ to $\partial B^{3}_{R}$ is bounded from below by $1$. Without loss of generality, we assume that $r_{0}$ is small enough with respect to $R$ and that $P$ is a 2-Lipschitz map (the reader may consult \cite[Theorem~4.8~(8)]{Federer_1959}). To lighten the notation, we denote $W=\{x \in B^{3}_{R}: \dist(x, \Omega) \leq r_{0}\}$ and $V=\{x \in B^{3}_{R}\setminus \Omega: \dist(x, \partial \Omega) \leq r_{0}\}$. For each $n \in \mathbb{N}$, define the map $f_{n}: B^{3}_{R}\to [0,+\infty)$ by $f_{n}=|Du_{n}|$ in $\Omega$, $f_{n}=|D(g\circ P)|$ in $V$ and $f_{n}=0$ in $B^{3}_{R} \backslash W$. Since $P$ is 2-Lipschitz and $g \in \mathcal{R}^{1}_{0}(\partial \Omega, \mathcal{N})$, using Lemma~\ref{lemma energ boundatum}, we observe that 
\begin{equation}\label{3j4905ji054jgun45iun}
(2-p_{n})\int_{B^{3}_{R}\setminus \Omega}f_{n}^{p_{n}}\diff x= (2-p_{n})\int_{V}f_{n}^{p_{n}}\diff x \leq C,
\end{equation}
where $C=C(g, \Omega, \mathcal{N})>0$ (here $C$ depends on $r_{0}$, $S(g)$ and the constant coming from \eqref{degen der}, where $\varphi=g$, but $r_{0}$ depends only on $ \Omega$ and the latter constant depends only on $g$). Then the proof of the assertion \ref{it_Eev4ahtip9aleetho0Avoo6i} follows by reproducing the proof of the assertion \ref{it_so3EPh9zoev2Gae0OoQuieg9} of Proposition~\ref{prop est for the differential inside domain} with minor modifications, namely, applying Proposition~\ref{prop goodest dyadic} with $f=f_{n}$, $p=p_{n}$, $V=B^{3}_{R}$, $C_{1}=\frac{2-p_{n}}{\eta}$ and $C_{2}=C$, where $n \in \mathbb{N}$ is large enough and $\eta, C>0$ depend only on $g, \Omega$, $\mathcal{N}$ and will be fixed later for the proof to work. Let us fix an arbitrary dyadic cube $Q$ such that $2Q\subset B^{3}_{R}$. Let $r>0$ be the sidelength of $Q$. To apply Proposition~\ref{prop goodest dyadic} with the above parameters, we need to prove the following: if $n \in \mathbb{N}$ is large enough and 
\begin{equation}\label{mf3ij0ij4i5gn5}
\int_{2Q}f_{n}^{p_{n}}\diff x \leq \frac{\eta r^{3-p_{n}}}{2-p_{n}},
\end{equation}
then 
\begin{equation}\label{mi3m5m5i4io490jg0ij5gn}
\int_{Q}f^{p_{n}}_{n}\diff x \leq Cr^{3-p_{n}}.
\end{equation}
First, notice the following. Fix an arbitrary $x_{0} \in S(g)$. Let $B^{\partial \Omega}_{\varrho}(x_{0})$ be the geodesic ball in $\partial \Omega$ with center $x_{0}$ and radius $\varrho>0$ such that every point in $S(g)\backslash \{x_{0}\}$ is far enough from $B^{\partial \Omega}_{\varrho}(x_{0})$. Then, using Proposition~\ref{Sandier lemma} and the facts that $p_{n} \in [3/2, 2)$ for each $n\in \mathbb{N}$ large enough and $B^{\partial \Omega}_{\varrho}(x_{0})$ is bilipschitz homeomorphic to $B^{2}_{\varrho}(x_{0})$, one has
\begin{equation*}
\int_{B^{\partial \Omega}_{\varrho}(x_{0})}|D_{\top} g|^{p_{n}}\diff \mathcal{H}^{2}+ \varrho \int_{\partial B^{\partial \Omega}_{\varrho}(x_{0})}|D_{\top} (g|_{\partial B^{\partial \Omega}_{\varrho}(x_{0})})|^{p_{n}}\diff \mathcal{H}^{1} \geq \frac{\mathcal{E}^{\textup{sg}}_{p_{n}/(p_{n}-1)}(g|_{\partial B^{\partial \Omega}_{\varrho}(x_{0})}) \varrho^{2-p_{n}}}{C_{0}(2-p_{n})},
\end{equation*}
where $\partial B^{\partial \Omega}_{\varrho}(x_{0})$ denotes the relative boundary of $B^{\partial \Omega}_{\varrho}(x_{0})$,  $C_{0}=C_{0}(\Omega)>0$ and we have used that $\tr_{\partial B^{\partial \Omega}_{\varrho}(x_{0})}(g)=g|_{\partial B^{\partial \Omega}_{\varrho}(x_{0})}$. Then, proceeding as in the proof of Lemma~\ref{lem lowenergy}, we deduce that 
\begin{equation}\label{i3jij5u8tu485utuhgj85u8u8u887y7hu5j95}
\int_{B^{\partial \Omega}_{\varrho}(x_{0})}|D_{\top} g|^{p_{n}}\diff \mathcal{H}^{2}+ \varrho \int_{\partial B^{\partial \Omega}_{\varrho}(x_{0})}|D_{\top} (g|_{\partial B^{\partial \Omega}_{\varrho}(x_{0})})|^{p_{n}}\diff \mathcal{H}^{1} \geq \frac{2\varepsilon_{0} \varrho^{2-p_{n}}}{2-p_{n}},
\end{equation}
for each $n\in \mathbb{N}$ large enough, where $\varepsilon_{0}=\varepsilon_{0}(C_{0}, \mathcal{N})>0$. Since $C_{0}$ depends only on $ \Omega$, we can assume that $\varepsilon_{0}$ depends only on $\Omega$ and $\mathcal{N}$. Also, since $g \in \mathcal{R}^{1}_{0}(\partial \Omega, \mathcal{N})$, we have the estimate
\[
\varrho \int_{\partial B^{\partial \Omega}_{\varrho}(x_{0})}|D_{\top} (g|_{\partial B^{\partial \Omega}_{\varrho}(x_{0})})|^{p_{n}}\diff \mathcal{H}^{1} \leq C^{\prime} \varrho \mathcal{H}^{1}(\partial B^{\partial \Omega}_{\varrho}(x_{0})),
\]
where $C^{\prime}>0$ is a constant depending only on the constant coming from \eqref{degen der}, which depends only on $g$. This, since $p_{n} \nearrow 2$ as $n \to +\infty$, implies that for each $n\in \mathbb{N}$ large enough (depending possibly only on $g$, $\Omega$ and $\mathcal{N}$),  
\begin{equation}\label{ij4i039u598ug9854h}
\varrho \int_{\partial B^{\partial \Omega}_{\varrho}(x_{0})}|D_{\top} (g|_{\partial B^{\partial \Omega}_{\varrho}(x_{0})})|^{p_{n}}\diff \mathcal{H}^{1} \leq \frac{\varepsilon_{0} \varrho^{2-p_{n}}}{2-p_{n}}.
\end{equation}
Combining \eqref{i3jij5u8tu485utuhgj85u8u8u887y7hu5j95} and \eqref{ij4i039u598ug9854h}, we get the next estimate
\begin{equation}\label{7y7y789u48ur7y3y743y498}
\int_{B^{\partial \Omega}_{\varrho}(x_{0})}|D_{\top} g|^{p_{n}}\diff \mathcal{H}^{2} \geq \frac{\varepsilon_{0} \varrho^{2-p_{n}}}{2-p_{n}}
\end{equation}
for each $n\in \mathbb{N}$ large enough. On the other hand, if $E\subset \partial \Omega$ is Borel and $\dist(E, S(g))\geq \varrho>0$, then, using \eqref{degen der}, we obtain
\begin{equation}\label{niu39j4j39j4j4093j90j9}
\int_{E}|D_{\top} g|^{p_{n}}\diff \mathcal{H}^{2}\leq \frac{C^{\prime} \mathcal{H}^{2}(E)}{\varrho^{p_{n}}}.
\end{equation} 
Next, observe that, since  $\Omega$ is a bounded Lipschitz domain having the unique nearest point property in $V$, there exist fairly small constants $0<a<\delta<1/2$ and a number $N_{0} \in \mathbb{N}\setminus \{0\}$, depending only on $\Omega$, such that the following holds. We need to distinguish between two further cases.

\medskip 

\noindent \emph{Case~1:} $Q \subset W$.  If $\dist(P(Q \cap V), S(g))\geq \delta r$, then we proceed as follows. 
Applying \eqref{niu39j4j39j4j4093j90j9} with $E=P(Q \cap V)$, we have
\begin{equation}\label{nf93j5893j58tj594hth8h}
\int_{Q \cap V} f^{p_{n}}_{n} \diff x \leq L r \int_{P(Q \cap V)}|D_{\top}g|^{p_{n}}\diff \mathcal{H}^{2} \leq C^{\prime \prime}r^{3-p_{n}},
\end{equation}
where $L=L(P)>0$ and $C^{\prime \prime}=C^{\prime \prime}(g, \Omega, \mathcal{N})>0$. Next, there exists at most $N_{0}$ balls $B^{3}_{ar}(x_{1}), \dotsc, B^{3}_{ar}(x_{k})$ such that $Q \cap \Omega \subset \bigcup_{i=1}^{k}B^{3}_{ar}(x_{i})\subset 2Q$ and for each $ i\in \{1,\dotsc, k\}$, either $B^{3}_{9ar/8}(x_{i}) \subset \Omega$ or $x_{i} \in \partial \Omega$ and $\dist(B^{3}_{9ar/8}(x_{i}), S(g))\geq \delta r/2$. We shall apply the corresponding $\eta$-compactness argument for each ball $B^{3}_{9ar/8}(x_{i})$, namely Lemma~\ref{heart} in the case when $B^{3}_{9ar/8}(x_{i}) \subset \Omega$ and Proposition~\ref{etacompboundary} otherwise. Let  $\eta, C^{\prime \prime \prime}>0$  be constants that are valid for both Lemma~\ref{heart} and Proposition~\ref{etacompboundary}, where $\kappa=8/9$, $p_{0}=3/2$ and the corresponding bilipschitz homeomorphism whose bilipschitz constant can depend only on $\Omega$. Thus, we can assume that $\eta, C^{\prime \prime \prime}$ depend only on $\Omega$ and $\mathcal{N}$. Assume that
\begin{equation}\label{mi9j8j8934hj8h3h43hr94j98juhgthg784h}
\int_{2Q}f^{p_{n}}_{n}\diff x \leq \frac{\eta r^{3-p_{n}}}{2(2-p_{n})}.
\end{equation}
In the case when $B^{3}_{9ar/8}(x_{i})\subset \Omega$, by Lemma~\ref{heart} and \eqref{mi9j8j8934hj8h3h43hr94j98juhgthg784h}, we immediately obtain the estimate
\begin{equation}
\label{hearti4j05i0ji045j50j0i}
\int_{B^{3}_{ar}(x_{i})}|Du_{n}|^{p_{n}}\diff x \leq C^{\prime \prime \prime}r^{3-p_{n}}.
\end{equation}
Let now $x_{i} \in \partial \Omega$. Using \eqref{nf93j5893j58tj594hth8h} and the fact that $p_{n} \nearrow 2$, we observe that
\begin{equation}\label{mjen439j948j9j43j93j49j}
\frac{9ar}{8} \int_{B^{3}_{\frac{9ar}{8}}(x_{i})\cap \partial \Omega}|D_{\top} g|^{p_{n}}\diff \mathcal{H}^{2} \leq \frac{\eta r^{3-p_{n}}}{2(2-p_{n})}
\end{equation}
for each $n\in \mathbb{N}$ large enough depending only on $g, \Omega$ and $\mathcal{N}$. Thus, in view of \eqref{mi9j8j8934hj8h3h43hr94j98juhgthg784h} and \eqref{mjen439j948j9j43j93j49j}, 
\[
\int_{B^{3}_{\frac{9ar}{8}}(x_{i}) \cap \Omega}|Du_{n}|^{p_{n}}\diff x + \frac{9ar}{8} \int_{B^{3}_{\frac{9ar}{8}}(x_{i}) \cap \partial \Omega}|D_{\top} g|^{p_{n}}\diff \mathcal{H}^{2} \leq \frac{\eta r^{3-p_{n}}}{2-p_{n}},
\]
for each $n \in \mathbb{N}$ large enough, which according to Proposition~\ref{etacompboundary}, implies that
\begin{equation}\label{mcoim43jjh95j8945hj}
\int_{B^{3}_{ar}(x_{i})}|Du_{n}|^{p_{n}}\diff x \leq C^{\prime \prime \prime}r^{3-p_{n}}.
\end{equation}
Combining \eqref{hearti4j05i0ji045j50j0i} and \eqref{mcoim43jjh95j8945hj}, we get 
\begin{equation}
\label{estimatecoveringballsheart}
\int_{Q \cap \Omega} f^{p_{n}}_{n} \diff x \leq N_{0} C^{\prime\prime \prime} r^{3-p_{n}}
\end{equation}
for each $n\in \mathbb{N}$ large enough. Taking into account the estimates \eqref{nf93j5893j58tj594hth8h} and \eqref{estimatecoveringballsheart}, we observe that \eqref{mf3ij0ij4i5gn5} implies \eqref{mi3m5m5i4io490jg0ij5gn}, where $C=C^{\prime \prime}+ N_{0} C^{\prime \prime \prime}$, which can depend only on $g$, $\Omega$ and $\mathcal{N}$.

If $\dist(P(Q \cap V), S(g))< \delta r$, then there exists $x_{0} \in S(g)$ such that $2Q\cap V$ contains a cylinder $U \simeq B^{\partial \Omega}_{ar}(x_{0}) \times (0,ar)$ such that, in view of \eqref{7y7y789u48ur7y3y743y498} and the definition of $f_{n}$ in $V$, it holds
\begin{equation}\label{mi48j09i85u89u595u90i-0-4i9ri439}
\int_{2Q}f^{p_{n}}_{n} \diff x \geq \int_{U}f^{p_{n}}_{n} \diff x \geq \frac{\varepsilon_{0} (a^{10}r)^{3-p_{n}}}{2-p_{n}}
\end{equation}
for each $n \in \mathbb{N}$ large enough. Thus, taking $\eta$ small enough with respect to $\varepsilon_{0}a^{15}$ so that Proposition~\ref{etacompboundary} applies with $2\eta$, under the condition that $\int_{2Q} f^{p_{n}}_{n}\diff x \leq \frac{\eta r^{3-p_{n}}}{2-p_{n}}$, we exclude the situation where $\dist(P(Q\cap V), S(g))<\delta r$. 

\medskip

\noindent \emph{Case~2:} $Q \cap (B^{3}_{R}\setminus W) \not = \emptyset$. If $Q \subset B^{3}_{R}\setminus W$, then, clearly, \eqref{mi3m5m5i4io490jg0ij5gn} holds, since $f_{n}=0$ on $Q$ by definition. Otherwise, we repeat the strategy of the \emph{Case 1} for $Q \cap W$ to conclude that for each $n \in \mathbb{N}$ large enough, \eqref{mf3ij0ij4i5gn5} implies \eqref{mi3m5m5i4io490jg0ij5gn} for some $\eta, C>0$ depending only on $g$, $\Omega$ and $\mathcal{N}$.

After all, we have showed that there exist constants $\eta, C>0$ depending only on $g, \Omega$ and $\mathcal{N}$ such that \eqref{mf3ij0ij4i5gn5} implies \eqref{mi3m5m5i4io490jg0ij5gn} for each $n \in \mathbb{N}$ large enough depending on $g, \Omega$ and $\mathcal{N}$. Therefore, we can apply Proposition~\ref{prop goodest dyadic} with $f=f_{n}$, $V=B^{3}_{R}$, $C_{1}=\frac{2-p_{n}}{\eta}$ and $C_{2}=C$ for each $n \in \mathbb{N}$ large enough. Taking into account \eqref{3j4905ji054jgun45iun} and the fact that $f_{n}=0$ in $B^{3}_{R} \backslash W$, we obtain \ref{it_Eev4ahtip9aleetho0Avoo6i}. 

The assertion \ref{it_ZoV0keeweev6aiboal0dudie} then comes from a diagonal argument, proceeding by the same way as in the proof of Proposition~\ref{prop est for the differential inside domain}~\ref{it_Aex0Aijushoetha0looKo6ei}.  This completes our proof of Proposition~\ref{prop quant behuptotheb}.
\end{proof}

\subsection{Proof of Theorem~\ref{Interiorbehavior1th}}
Now we prove our first main theorem.
\begin{proof}[Proof of Theorem~\ref{Interiorbehavior1th}]
 The proof follows from Propositions~\ref{conv to harm},~\ref{prop 1-var},~\ref{prop structure S_*},~\ref{prop est for the differential inside domain},~\ref{min of sing},~\ref{example 1},~\ref{prop quant behuptotheb} and Corollary~\ref{cor 5.17}.
\end{proof}

\subsection{Boundary repulsion property}\label{section_Wai9niecoo5AhkaciulufeeM}
In this subsection, we derive the Pohozaev-type identity for $p$-minimizers and prove that the varifold associated to the limit of the stress-energy tensors of $p$-minimizers as $p\nearrow 2$ is \textit{minimizing area to first order} in the Brian White sense (see \cite{White_2010}). As a consequence, we prove that if $\overline{\Omega}$ is strongly convex at every point of $\partial \Omega$, the boundary datum $g$ is an element of $\mathcal{R}^{1}_{0}(\partial \Omega, \mathcal{N})$ and $g \in C^{1}(\partial \Omega \setminus S(g), \mathcal{N})$, then $S_{*}\cap \partial \Omega=S(g)$, where $S_{*}$ is the energy concentration set supporting the varifold and  $S(g)$ is the singular set of $g$. 

We begin by establishing the boundary stress-energy tensor identity. For convenience, if $U \subset \mathbb{R}^{3}$ is open, we define for $u \in W^{1,p} (U, \mathbb{R}^{\nu})$ the map 
$T^{p}_{u} : U \to \mathbb{R}^{3}\otimes \mathbb{R}^{3}$ by
\begin{equation}\label{eq_cai5Chohqu5eyuciezeonge4}
    T^{p}_{u}=\frac{|Du|^{p}}{p}\mathrm{Id}-\frac{Du\otimes Du}{|Du|^{2-p}},
\end{equation}
where $Du\otimes Du=(Du)^{\mathrm{T}}Du$. If $U$ is a bounded Lipschitz domain and $u$ is a minimizing \(p\)-harmonic map in $U$, we have $\mathrm{div} (T^p_{u})=0$ in $\mathscr{D}^{\prime}(U, \mathbb{R}^{3})$ (see \eqref{integralidentity}).
\begin{prop}\label{prop int by parts tensor}
    Let $\Omega \subset \mathbb{R}^{3}$ be a bounded domain of class $C^2$, $x_{0} \in \partial \Omega$, $r>0$ be fairly small and $p \in [1,+\infty)$. If $u \in W^{1, p}(\Omega\cap B^{3}_{r}(x_{0}), \mathbb{R}^{\nu})$ satisfies $\mathrm{div}(T^{p}_{u})=0$ in $\mathscr{D}^{\prime}(\Omega\cap B^{3}_{r}(x_{0}), \mathbb{R}^{3})$ and if \(Du\) is continuous at every point of \(\partial \Omega \cap B^{3}_{r}(x_{0})\), then for each $\xi \in C^{1}_{c}(B^{3}_{r}(x_{0}), \mathbb{R}^{3})$,
    \begin{equation}\label{int by parts}
        \int_{\Omega \cap B^{3}_{r}(x_{0})}T^{p}_{u}:D\xi \diff x = \int_{\partial \Omega \cap B^{3}_{r}(x_{0})}\langle \nu, T^{p}_{u}[\xi]\rangle \diff \mathcal{H}^{2},
    \end{equation}
    where $\nu \in C^{1}(\partial \Omega, \mathbb{S}^{2})$ stands for the outward pointing unit normal vector field to $\partial \Omega$.
\end{prop}
    \begin{proof} To lighten the notation, denote $B=B^{3}_{r}(x_{0})$. Observe that $\partial \Omega$ is a $2$-dimensional compact Riemannian manifold of class $C^{2}$ without boundary. Hence there exists $r_{0}=r_{0}(\partial \Omega)>0$ such that the function $f(x)=\dist(x, \partial \Omega)$, $x \in \Omega$ is of class $C^{1}$ in $V=\{x \in \Omega: \dist(x, \partial \Omega)< r_{0}\}$ (see \cite[Theorem~4.12]{Federer_1959}) and, furthermore, $|Df(x)|=1$ for $x \in V$ (see \cite[4.8~(5), 4.8~(3)]{Federer_1959}).  Let us fix a sufficiently small $\varepsilon>0$ and define
    \begin{equation*}
        g_{\varepsilon}(t)=
        \begin{cases}
            1 & \text{if}\,\ \,\  t \in [\varepsilon, +\infty)\\
            \frac{t}{\varepsilon} & \text{if}\,\ \,\ t \in [0,\varepsilon].
        \end{cases}
    \end{equation*}
    Since $g_{\varepsilon}$ is a Lipschitz function on $[0, +\infty)$, for each $\xi \in W^{1,\infty}_{0}(B,\mathbb{R}^{3})$, it is clear that the function $\varphi_{\varepsilon}(x)\coloneqq g_{\varepsilon}(f(x))\xi(x)$ is an element of $W^{1,\infty}_{0}(\Omega\cap B,\mathbb{R}^{3})$. Since $\mathrm{div}(T^{p}_{u})=0$ in $\mathscr{D}^{\prime}(\Omega\cap B,\mathbb{R}^{3})$, $u \in W^{1, p}(\Omega\cap B, \mathbb{R}^{\nu})$ and \(Du\) is continuous at every point of \(\partial \Omega \cap B\), by approximation, we have $\int_{\Omega \cap B}T^{p}_{u}:D\xi \diff x =0$ for each $\xi \in W^{1,\infty}_{0}(\Omega\cap B,\mathbb{R}^{3})$. Thus, using the function $\varphi_{\varepsilon}$ as a test function in the weak formulation of $\mathrm{div}(T^{p}_{u})=0$, we get
    \begin{align*}
        0=&\int_{\Omega \cap B} T^{p}_{u}: g_{\varepsilon}(f)D\xi \diff x +  \int_{\Omega \cap B} \langle D(g_{\varepsilon}(f)),T^{p}_{u}[\xi]\rangle \diff x \\ =& \int_{\Omega \cap B} T^{p}_{u}: g_{\varepsilon}(f)D\xi \diff x +  \int_{\Omega \cap B} \langle g^{\prime}_{\varepsilon}(f) D f, T^{p}_{u}[\xi]\rangle \diff x.
    \end{align*}        
    Letting $\varepsilon \searrow 0$  and using the Lebesgue dominated convergence theorem, we have
    \begin{equation} \label{A.3}
        \begin{split}
            \int_{\Omega \cap B} T^{p}_{u} : g_{\varepsilon}(f)D\xi \diff x & \to \int_{\Omega \cap B} T^{p}_{u}:D\xi\diff x.
        \end{split}
    \end{equation}
    On the other hand, using the coarea formula (see \cite[Theorem~3.2.22~(3)]{Federer}), the fact that $|Df|=1$ in $V$ and changing the variables, we get
    \begin{equation}
     \label{A.4}
     \begin{split}
        \int_{\Omega \cap B} \langle g^{\prime}_{\varepsilon}(f)D f, T^{p}_{u} [\xi] \rangle\diff x
        & = \frac{1}{\varepsilon} \int^{\varepsilon}_{0} \diff t \int_{\{f=t\} \cap B} \langle D f(\sigma), T^{p}_{u}[\xi](\sigma)\rangle\diff \mathcal{H}^{2}(\sigma) \\
        &=\frac{1}{\varepsilon} \int^{\varepsilon}_{0} \diff t \int_{\partial \Omega \cap B} \langle D f(y-t\nu(y)), T^{p}_{u}[\xi](y)\rangle \diff \mathcal{H}^{2}(y) + o(1)_{\varepsilon \searrow 0}\\
        & \to -\int_{\partial \Omega \cap B} \langle \nu, T^{p}_{u}[\xi]\rangle \diff \mathcal{H}^{2}, 
        \end{split}
    \end{equation}
    as $\varepsilon \searrow 0$. 
    In fact, since $u \in W^{1, p}(\Omega\cap B, \mathbb{R}^{\nu})$, \(Du\) is continuous at every point of \(\partial \Omega \cap B\), $\xi \in C^{1}_{c}(B, \mathbb{R}^{3})$ and $D f(y-t\nu(y))\to -\nu(y)$ as $t\searrow 0$ for each $y \in \partial \Omega$, the function
    \[
    r\in (0,a] \mapsto \Psi(r)\coloneqq  \int^{a}_{0} \diff t \int_{\partial \Omega \cap B} \langle Df(y-t\nu(y)), T^{p}_{u}[\xi](y)\rangle \diff \mathcal{H}^{2}(y),
    \]
    where $a>0$ is fairly small, has the right derivative at $0$ equal to $\Psi^{\prime}(0+)=-\int_{\partial \Omega \cap B} \langle \nu,  T^{p}_{u}[\xi]\rangle \diff \mathcal{H}^{2}$. Combining \eqref{A.3} and \eqref{A.4}, we deduce the desired stress-energy identity, which completes our proof of Proposition~\ref{prop int by parts tensor}.
\end{proof}
Now we deduce the normal boundary estimate.
\begin{prop}\label{prop bound normal}
    Let $\Omega \subset \mathbb{R}^{3}$ be a bounded domain of class $C^{2}$, $x_{0} \in \partial \Omega$, $r>0$ be fairly small and $p \in (1,2]$. Let $K\csubset \partial \Omega \cap B^{3}_{r}(x_{0})$ be relatively open. If $u \in W^{1, p}(\Omega\cap B^{3}_{r}(x_{0}), \mathbb{R}^{\nu})$ satisfies $\mathrm{div}(T^{p}_{u})=0$ in $\mathscr{D}^{\prime}(\Omega\cap B^{3}_{r}(x_{0}), \mathbb{R}^{3})$ and if \(Du\) is continuous at every point of \(\partial \Omega \cap B^{3}_{r}(x_{0})\), then 
    \begin{equation}
    \int_{K}|Du|^{p}\diff \mathcal{H}^{2}\leq \frac{C}{p-1}\left(\int_{\Omega\cap B^{3}_{r}(x_{0})}\frac{|Du|^{p}}{p}\diff x + \int_{\partial \Omega \cap B^{3}_{r}(x_{0})}\frac{|D_{\top}u|^{p}}{p}\diff \mathcal{H}^{2}\right),
    \end{equation}
    where $C=C(K)>0$. 
\end{prop}

To prove Proposition~\ref{prop bound normal}, we need the next estimate.
\begin{lemma}\label{lemma normal der}
Let $\Omega \subset \mathbb{R}^{3}$ be a bounded domain of class $C^{2}$, $x_{0} \in \partial \Omega$, $r>0$ be fairly small and $p \in [1,2]$. Let $u \in W^{1, p}(\Omega\cap B^{3}_{r}(x_{0}), \mathbb{R}^{\nu})$ and \(Du\) be continuous at every point of \(\partial \Omega \cap B^{3}_{r}(x_{0})\). Then the following estimate holds
\begin{equation}\label{est normal der}
    (p-1)|D_{\perp} u|^{p}\leq -p\langle\nu, T^{p}_{u}[\nu]\rangle+(3-p)|D_{\top}u|^{p}
\end{equation}
on $\partial \Omega \cap B^{3}_{r}(x_{0})$, where $\nu \in C^{1}(\partial \Omega, \mathbb{S}^{2})$ stands for the outward pointing unit normal vector field to $\partial \Omega$. 
\end{lemma}
\begin{proof}
    Applying successively  Young's inequality for products,
    the definition of the stress-energy  tensor  $T^{p}_{u}$ and the subadditivity of the function \(t \in \intvo{0}{\infty} \mapsto t^{p/2}\), we have
    \begin{equation*}
    \begin{split}
        |D_{\perp} u|^{p}&=|D_{\perp} u|^{p}|Du|^{\frac{p(p-2)}{2}}|Du|^{\frac{p(2-p)}{2}}\\
        &\leq \frac{p}{2}|D_{\perp}u|^{2}|Du|^{p-2}+\frac{2-p}{2}|Du|^{p}\\
        &=-\frac{p}{2}\biggl(\frac{|Du|^{p}}{p}-\frac{|D_{\perp}u|^{2}}{|Du|^{2-p}}\biggr) + \frac{3-p}{2}|Du|^{p}\\
        &=-\frac{p}{2}\langle\nu, T^{p}_{u}[\nu]\rangle +\frac{3-p}{2}(|D_{\top}u|^{2}+|D_{\perp}u|^{2})^{\frac{p}{2}}\\
        &\leq -\frac{p}{2}\langle\nu, T^{p}_{u}[\nu]\rangle +\frac{3-p}{2}(|D_{\top}u|^{p}+|D_{\perp}u|^{p}),
    \end{split}
    \end{equation*}
    which yields \eqref{est normal der} by rearranging the terms.
\end{proof}
\begin{proof}[Proof of Proposition~\ref{prop bound normal}]
    Let $\theta \in C^{1}(\partial \Omega, \mathbb{R})$ be such that $\theta=1$ on $K$, $0\leq \theta\leq 1$ and $\theta=0$ in $\partial \Omega \setminus B^{3}_{r}(x_{0})$. Fix $\xi \in C^{1}_{c}(B^{3}_{r}(x_{0}), \mathbb{R}^{3})$ such that $\xi=\theta \nu$ on $\partial \Omega$. 
    Applying Proposition~\ref{prop int by parts tensor} and using the Cauchy-Schwarz inequality, and also the facts that $|\mathrm{Id}|=\sqrt{3}$, $|Du \otimes Du|=|Du|^{2}$, we get
    \begin{equation}
     \label{ineq tensor est en balldomain}
     \begin{split}
        \left|\int_{\partial \Omega \cap B^{3}_{r}(x_{0})}\theta \langle \nu, T^{p}_{u}[\nu]\rangle\diff \mathcal{H}^{2} \right|&=\left|\int_{\Omega \cap B^{3}_{r}(x_{0})}T^{p}_{u}:D\xi \diff x\right| \\
        & = \left|\int_{\Omega \cap B^{3}_{r}(x_{0})}\left(\frac{|Du|^{p}}{p}\mathrm{Id}-\frac{Du\otimes Du}{|Du|^{2-p}}\right):D\xi \diff x\right| \\
        &\leq \int_{\Omega \cap B^{3}_{r}(x_{0})}\left(\frac{\sqrt{3}}{p}+1\right)|Du|^{p}|D\xi|\diff x\\
        &\leq (2+\sqrt{3})\|D\xi\|_{L^{\infty}(B^{3}_{r}(x_{0}), \mathbb{R}^{3}\otimes \mathbb{R}^{3})}\int_{\Omega \cap B^{3}_{r}(x_{0})}\frac{|Du|^{p}}{p}\diff x. 
        \end{split}
        \end{equation}
    Setting $E=\{x \in \partial \Omega \cap B^{3}_{r}(x_{0}): (p-1)|D_{\perp}u(x)|^{2}\geq |D_{\top} u(x)|^{2}\}$, observe that $-\langle \nu, T^{p}_{u}[\nu] \rangle \geq 0$ on $E$ and hence
    \begin{equation}\label{fineqboundest}
        -\frac{p}{p-1}\int_{K \cap E} \langle \nu, T^{p}_{u}[\nu] \rangle \diff \mathcal{H}^{2} \leq -\frac{p}{p-1} \int_{E} \theta \langle \nu, T^{p}_{u}[\nu] \rangle \diff \mathcal{H}^{2},
    \end{equation}
    where we have used that $K\csubset \partial \Omega \cap B^{3}_{r}(x_{0})$, $\theta=1$ on $K$ and $0\leq \theta \leq 1$. On the other hand, on $(\partial \Omega \cap B^{3}_{r}(x_{0})) \setminus E$ it holds
    \begin{equation}
     \label{esttensoranotherpartbound}
     \begin{split}
        0 \leq \langle \nu, T^{p}_{u}[\nu] \rangle &= |Du|^{p-2}\biggl(\frac{|D_{\top}u|^{2}+|D_{\perp}u|^{2}}{p}-|D_{\perp}u|^{2}\biggr)\\
        & \leq |D_{\top} u|^{p-2}\biggl(\frac{|D_{\top}u|^{2}+(p-1)|D_{\perp} u|^{2}}{p}\biggr) \leq \frac{2}{p}|D_{\top}u|^{p}. 
        \end{split}
    \end{equation}
    Hereinafter in this proof, $C$ denotes an absolute positive constant that can be different from line to line.  Using Lemma~\ref{lemma normal der}, the estimates \eqref{ineq tensor est en balldomain}, \eqref{fineqboundest} and \eqref{esttensoranotherpartbound}, we deduce the following chain of estimates
    \begin{align*}
        \int_{K}|Du|^{p}\diff \mathcal{H}^{2} &\leq \int_{K}(|D_{\top}u|^{p} +|D_{\perp}u|^{p})\diff \mathcal{H}^{2}  \\ 
        &\leq  - \frac{p}{p-1} \int_{K}\langle \nu, T^{p}_{u}[\nu]\rangle \diff \mathcal{H}^{2} + \frac{C}{p-1}\int_{K}\frac{|D_{\top}u|^{p}}{p}\diff \mathcal{H}^{2}\\
        &\leq -\frac{p}{p-1} \int_{\partial \Omega \cap B^{3}_{r}(x_{0})}\theta \langle \nu, T^{p}_{u}[\nu] \rangle \diff \mathcal{H}^{2}+ \frac{C}{p-1}\int_{\partial \Omega \cap B^{3}_{r}(x_{0})}\frac{|D_{\top} u|^{p}}{p}\diff \mathcal{H}^{2} \\
        &\leq \frac{C}{p-1}\biggl(\|D\xi\|_{\infty} \int_{\Omega \cap B^{3}_{r}(x_{0})}\frac{|Du|^{p}}{p}\diff x + \int_{\partial \Omega \cap B^{3}_{r}(x_{0})}\frac{|D_{\top} u|^{p}}{p}\diff \mathcal{H}^{2}\biggr),
    \end{align*}
    which completes our proof of Proposition~\ref{prop bound normal}.
\end{proof}
Next, we analyze the boundary and inward perturbations for the limit stress-energy tensor, and we prove that the set supporting the associated varifold to this tensor is \textit{minimizing area to first order} in the Brian White sense. First, we analyze the boundary perturbations.
\begin{prop}\label{prop bound perturbations}
    Let $\Omega \subset \mathbb{R}^{3}$ be a bounded domain of class $C^{2}$, $x_{0} \in \partial \Omega$, $r>0$ be fairly small and $p \in [1,+\infty)$. If $u \in W^{1, p}(\Omega\cap B^{3}_{r}(x_{0}), \mathbb{R}^{\nu})$ satisfies $\mathrm{div}(T^{p}_{u})=0$ in $\mathscr{D}^{\prime}(\Omega\cap B^{3}_{r}(x_{0}), \mathbb{R}^{3})$ and if \(Du\) is continuous at every point of \(\partial \Omega \cap B^{3}_{r}(x_{0})\), then for each $\xi \in C^{1}_{c}(B^{3}_{r}(x_{0}), \mathbb{R}^{3})$,
    \begin{equation}\label{est bound per}
        \biggl|\int_{\Omega \cap B^{3}_{r}(x_{0})} T^{p}_{u}:D\xi \diff x\biggr| \leq \left(1+\frac{1}{p}\right)\|\xi\|_{L^{\infty}(\partial \Omega \cap B^{3}_{r}(x_{0}), \mathbb{R}^{3})} \int_{\partial \Omega \cap B^{3}_{r}(x_{0})}|Du|^{p} \diff \mathcal{H}^{2}.
    \end{equation}
    \end{prop}
    \begin{proof}
        Let $\nu \in C^{1}(\partial \Omega, \mathbb{S}^{2})$ be the outward pointing unit normal vector field to $\partial \Omega$ and let $\xi \in C^{1}_{c}(B^{3}_{r}(x_{0}), \mathbb{R}^{3})$. Applying the triangle and the Cauchy-Schwarz inequalities, we have 
        \[
        |\langle \nu, T^{p}_{u}[\xi]\rangle|\leq (1+1/p)|Du|^{p}|\xi|
        \]
        on $\partial \Omega \cap B^{3}_{r}(x_{0})$. Using this estimate and Proposition~\ref{prop int by parts tensor}, we complete our proof of Proposition~\ref{prop bound perturbations}.
    \end{proof}
\begin{prop}\label{prop div measure}
    Let $\Omega \subset \mathbb{R}^{3}$ be a bounded domain of class $C^{2}$, $x_{0} \in \partial \Omega$, $r>0$ be fairly small, $(p_{n})_{n\in \mathbb{N}}\subset [1,2)$ and $p_{n} \nearrow 2$ as $n\to +\infty$. Let $(u_{n})_{n\in \mathbb{N}}$ be a sequence of mappings such that $u_{n} \in W^{1,p_{n}}(\Omega \cap B^{3}_{r}(x_{0}), \mathbb{R}^{\nu})$,  \(Du_n\) be continuous at every point of \(\partial \Omega \cap B^{3}_{r}(x_{0})\) and $\mathrm{div}(T^{p_{n}}_{u_{n}})=0$ in $\mathscr{D}^{\prime}(\Omega \cap B^{3}_{r}(x_{0}), \mathbb{R}^{3})$. Assume also that $(2-p_{n})T^{p_{n}}_{u_{n}} \overset{*}{\rightharpoonup} T_{*}$ weakly* in $(C(\overline{\Omega \cap B^{3}_{r}(x_{0})}, M_{3}(\mathbb{R})))^{\prime}$. If 
    \begin{equation}\label{cond1lim}
        \limsup_{n\to +\infty} (2-p_{n})\int_{\Omega \cap B^{3}_{r}(x_{0})}\frac{|Du_{n}|^{p_{n}}}{p_{n}}\diff x <+\infty
    \end{equation}
    and
    \begin{equation}\label{cond2lim}
        \limsup_{n\to +\infty} (2-p_{n})\int_{\partial \Omega \cap B^{3}_{r}(x_{0})}\frac{|D_{\top} u_{n}|^{p_{n}}}{p_{n}}\diff \mathcal{H}^{2} <+\infty,
    \end{equation}
    then there exists a constant $C>0$ such that for each $\xi \in C^{1}_{c}(B^{3}_{r}(x_{0}), \mathbb{R}^{3})$ we have 
    \begin{equation}\label{boundary est limit}
        \biggl\lvert \int_{\Omega \cap B^{3}_{r}(x_{0})} T_{*}:D\xi \diff x \biggr\rvert \leq C\|\xi\|_{L^{\infty}(\partial \Omega \cap B^{3}_{r}(x_{0}), \mathbb{R}^{3})}.
    \end{equation}
\end{prop}
\begin{proof}
    By the weak* convergence and the convexity, and also using Proposition~\ref{prop bound normal}, Proposition~\ref{prop bound perturbations} and the estimates \eqref{cond1lim}, \eqref{cond2lim}, we obtain the following chain of estimates
    \begin{equation*}
    \begin{split}
     \biggl\lvert \int_{\Omega \cap B^{3}_{r}(x_{0})}T_{*}:D\xi \diff x \biggr\rvert 
     &\leq \liminf_{n\to +\infty}\, \biggl\lvert (2-p_{n})\int_{\Omega \cap B^{3}_{r}(x_{0})}T^{p_{n}}_{u_{n}}:D\xi \diff x \biggr\rvert \\
     &\leq  \limsup_{n\to +\infty}2\|\xi\|_{L^{\infty}(\partial \Omega \cap B^{3}_{r}(x_{0}), \mathbb{R}^{3})}(2-p_{n})
     \int_{ \partial \Omega \cap B^{3}_{r}(x_{0})}|Du_{n}|^{p_{n}} \diff \mathcal{H}^{2} \\
      &\leq \limsup_{n\to +\infty} C^{\prime}\|\xi\|_{L^{\infty}(\partial \Omega \cap B^{3}_{r}(x_{0}), \mathbb{R}^{3})}(2-p_{n})\\
      &\hspace{6em}\biggl(\int_{\Omega \cap B^{3}_{r}(x_{0})}\frac{|Du_{n}|^{p_{n}}}{p_{n}}\diff x + \int_{\partial \Omega \cap B^{3}_{r}(x_{0})}\frac{|D_{\top} u_{n}|^{p_{n}}}{p_{n}}\diff \mathcal{H}^{2}\biggr)\\
      &\leq C\|\xi\|_{L^{\infty}(\partial \Omega \cap B^{3}_{r}(x_{0}), \mathbb{R}^{3})},
    \end{split}
    \end{equation*}
    where $C^{\prime}, C>0$ are constants independent of $\xi$. This leads to \eqref{boundary est limit} and thereby completes our proof of Proposition~\ref{prop div measure}.
\end{proof}
Now we analyze the inward perturbations.
\begin{prop}\label{prop div nonnegative}
    Let $\Omega \subset \mathbb{R}^{3}$ be a bounded domain of class $C^{2}$, $x_{0} \in \partial \Omega$, $r>0$ be fairly small and $p \in [1,2]$. If $u \in W^{1,p}(\Omega \cap B^{3}_{r}(x_{0}), \mathbb{R}^{\nu})$ satisfies $\mathrm{div}(T^{p}_{u})=0$ in $\mathscr{D}^{\prime}(\Omega\cap B^{3}_{r}(x_{0}), \mathbb{R}^{3})$ and if \(Du\) is continuous at every point of \(\partial \Omega \cap B^{3}_{r}(x_{0})\), then for each $\xi \in C^{1}_{c}(B^{3}_{r}(x_{0}), \mathbb{R}^{3})$ such that $\langle \xi, \nu \rangle \leq 0$ on $\partial \Omega \cap B^{3}_{r}(x_{0})$, where $\nu \in C^{1}(\partial \Omega, \mathbb{S}^{2})$ denotes the outward pointing unit normal vector field to $\partial \Omega$, we have
    \begin{equation*}
        \int_{\Omega \cap B^{3}_{r}(x_{0})} T^{p}_{u}:D\xi \diff x \geq -\|\xi\|_{L^{\infty}(U, \mathbb{R}^{3})}\int_{U}\left((3-p)\frac{|D_{\top}u|^{p}}{p}+|D_{\perp}u|^{p-1}|D_{\top}u|\right) \diff \mathcal{H}^{2},
    \end{equation*}
    where $U=\partial \Omega \cap B^{3}_{r}(x_{0})$.
\end{prop}
\begin{proof}
    Using Lemma~\ref{lemma normal der} and the fact that $\langle \xi, \nu \rangle \le 0$ on $U$, we deduce that
    \begin{align*}
    \langle \nu, T^{p}_{u}[\xi] \rangle & = \langle \nu, \xi \rangle \langle \nu, T^{p}_{u}[\nu]\rangle - |Du|^{p-2}\langle Du[\xi-\langle \nu, \xi \rangle \nu], Du[\nu] \rangle \\
    & \geq \langle \nu, \xi \rangle \left(\frac{1-p}{p}\right) |D_{\perp} u|^{p}+\langle \nu, \xi \rangle \left(\frac{3-p}{p}\right)|D_{\top} u|^{p}-\|\xi\|_{L^{\infty}(U, \mathbb{R}^{3})}|D_{\perp}u|^{p-1}|D_{\top}u|\\
    & \geq -\|\xi\|_{L^{\infty}(U, \mathbb{R}^{3})}\left(\frac{3-p}{p}\right)|D_{\top} u|^{p}-\|\xi\|_{L^{\infty}(U, \mathbb{R}^{3})}|D_{\perp}u|^{p-1}|D_{\top}u|.
\end{align*}
The conclusion then follows from Proposition~\ref{prop int by parts tensor}.
\end{proof}
The next proposition says that the set supporting the varifold associated to the tensor $T_{*}$ of Proposition~\ref{prop div measure} is \textit{minimizing area to first order} in the Brian White sense (see \cite{White_2010}).
\begin{prop}\label{prop min BW}
    Let $\Omega \subset \mathbb{R}^{3}$ be a bounded domain of class $C^{2}$, $x_{0} \in \partial \Omega$ and $r>0$ be fairly small. Let $(p_{n})_{n\in \mathbb{N}}$, $(u_{n})_{n \in \mathbb{N}}$ and $T_{*}$ be as in Proposition~\ref{prop div measure}. Let $\nu \in C^{1}(\partial \Omega, \mathbb{S}^{2})$ be the outward pointing unit normal vector field to $\partial \Omega$. Assume that 
    \begin{equation}\label{tangent limit zero}
      \lim_{n\to +\infty}(2-p_{n})\int_{\partial \Omega \cap B^{3}_{r}(x_{0})}\frac{|D_{\top}u_{n}|^{p_{n}}}{p_{n}}\diff \mathcal{H}^{2}=0.
\end{equation}
Then for each $\xi \in C^{1}_{c}(B^{3}_{r}(x_{0}), \mathbb{R}^{3})$ such that $\langle \xi, \nu \rangle \le 0$ on $\partial \Omega \cap B^{3}_{r}(x_{0})$,
\begin{equation*}
    \int_{\Omega \cap B^{3}_{r}(x_{0})}T_{*}:D\xi \diff x \geq 0.
\end{equation*}
\end{prop}
\begin{proof}
    Set $U= \partial \Omega \cap B^{3}_{r}(x_{0})$. Taking into account \eqref{cond1lim}, \eqref{cond2lim} and Proposition~\ref{prop bound normal}, we get
    \begin{equation}\label{total der est}
        A\coloneqq \limsup_{n\to +\infty} (2-p_{n}) \int_{U}|Du_{n}|^{p_{n}} \diff \mathcal{H}^{2} <+\infty.
    \end{equation}
    Next, using Proposition~\ref{prop div nonnegative}, H\"older's inequality, \eqref{total der est} and \eqref{tangent limit zero}, we deduce that
    \begin{equation*}
    \begin{split}
        \int_{\Omega \cap B^{3}_{r}(x_{0})}T_{*}:D\xi \diff x &\geq -\|\xi\|_{L^{\infty}(U, \mathbb{R}^{3})}\limsup_{n\to +\infty} (2-p_{n}) \int_{U}\frac{|D_{\top} u_{n}|^{p_{n}}}{p_{n}}\diff \mathcal{H}^{2} \\ & \,\ \,\ \,\ -\|\xi\|_{L^{\infty}(U, \mathbb{R}^{3})}\limsup_{n\to +\infty} (2-p_{n})\|Du_{n}\|_{L^{p_{n}}(U, \mathbb{R}^{\nu}\otimes \mathbb{R}^{3})}^{1-\frac{1}{p_{n}}} \|D_{\top}u_{n}\|_{L^{p_{n}}(U, \mathbb{R}^{\nu}\otimes \mathbb{R}^{3})}\\
        &\geq -\|\xi\|_{L^{\infty}(U, \mathbb{R}^{3})}A^{\frac{1}{2}} \left(\limsup_{n\to +\infty} (2-p_{n})\int_{U}\frac{|D_{\top}u_{n}|^{p_{n}}}{p_{n}}\diff \mathcal{H}^{2}\right)^{\frac{1}{2}} = 0,
        \end{split}
    \end{equation*}
    which completes our proof of Proposition~\ref{prop min BW}.
\end{proof}
We shall say that $\overline{\Omega}$ is \emph{strongly convex} at a point $x_{0} \in \partial \Omega$ whenever the two  principal curvatures of $\partial \Omega$ at $x_{0}$, with respect to the inward pointing unit normal vector to $\partial \Omega$ at $x_{0}$, are positive. If $\Omega$ is convex and $\partial \Omega$ has the second fundamental form that is positive semidefinite with respect to the inward pointing unit normal vector field to $\partial \Omega$, then $\overline{\Omega}$ is strongly convex at $x_{0} \in \partial \Omega$ whenever the Gaussian curvature of $\partial \Omega$ is positive at \(x_0\).
   
As a result of our analysis in this subsection, we obtain the next boundary repulsion property.
\begin{theorem}\label{thm boundary concentration}
    Let $\Omega \subset \mathbb{R}^{3}$ be a bounded domain of class $C^{2}$ such that $\overline{\Omega}$ is strongly convex at every point of $\partial \Omega$. Let $g \in \mathcal{R}^{1}_{0}(\partial \Omega, \mathcal{N})$ be such that $g \in C^{1}(\partial \Omega \setminus S(g), \mathcal{N})$. Let $(p_{n})_{n\in \mathbb{N}} \subset [1,2)$, $p_{n} \nearrow 2$ as $n \to +\infty$ and $(u_{n})_{n \in \mathbb{N}}$ be a sequence of $p_{n}$-minimizers in $\Omega$ such that for each $n \in \mathbb{N}$, $\operatorname{tr}_{\partial \Omega}(u_{n})=g$. 
    Assume that for each \(n \in \mathbb{N}\), \(D u_n\) is continuous on \(\partial \Omega \setminus S (g)\). 
    Then there exists $\mu_{*} \in (C(\overline{\Omega}))^{\prime}$ such that, up to a subsequence, \eqref{wcm} holds. Moreover, $S_{*}$ lies in the convex hull of $S(g)$ and $S_{*}\cap \partial \Omega =S(g)$,  where $S_{*}$ is the support of $\mu_{*}$.
\end{theorem}
\begin{rem}
According to \cite{H-L}, every $p$-minimizer $u$ in $\Omega$ is locally $C^{1,\alpha}$ regular in $\Omega \backslash S(u)$ for some relatively closed subset $S(u)$ of $\Omega$ which has Hausdorff dimension at most 1 when $p \in (1,2]$. Furthermore, $p$-minimizers with smooth boundary values on smooth domains are H\"{o}lder continues near the boundary and, in view of \cite{SU}, it is reasonable to assume that they are $C^{1}$ regular near the boundary.  In Theorem~\ref{thm boundary concentration}, we assume that the gradient of a $p$-minimizer is continuous on a portion of the boundary, where the boundary datum is smooth. This assumption is significantly less demanding than the assumption that the minimizer is $C^{1}$ regular outside the union of its singular set with the singular set of the boundary datum.
\end{rem}
\begin{proof}[Proof of Theorem~\ref{thm boundary concentration}]
    Using Proposition~\ref{global energy bound}, the fact that $g \in \mathcal{R}^{1}_{0}(\partial \Omega, \mathcal{N})\subset W^{1/2,2}(\partial \Omega, \mathcal{N})$ and the Banach-Alaoglu theorem, we obtain $\mu_{*} \in (C(\overline{\Omega}))^{\prime}$ such that, up to a subsequence, \eqref{wcm} holds. 
    Fix an arbitrary point $x_{0} \in \partial \Omega \setminus S(g)$ and a fairly small $r>0$ such that $\smash{\overline{B}}^{3}_{r}(x_{0})\cap S(g)=\emptyset$. In view of Proposition~\ref{global energy bound}, the fact that $g \in W^{1/2,2}(\partial \Omega, \mathcal{N})$ and Lemma~\ref{lemma energ boundatum}, \eqref{cond1lim} and \eqref{cond2lim} hold. Then, taking into account \eqref{est tv 5.9}, Proposition~\ref{global energy bound} and using the Banach-Alaoglu theorem, we deduce that there exists $T_{*} \in (C_{0}(\overline{\Omega \cap B^{3}_{r}(x_{0})}, M_{3}(\mathbb{R})))^{\prime}$ such that, up to a subsequence (still denoted by the same index), $(2-p_{n})T^{p_{n}}_{u_{n}} \overset{*}{\rightharpoonup} T_{*}$ weakly* in $(C_{0}(\overline{\Omega \cap B^{3}_{r}(x_{0})}, M_{3}(\mathbb{R})))^{\prime}$. Keeping the same notation, let us define $T_{*}(\cdot)= T_{*}(\cdot \cap \overline{\Omega} \cap B^{3}_{r}(x_{0})) \in (C_{0}(B^{3}_{r}(x_{0}), M_{3}(\mathbb{R})))^{\prime}$. By Proposition~\ref{prop div measure}, $\mathrm{div}(T_{*}) \in (C_{0}(B^{3}_{r}(x_{0}), \mathbb{R}^{3}))^{\prime}$. Then, according to \cite[Proposition~3.1]{ArroyoRabasa_DePhilippis_Hirsch_Rindler_2019}, there exists an $\mathcal{H}^{1}$-rectifiable set $R\subset B^{3}_{r}(x_{0}) \cap \overline{\Omega}$ and a map $\lambda \in L^{\infty}(R, M_{3}(\mathbb{R}))$ satisfying the following conditions. For $\mathcal{H}^{1}$-a.e.\  $x_{0} \in R$, $|\lambda(x_{0})|=1$, $\lambda(x_{0}) \in \mathrm{G}(3,1)$ is the orthogonal projection matrix onto $T_{x_{0}}R$  such that $T_{*}\mres\{\Theta^{*}_{1}(|T_{*}|)>0\}(d x)= \Theta^{*}_{1}(|T_{*}|,x)\lambda(x)\mathcal{H}^{1}\mres R (d x)$, where $T_{x_{0}}R$ is the approximate tangent plane to $R$ at $x_{0}$. One has that $|T_{*}|=\mu_{*}$, $\Theta^{*}_{1}(|T_{*}|)=\Theta_{1}(\mu_{*})$ and $R=S_{*}\cap B^{3}_{r}(x_{0})$ (we refer the reader to the proof of Proposition~\ref{prop 1-var}). In particular, we deduce that 
    $V_{*}(d A, d x)=\delta_{\lambda_{*}(x)}(d A)\otimes |T_{*}|(d x)$
        is a 1-dimensional rectifiable varifold (see \cite{Allard, DPDRG}). Since $g \in C^{1}(\partial \Omega \cap B^{3}_{r}(x_{0}), \mathcal{N})$, $(u_{n})_{n\in \mathbb{N}}$ fulfill the condition \eqref{tangent limit zero} of  Proposition~\ref{prop min BW}, which says that $V_{*}$ is \emph{minimizing area to first order} in $(\Omega \cap B^{3}_{r}(x_{0})) \cup (\partial \Omega \cap B^{3}_{r}(x_{0}))$ in the Brian White sense (see \cite{White_2010}). Then, by \cite[Theorem~1]{White_2010}, $x_{0} \not \in S_{*}$. Thus, $S_{*} \cap \partial \Omega  \subset S(g)$. Denote by \(E\) the convex hull of \(S(g)\). By contradiction, assume that $S_{*} \cap (\Omega\setminus E) \neq \emptyset$. Then, since $S_{*} \cap \partial \Omega  \subset S(g)$, the set $M$ of points $x_{0} \in S_{*}$ such that $\dist(x_{0}, E)=\max\{\dist(x,E):x \in S_{*}\}$ is a subset of $\Omega \setminus E$. Fix a point $x_{0} \in M$ such that the segments of $S_{*}$ emanating from $x_{0}$ do not lie in the same plane. Then these segments lie in the same half-space (lying under a 2-plane parallel to some face of $E$) and form a vertex, which cannot happen in view of the stationarity of the varifold coming from Proposition~\ref{prop 1-var} (see Corollary~\ref{cor 5.18}).  Therefore, $S_{*} \subset E$. 
    
    Next, let us prove that $S(g)\subset S_{*}\cap \partial \Omega$. Using Corollary~\ref{cor conv uptotheboundary} and the fact that $S_{*}\cap \partial \Omega \subset S(g)$, for all fairly small $\delta>0$, we have $\tr_{\partial \Omega \setminus S^{\delta}_{g}}(u_{*})=g|_{S^{\delta}_{g}}$, where $u_{*}$ is a map coming from Corollary~\ref{cor conv uptotheboundary} and $S^{\delta}_{g}=\{x \in \partial \Omega: \dist(x, S(g)) \geq \delta\}$. For each $a \in S(g)$, denoting by $\partial B^{\partial \Omega}_{\varrho}(a)$ the boundary of the geodesic ball in $\partial \Omega$ with center $a$ and radius $\varrho>0$, we know that $u_{*}|_{\partial B^{\partial \Omega}_{\varrho}(a)}=g|_{\partial B^{\partial \Omega}_{\varrho}(a)}$ is not nullhomotopic for each sufficiently small $\varrho>0$. This, together with the fact that $S_{*} \subset E$ and Corollary~\ref{cor 5.17}, implies that $a \in S_{*} \cap \partial \Omega$ and hence $S(g)\subset S_{*}\cap \partial \Omega$. Therefore, $S_{*}\cap \partial \Omega = S(g)$, which completes our proof of Theorem~\ref{thm boundary concentration}.
\end{proof}

\subsection{Structure of the singular set and the limit measure}

\begin{theorem}\label{prop structure chains}
	Let $\Omega$, $g$, $S(g)$, $(u_{n})_{n\in \mathbb{N}}$, $\mu_{*}$ and $S_{*}$ be as in Theorem~\ref{thm boundary concentration}. Let $u_{*}$ be the map of Proposition~\ref{conv to harm} being an element of $W^{1,2}_{\loc}(\Omega\setminus S_{*}, \mathcal{N})$. Then $S_{*}$ is a finite union of finite chains of segments connecting the points of $S(g)$ and lying in the convex hull of $S(g)$. 
More precisely, there exists a finite set of segments $L_{1},\dotsc,L_{n}\subset \overline{\Omega}$ such that $S_{*}=\bigcup_{i=1}^{n}L_{i}$ and $u_{*}\in C^{\infty}(\Omega \setminus (\bigcup_{i=1}^{n}L_{i} \cup S_{0}), \mathcal{N})$, where $S_{0}$ is a locally finite subset of $\Omega \setminus \bigcup_{i=1}^{n}L_{i}$. 
If $x$ is a point inside $L_{i}$, then $\Theta_{1}(\mu_{*},x)=\mathcal{E}^{\mathrm{sg}}_{2}([u_{*}|_{\partial D}])=:\lambda_{i}$, where $D$ is a closed 2-disk in $\Omega$ such that $D\cap S_{*}=\{x\}$, $\partial D \cap S_{0}=\emptyset$ and $x$ is an interior point of $D$. If $x$ is a branching point of $S_{*}$, namely there exists $L_{i_{1}},\dotsc,L_{i_{k}}$ emanating from $x$, then $\sum_{j=1}^{k}\lambda_{i_{j}}v_{i_{j}}=0$, where $v_{i_{j}}$ is the unit direction vector of $L_{i_{j}}$ pointing from $x$. 
Furthermore, $\mu_{*}=\sum_{i=1}^{n}\lambda_{i}\mathcal{H}^{1}\mres L_{i}$. 
Finally, for each $i\in \{1,\dotsc,n\}$, each endpoint of $L_{i}$ is either a branching point of $S_{*}$ lying in $\Omega$ or it belongs to $S(g)$. If $x$ is an endpoint of some $L_{i}$ lying in $S(g)$, then $[g|_{\partial D_{x}}]=[u_{*}|_{\partial D}]$, where $D_{x}$ is the geodesic disk on $\partial \Omega$ with center $x$ and lying far from $S(g)\setminus \{x\}$, and $D$ is a closed 2-disk in $\Omega$ centered inside $L_{i}$ with $\partial D$ lying outside the convex hull of $S(g)$ and outside of $S_{0}$.
\end{theorem}

\begin{proof}
Denote by \(E\) the convex hull of \(S(g)\). The fact that $S_{*} \subset E$ comes from Theorem~\ref{thm boundary concentration}.
Fix an arbitrary $x_{0} \in S(g)$. Further in this proof, the center of each ball and sphere is at $x_{0}$. In view of \eqref{C2}, for each $r_{0}>0$,
\begin{equation*}
\limsup_{n\to +\infty} (2-p_{n}) \int_{B^{3}_{r_{0}} \cap \Omega} \frac{|Du_{n}|^{p_{n}}}{p_{n}}\diff x < +\infty.
\end{equation*}
Then, up to extracting a subsequence (not relabeled) and using the coarea formula, together with the Fatou lemma, we deduce that $\tr_{\partial B^{3}_{r}\cap \Omega}(u_{n})=u_{n}|_{\partial B^{3}_{r} \cap \Omega} \in W^{1,p_{n}}(\partial B^{3}_{r}\cap \Omega, \mathcal{N})$ and
\begin{equation*}
\limsup_{n\to +\infty}(2-p_{n}) \int_{\partial B^{3}_{r} \cap \Omega} \frac{|Du_{n}|^{p_{n}}}{p_{n}}\diff \mathcal{H}^{2} < +\infty
\end{equation*}
for $\mathcal{H}^{1}$-a.e.\  $r \in (0,r_{0}]$.
Let $r_{0}>0$ be small enough so that $\smash{\overline{B}}^{3}_{r_{0}}$ does not contain any point of $S(g)$ except of $x_{0}$. Furthermore, by virtue of Proposition~\ref{prop structure S_*} and the fact that $S_{*} \subset E$, we can also assume that $\partial B^{3}_{r_{0}}$ does not contain any branching point of $S_{*}$. Next, we recall that $\mu_{*}\mres \Omega (d x) = \Theta_{1}(x) \mathcal{H}^{1}\mres (S_{*}\cap \Omega)(d x)$, where the density $\Theta_{1}$ is defined in \eqref{defofdensity}. For each $r \in (0,r_{0}]$, define $M(r)=\int_{S_{*}\cap \partial B^{3}_{r}} \Theta_{1}(x)\diff \mathcal{H}^{0}(x)$, where $\mathcal{H}^{0}$ denotes the $0$-dimensional Hausdorff measure in $\mathbb{R}^{3}$, namely, for a discrete set of points $X\subset \mathbb{R}^{3}$, $\mathcal{H}^{0}(X)=\# X$ (the cardinality of $X$). In view of Lemma~\ref{lem cont singenergy}, Proposition~\ref{prop density dv}, Proposition~\ref{prop structure S_*} and the fact that $S_{*}\subset E$,  $M$ is a bounded from below piecewise constant function with a discrete set of values. Indeed, there exists an at most countable collection of intervals $\{(a_{i}, b_{i}]: i\in \mathbb{N}\cap [0,i_{0}]\}$, where $i_{0}\in \mathbb{N}\cup \{+\infty\}$, such that $b_{0}=r_{0}$, $(0, r_{0}]=\bigcup_{i=0}^{i_{0}}(a_{i}, b_{i}]$ and for each $i\in \mathbb{N}\cap [0,i_{0}]$: $M$ is constant on $(a_{i}, b_{i})$; $a_{i}=b_{i+1}$; for each $r \in (a_{i}, b_{i})$, $\partial B^{3}_{r}$ does not meet any branching point of $S_{*}$; $\partial B^{3}_{a_{i}}$ meets a branching point of $S_{*}$. By definition of $r_{0}$, $M$ is constant on $(a_{0}, b_{0}]$. Notice that, by the upper semicontinuity of $\Theta_{1}$ (see, for instance, \cite[2(8)]{Allard-Almgren}) and Corollary~\ref{cor 5.17}, for each $i\in \mathbb{N} \cap [0,i_{0}]$ and $x \in (a_{i}, b_{i}) \cup (a_{i+1}, b_{i+1})$, $M(x)\leq M(b_{i})$. Also, taking into account Proposition~\ref{prop structure S_*} and the facts that $S_{*}\subset E$, $E \cap \partial \Omega =S(g)$ (since $\overline{\Omega}$ is strongly convex at every point of $\partial \Omega$), we observe that $\mu_{*}(\partial (B^{3}_{r}\cap \Omega))=0$ for each $r \in (0,r_{0}]$. This, together with the weak* convergence of $\mu_{n}$ to $\mu_{*}$ (see \eqref{wcm}), implies that for each $i \in \mathbb{N} \cap [0,i_{0}]$,
\begin{equation*}
\mu_{n}((B^{3}_{b_{i}}\setminus B^{3}_{a_{i}})\cap \Omega) -\mu_{*}((B^{3}_{b_{i}}\setminus B^{3}_{a_{i}})\cap \Omega) =\int_{a_{i}}^{b_{i}}\diff r \biggl( (2-p_{n})\int_{\partial B^{3}_{r} \cap \Omega} \frac{|Du_{n}|^{p_{n}}}{p_{n}} \diff \mathcal{H}^{2} -M(r)\biggr) \to 0
\end{equation*}
as $n \to +\infty$, where we have used the coarea formula. Applying the Fatou lemma, up to a subsequence (not relabeled), we get
\begin{equation}\label{casesfiniteseg}
\int_{a_{i}}^{b_{i}}\diff r \biggl(\limsup_{n\to +\infty} (2-p_{n})\int_{\partial B^{3}_{r} \cap \Omega} \frac{|Du_{n}|^{p_{n}}}{p_{n}}\diff \mathcal{H}^{2} -M(r)\biggr)\leq 0.
\end{equation}
Fix an arbitrary $i\in \mathbb{N} \cap [0,i_{0}]$. Then either for $\mathcal{H}^{1}$-a.e.\  $r \in [a_{i}, b_{i}]$, 
\begin{equation}\label{ineq00001111}
\limsup_{n\to +\infty} (2-p_{n})\int_{\partial B^{3}_{r} \cap \Omega} \frac{|Du_{n}|^{p_{n}}}{p_{n}}\diff \mathcal{H}^{2} \geq M(r),
\end{equation}
or
there exists a set $I\subset [a_{i}, b_{i}]$ of positive $\mathcal{H}^{1}$-measure such that for each $r \in I$, \eqref{ineq00001111} does not hold. In the case when \eqref{ineq00001111} holds for $\mathcal{H}^{1}$-a.e.\  $r \in [a_{i}, b_{i}]$, according to \eqref{casesfiniteseg},  we have
\begin{equation*}
\limsup_{n\to +\infty} (2-p_{n})\int_{\partial B^{3}_{r} \cap \Omega} \frac{|Du_{n}|^{p_{n}}}{p_{n}}\diff \mathcal{H}^{2}=M(r)
\end{equation*}
for $\mathcal{H}^{1}$-a.e.\  $r \in [a_{i}, b_{i}]$. 

Thus, there exists $r \in (a_{i}, b_{i})$ such that $\tr_{\partial B^{3}_{r} \cap \Omega}(u_{n})=u_{n}|_{\partial B^{3}_{r} \cap \Omega} \in W^{1, p_{n}}(\partial B^{3}_{r} \cap \Omega, \mathcal{N})$ for each $n$ and 
\begin{equation}\label{usefulestimenergdensity}
\limsup_{n\to +\infty} (2-p_{n}) \int_{\partial B^{3}_{r} \cap \Omega} \frac{|Du_{n}|^{p_{n}}}{p_{n}}\diff \mathcal{H}^{2} \leq M(r).
\end{equation}
Next, we construct a competitor $v_{n}$ for $u_{n}$. Define $v_{n}=u_{n}$ in $\Omega \setminus B^{3}_{r}$. For each $x$ in the cone $K=\{x \in B^{3}_{r} \cap \Omega: (x_{0}+r(x-x_{0})/|x-x_{0}|) \in \partial B^{3}_{r} \cap \overline{\Omega}\}$, set 
\begin{equation}\label{homdefnvnun}
v_{n}(x)=\operatorname{tr}_{\partial B^{3}_{r}\cap \overline{\Omega}}(u_{n})\left(x_{0}+\frac{r(x-x_{0})}{|x-x_{0}|}\right).
\end{equation}
In the region $V=(B^{3}_{r} \cap \Omega) \setminus K$, we define $v_{n}$ by homotopy, using \cite[Theorem~2]{Bethuel-Demengel}. More precisely, fix a point $x \in V$ and consider a 2-plane passing through $x$, which is orthogonal to the axis of the cone $K$. The intersection of this 2-plane with $V$ is an annulus whose boundary consists of two curves $\gamma_{1} \subset \partial \Omega$ and $\gamma_{2} \subset \partial K$. By definition, $g|_{\gamma_{1}}$ is continuously homotopic to $v_{n}|_{\gamma_{2}}$. Then, by \cite[Theorem~2]{Bethuel-Demengel}, on this annulus there exists a $W^{1,2}$ map  with values in $\mathcal{N}$ and with boundary values $g|_{\gamma_{1}}$ and $v_{n}|_{\gamma_{2}}$ on $\gamma_{1}$ and $\gamma_{2}$, respectively. Observe that the set of values of $v_{n}|_{\gamma_{2}}$ coincides with the set of values of $g|_{\partial B^{3}_{r} \cap \partial \Omega}$. Continuing this process, namely filling $V$, we obtain a map $v_{n}|_{V}$ such that $v_{n}|_{V} \in W^{1,2}(V \setminus \smash{\overline{B}}^{3}_{\delta}, \mathcal{N})$ for each $ \delta \in (0,r)$ and 
\begin{equation}\label{energyzeroinlimitinE}
\limsup_{n\to +\infty} (2-p_{n})\int_{V} \frac{|D(v_{n}|_{V})|^{p_{n}}}{p_{n}}\diff x =0,
\end{equation}
which is possible due to the fact that $g \in \mathcal{R}^{1}_{0}(\partial \Omega, \mathcal{N})$ (see Lemma~\ref{lemma energ boundatum}), since the $p_{n}$-energy of $v_{n}|_{V}$ in $V\cap B^{3}_{\delta}$ is of order $\delta/(2-p_{n})$ for $\delta \in (0,r)$. Thus, we have constructed a competitor $v_{n}$ for $u_{n}$. 

Next, again by the weak* convergence of $\mu_{n}$ to $\mu_{*}$ and the fact that $\mu_{*}(\partial (B^{3}_{r} \cap \Omega))=0$, we have
\begin{equation}\label{againweakconvergenceofmeasures}
\int_{0}^{r}M(\varrho)\diff \varrho = \mu_{*}(B^{3}_{r} \cap \Omega) = \lim_{n\to+\infty} (2-p_{n})\int_{B^{3}_{r} \cap \Omega} \frac{|Du_{n}|^{p_{n}}}{p_{n}}\diff x.
\end{equation}
Since $v_{n}$ is a competitor for $u_{n}$, $v_{n}=u_{n}$ in $\Omega \setminus B^{3}_{r}$, using \eqref{usefulestimenergdensity}, \eqref{homdefnvnun}, \eqref{energyzeroinlimitinE} and \eqref{againweakconvergenceofmeasures}, we deduce that
\begin{equation}
\label{niceestimforM}
\begin{split}
\int_{0}^{r}M(\varrho)\diff \varrho & \leq  \limsup_{n\to +\infty}(2-p_{n}) \int_{K} \frac{|Dv_{n}|^{p_{n}}}{p_{n}}\diff x\\ &\leq \limsup_{n\to +\infty}\int_{0}^{r} \varrho^{2-p_{n}} \diff \varrho\, (2-p_{n})\int_{\partial B^{3}_{r}\cap \Omega}\frac{|D_{\top}u_{n}|^{p_{n}}}{p_{n}} \diff \mathcal{H}^{2} \leq r M(r),
\end{split}
\end{equation}
where we have also used that $p_{n} \nearrow 2$ and $|D_{\top}u_{n}|\leq |Du_{n}|$. We can choose $i \in \mathbb{N} \cap [0,i_{0}]$ and $r>0$ so that, apart from \eqref{niceestimforM}, $M(r)$ is the minimal value of $M$ on $(0,r_{0}]$. In fact, since $M$ is bounded from below and has a discrete set of values, it achieves its minimum at every point of the interval $(a_{i}, b_{i})$ for some $i \in \mathbb{N} \cap [0, i_{0}]$ (notice that it can achieve the minimum at $a_{i}$ or $b_{i}$, but in this case it achieves its minimum at every point of $(a_{i}, b_{i})$, since $M(x) \leq \min\{M(a_{i}), M(b_{i})\}$ as was observed earlier). Thus, taking such $i \in \mathbb{N}\cap [0,i_{0}]$ and $r \in (a_{i}, b_{i})$, it holds $M=M(r)$ on $(0,r]$, because otherwise \eqref{niceestimforM} would not hold. Then, since $M$ has a discrete set of values, there exists a bounded set $J \subset \mathbb{N}$ such that for each $ \varrho \in (0,r]$, the number of branching points of $S_{*} \cap \partial B^{3}_{\varrho}$ belongs to $J$. This, together with the stationarity of the associated varifold $V_{*}$ (see Corollary~\ref{cor 5.18}), implies that the only possible situation describing $S_{*} \cap \smash{\overline{B}}^{3}_{r}$ is when $S_{*} \cap \smash{\overline{B}}^{3}_{r}$ is the union of closed straight line segments connecting the points of $S_{*} \cap \partial B^{3}_{r}$ with $x_{0}$ and whose relative interiors are mutually disjoint. 

We complete our proof of Theorem~\ref{prop structure chains} using Propositions~\ref{conv to harm},~\ref{prop structure S_*}, Corollaries~\ref{cor 5.17},~\ref{cor 5.18},~\ref{cor conv uptotheboundary} and Theorem~\ref{thm boundary concentration}. 
\end{proof}

The next corollary is a direct consequence of Theorem~\ref{prop structure chains}, so we omit its proof.
\begin{cor}\label{cor on structure}
    If the set $S(g)$ of Theorem~\ref{prop structure chains} lies in a 2-plane, then $S_{*}$ is the segment  connecting the points of $S(g)$ if $S(g)$ has only two points, is contained in the triangle with vertices from $S(g)$ if $S(g)$ has only three points and is contained in the polygon with vertices from $S(g)$ if $S(g)$ has more than four points. 
    \end{cor}
We provide a revised version of Corollary~\ref{cor on structure} for the case when $S(g)$ has only two points, also without proof.
\begin{cor}\label{example 2}
        Let $\Omega$, $g$ and $(u_{n})_{n\in \mathbb{N}}$ be as in Theorem~\ref{thm boundary concentration}. Let $S(g)=\{a,b\}$. Then $S_{*}=[a,b]$ and for each $x \in (a,b)$, $\Theta_{1}(\mu_{*},x)=\mathcal{E}^{\mathrm{sg}}_{2}([g|_{\partial D_{a}}])=\mathcal{E}^{\mathrm{sg}}_{2}([g|_{\partial D_{b}}])$, where for $z\in \{a,b\}$, $D_{z}$ is a geodesic (closed) disk in $\partial \Omega$  with center $z$,  not containing $b$ if $z=a$ and $a$ if $z=b$. Furthermore, $\mu_{*}=\mathcal{E}^{\mathrm{sg}}_{2}([g|_{\partial D_{a}}]) \mathcal{H}^{1}\mres [a,b]$.
\end{cor}

\subsection{Minimality of the singular set}
Now we prove the global minimality of $S_{*}$.
\begin{prop}\label{prop global min}
Let $\Omega$, $g$, $S(g)$, $(u_{n})_{n\in \mathbb{N}}$, $S_{*}$ and $\mu_{*}$ be as in Theorem~\ref{thm boundary concentration}. Let $u_{*}$ be the map of Proposition~\ref{conv to harm} being an element of $W^{1,2}_{\loc}(\Omega\setminus S_{*}, \mathcal{N})$. Then $S_{*} \in \mathscr{C}(\Omega, g)$ and for each $S \in \mathscr{C}(\Omega, g)$,
\begin{equation*}
\mu_{*}(S_{*})=\mathbb{M}(g, S_{*}) \leq \mathbb{M}(g, S).
\end{equation*}
\end{prop}
\begin{proof}
According to Theorem~\ref{prop structure chains}, $S_{*} \in \mathscr{C}(\Omega, g)$ (see Definition~\ref{def admissible chain}). Let $S \in \mathscr{C}(\Omega, g)$. By reproducing the \emph{Step~2} of the proof of Proposition~\ref{min of sing}, we obtain a sequence of competitors $v_{n} \in W^{1,p_{n}}(\Omega, \mathcal{N})$ for $u_{n}$ such that
\begin{equation}\label{estimcommassenergysegments}
\limsup_{n\to +\infty} (2-p_{n}) \int_{\Omega} \frac{|Dv_{n}|^{p_{n}}}{p_{n}}\diff x \leq \mathbb{M}(g, S),
\end{equation}
such  reproduction is possible thanks to Lemma~\ref{lemma energ boundatum}. By Theorem~\ref{prop structure chains} and Definition~\ref{defnofmass}, $\mu_{*}(S_{*})=\mathbb{M}(g, S_{*})$.  Using this, $\eqref{wcm}$, the fact that $v_{n}$ is a competitor for $u_{n}$ and \eqref{estimcommassenergysegments}, we obtain
\begin{align*}
\mathbb{M}(g, S_{*})= \lim_{n\to +\infty}(2-p_{n})\int_{\Omega}\frac{|Du_{n}|^{p_{n}}}{p_{n}}\diff x \leq \limsup_{n\to +\infty} (2-p_{n}) \int_{\Omega} \frac{|Dv_{n}|^{p_{n}}}{p_{n}}\diff x \leq \mathbb{M}(g, S),
\end{align*}
which completes our proof of Proposition~\ref{prop global min}.
\end{proof}
For convenience, we provide the definition of a \emph{minimal connection}.
\begin{defn}\label{defnminconnection}
A minimal connection of a set $\{a_{1},\dotsc,a_{2n}\} \subset \mathbb{R}^{3}$ is a union $S$ of mutually disjoint closed segments, each of which connects two points from $\{a_{1},\dotsc,a_{2n}\}$ so that the sum of the lengths of these segments is minimal among all possible such connections.
\end{defn}

\begin{prop}\label{example 3}
Let $\Omega, g, S(g), S_{*}$ be as in Theorem~\ref{thm boundary concentration}, $S(g) \not = \emptyset$ and $\pi_{1}(\mathcal{N}) \simeq \mathbb{Z}/2\mathbb{Z}$. Then $S(g)$ has an even number of elements and $S_{*}$ is a minimal connection of $S(g)$ in the sense of Definition~\ref{defnminconnection}. 
\end{prop}
\begin{proof}
    First, notice that $\# S(g)$ is even, because otherwise, using the fact that $\pi_{1}(\mathcal{N}) \simeq \mathbb{Z}/2\mathbb{Z}$, we would  obtain a map which is not nullhomotopic, but having zero degree, which is a contradiction. Next, in view of Corollary~\ref{cor branching points} and Theorem~\ref{prop structure chains}, $S_{*}$ is a finite union of segments whose relative interiors are mutually disjoint, and each of these segments connects two points from $S(g)$.  Let $S$ be a union of mutually disjoint closed segments, each of which connects two points from $S(g)$.  Then, according to Proposition~\ref{propextoutsideplans} and Definition~\ref{def admissible chain},  $S \in \mathscr{C}(\Omega, g)$. By Proposition~\ref{prop global min}, 
    \[
    \mathbb{M}(g, S_{*}) \leq \mathbb{M}(g, S).
    \]
Since $\pi_{1}(\mathcal{N})$ has only one nontrivial element, the above estimate implies that $S_{*} \in \mathscr{C}(\Omega, g)$ is a minimal connection of $S(g)$ in the sense of Definition~\ref{defnminconnection}. This completes our proof of Proposition~\ref{example 3}.
\end{proof}
\subsection{Proof of Theorem~\ref{theorem_intro_global}} 
Finally, we prove our second main theorem.
\begin{proof}[Proof of Theorem~\ref{theorem_intro_global}]
The proof follows from Propositions~\ref{prop quant behuptotheb},~\ref{prop global min} and Theorem~\ref{prop structure chains}.
\end{proof}
  \appendix
\section{Auxiliary results}
We provide a proof of the following lemma leading to Corollary~\ref{cor interpolation}, which was used in the proofs of Lemmas~\ref{lem lifting linear},~\ref{lem lifting nonlinear}.
\begin{lemma}\label{interpol} 
    Let $m \in \mathbb{N}\setminus \{0,1\}$, $p\in (1,+\infty)$ and $U \subset \mathbb{R}^{m-1}$ be open and convex. Let $\varphi:U\to \mathbb{R}^{m}$ be a bilipschitz mapping and $E=\varphi(U)$. Then there exists a constant $C=C(m)>0$ such that for every $u\in W^{1,p}(E, \mathbb{R}^{k})$,
    \begin{equation}\label{main int estimate}
        |u|^{p}_{W^{1-1/p,p}(E)}\leq  \frac{2^{p}}{p-1}L^{4m-5+2p}C\|u\|_{L^{p}(E)}\,\|D _{\top}u\|_{L^{p}(E)}^{p-1}
    \end{equation}
    and
    \begin{equation}\label{est 2.8}
        |u|^{p}_{W^{1-1/p,p}(E)}\leq 2^{p}L^{2m-3+p}C\Bigl(\frac{2}{\diam(U)}\Bigr)^{p-1}\|u\|^{p}_{L^{p}(E)}+ L^{4m-5+2p}C\Bigl(\frac{\diam(U)}{2}\Bigr)\|D_{\top}u\|^{p}_{L^{p}(E)},
    \end{equation}
    where $|u|^{p}_{W^{1-1/p,p}(E)}$ denotes the integral $\int_{E}\int_{E}\frac{|u(y)-u(x)|^{p}}{|y-x|^{m-2+p}}\diff \mathcal{H}^{m-1}(y)\diff \mathcal{H}^{m-1}(x)$ and $L>0$ is a bilipschitz constant of $\varphi$.
\end{lemma}
\begin{proof}
    Let $\|D_{\top}u\|_{L^{p}(E)}>0$, because otherwise \eqref{main int estimate} and \eqref{est 2.8} are satisfied. According to \cite[Theorem~3.2.5,~p.~244]{Federer}, we have
    \begin{equation}\label{ineq parametrization}
        |u|^{p}_{W^{1-1/p,p}(E)}=\int_{U}\int_{U}\frac{|u(\varphi(y))-u(\varphi(x))|^{p}}{|\varphi(y)-\varphi(x)|^{m-2+p}}J_{m-1}(\varphi(y))J_{m-1}(\varphi(x))\diff y\diff x,
    \end{equation}
    where $J_{m-1}$ denotes the $(m-1)$-dimensional Jacobian (see, for instance, Section~3.2.1, p.~241 in \cite{Federer}).
    We define for each $\delta>0$ the sets \[W_{\delta}=\{(x,y)\in U^{2}: |y-x|\geq \delta\}\quad \text{and} \quad F_{\delta}=\{(x,y) \in U^{2}: |y-x|<\delta\}. \] Then $U^{2}=W_{\delta}\cup F_{\delta}$. Hereinafter in this proof, $C$ denotes a positive constant that can only depend on $m$ and can be different from line to line.  We first prove \eqref{main int estimate}. For each $\delta>0$,
    \begin{equation}
    \label{ineq 0.2}
    \begin{split}
        \iint_{W_{\delta}}&\frac{|u(\varphi(y))-u(\varphi(x))|^{p}}{|\varphi(y)-\varphi(x)|^{m-2+p}}J_{m-1}(\varphi(y)) J_{m-1}(\varphi(x))\diff y  \diff x \\ & \leq 2^{p-1}\iint_{W_{\delta}}\frac{|u(\varphi(y))|^{p}+|u(\varphi(x))|^{p}}{|\varphi(y)-\varphi(x)|^{m-2+p}}J_{m-1}(\varphi(y)) J_{m-1}(\varphi(x))\diff y\diff x \\ &\leq 2^{p-1}L^{m-2+p}\iint_{W_{\delta}}\frac{|u(\varphi(y))|^{p}+|u(\varphi(x))|^{p}}{|y-x|^{m-2+p}}J_{m-1}(\varphi(y)) J_{m-1}(\varphi(x))\diff y\diff x \\ &\leq 2^{p}L^{2m-3+p}\int_{E}|u|^{p}\diff\mathcal{H}^{m-1} \int_{\mathbb{R}^{m-1}\setminus B^{m-1}_{\delta}}\frac{\diff z}{|z|^{m-2+p}}\\
        &=\frac{ 2^{p}CL^{2m-3+p}}{(p-1)\delta^{p-1}}\int_{E}|u|^{p}\diff \mathcal{H}^{m-1},   
        \end{split}
        \end{equation}
    where we have used that $\varphi$ is bilipschitz with the bilipschitz constant $L$ and \cite[Theorem~3.2.5]{Federer}. Next, using Jensen's inequality, Fubini's theorem and the fact that $\varphi$ is bilipschitz with the bilipschitz constant $L$, we have
    \begin{equation}
\label{ineq 0.3}
\begin{split}
        \iint_{F_{\delta}}&\frac{|u(\varphi(y))-u(\varphi(x))|^{p}}{|\varphi(y)-\varphi(x)|^{m-2+p}} J_{m-1}(\varphi(y)) J_{m-1}(\varphi(x))\diff y \diff x \\& = \iint_{F_{\delta}}\left| \int^{1}_{0}\frac{\partial }{\partial t} u(\varphi((1-t)x+ty))\diff t\right|^{p} \frac{J_{m-1}(\varphi(y)) J_{m-1}(\varphi(x))}{|\varphi(y)-\varphi(x)|^{m-2+p}}\diff y\diff x\\
        &\leq L^{4m-5+2p}\int^{1}_{0}\left(\iint_{F_{\delta}}\frac{|D_{\top}u(\varphi((1-t)x+ty))|^{p}|y-x|^{p}}{|y-x|^{m-2+p}} J_{m-1}(\varphi((1-t)x+ty))\diff y\diff x\right) \diff t\\ & \leq L^{4m-5+2p}\int^{1}_{0}\left(\iint_{F_{\delta}}\frac{|D_{\top}u(\varphi((1-t)x+ty))|^{p}}{|y-x|^{m-2}} J_{m-1}(\varphi((1-t)x+ty))\diff y\diff x\right) \diff t.  
        \end{split}
    \end{equation}    
    Performing the change of variable $z=(1-t)x+ty$ in the integral with respect to $x$, we have
    \begin{equation}\label{eq 0.4}
        \begin{split}
            \iint_{F_{\delta}}\frac{|D_{\top}u(\varphi((1-t)x+ty))|^{p}}{|y-x|^{m-2}}&J_{m-1}(\varphi((1-t)x+ty))\diff y\diff x\\ & \qquad \qquad = \frac{1}{1-t}\iint_{F_{(1-t)\delta}}\frac{|D_{\top}u(\varphi(z))|^{p}}{|y-z|^{m-2}} J_{m-1}(\varphi(z))\diff y\diff z.
        \end{split}
    \end{equation}
    Therefore, integrating \eqref{eq 0.4} with respect to $t\in [0,1]$, we get
    \begin{equation}
    \label{ineq 0.5}
    \begin{split}
        \int^{1}_{0}&\left(\iint_{F_{\delta}}\frac{|D_{\top}u(\varphi((1-t)x+ty))|^{p}}{|y-x|^{m-2}}J_{m-1}(\varphi((1-t)x+ty))\diff y\diff x\right) \diff t \\ & = \int^{1}_{0}\frac{1}{1-t}\iint_{F_{(1-t)\delta}}\frac{|D_{\top}u(\varphi(z))|^{p}}{|y-z|^{m-2}} J_{m-1}(\varphi(z))\diff y\diff z \diff t\\ & \leq \int_{E}|D_{\top}u|^{p}\diff \mathcal{H}^{m-1} \int^{1}_{0}\frac{1}{1-t}\int_{B^{m-1}_{(1-t)\delta}}\frac{1}{|\xi|^{m-2}}\diff \xi  \diff t\\ &\leq C\delta\int_{E}|D_{\top}u|^{p}\diff \mathcal{H}^{m-1}.       
        \end{split}
    \end{equation}
    Then, taking $\delta>0$ such that 
    \[
    \delta=\frac{\|u\|_{L^{p}(E)}}{\|D_{\top}u\|_{L^{p}(E)}},
    \]
    we obtain \eqref{main int estimate}  from \eqref{ineq parametrization}, \eqref{ineq 0.2}, \eqref{ineq 0.3} and \eqref{ineq 0.5}. 
    
    Next, we prove \eqref{est 2.8}. Let us now choose $\delta=\diam(U)/2$. Then, changing $\int_{\mathbb{R}^{m-1}\setminus B^{m-1}_{\delta}}\frac{\diff z}{|z|^{m-2+p}}$ to $\int_{B^{m-1}_{2\delta}\setminus B^{m-1}_{\delta}}\frac{\diff z}{|z|^{m-2+p}}$ and using that $\frac{1-(1/2)^{p-1}}{p-1}\leq \ln(2)$ in \eqref{ineq 0.2}, we obtain
    \begin{multline*}
        \iint_{W_{\delta}}\frac{|u(\varphi(y))-u(\varphi(x))|^{p}}{|\varphi(y)-\varphi(x)|^{m-2+p}}J_{m-1}(\varphi(y)) J_{m-1}(\varphi(x))\diff y  \diff x\\ \leq 2^{p}\ln(2)L^{2m-3+p}C\left(\frac{1}{\delta}\right)^{p-1}\int_{E}|u|^{p}\diff \mathcal{H}^{m-1},
    \end{multline*}
    which, together with \eqref{ineq 0.5}, implies \eqref{est 2.8}. This completes our proof of Lemma~\ref{interpol}.
\end{proof}
\begin{cor}\label{cor interpolation}
    Let $p\in (1,2]$. Then there exists a constant $C=C(m)>0$ such that for every $u\in W^{1,p}(\partial B^{m}_{r}(x_{0}), \mathbb{R}^{k})$,
    \begin{equation}\label{ineq linear}
        |u|^{p}_{W^{1-1/p,p}(\partial B^{m}_{r}(x_{0}))}\leq C r \|D_{\top}u\|^{p}_{L^{p}(\partial B^{m}_{r}(x_{0}))}.
    \end{equation}
    Furthermore, if $\mathcal{Y}$ is a compact Riemannian manifold, which is isometrically embedded into    $\mathbb{R}^{k}$, $p\in [p_{0},2]$ for some $p_{0}\in (1,2]$ and $u \in W^{1,p}(\partial B^{m}_{r}(x_{0}), \mathcal{Y})$, then there exists a constant $C=C(p_{0},m,\mathcal{Y})>0$ such that
    \begin{equation}\label{ineq nonlinear}
        |u|^{p}_{W^{1-1/p,p}(\partial B^{m}_{r}(x_{0}))} \leq C r^{\frac{m-1}{p}} \|D_{\top}u\|^{p-1}_{L^{p}(\partial B^{m}_{r}(x_{0}))}.
    \end{equation}
\end{cor}
\begin{proof}[Proof of Corollary~\ref{cor interpolation}] First, let us prove \eqref{ineq linear}. By scaling and translation,  it is enough to show that for each $u \in W^{1,p}(\mathbb{S}^{m-1}, \mathbb{R}^{k})$,
    \begin{equation}\label{est 2.16}
        |u|^{p}_{W^{1-1/p,p}(\mathbb{S}^{m-1})}\leq C \|D_{\top} u\|^{p}_{L^{p}(\mathbb{S}^{m-1})}.
    \end{equation}
    Recall that $m\geq 2$ and for each \(j \in \{1, \dotsc, m\}\), define  
    \(E_{2j - 1} = \{ x \in \mathbb{S}^{m - 1} : x_j < 3/4\}\) and 
    \(E_{2j} = \{ x \in \mathbb{S}^{m - 1} : x_j > -3/4\}\). For every \(x \in \mathbb{S}^{m - 1}\), there is at most one \(i \in \{1, \dotsc, 2m\}\) such that \(x \not \in E_i\).
    Thus, 
    \(\mathbb{S}^{m-1}\times \mathbb{S}^{m-1}=\bigcup_{i=1}^{2m}E_{i}\times E_{i}\). Moreover, the \(E_i\)'s are bilipschitz images of the unit ball \(B^{m -1}_1\).
    
    Hereinafter in this proof, $C$ will denote a positive constant that can depend only on $m$ and can be different from line to line.  Then, in view of \eqref{est 2.8} with $U=B^{m-1}_{1}$, we have
    \begin{equation}\label{est 2.17}
        |u|^{p}_{W^{1-1/p,p}(E_{i})}\leq C (\|u\|^{p}_{L^{p}(E_{i})}+\|D_{\top}u\|^{p}_{L^{p}(E_{i})}).
    \end{equation}
    Applying \eqref{est 2.17} with $v=u-\frac{1}{\mathcal{H}^{m-1}(E_{i})}\int_{E_{i}}u\diff \mathcal{H}^{m-1}$ and using the Poincaré inequality for $v$ (observe that the constant in the Poincaré inequality can be chosen independently of $p \in (1, 2]$), we get the estimate $|u|^{p}_{W^{1-1/p,p}(E_{i})}\leq C \|D_{\top}u\|^{p}_{L^{p}(E_{i})}$. Then, summing over $i$, we obtain \eqref{est 2.16}, which implies \eqref{ineq linear}.
    
    Next, we prove \eqref{ineq nonlinear}. To this end, we show that if $p \in [p_{0},2]$ and $u \in W^{1,p}(\partial B^{m}_{r}(x_{0}), \mathcal{Y})$, then
    \begin{equation}\label{ineq 0.8}
        |u|^{p}_{W^{1-1/p,p}(\partial B^{m}_{r}(x_{0}))}\leq C^{\prime} \|u\|_{L^{p}(\partial B^{m}_{r}(x_{0}))} \|D_{\top} u\|^{p-1}_{L^{p}(\partial B^{m}_{r}(x_{0}))},
    \end{equation}
    where $C^{\prime}=C^{\prime}(p_{0},m)>0$. Notice that if \eqref{ineq 0.8} would  hold, then it would be ``scale-invariant''. Thus, it is enough to prove \eqref{ineq 0.8} in the case when $\partial B^{m}_{r}(x_{0})=\mathbb{S}^{m-1}$. Applying Lemma~\ref{interpol} with $E=E_{i}$ as above for each $i \in \{1,\dotsc,2m\}$  and then summing over $i$, we get
    \begin{equation}\label{ineq 0.10}
        |u|^{p}_{W^{1-1/p,p}(\mathbb{S}^{m-1})}\leq \frac{C}{p_{0}-1} \|u\|_{L^{p}(\mathbb{S}^{m-1})} \|D_{\top}u\|_{L^{p}(\mathbb{S}^{m-1})}^{p-1}.
    \end{equation}
    Setting $C^{\prime}=C/(p_{0}-1)>0$, \eqref{ineq 0.10} yields \eqref{ineq 0.8}. The inequality~\eqref{ineq nonlinear} is a direct consequence of \eqref{ineq 0.8} and the fact that $\mathcal{Y}$ is a compact Riemannian manifold, which is isometrically embedded into $\mathbb{R}^{k}$. This completes our proof of Corollary~\ref{cor interpolation}.
\end{proof}

\bibliography{bib01}
\bibliographystyle{amsplain}
\end{document}